\documentclass[10pt]{article}
\usepackage{authblk}

\usepackage[a4paper,top=3cm,bottom=2cm,left=2.6cm,right=2.6cm,marginparwidth=1.75cm]{geometry}

\usepackage{amsmath}
\usepackage{amsthm}
\usepackage{amssymb}
\usepackage{verbatim}
\usepackage{mathtools}
\usepackage{enumitem}
\usepackage[colorlinks]{hyperref}
\usepackage{xcolor}
\usepackage{txfonts}
\theoremstyle{plain}% default
\newtheorem{theorem}{Theorem}[section]
\newtheorem{lemma}[theorem]{Lemma}
\newtheorem{proposition}[theorem]{Proposition}
\newtheorem{corollary}[theorem]{Corollary}
\newtheorem{definition}[theorem]{Definition}

\newtheorem{remark}[theorem]{Remark}
\newtheorem{assumption}[theorem]{Assumption}
\newcommand{\RRR}{\color{black}}
\newcommand{\BBB}{\color{black}}

\newcommand{\N}{\mathbb{N}}
\newcommand{\Z}{\mathbb{Z}}

\newcommand{\R}{\mathbb{R}}
\newcommand{\C}{\mathbb{C}}
\newcommand{\A}{\mathcal{A}}
\newcommand{\funcA}{\mathfrak{a}}
\newcommand{\funcB}{\mathfrak{b}}

\newcommand{\funcC}{\mathcal{C}}

\newcommand{\simgrad}{\sym\nabla}

\newcommand{\epsh}{{\varepsilon_h}}
\newcommand{\epshtwo}{{\varepsilon_h^2}}
\newcommand{\epshfour}{{\varepsilon_h^4}}
\newcommand{\eps}{{\varepsilon}}
\newcommand{\weak}{{\rightharpoonup}}
\newcommand{\weakL}{{\,\xrightharpoonup{L^2\,}\,}}

\newcommand{\drtwoscale}{{\,\xrightharpoonup{{\rm dr}-2\,}\,}}
\newcommand{\drtwoscalet}{{\,\xrightharpoonup{{\rm t,dr}-2\,}\,}}
\newcommand{\strongdrtwoscale}{{\,\xrightarrow{{\rm dr}-2\,}\,}}
\newcommand{\strongdrtwoscalet}{{\,\xrightarrow{{\rm t,dr}-2\,}\,}}
\DeclareMathOperator{\sym}{sym}
\DeclareMathOperator{\dist}{dist}

\DeclareMathOperator{\tr}{tr}
\DeclareMathOperator{\skeww}{skew}
\newcommand{\mat}[1]{\boldsymbol #1}
\newcommand{\vect}[1]{\boldsymbol #1}
\newcommand{\dd}{{\mathrm{d}}}

\let\oldsqrt\sqrt
\def\sqrt{\mathpalette\DHLhksqrt}
\def\DHLhksqrt#1#2{%
\setbox0=\hbox{$#1\oldsqrt{#2\,}$}\dimen0=\ht0
\advance\dimen0-0.2\ht0
\setbox2=\hbox{\vrule height\ht0 depth -\dimen0}%
{\box0\lower0.4pt\box2}}

\AtBeginDocument{
  \let\div\relax
  \DeclareMathOperator{\div}{div}
}

\begin{document}
	
	\title{\sc Spectral and evolution analysis of composite elastic plates with high contrast}

\def\correspondingauthor{\footnote{Corresponding author: k.cherednichenko@bath.ac.uk}}

\author[1]{Marin Bu\v{z}an\v{c}i\'{c}} 
\author[2]{Kirill Cherednichenko\correspondingauthor{}}
\author[3]{Igor Vel\v{c}i\'{c}}
\author[3]{Josip \v{Z}ubrini\'{c}}
\affil[1]{Faculty of Chemical Engineering and Technology, University of Zagreb, \linebreak Maruli\'{c}ev trg 19, 10000 Zagreb, Croatia}
\affil[2]{Department of Mathematical Sciences, University of Bath, \linebreak Claverton Down, Bath, BA2 7AY, United Kingdom}
\affil[3]{Faculty of Electical Engineering and Computing, University of Zagreb, \linebreak Unska 3, 10000 Zagreb, Croatia}
	%% Language and font encodings

\maketitle
\begin{abstract}
	We analyse the behaviour of thin composite plates whose material properties vary periodically in-plane and possess a high degree of contrast between the individual components. Starting from the equations of three-dimensional linear elasticity that describe soft inclusions embedded in a relatively stiff thin-plate matrix, we derive the corresponding asymptotically equivalent two-dimensional plate equations. Our approach is based on recent results concerning decomposition of deformations with bounded scaled symmetrised gradients. Using an operator-theoretic approach, we calculate the limit resolvent and analyse the associated limit spectrum and effective evolution equations. We obtain our results under various asymptotic relations between the size of the soft inclusions (equivalently, the period) and the plate thickness as well as under various scaling combinations between the contrast, spectrum, and time. In particular, we demonstrate significant qualitative differences between the asymptotic models obtained in different regimes.	
\end{abstract}
 	
\section{Introduction}
Derivation of limit models for thin structures in linear and non-linear elasticity is a well-established topic (for example, for the approach via formal asymptotics, see \cite{ciarlet2,ciarlet1} and references therein). 
As part of recent related activity, there appeared a number of works that derive models of (highly) heterogeneous thin structures by simultaneous homogenisation and dimension reduction,
%%, in both linear and non-linear settings 
see \cite{BukalVel,Gri20,Neukamm10,Neukamm13,Vel14a}; for the older work see also \cite{Caill}. The present paper aims at making a further contribution to this body of work, by addressing the derivation of effective models for thin plates with high-contrast inclusions in the context of spectral and evolution analysis (see also \cite{rohan2,rohan1}). Simultaneously with the above activity in relation to the analysis of thin structures, the past two decades have seen a growing interest to the analysis of materials with high-contrast inclusions (for early papers on this subject, see \cite{Zhikov2000, bouchitte,zhikov2005}) that exhibit frequency-dependent material properties (equivalently, time-nonlocal evolution), which is representative of what one may refer to as ``metamaterial" behaviour \cite{Phys_book}. Furthermore, as was recently discussed in \cite{chered1}, high contrast in material parameters corresponds to regimes of length-scale interactions, when parts of the medium exhibit resonant response to an external field. As a result, such composites possess macroscopic, or "effective", material properties not commonly found in nature, such as \RRR the time \BBB non-locality mentioned above (leading to "memory" effects) or negative refraction, which motivates their use in the context of electromagnetic or acoustic wave propagation for the development of novel devices with cloaking and superlensing properties. Due to the dependence of the effective parameters on frequency, the wave propagation spectrum of these materials has a characteristic band-gap structure (i.e., waves of some frequencies do not propagate through the material, see also \cite{avila08,smysh09}). 

There have been several works dealing with high-contrast inclusions in the context of elasticity: spectral analysis on bounded domains is given in \cite{avila08}, in the whole space in \cite{zhikov2012}, see also \cite{smysh09} for treating partial degeneracy (when ``directional localisation" takes place), for different models of high-contrast plates (where the starting equations are two-dimensional equations for an ``infinitely thin" elastic plate), see \cite{rohan2,rohan1}. In subsequent developments, \cite{cooper} deals with high-contrast inclusions with partial degeneracy, when only one of several material constants (namely, the shear modulus) is relatively small, \cite{chered3} discusses the limit spectrum of planar elastic frameworks made of rods and filled with a soft material, and \cite{chered4} derives an effective model for the case of high-contrast inclusions in the stiff matrix in the context of non-linear elasticity, under an assumption of small loads. In the more recent push towards a quantitative description of metamaterials, elliptic differential equations with high contrast have been analysed in the sense of approximating the associated resolvent with respect to the operator norm (see \cite{chered2,chered1}). In the related papers, using the Gelfand transform as a starting point, a new operator family was constructed that approximates the resolvent of the original one and that cannot be obtained directly from the standard limit operator inferred from the earlier qualitative analysis. However, these results are by now obtained only for the whole-space setting and for the particular case of the diffusion operator. In relation to quantifying the resolvent behaviour with respect to the operator norm, we should also mention \cite{krejcirik},  where the dimension reduction for a class of differential operators is carried out in the abstract setting (on a finite domain) and \cite{chered5}, where thin infinite elastic plates in moderate contrast are analysed. 

In terms of understanding the structure of two-scale limits of partial differential operators with high contrast, we refer to \cite{kamotski}, where an approach to spectral analysis and its consequences for materials with high-contrast inclusions (including partial degeneracies) on bounded domains is presented, via two-scale convergence. While addressing the description of the limit spectrum only partially, \cite{kamotski} provided a general framework for the analysis of the limit resolvent, on which new results concerning elasticity and other physically relevant setups could subsequently build. Finally, the subject of homogenisation of stochastic high-contrast media, which naturally follows the analysis of periodic setups, was recently initiated in \cite{cherd1} and further developed in \cite{cherd2}.

The main goal of the present paper is to study the behaviour of high-contrast finite plates from the qualitative point of view by dimension reduction from three-dimensional linearised elasticity (including an analysis of the limit spectrum and hyperbolic time evolution), identifying all possible asymptotic regimes, while putting the derivation of the corresponding limit models on the operator-theoretic footing and emphasising the peculiarities of individual scalings. In a future work we intend to carry out a quantitative analysis (in the sense of operator-norm resolvent approximation, in the whole-space setting) of the asymptotic regimes discussed here, using the techniques of \cite{chered1}. This will, in most cases, guarantee the convergence of spectra to the spectrum of the corresponding limit operator derived in the present paper.  However, for the regime when the period of material oscillations is much smaller than the thickness of the body, spectral pollution is shown to occur for the limit operator
%%which shows that
%%even in relation to the kind of limit model we obtain here, 
%%(which can be seen as an analogue of the model of \cite{Zhikov2000} obtained by two-scale convergence), 
suggested by the present work (i.e., the limiting spectrum has points outside the spectrum of the limit operator). This will make quantitative (i.e., operator-norm based) approaches ever more indispensable for sharp multiscale analysis.

In what follows, we assume that all elastic moduli of the soft component are of the same order (unlike in \cite{smysh09, cooper}). While we do not apply any additional scaling to either elastic moduli or the mass densities, we do discuss models obtained on different time scales in the context of hyperbolic evolution. Note that this kind of time scaling is indeed sometimes interpreted as a scaling of the mass density (see \cite{ciarlet1, Raoult}). %%%In contrast,  we choose to assume that the densities of the both materials (soft and stiff) are on the same scale and introduce a time scaling instead.

Dimension reduction in elasticity always requires a special treatment, due to the degeneracy of the problem as a consequence of the fact that the constant in the Korn's inequality blows up as the domain thickness goes to zero. From the point of view of spectral analysis, the operator of the associated problem on a rescaled domain of finite thickness has spectrum of order $h^2,$ where $h$ is the thickness of the original thin body, with the eigenfunctions describing out-of-plane (``bending") deformations. This can be physically interpreted as follows: bending deformation %(deformations in the out-of-plane direction) 
can be finite but with small energy, while the magnitude of the (in-plane) stretching deformations is scaled as the magnitude of their energy. Thus for finite plates there are two orders of spectrum: the spectrum of order $h^2$ and the spectrum of order one. (On an infinite plate there is no natural way to scale the spectrum, see \cite{chered5}.) As a result, in the evolution analysis one scales time (or mass density) accordingly, in order to see the evolution of the bending deformation. As a consequence of small, slow in time in-plane forces, in-plane motion is partially quasistatic (see Remark \ref{remev-1} below).

Furthermore, we show that in order to see in-plane motion with inertia one has to allow in-plane forces acting in the real (i.e., unscaled) time --- to the best of our knowledge, such models have not been discussed in the literature, even in the simplest case of a homogeneous plate, see Remark \ref{remev1}. In a model of this type, when we admit \RRR an \BBB out-of-plane force (of the same order of magnitude as the in-plane force components), in the limit we obtain a degenerate equation for the out-of-plane motion, since there is no elastic resistance to it (see Remark \ref{remev3}). 

The behaviour of plates with high-contrast inclusions also has its peculiarities. As demonstrated in the present work, the effect of the inclusions is not seen on the long timescale under the standard scaling of the corresponding elastic coefficients (i.e., the case of coefficients of order $\varepsilon^2$, where $\varepsilon$ is the dimensionless characteristic cell size), see Remark \ref{remev-1}. In order to obtain standard effects of high-contrast in the long-time evolution, one needs to scale the coefficients differently, see Section \ref{sectionltbhc1}, Section \ref{sectionltbhc2}, Remark \ref{remev4} and Remark \ref{remev5}. 

Adopting the operator-theoretic perspective, we start by deriving the limit resolvent in different scaling regimes. %In order to achieve that 
To that end, we combine suitable decompositions of deformations that have bounded symmetrised gradients with some special properties of two-scale convergence (see Appendix and the references therein). The limit resolvent is always degenerate (this is caused by the different scaling of (out-of-plane) bending and (in-plane) stretching deformations, as explained above). Moreover, as is usual in derivation of asymptotic models for thin structures, one needs to take care of \RRR possible different \BBB scalings of the applied \RRR loads \BBB. These are not necessary for the analysis of the limit spectrum (see Section \ref{limitspectrum}), but are essential if we want to analyse evolution equations with appropriate loads (see Section \ref{limeveqsec}). Here we obtain different models depending on the effective parameter $\delta \in [0,\infty]$, which is the limit ratio between the thickness of the domain $h$ and the period $\eps$ (equivalently, between $h$ and the size of the inclusions), where $\varepsilon$ tends to zero simultaneously with $h,$ as well as on two characteristic orders of spectrum $h^2$ and one (equivalently, time-scales of order $1/h^2$ and one). In order to obtain high-contrast effects for ``small'' spectrum (i.e., on the long time scale in evolution), we also treat a non-standard scaling of the coefficients of high-contrast inclusions. We emphasize that \RRR such \BBB scaling is different for the case $\delta=0$ (small inclusions behave like small plates) and the case $\delta > 0$ (for the former the the scaling is $\varepsilon^4,$ while for the latter it is $\varepsilon^2h^2$).   

In order to derive the limit spectrum, we employ elements of the approach of \cite{Zhikov2000, zhikov2005}. Surprisingly, in the regime $\delta=\infty$ the limit spectrum does not coincide with the spectrum of the limit operator, which necessitates additional analysis (see Section \ref{sectionspsur} and Remark \ref{remivan101}). This, however, is not specific for elasticity and would also happen if one carried out simultaneous high-contrast homogenisation and dimension reduction for the diffusion equation. 

Suitably adapting the approach of \cite{pastukhova} to dimension reduction in linear elasticity (see Appendix for details), we use our results on resolvent convergence to derive appropriate limit evolution equations. To infer weak convergence of solutions from the weak convergence of initial conditions and loads, we use the fact that the resolvent is the Laplace transform of the evolution operator, while for deriving strong convergence of solutions for all times $t$ (from the strong convergence of initial conditions and loads), one needs to show the strong convergence of exponential functions on the basis of the strong convergence of resolvents. Both these implications are analysed in \cite{pastukhova} in an abstract form, which guides our study in the specific context of dimension reduction. 

Note that, due to the presence of the discrete spectrum of order $h^2,$ an order one scaling of the operator leads to the limit spectrum covering the whole positive real line (see \cite{castro1}). For this reason,  in the case when the coefficients on the soft inclusions are of order $\varepsilon^2$ (where $\varepsilon,$ as before, represents the dimensionless cell size) we separate the subspaces of bending and stretching deformations.
 %%live and subspace where stretching deformations live. 
 This can be done only under additional symmetry assumptions on the elasticity tensor (see Section \ref{secmembsp}). The structure of the limit spectrum for the operator acting in the space of stretching deformations is similar to that of the limit spectrum in high-contrast three-dimensional elasticity (see \cite{avila08,zhikov2012}), where the ``Zhikov function" describing the limit frequency dispersion is matrix-valued (see Remark \ref{remlsm}) in the situations when $\delta<\infty$. For the case of coefficients of order $\varepsilon^4$ and $\varepsilon^2h^2$, which correspond to $\delta=0$ and $\delta \in (0,\infty)$ respectively, the structure of the limit spectrum of order $h^2$ is again similar to that of standard high-contrast materials (see Section \ref{sectionbgl}). However, the Zhikov function is scalar in these cases.  The spectral analysis of the case $\delta=\infty$, as already mentioned, requires special treatment (and uses some of ideas presented in \cite{AllaireConca98}).  

In what follows, we first present the results (effective tensors, limit resolvent,  limit spectrum, limit evolution equations in different regimes, see Sections \ref{notation}, \ref{main_results}) and then provide the proofs of all statements, see Section \ref{proofs}. In the Appendix we give some useful claims about decomposition of displacements with bounded scaled symmetric gradients, two-scale convergence, extension operators and operator theoretical approach to high-contrast.

\section{Notation and setup}
\label{notation} 
Let $\omega\subset\R^2$ be a bounded Lipschitz domain and consider the open interval $I =(-1/2, 1/2)\subset \R.$ Given a small positive number $h>0$, we define a three-dimensional plate
\[
\Omega^h := \omega \times (h I),
\]
whose boundary consists of the lateral surface $\Gamma^h:=\partial\omega \times (h I)$ and the transverse boundary %%$\Gamma^h := 
$\omega \times \partial(h I)$.
We assume that the part of the boundary of $\Omega^h$ on which the Dirichlet (zero-displacement) boundary condition is set has the form $\Gamma_{\rm D}^h: = \gamma_{\rm D} \times (hI)\subset\Gamma^h,$
%% \subset \partial\omega \times (hI)$, 
where $\gamma_{\rm D}\subset\omega$ has positive (one-dimensional) measure. 
%It is also possible to analyse a non-zero Dirichlet boundary condition, but this would require imposing additional symmetry conditions in those parts of our study where we use symmetries of the solution. Since the related generalisation is straightforward, we omit it. 

For a vector $\vect a\in \R^k,$ we denote by $a_j,$ $j=1, \dots, k,$ its components, so 
$\vect a=(a_1,\dots, a_k).$ Similarly, the entries of a matrix $\vect A\in \R^{k\times k},$ are referred to as $A_{ij},$ $i,j=1, \dots, k.$  We denote by $x = (x_1, x_2, x_3) =: (\hat{x}, x_3)$ the standard Euclidean coordinates in $\mathbb{R}^3.$ (Note that we reserve the boldface for vectors and matrices representing elastic displacements and their gradients and regular type for coordinate vectors in the corresponding reference domains.) The unit basis vectors in $\R^k$ are denoted by $\vect e_i$, $i=1,\dots,k$. Furthermore, for $\vect a,\vect b \in \R^k$ we denote by $\vect a \otimes \vect b \in \R^{k \times k}$, the matrix whose $ij$-entry is $a_ib_j:$
$$ 
\vect a \otimes \vect b=\{a_i b_j\}_{ij=1}^k.
$$
%%, \quad i,j=1, \dots,k. $$
For $\vect A \in \R^{k \times l }$, by $\vect A^\top$ we denote its transpose and for the case $k=l$ we denote by $\sym\vect A=(\vect A+{\vect A}^\top)/2$ the ``symmetrisation" of $\vect A.$

For an operator ${\mathcal A}$ (or a bilinear form $a$) the domain of ${\mathcal A}$ (respectively $a$)  is denoted by ${\mathcal D}({\mathcal A})$ (respectively ${\mathcal D}(a)$). 

Throughout the paper, we use the notation $\varepsilon_h$ interchangeably with $\varepsilon,$ to emphasize the fact that $\varepsilon$ goes to zero simultaneously with $h.$ 

Furthermore, when indicating a function space $X$ in the notation for a norm $\Vert\cdot\Vert_X,$ we omit the physical domain on which functions in $X$ are defined whenever it is clear from the context. For example, we often write $\Vert\cdot\Vert_{L^2},$ $\Vert\cdot\Vert_{H^1}$  instead of $\Vert\cdot\Vert_{L^2(\Omega;{\mathbb R}^k)},$ $\Vert\cdot\Vert_{H^1(\Omega;{\mathbb R}^k)},$ $k=2,3.$

Finally, we use the label $C$ for all constants present in estimates for functions in various sets. In such cases $C$ can be shown to admit some positive value independent of the function being estimated.  

\subsection{Differential operators of linear elasticity }

Consider the reference cell $Y:= [0, 1)^2.$ Let $Y_0 \subset Y$ be an open set with Lipschitz boundary (unless otherwise stated) such that its closure is a subset of the interior of $Y,$ and set $Y_1 = Y \setminus Y_0$.  We denote by $\chi_{Y_0}$ the characteristic function of $Y_0$ and by $\chi_{Y_1}$ the characteristic function of $Y_1$.  For any subset of $A \subset \R^k$, we denote by $\chi_A$ the characteristic function of the set $A$. 
The domain $\Omega^h$ is then divided into two subdomains $\Omega^{h,\epsh}_0$ and $\Omega^{h,\epsh}_1$:
\begin{eqnarray*}
\Omega^{h,\epsh}_0 :=  \bigcup_{z \in \Z^2:\epsh(Y+z)\subset \omega}\bigl\{\varepsilon_h(Y_0 + z)\times hI\bigr\}, \quad\qquad 
\Omega^{h,\epsh}_1 := \Omega^h\backslash \Omega_0^{h,\epsh}. 
\end{eqnarray*}
%Equivalently:
%\[
%\Omega^\epsh_0 := \left\{ x \in \Omega^h : %\chi_{Y_0}\left(\tfrac{\hat{x}}{\eps}\right) = 1 \right\},
%\quad
%\Omega^\epsh_1 := \left\{ x \in \Omega^h : %\chi_{Y_1}\left(\tfrac{\hat{x}}{\eps}\right) = 1 \right\},
%\]
Furthermore, we denote 
\begin{equation*}
\Omega^\epsh_0 := \Omega^{1,\epsh}_0, \qquad \Omega^\epsh_1:= \Omega^{1,\epsh}_1.
\end{equation*}
By $\rho^{h, \epsh}$ we denote function representing the mass density of the medium. We then define 
$$\rho^{h,\epsh}(x)=\rho_0\left({\hat{x}}/{\epsh}\right)\chi_{\Omega_0^{h,\epsh}}+\rho_1\left({\hat{x}}/{\epsh}\right)\chi_{\Omega_1^{h,\epsh}}, \qquad x \in \Omega^h,$$
where $\rho_0,\rho_1$ are periodic positive bounded functions, defined on $Y_0$ and $Y_1$ respectively and extended via periodicity. Namely, there exist $c_1,c_2>0$ such that 
$$ c_1<\rho_0(y)<c_2\quad \forall y\in Y_0 , \quad\qquad c_1<\rho_1(y)<c_2\quad \forall y\in Y_1.$$
We also denote $\rho:=\rho_0 \chi_{Y_0}+\rho_1\chi_{Y_1}$,  $\rho^{\epsh}:=\rho^{1,\epsh}$. 
 We make use of the variational space with zero Dirichlet boundary conditions, defined as:
\[
H^1_{\Gamma_{\rm D}^h}(\Omega^h,\R^3) := \left\{\vect v \in H^1(\Omega^h;\R^3) :\vect v = 0 \ \mbox{on} \ \Gamma^h_{\rm D}\right\}.
\]
The elastic properties of periodically heterogeneous material are stored in the elasticity tensor $\C^{\mu_h}$, which is assumed to be of the form:
\begin{equation}
\C^{\mu_h}(y) =
\left\{\begin{array}{ll}
\C_1(y), & y \in Y_1, \\[0.4em]
\mu_h^2\, \C_0(y), & y \in Y_0.
\end{array} \right. 
\label{coeff_def}
\end{equation}
where $\mu_h$ is a parameter that goes to zero simultaneously with $h,\varepsilon_h$. The tensor $\C^{\mu_h}$ is then extended to $\R^2$ via $Y$-periodicity.  
The tensors $\C_0$ and $\C_1$ are assumed to be uniformly positive definite on symmetric matrices, namely there exists
$\nu>0$ such that
\begin{equation} \label{coercivity} 
\nu|\xi|^2 \leq \C_{0,1}(y)\xi : \xi \leq \nu^{-1}|\xi|^2 \qquad \forall \xi \in \R^{3\times 3},\ \xi^\top = \xi. 
\end{equation} 
It is well known that for a hyperelastic material the following symmetries hold, which we assume henceforth:
$$\C_{\alpha,ijkl}=\C_{\alpha,jikl}=\C_{\alpha,klij},\quad i,j,k,l\in\{1,2,3\},\qquad  \alpha \in \{0,1\}. $$

The focus of our analysis is the differential operator of linear elasticity $\A_\epsh^h$ corresponding to the differential expression
%\begin{eqnarray*}
\[
%%\mathcal{A}^h_\epsh:=
-\bigl(\rho^{h,\epsh}\bigr)^{-1}\div\left(\C^{\mu_h}({\hat{x}}/{\epsh}) \simgrad\right), 
%%\qquad %%{\color{red}\mathcal{D}(\mathcal{A}^h_\epsh)}\subset  %%H_{\Gamma_{\rm D}^h}^1(\Omega^h;\R^3).
\] 
It is defined as an unbounded operator in $L^2(\Omega^h,\R^3)$ (where the inner product is weighted\footnote{$(\vect u, \vect v)_{\epsh}:=\int_{\Omega^h}\rho^{h,\epsh} \vect u \vect v $. } by the mass density function $\rho^{h,\epsh}$) with domain
\[
\mathcal{D}(\mathcal{A}^h_\epsh)\subset  H_{\Gamma_{\rm D}^h}^1(\Omega^h;\R^3),
\]
via the bilinear form
\[
%%\\[0.7em] \nonumber
a^h_\epsh( \vect U, \vect V):=\int\limits_{\Omega^h} \C^\epsh\biggl(\dfrac{\hat{x}}{\epsh}\biggr) \simgrad \vect U(x) : \simgrad \vect V(x) \,dx, \qquad  \vect U, \vect V\in\mathcal{D}(a_\epsh^h)=H^1_{\Gamma^h_{\rm D}}(\Omega^h;\R^3)=\mathcal{D}\bigl((\mathcal{A}^h_\epsh)^{1/2}\bigr),
\]
%\end{eqnarray*}
with zero Dirichlet boundary condition on the part $\Gamma_{\rm D}^h$ of the boundary, which corresponds to the partially clamped case. For a given pair $(h, \varepsilon_h)$ we denote by $\vect U^\epsh$ any ``deformation field" on \RRR$\Omega^h,$ \BBB i.e., the solution to the integral identity 
\[
a^h_\epsh({\vect U}^\epsh, \vect V)=\int_{\Omega^h}{\vect F}(x)\cdot{\vect V}(x)dx\qquad\forall {\vect V}\in H^1_{\Gamma^h_{\rm D}}(\Omega^h;\R^3),
\]
for some ${\vect F}\in L^2(\Omega^h;\R^3).$
%\BBB

We assume that the following limits for the ratio of the period $\epsh$ and the thickness $h$ exist:
\begin{equation}
\lim\limits_{h \to 0} \frac{h}{\epsh} =:\delta \in [0,\infty], \qquad \lim\limits_{h \to 0} \frac{h}{{\epshtwo}} =:\kappa \in [0,\infty]
\label{delta_kappa}
\end{equation}
and will discuss different asymptotic regimes in terms of the values of $\delta,$ $\kappa.$

The asymptotic regime $\mu_h = O(1)$ corresponds to the standard case of moderate-contrast (i.e., uniformly elliptic) homogenisation. However, in the present paper we are interested in the ``critical" case  $\mu_h = \epsh$, which corresponds to high contrast in material coefficients. In addition to this, due to the the dimension reduction in elasticity, higher orders of contrast will also be of interest, namely $\mu_h = \epsh h$  for $\delta>0$ and $\mu_h={\epshtwo}$ for $\delta=0,$ see the table in Section \ref{table_sec}. 

Parts of the following assumption will be used occasionally to \RRR showcase \BBB special situations.
\begin{assumption}\label{assumivan1}  
	%%The following material symmetries are defined.  
	\begin{enumerate}
		\renewcommand\labelenumi{(\theenumi)} 
		\item 
		The elasticity tensor is planar symmetric: 
		\begin{equation*}
		\C_{\alpha,ijk3} = 0, \C_{\alpha,i333} = 0, %%\text{ for }  
		\quad i,j,k \in \{1,2\},\qquad \alpha \in \{0,1\}.
		\end{equation*}
		\item The inclusion set $Y_0$ has a ``centre point" $y^0=(y^0_1,y^0_2)\in Y_0$, such that $Y_0$ is symmetric with respect to the lines $y_1=y^0_1$, $y_2=y^0_2$. We also assume that the elasticity tensor $y \mapsto \C_0(y)$ and density $y \mapsto \rho_0(y)$ are invariant under the corresponding symmetry transformations. 
		\item The inclusion set $Y_0$ is invariant under the rotations with respect to the angle $\pi/2$ around the point $(y^0_1,y^0_2)$. Additionally, assume that the following material symmetries hold:
		\begin{equation*}   
		\C_{0,11ij}=\C_{0,22ij}, \quad \C_{0,12kk}=0,\quad i,j,k\in\{1,2,3\},
		\end{equation*} 
		and that the function $y \mapsto \rho_0(y)$ is symmetric with respect to the rotation through $\pi/2$ around the point $(y^0_1,y^0_2)$. 	
\end{enumerate} 
\end{assumption}
We define the following subspaces of $L^2(\Omega^h;\R^3)$:
\begin{equation*}
\begin{aligned}
&L^{2, \rm bend}(\Omega^h;\R^3) := \left\{\vect V = (V_1,  V_2,  V_3) \in L^2(\Omega^h;\R^3); \quad  V_1,  V_2 \mbox{ are odd  w.r.t. } x_3,\quad  V_3 \mbox{ is even  w.r.t. } x_3 \right\}, \\[0.3em]
&L^{2, \rm memb}(\Omega^h;\R^3) := \left\{\vect V = (V_1, V_2, V_3) \in L^2(\Omega^h;\R^3); \quad  V_1,  V_2 \mbox{ are even  w.r.t. } x_3,\quad  V_3 \mbox{ is odd  w.r.t. } x_3 \right\}.
\end{aligned}
\end{equation*}
Similarly, we define $L^{2, \rm bend}(\Omega \times Y;\R^3)$, $L^{2, \rm memb}(\Omega \times Y;\R^3)$, $L^{2, \rm bend}(I \times Y_0;\R^3)$, $L^{2, \rm memb}(I \times Y_0;\R^3)$.  
\begin{remark} 
	\label{inv_remark}
	Part (1) of Assumption \ref{assumivan1} is needed to infer that the spaces $L^{2, \rm bend}(\Omega^h;\R^3)$, $L^{2, \rm memb}$ $(\Omega^h; \R^3)$ are invariant for the operator $\mathcal{A}^h_{\epsh}$.  
	%The second part of the Assumption \ref{assumivan1} 
	Part (2) of the same assumption will additionally be used when we want to infer that the values of the Zhikov function $\beta,$ see (\ref{razgovor1}), are diagonal matrices, %%{\color{red}(in the case 2),}
	and part (3) will be used in combination with parts (1) and (2) when we want to infer that the $(1,1)$ and $(2,2)$ entries of the Zhikov function are equal.
	%% {\color{red}(in the case 3).}
	Although we do not assume the dependence on the $x_3$ variable, our analysis can be easily extended to this case (at the expense of obtaining more complex limit equations in some cases). In the case of planar symmetries, a natural assumption would be that the elasticity tensor is even in the $x_3$ variable. 
\end{remark}

In order to work in a fixed domain $\Omega := \Omega^1$,  $\Gamma:=\Gamma^1,$ $\Gamma_{\rm D}:= \Gamma_{\rm D}^1$, we apply the change of variables $$(x_1,x_2,x_3) := (x_1^h,x_2^h, h^{-1}x^h_3), \quad (x_1^h,x_2^h,x^h_3)\in \Omega^h,$$ and define $\vect u^\epsh(x) :={\vect U}^\epsh(x^h)$. In the new variables, we will be dealing with a scaled symmetrized gradient and scaled divergence, given by
\[
\simgrad \vect U^\epsh(x^h) = \simgrad_h \vect u^\epsh(x),\qquad  \div \vect U^\epsh(x^h)=\tr \nabla_h \vect u^{\epsh} (x)=:\div_h \vect u^{\epsh}(x),   
\]
where for a given function ${\vect u}$ we use the notation $\nabla_h \vect u := \left(\nabla_{\hat{x}} \vect u |\,h^{-1}\partial_{x_3} \vect u\right)$ for the gradient scaled ``transversally", and $\tr$ denotes the trace of a matrix. 
Thus, we are dealing with an operator $\mathcal{A}_\epsh$ in $L^2(\Omega;\R^3)$ (where the inner product is defined with the weight function $\rho^\epsh$) whose differential expression and domain are given by 
%%defined through the bilinear form:
%\begin{eqnarray*}
%%\mathcal{A}_\epsh \vect u &:=& 
\[
 -\bigl(\rho^{\epsh}\bigr)^{-1}\div_h\left(\C^\epsh({\hat{x}}/{\epsh}) \simgrad_h\right), \qquad\quad  \mathcal{D}(\mathcal{A}_\epsh) \subset H^1_{\Gamma_{\rm D}}(\Omega;\R^3),
 \]
respectively. The operator $\mathcal{A}_\epsh$ is defined by the form 
 %%\\ \nonumber
\[
a_\epsh( \vect u, \vect v):= 
%&:=& 
\int\limits_{\Omega} \C^\epsh\biggl(\dfrac{\hat{x}}{\epsh}\biggr) \simgrad_h \vect u(x) : \simgrad_h \vect v(x) \,dx, \qquad\quad \vect u, \vect v\in\mathcal{D}(a_\epsh)=H^1_{\Gamma_{\rm D}}(\Omega;\R^3) = \mathcal{D}\bigl(\mathcal{A}_\epsh^{1/2}\bigr). 
%\end{eqnarray*}
\]

As in Remark \ref{inv_remark}, under Assumption \ref{assumivan1}\,(1) the spaces $L^{2, \rm bend}(\Omega;\R^3)$, $L^{2, \rm memb}(\Omega;\R^3)$ are invariant for the operator $\mathcal{A}_{\epsh}$. 
\begin{comment} 
The link between these two objects is the following:
\begin{equation*}
a_\epsh(\vect u,\vect v) = \langle \mathcal{A}_\epsh \vect u, \vect v \rangle,\quad  \forall \vect u,\vect v \in \mathcal{D}(\mathcal{A}_\epsh),
\end{equation*}
\begin{equation*}
a_\epsh(\vect u,\vect v) = \langle \mathcal{A}_\epsh^{1/2}\vect u, \mathcal{A}_\epsh^{1/2} \vect v \rangle,\quad  \forall \vect u,\vect v \in \mathcal{D}(\mathcal{A}_\epsh^{1/2}).
\end{equation*}
\end{comment} 
We will also say that the operator $\mathcal{A}_{\epsh}$ represents the bilinear form $a_{\epsh}$ (a symmetric bilinear form defines a self-adjoint densely defined unbounded operator, see, e.g., \cite{Schmudgen}). 
In connection with $\mathcal{A}_{\epsh}$ we define the operator $\tilde{\mathcal{A}}_{\epsh}$ as the restriction of $\mathcal{A}_{\epsh}$  onto the space $L^{2, \rm memb}(\Omega;\R^3)$.  Additionally, we define the  self-adjoint operators $\mathring{\mathcal{A}}_{\epsh}$  in $L^2(I \times Y_0;\R^3)$ whose differential expression and domain are given by
$$ 
%%\mathring{\mathcal{A}}_{\epsh} \vect u:=
-\rho_0^{-1}\div_{\frac{h}{\epsh}}\left(\C_0(y) \sym  \nabla_{\frac{h}{\epsh}}\right) ,\qquad\quad\mathcal{D} (\mathring{\mathcal{A}_{\epsh}})\subset H^1_{00}(I \times Y_0;\R^3), $$
as the operators represented by the respective bilinear forms 
$$ \mathring{a}_{\epsh} (\vect u, \vect v)= \int_{I \times Y_0}\C_0(y) \sym \nabla_{\frac{h}{\epsh}} \vect u:\sym \nabla_{\frac{h}{\epsh}} \vect v dx_3 dy, \qquad \vect u, \vect v\in{\mathcal D}(\mathring{a}_{\epsh})= \bigl(H_{00}^1(I \times Y_0;\R^3)\bigr)^2,
%%\times H_{00}^1(I \times Y_0;\R^3),
$$
%%\to \R ,$$
where $H_{00}^{1} (I \times Y_0;\R^k)$ stands for the subspace of $H^1(I \times Y_0;\R^k)$ consisting of functions with zero trace on $I \times \partial Y_0$.
% We also define the operator
Finally, we define  $\mathring{\tilde{\mathcal{A}}}_{\epsh}$ as the operator corresponding to the same differential expression as $\mathring{\mathcal{A}}_{\epsh}$ but acting in the space $L^{2, \rm memb}(I \times Y_0;\R^3),$ hence representing an appropriate bilinear form  
\[
\mathring{\tilde{a}}_{\epsh}:  \left(H_{00}^1(I \times Y_0;\R^3)\right)^2 \cap \left( L^{2, \rm memb} (I \times Y_0;\R^3)\right)^2 \to \R.
\]

\subsection{Additional notation}

%%For $x,y \in \R^n$, 
%%by $(x,y)$ we denote 
The inner product of $x,y \in \R^n$ is denoted by $(x, y):= \sum_{i=1}^n x_iy_i$. For a function $ f \in L^1(A)$ (and similarly for $\vect f \in L^1(A;\R^3)$), we denote by 
\begin{equation}
\fint_A  f:=\frac{1}{|A|}\int_A  f,
\label{meanf}
\end{equation}
 its mean over $A.$ We will also use the shorthand notation 
%%(also extended to vector-valued functions)
\begin{equation}
\overline{ f}:=\int_I  f(x_3) dx_3,\qquad \langle  f \rangle := \int\limits_{Y}  f(y) \,dy, \qquad \vect f_{\!*}  := 
\left(\begin{array}{c}
\vect f_1 \\
\vect f_2
\end{array}\right),
\label{aux_notation}
\end{equation}
where in the last expression it is assumed that $\vect f$ is a (three-component) vector-valued function. In line with (\ref{meanf}), the notation $\overline{f}$ and $\langle f\rangle$ is naturally extended to vector-valued functions. 

Next, denote by $\iota$ the ``embedding'' operator
\begin{equation}
\iota:\R^{2\times 2} \to \R^{3\times 3}, \qquad \iota
\begin{pmatrix}
 a_{11}  &  a_{12} \\[0.2em]  a_{21}  &  a_{22}
\end{pmatrix}:= 
\begin{pmatrix}
 a_{11} &  a_{12} & 0 \\[0.2em]  a_{21} &  a_{22}  & 0 \\[0.2em] 0 & 0 & 0
\end{pmatrix}.
\label{iota_one}
\end{equation}
Similarly, we define an operator $\iota:\R^{3\times 2}\to \R^{3 \times 3}.$ We use the same notation for this operator and the operator defined in (\ref{iota_one}),
as it will be clear from the context which of the two embeddings is used in each particular case. 
For $\vect a \in \R^3$ we denote by $\iota_1$ the mapping
%% $\iota_1:\R^3 \to \R^{3 \times 3}$ given by 
$$ \iota_1:\R^3 \to \R^{3 \times 3},\qquad \iota_1(\vect a)= \begin{pmatrix}\begin{matrix}0 \end{matrix} & \begin{matrix}  a_1  \\  a_2 \end{matrix}  \\ \begin{matrix}  a_1 &  a_2 \end{matrix} &  a_3\end{pmatrix}. $$
Furthermore, for $l>0$ we define the ``scaling" matrix
\begin{equation*}
\pi_{l} :=  \begin{pmatrix}
l & 0 & 0 \\[0.2em] 0 & l & 0 \\[0.2em]0  & 0 & 1 \\
\end{pmatrix}.
\end{equation*}
We also define the space $H^1_{\gamma_{\rm D}}(\omega; \R^2)$ of ${\mathbb R}^2$-valued $H^1$ functions vanishing on $\gamma_{\rm D}$  and the space $H^2_{\gamma_{\rm D}}(\omega)$ of scalar $H^2$ functions vanishing on $\gamma_{\rm D}$ together with their first derivatives.

In what follows, we denote by $\mathcal{Y}$ the flat unit torus in $\R^2$, by $\mathcal{Y}_1$ the flat unit torus in $\R^2$ with a hole corresponding to the set $Y_1,$ 
by $\R^{n \times n}_{\sym}$ the space of symmetric matrices, $\R^{n \times n}_{\skeww}$ the space of skew-symmetric matrices,  by $\vect  I_{n \times n}$ the unit matrix in $\R^{n \times n}$, and by  
$\delta_{\alpha \beta}$ the Kronecker delta function. 
Furthermore, $H^1(\mathcal{Y})$, $H^2(\mathcal{Y}) $  denote the spaces of periodic functions in $H^1(Y)$, $H^2(Y).$ Similarly, we denote by  
$H^1(I \times  \mathcal{Y})$ the space of functions in $H^1(I \times  Y)$ that are periodic in $y\in Y.$ 
The spaces $\dot{H}^1(\mathcal{Y})$,  $\dot{H}^1(I \times \mathcal{Y} )$ are defined to consist of functions in $H^1(\mathcal{Y})$, $H^1(I \times \mathcal{Y})$ whose mean value is zero.    
Similarly, we define the spaces $H^k(\mathcal{Y}_1)$ for $k=1,2$. 
 Note that every function in $H_{00}^1(I \times Y_0;\mathbb{R}^k)$  can be naturally extended by zero to a function in $H^1(I \times \mathcal{Y};\R^k)$.

The space $C^{k}(\mathcal{Y})$
denotes the space of smooth functions on the torus $\mathcal{Y}$ that have continuous derivatives up to order $k.$  In a similar way we define the space $C^k(I \times \mathcal{Y})$. Furthermore, $C^k_{00}(I \times Y_0)$ denotes the space of $k$-differentiable functions on $I\times Y_0$ whose derivatives up to order $k$ are zero on $I \times \partial Y_0$. 
For $A \subset \R^n$, the space $C^k_{\rm c}(A)$ consists of functions with compact support in $A$ that have continuous derivatives up to order $k.$
%In order to represent the functions supported in the soft component of $Y$ we use the following space
%\begin{equation*}
%    \mathcal{X}:= \{ \vect \phi \in H^1(\mathcal{Y};\R^3)); \vect %\phi_{|Y_1} = 0 \}.
%\end{equation*}
%

For a function $\vect u \in H^1(I \times Y ;\R^3),$ we use the notation $\widetilde{\nabla}_{\delta}$ for the ``anisotropically scaled" gradient whose third column is obtained from the usual gradient by scaling with $\delta^{-1}:$
$$
 \widetilde{\nabla}_{\delta} \vect u := \left(\nabla_y \vect u |\,\delta^{-1}\partial_{x_3} \vect u\right).
$$ 
Next, for  $\vect{\varphi} \in L^2(\omega; H^1(I \times \mathcal{Y};\R^3))$, we denote 
$$ 
\funcC_{\delta}(\vect \varphi)= \sym \widetilde{\nabla}_{\delta} \vect \varphi, 
$$
\begin{comment} 
The crucial component of the analysis of high contrast media is the operator $\mathcal{A}_0$, referred to as Bloch operator, associated with the bilinear form:
\begin{equation*}
\begin{split}
& a_0( \vect u, \vect v) :=  \int\limits_{\Omega}\int\limits_{Y} \C_0(y) \sym\widetilde{\nabla}_{2,\gamma}\, \vect u(x,y) : \sym\widetilde{\nabla}_{2,\gamma}\, \vect v(x,y) \,dy\,dx, \\
&  \quad a_0 : L^2(\omega;H^1(I;\mathcal{X}))) \times L^2(\omega;H^1(I;\mathcal{X}))) \to \R,
\end{split}
\end{equation*}
namely:
\begin{equation*}
\mathcal{A}_0 := -\widetilde{\div}_{2,\gamma}(\sym \widetilde{\nabla}_{2,\gamma}(\cdot)), \quad  \mathcal{A}_0 : L^2(\omega;H^1(I;\mathcal{X})) \to L^2(\Omega \times Y;\R^3).
\end{equation*}
\end{comment} 
and for $\vect{\varphi}_1 \in L^2(\omega;H^1(\mathcal{Y};\R^2))$, $\varphi_2 \in L^2(\omega;H^2(\mathcal{Y}))$, $\vect g \in L^2 (\Omega \times Y; \R^3)$, we use the notation  
\begin{equation} \funcC_0(\vect{\varphi}_1,\varphi_2,\vect g)(x,y):= \begin{pmatrix}\begin{matrix} \sym \nabla_y \vect{\varphi}_1(\hat{x},y)-x_3 \nabla^2_y \varphi_2  (\hat{x},y) \end{matrix} & \begin{matrix} g_1 (x,y) \\[0.3em] g_2(x,y) \end{matrix}  \\[1em] \begin{matrix} g_1 (x,y)&  g_2(x,y) \end{matrix} & g_3(x,y) \end{pmatrix}. 
	\label{C_0_def}
\end{equation}	
Furthermore, for $\vect{w} \in L^2(\Omega;\dot{H}^1(\mathcal{Y};\R^3))$, $\vect g \in L^2 (\Omega; \R^3)$, we define
$$ \funcC_{\infty}(\vect w,\vect g)(x,y):= \begin{pmatrix}\begin{matrix} \sym \nabla_y \vect{w}_*  (x,y) \end{matrix} & \begin{matrix}  g_1 (x)+\partial_{y_1}  w_3(x,y) \\[0.3em] g_2(x)+ \partial_{y_2} w_3(x,y) \end{matrix}  \\[1.0em] \begin{matrix}  g_1 (x)+\partial_{y_1} w_3(x,y)&  g_2(x)+\partial_{y_2} w_3(x,y) \end{matrix} &  g_3(x) \end{pmatrix}, 
$$
where $\vect{w}_*$ is defined via (\ref{aux_notation}).

For different values of $\delta,\kappa$, we introduce the spaces
%\begin{eqnarray*} 
\begin{align*}
{\mathfrak C}_{\delta}(\Omega \times Y)&:=\begin{cases} 
\left\{\funcC_{\delta} (\vect \varphi): \vect \varphi \in L^2(\omega; H^1(I \times \mathcal{Y};\R^3))\right\}, \qquad 
%%\textrm{ for }
 \delta \in (0,\infty),\\[0.35em]
\left \{\funcC_0(\vect{\varphi}_1,\varphi_2,\vect g): \vect{\varphi}_1 \in L^2(\omega; H^1(\mathcal{Y};\R^2)), \varphi_2 \in L^2(\omega; H^2(\mathcal{Y})), \vect g \in L^2 (\Omega \times Y; \R^3)\right\},\ \ \ 
%%\textrm{ for } 
\delta=0, \\[0.35em]
\left\{\funcC_{\infty}(\vect w,\vect g): \vect{w} \in L^2(\Omega;\dot{H}^1(\mathcal{Y};\R^3)), \vect g \in L^2 (\Omega \times Y; \R^3)\right\}, \qquad 
%%\textrm{ for } 
\delta=\infty; 
\end{cases} 
\\[0.7em]
V_{1,\delta,\kappa} (\omega \times Y)&:=\begin{cases} 
H^1_{\gamma_{\rm D}} (\omega;\R^2)\times L^2(\omega), \qquad
%% \textrm{for } 
\delta \in [0,\infty],\   
%%\textrm{ or }  \delta=0,\ \ 
\kappa=\infty, \\[0.35em]
%H^1_{\gamma_{\rm D}} (\omega;\R^2)\times L^2(\omega), \qquad 
%\textrm{for }
% \delta=\infty, \\[0.35em]
%H^1_{\gamma_{\rm D}} (\omega;\R^2)\times L^2(\omega), \qquad
% \textrm{for }
% \delta=0,\ \ \kappa=\infty,\\[0.35em]
H^1_{\gamma_{\rm D}} (\omega;\R^2)\times  L^2(\omega;H^2(\mathcal{Y}_1)\times L^2(Y_0)),\qquad  %\textrm{for } 
\delta=0,\ \ \kappa \in (0,\infty), \\[0.35em]
H^1_{\gamma_{\rm D}} (\omega;\R^2) \times  L^2(\omega \times Y), 
\qquad 
%\textrm{for } 
\delta=0,\ \ \kappa=0;
\end{cases}
\\[0.7em]
%%\end{eqnarray*}
%%\begin{eqnarray*}
V_{2,\delta} (\Omega \times Y_0)&:=\begin{cases}
L^2(\omega;H^1_{00} (I \times Y_0;\R^3)),\qquad  %\textrm{for }
 \delta \in (0,\infty), \\[0.35em]
L^2(\Omega;H^1_{0} (Y_0;\R^3)),
\qquad 
%%\textrm{for } 
\delta=\infty,\\[0.35em]
L^2(\omega;H^1_{0} (Y_0;\R^2)) \times L^2(\omega \times Y_0) ,
\qquad 
%%\textrm{for }
 \delta=0;
\end{cases}
 \\[0.7em]
H_{\delta,\kappa}(\Omega \times Y)&:=\begin{cases}
L^2(\omega;\R^3)+L^2(\Omega \times Y_0;\R^3), \qquad %\textrm{ for } 
\delta \in (0,\infty],\ \ \kappa=\infty, \\[0.35em]
L^2(\omega;\R^3)+L^2(\omega \times Y_0;\R^3), \qquad
% \textrm{ for } 
\delta =0,\ \ \kappa=\infty, \\[0.35em]
\left(L^2(\omega;\R^2)+L^2(\omega \times Y_0;\R^2)\right)\times L^2(\omega \times Y) , \qquad %\textrm{ for } 
\delta =0,\ \ \kappa \in [0,\infty), 	
\end{cases} 			
\end{align*}
where functions defined on $\omega$ are assumed to be constant across the plate whenever they are considered in $\Omega.$ (In other words, $L^2(\omega \times Y)$ is treated as naturally embedded in $L^2(\Omega \times Y).$)
%\end{eqnarray*} 	
%
We denote by $P_{\delta,\kappa}$ and $P^0$ the orthogonal projections $P_{\delta,\kappa}:L^2(\Omega\times Y;\R^3)\to H_{\delta,\kappa}(\Omega \times Y)$ and $P^0:L^2(\Omega \times Y) \to L^2(\omega)+L^2(\omega \times Y_0),$ respectively.  
The mappings $$L^2(\Omega)+L^2(\Omega \times Y_0) \ni u(x)+\mathring{u} (x,y )\mapsto u(x) \in L^2(\Omega)$$ %%is labelled by $S_1$ 
and 
%%the mapping 
$$L^2(\Omega)+L^2(\Omega \times Y_0) \ni u(x)+\mathring{u} (x,y )\mapsto \mathring{u}(x,y) \in L^2(\Omega\times Y_0)$$  are labelled by $S_1$ and $S_2,$ respectively. 
 For Hilbert spaces $V, W$ and a linear operator $\mathcal{A}:V\to W,$ we denote by $\mathcal{R}(\mathcal{A}) \subset W$ its range, and for a linear operator $\mathcal{A}:V \to V,$ we denote by $\sigma (\mathcal{A})$ its spectrum. Furthermore,   $\sigma_{\rm ess}(\mathcal{A})$ and $\sigma_{\rm disc}(\mathcal{A})$ denote the essential and discrete spectrum of $\mathcal{A},$ respectively.  Throughout, we denote by $\mathcal{I}$ the identity operator on the appropriate ambient space.
%By $i(\cdot,\cdot)$ we denote the scalar multiplication on the space $L^2(\Omega \times Y)$ (this also includes the subspaces embedded in it).

For the definition of two-scale convergence, the related notation and properties of importance for our analysis, we refer the reader to Appendix (for the basic properties and introduction, see also \cite{Allaire-92}). Finally, for a Hilbert space $V$, we denote by $V^*$ its dual, and $\rightharpoonup,$ $\to$ denote, respectively, the weak and strong convergence.

\subsection{Section guide for different scaling regimes}

\label{table_sec}

The table below shows the different scalings considered in this paper for the period of oscillations $\varepsilon_h$ with respect to the thickness $h$ as well as  appropriate scalings of the contrast, time, and spectrum.

\

%{\small

\noindent
\begin{tabular}{|l|c|c|c|c|c|c|c|c|c|}
	%\begin{tabular}{|l|l|c|c|c|}
	\hline
	\ & Time
	&
	\parbox[t]{1.1cm}{$h\ll\varepsilon_h$ \\$(\delta=0)$ \\ \ }
	%$\begin{array}{c}h\ll\varepsilon_h\\ (\delta=0)\end{array}$
	& Spec
	%%\parbox[t]{1cm}{Spec-\\ trum\\ \ }
	& Time &
	\parbox[t]{1.7cm}{\ \ \ $h\sim\varepsilon_h$ \\$\bigl(\delta\in(0,\infty)\bigr)$ \\ \ }
	& Spec & Time & 
	\parbox[t]{1.3cm}{$h\gg\varepsilon_h$ \\$(\delta=\infty)$ \\ \ }
	%%$\begin{array}{c}h\gg\varepsilon_h\\ (\delta=\infty)\end{array}$ 
	& Spec\\
	\hline
		\parbox[t]{1.3cm}{\ \\$\mu_h=\varepsilon_h$}& 
\parbox[t]{1cm}{Non-\\scaled:\\ \ref{realtimeevol}\\ \ }
	&\parbox[t]{1cm}{$\tau=0:$\\ \ref{hlleps}.A\\ \ }& \ref{secmembsp} & \parbox[t]{1cm}{Long:\\ \ref{section341}\\--------\\Non-\\scaled:\\ \ref{realtimeevol}\\ \ } &
	\parbox[t]{1cm}{$\tau=2:$\\ \ref{sectionheps}.A\\--------\\$\tau=0:$\\ \ref{sectionheps}.B\\ \ }
   & \parbox[t]{1cm}{$\tau=2:$\\ \ref{section332}\\--------\\$\tau=0:$\\ \ref{secmembsp}\\ \ } & \parbox[t]{1cm}{Non-\\scaled:\\ \ref{realtimeevol}\\ \ } & 
   \parbox[t]{1cm}{$\tau=0:$\\ \ref{epsllh}.A\\ \ }
   & \ref{sectionspsur}\\
	\hline
	\parbox[t]{1.3cm}{\ \\$\mu_h=\varepsilon_hh$}&***** & ***** & ****& 
	\parbox[t]{1cm}{Long:\\ \ref{sectionltbhc1}\\ \ }
	 & \parbox[t]{1cm}{$\tau=2:$\\ \ref{sectionheps}.C\\ \ }
	 & \ref{sectionbgl} &
	 \parbox[t]{1cm}{Long:\\ \ref{sectionltbhc1}\\ \ }
	  & 
	  \parbox[t]{1cm}{$\tau=2:$\\ \ref{epsllh}.B\\ \ }
	  &\ref{sectionspsur}\\
	\hline
	\parbox[t]{1.3cm}{\ \\$\mu_h=\varepsilon_h^2$}& 
	\parbox[t]{1cm}{Long:\\ \ref{sectionltbhc2}\\ \ }
	 &
	 \parbox[t]{1cm}{$\tau=2:$\\ \ref{hlleps}.B\\ \ }
	  & \ref{sectionbgl}& ***** & ***** & ***** & ***** & ***** & ****\\
	\hline
\end{tabular}
%}

%%\

%%{\color{red} CAPTION} 

\section{Main results}
\label{main_results} 
\subsection{Effective elasticity tensors}
\label{effective}
In this section we will define limit elasticity tensors that will appear in various regimes. For $\delta \in (0,\infty),$ we define a symmetric tensor $ \C^{\rm hom}_{\delta}$ via
\begin{equation}
\label{tensorminimization}
\begin{split}
 \C^{\rm hom}_{\delta}(\vect A&,\vect B):(\vect A,\vect B):= \\[0.35em] 
& = \min_{\vect \varphi \in H^1(I\times \mathcal{Y}; \R^3)}  \int\limits_I \int\limits_{Y_1} \C_1(y) \left[\iota\left(\vect A -x_3 \vect B\right) + \sym\widetilde{\nabla}_{\delta}\, \vect \varphi\right] : \left[\iota\left(\vect A -x_3 \vect B\right) + \sym\widetilde{\nabla}_{\delta}\, \vect \varphi\right] \,dy \,dx_3,\\[0.35em]
&\qquad\qquad\qquad\vect A, \vect B \in \mathbb{R}^{2 \times 2}_{\sym},
\end{split}
\end{equation}
as well as tensors $\C^{\rm memb}_{\delta},$ $\C^{\rm bend}_{\delta}$ via
\begin{equation*}
\begin{split}
\C^{\rm memb}_{\delta} \vect A: \vect A := \C^{\rm hom}_{\delta} (\vect A,\vect 0):(\vect A,\vect 0),\quad \vect A\in \mathbb{R}^{2 \times 2}_{\sym},\quad\qquad
\C^{\rm bend}_{\delta} \vect B :\vect B := \C^{\rm hom}_{\delta}(\vect 0,\vect B):(\vect 0,\vect B),\quad \vect B \in \mathbb{R}^{2 \times 2}_{\sym}.
\end{split}
\end{equation*}
\begin{remark}\label{remigor1} 
	Under an additional assumption on the material symmetries, namely Assumption \ref{assumivan1}\,~(1), the tensor $\C_{\delta}^{\rm hom}$ can be written as the orthogonal direct sum 
	$$\C_{\delta}^{\rm hom} = \C_{\delta}^{\rm memb} \oplus \C_{\delta}^{\rm bend},$$
	in the sense that
	\begin{equation*}
	\C^{\rm hom}_{\delta} = \begin{bmatrix} \C_{\delta}^{\rm memb} & 0 \\[0.35em] 0 & \C_{\delta}^{\rm bend}\end{bmatrix},
	\end{equation*}
i.e.,
	\begin{equation*}
	\C^{\rm hom}_{\delta}(\vect A,\vect B):(\vect A,\vect B) =  \C_{\delta}^{\rm memb} \vect A: \vect A + \C_{\delta}^{\rm bend} \vect B:\vect B,\qquad \vect A, \vect B \in \mathbb{R}^{2 \times 2}_{\sym}.
	\end{equation*}
\end{remark}
For the case $\delta=0$ the following tensor $\C^{\rm hom,r}$ will be important (in this case we assume that $Y_0$ is of class $C^{1,1}$):
%% For $\vect A, \vect B \in \R^{2 \times 2}_{\sym}$ we define 
\begin{equation}
\label{tensorminimizationrrr}
\begin{split}
&\C^{\rm hom,r}(\vect A,\vect B):(\vect A,\vect B):= \\[0.35em] & = \min  \int\limits_I \int\limits_{Y_1} \C_1(y) \left[\iota\left(\vect A -x_3 \vect B\right) + \funcC_0(\vect{\varphi}_1,\varphi_2,\vect g)(x_3,y)\right] : \left[\iota\left(\vect A -x_3 \vect B\right) + \funcC_0(\vect{\varphi}_1,\varphi_2,\vect g)(x_3,y)\right] \,dy \,dx_3,\\[-0.25em]
&\qquad\qquad\qquad\qquad\vect A, \vect B \in \mathbb{R}^{2 \times 2}_{\sym},
\end{split}
\end{equation}
where the minimum is taken over  $\vect{\varphi}_1 \in \dot{H}^1(\mathcal{Y};\R^2)$, $\varphi_2 \in \dot{H}^2(\mathcal{Y})$, $\vect g \in L^2 ( I \times Y, \R^3).$ Note that in (\ref{tensorminimizationrrr}) the definition (\ref{C_0_def}) of $\mathcal{C}_0$ is used with $\vect{\varphi}_1,$ $\varphi_2,$ $\vect g$ independent of $\hat{x}$. Furthermore, we define a tensor function $\C_0^{\rm red}(y),$ $y\in Y_0,$ by the formula
\begin{equation}
\begin{split} \label{nada10000}
 \C_0^{\rm red}(y)(\vect A&,\vect B):(\vect A,\vect B):= \\[0.35em] & = \min_{\vect g \in L^2(I;\R^3)}  \int\limits_I  \C_0(y) \left[\iota\left(\vect A -x_3 \vect B\right) + \iota_1(\vect g(x_3))\right] : \left[\iota\left(\vect A -x_3 \vect B\right) + \iota_1 (\vect g(x_3))\right] \,dx_3,\\[0em]
&\qquad\qquad\qquad\qquad\vect A, \vect B \in \mathbb{R}^{2 \times 2}_{\sym}.
\end{split}
\end{equation}
In addition, for $\alpha=0,1$
 %%and $\vect A \in \R^{2 \times 2}_{\sym}$,
  we define a tensor-valued function $\C_{\alpha}^{\rm r}(y)$, $y \in Y$, via the formula 
$$\C_{\alpha}^{\rm r}(y)\vect A:\vect A=\min_{\vect d \in \R^3}\C_{\alpha} (y)[\iota(\vect A)+\iota_1 (\vect d)]:[\iota(\vect A)+\iota_1 (\vect d)],\qquad \vect A \in \R^{2 \times 2}_{\sym},\quad y\in Y_\alpha.
$$
\begin{remark} \label{nada10001} 
	
	It is easily seen that for a $\vect{\varphi}_1,$ $\varphi_2,$ $\vect g$ on which the minimum in \eqref{tensorminimizationrrr} is attained, one has $\vect g(x_3, y)=\vect g_0(y)+x_3 \vect g_1(y),$ for some $\vect g_0,$ $\vect g_1\in L^2(Y, {\mathbb R}^3).$ It follows that 
	$$ \C^{\rm hom,r}(\vect A,\vect B):(\vect A,\vect B)=
	\C_{1}^{\rm memb,r}\vect A:\vect A+\C_{1}^{\rm bend,r} \vect B:\vect B, \qquad  \vect A, \vect B \in \R^{2 \times 2}_{\sym},$$
	where 
	\begin{eqnarray*}
		\C_{1}^{\rm memb,r}\vect A:\vect A&:=& \C^{\rm hom,r}(\vect A,\vect 0):(\vect A,\vect 0)\\[0.3em] &=&  \min_{\vect \varphi_1 \in \dot{H}^1(\mathcal{Y};\R^2)} \int_{Y_1}  \C_{1}^r(y)[\vect A+\nabla_y \vect \varphi_1(y)]:[\vect A+\nabla_y \vect \varphi_1(y)]\, dy, \qquad  \vect A \in \R^{2 \times 2}_{\sym},\\[1em]
		\C_{1}^{\rm bend,r}\vect B:\vect B&:=& \C^{\rm hom,r}(\vect 0,\vect B):(\vect 0,\vect B)\\[0.3em] &=& \min_{\varphi \in \dot{H}^2(\mathcal{Y})} \frac{1}{12} \int_{Y_1}  \C_{1}^r(y)[\vect B+\nabla_y^2 \varphi(y)]:[\vect B+\nabla_y^2 \varphi(y)]\,dy, \qquad   \vect B \in \R^{2 \times 2}_{\sym}. 
	\end{eqnarray*} 
	Similarly to the above, it is seen that the minimum in \eqref{nada10000} is attained on the vector fields of the form $\vect g(x_3)= \vect g_0+x_3 \vect g_1,$ where $\vect g_0,$ $\vect g_1\in{\mathbb R}^3.$ Furthermore, we have the following decomposition:
	$$ \C_0^{\rm red}(y)(\vect A,\vect B):(\vect A,\vect B)=
	\C_{0}^{\rm memb,r}(y)\vect A:\vect A+\C_{0}^{\rm bend,r}(y) \vect B:\vect B, \qquad  \vect A, \vect B \in \R^{2 \times 2}_{\sym},\qquad y\in Y_0,$$
	where 
	$$ \C_{0}^{\rm memb,r}(y)\vect A:\vect A:= \C_{ 0}^r(y)\vect A:\vect A, \quad \C_{0}^{\rm bend,r}(y)\vect B:\vect B:=\frac{1}{12}\C_{0}^{\rm r}(y)\vect B:\vect B, \qquad  \vect A, \vect B \in \R^{2 \times 2}_{\sym}, \qquad y\in Y_0.$$	

\end{remark}	

For the case $\delta=\infty$, a tensor $\C^{\rm hom,h}$ will be important, which is defined by 
\begin{equation}
\label{tensorminimizationrrrh}
\begin{split}
\C^{\rm hom,h}(\vect A&,\vect B):(\vect A,\vect B):= \\[0.3em] & = \min  \int\limits_I \int\limits_{Y_1} \C(y) \left[\iota\left(\vect A -x_3 \vect B\right) + \funcC_{\infty}(\vect w,\vect g)(x_3,y)\right] : \left[\iota\left(\vect A -x_3 \vect B\right) + \funcC_{\infty}(\vect w,\vect g)(x_3,y)\right] \,dy \,dx_3,\\[-0.3em]
&\qquad\qquad\qquad\qquad\vect A, \vect B \in \R^{2 \times 2}_{\sym},
\end{split}
\end{equation}
where the minimum is taken over  $\vect w \in L^2 (I;\dot{H}^1(\mathcal{Y};\R^3))$, $\vect g \in L^2 (I; \R^3).$ (As in the case of the expression ${\mathcal C}_0$ entering (\ref{tensorminimizationrrr}), for the expression $\mathcal{C}_{\infty}$ in (\ref{tensorminimizationrrrh})  we take the functions $\vect w,$ $\vect g$ to be independent of $\hat{x}.$)
\begin{remark} \label{nada10001h} 
	It is easily seen that the minimum in \eqref{tensorminimizationrrrh} is attained on $\vect g=\vect g_0+x_3 \vect g_1 $, $\vect w=\vect w_0(y)+x_3 \vect w_1(y),$ for some ${\vect g}_0,$ ${\vect g}_1\in{\mathbb R}^3,$ $\vect w_0,$ $\vect w_1\in L^2(Y; {\mathbb R}^3).$ It follows that 
	%%for $\quad  \vect A, \vect B \in \R^{2 \times 2}_{\sym},$ one has
	$$ \C^{\rm hom,h}(\vect A,\vect B):(\vect A,\vect B)=
	\C^{\rm memb,h}\vect A:\vect A+\C^{\rm bend,h} \vect B:\vect B,  \qquad \vect A, \vect B \in \R^{2 \times 2}_{\sym},$$
	where 
	\begin{eqnarray*}
		\C^{\rm memb,h}\vect A:\vect A&:=&
		\C^{\rm hom,h}(\vect A,\vect 0):(\vect A,\vect 0) \\[0.35em] &=&
		\min_{\vect w \in H^1(\mathcal{Y},\mathbb{R}^3),\vect g \in \mathbb{R}^3 }  \int_{Y_1}  \C(y)[\vect A+\mathcal{C}_{\RRR \infty\BBB}(\vect w, \vect g)]:[\vect A+\mathcal{C}_{\RRR \infty \BBB} (\vect w, \vect g)]\, dy, \qquad \vect A \in \R^{2 \times 2}_{\sym}, \\[1em]
		\C^{\rm bend,h}\vect B:\vect B&:=&
		\C^{\rm hom,h}(\vect 0,\vect B):(\vect 0,\vect B) \\[0.35em] &=&
		\min_{\vect w \in H^1(\mathcal{Y},\mathbb{R}^3),\vect g \in \mathbb{R}^3} \frac{1}{12} \int_{Y_1}  \C(y)[\vect B+\mathcal{C}_{\RRR \infty\BBB}(\vect w, \vect g)]:[\vect B+\mathcal{C}_{\RRR \infty \BBB}(\vect w, \vect g)]\, dy, \qquad  \vect B \in \R^{2 \times 2}_{\sym}. 
	\end{eqnarray*} 
\end{remark}		
The following proposition is proved in Section \ref{sectinproof3.1}. 
\begin{proposition} \label{propvecer} 
	The tensor $\C^{\rm hom}_{\delta}$ (and consequently the tensors $\C_{\delta}^{\rm memb}$, $\C_{\delta}^{\rm bend}$ as well) is bounded and coercive, i.e., there exists $\nu>0$ such that 
	$$
	\nu\bigl(|\vect A|^2+|\vect B|^2\bigr) \leq \C^{\rm hom}_{\delta} (\vect A,\vect B):(\vect A,\vect B)  \leq \nu^{-1}\bigl(|\vect A|^2+|\vect B|^2\bigr) \qquad \forall \vect A,\vect B \in \R^{2 \times 2}_{\sym}. 
	$$ 	
	Analogous claims are valid for tensors  $\C^{\rm hom,r}$, $\C^{\rm hom,h}$, $\C_0^{\rm red}$ (and consequently tensors  $\C_{1}^{\rm memb,r} $, $	\C_{1}^{\rm bend,r}$, $	\C^{\rm memb,h}$, $	\C^{\rm bend,h}$, $	\C_{0}^{\rm memb,r}$, $	\C_{0}^{\rm bend,r}$).
\end{proposition}	
\subsection{Limit resolvent equations} 
\label{limreseq} 
Our starting point is the following resolvent formulation. For $\tau, \lambda>0$ and a given $\vect f^{\epsh} \in L^2(\Omega;\R^3),$ find $\vect u^\epsh \in H^1_{\Gamma_{\rm D}}(\Omega;\R^3)$ such that the following variational formulation holds:
\begin{equation}
\label{resolventproblemstarting}
\frac{1}{h^{\tau}}\int\limits_{\Omega} \C^{\mu_h}\biggl(\dfrac{\hat{x}}{\epsh}\biggr) \simgrad \vect u^\epsh : \simgrad_h \vect v \,dx
+
\lambda \int\limits_{\Omega}\rho^{\epsh} \vect u^\epsh \cdot \vect v\,dx 
=
\int\limits_{\Omega} \vect  f^\epsh \cdot \vect v\,dx \qquad \forall \vect v \in H^1_{\Gamma_{\rm D}}(\Omega;\R^3).
\end{equation}

We derive the limit resolvent equation, as $h\to0,$ depending on various assumptions about the parameter  $\delta=h/\eps_h,$ the exponent $\tau,$ and the scaling of the load density $\vect f^\epsh$. In Section \ref{limitspectrum} we discuss implications of these results for the limit spectrum and evolution equations. Different scalings of the operator will, in particular, yield different scalings of the spectrum and the time variable (or mass density) in the evolution problems.
%% We will also discuss different scaling of forces in different regimes. 
Note that the load density scaling will also depend on the asymptotic regime considered. 

It is standard in the theory of plates that one discusses limit equations (both static and dynamic) depending on an appropriate scaling of the external loads. Furthermore, we will see that the limit resolvent equation is always degenerate in some sense. From the mathematical point of view, this is a consequence of the fact that for thin domains the constant in Korn's inequality blows up and by further analysis one can see that this implies that the out-of-plane and in-plane components of the solution are scaled differently in the limit problem. From the physical point of view, it is much easier (i.e., energetically more convenient) for the plate to bend then to stretch. As a result, bending and membrane waves propagate  through the medium on different time scales. The effect of high-contrast is also non-negligible, yielding different behaviour depending on the asymptotic regime: the small elastic inclusions behave like three-dimensional objects (regime $h \sim \epsh$) or like small thin plates (regime $h \ll \epsh$). We next present our main results. 
\subsubsection{Asymptotic regime $h \sim \varepsilon_h$}
\label{sectionheps}
\vspace{+1.2ex}
\noindent\textbf{A. ``Bending" scaling: $\mu_h=\epsh,\tau=2$} 
\vspace{1.3ex}

The following proposition provides an appropriate compactness result, namely a bound on the sequence of energies for a fixed value of $\delta,$ see (\ref{delta_kappa}), and its consequences in terms of two-scale convergence.
\begin{proposition}\label{propdinsi1}
	%Let a pair $(h, \eps_h)$ be an element of a sequence 
	Consider a sequence $\{(h, \eps_h)\}$ such that  $\delta=\lim_{h\to0} h/\eps_h\in (0,\infty)$, and suppose that $\mu_h=\epsh$, $\tau=2$. The following statements hold:
	\begin{enumerate} 
		\item  
		There exists $C>0$, independent of $h$, such that for any sequence
		  $(\vect f^{\varepsilon_h})_{h>0}\subset L^2(\Omega;\R^3)$ of load densities and the corresponding solutions $\vect u^\epsh$ to the resolvent problem \eqref{resolventproblemstarting} one has  
		  	$$
		  	h^{-2}a_{\epsh} (\vect u^{\epsh},\vect u^{\epsh})+\|\vect u^{\epsh}\|^2_{L^2}< C\bigl\|\pi_h \vect f^{\epsh}\bigr\|^2_{L^2}.
		  	$$
		\item 
	If $$\limsup_{h \to 0} \left(h^{-2}a_{\epsh} (\vect u^{\epsh},\vect u^{\epsh})+\|\vect u^{\epsh}\|^2_{L^2}\right)< \infty, \qquad   (\vect u^{\epsh})_{h>0}\subset H^1_{\Gamma_{\rm D}}(\Omega;\R^3),$$ 	
	then there exist functions $\vect \funcA \in H_{\gamma_{\rm D}}^1(\omega;\R^2)$, $\funcB \in H_{\gamma_{\rm D}}^2(\omega)$, $\funcC \in {\RRR {\mathfrak C}\BBB}_{\delta}(\Omega \times Y)$, $\mathring{\vect u}\in V_{2,\delta} (\Omega \times Y_0)$,  such that for a subsequence, which we continue labelling with $\epsh$, one has
	\begin{equation}
	\begin{split} \label{nakcomp1} 
	\vect u^\epsh & = \Tilde{\vect u}^\epsh + \mathring{\vect u}^\epsh, \quad \Tilde{\vect u}^{\epsh},\mathring{\vect u}  \in H^1_{\Gamma_{\rm D}}(\Omega;\R^3),\quad \mathring{\vect u}^\epsh |_{\Omega_1^{\varepsilon_h}} =0, \\
	\pi_{1/h}\Tilde{\vect u}^\epsh & {\,\xrightarrow{L^2\,}\,}\bigl(
	 \funcA_1(\hat{x}) - x_3 \partial_1 \funcB(\hat{x}), 
	\funcA_2(\hat{x}) - x_3 \partial_2 \funcB(\hat{x}),
	\funcB(\hat{x})\bigr)^\top, \\
	h^{-1}\mathring{\vect u}^\epsh(x) & \drtwoscale  \mathring{\vect u}(x,y),
	\\
	h^{-1}\simgrad_h \Tilde{\vect u}^\epsh(x) & \drtwoscale \, \iota\bigl(\simgrad_{\hat{x}} \vect \funcA(\hat{x})-x_3 \nabla_{\hat{x}}^2\funcB(\hat{x})\bigr) + \funcC(x,y),
	\\
	\epsh h^{-1}\simgrad_h \mathring{\vect u}^\epsh(x) & \drtwoscale \sym\widetilde{\nabla}_{\delta}\, \mathring{\vect u}(x,y),
	\end{split}
	\end{equation}
where $\drtwoscale$ stands for the ``dimension-reduction two-scale convergence" defined in Appendix \ref{twoscale_conv}.
	\item If, additionally to 2, one assumes that $$\lim_{h \to 0} \left(h^{-2}a_{\epsh} (\vect u^{\epsh}, \vect u^{\epsh})+ \|\vect u^{\epsh}\|^2_{L^2}\right)=a_{\delta}^{\funcB} (\funcB,\funcB)+\|\funcB\|^2_{L^2},$$ where the form $a_{\delta}^{\funcB}$ is defined in \eqref{missaglic1}, then one has the strong two-scale convergence ({\it cf.} Appendix \ref{twoscale_conv})
\begin{equation*} 
\pi_{1/h} \vect{u}^{\epsh} \strongdrtwoscale 
\bigl(\funcA_1(\hat{x}) - x_3 \partial_1 \funcB(\hat{x}), 
 \funcA_2(\hat{x}) - x_3 \partial_2 \funcB(\hat{x}),
\funcB(\hat{x})\bigr)^\top,
\end{equation*}	
with $\vect \funcA=\vect \funcA^{\funcB}$ (for the definition of $\vect \funcA^{\funcB}$ see \eqref{anketa1} below). 
	\end{enumerate} 
\end{proposition}	
\begin{remark} \label{rucak11} 
	It can be seen from the proof of Proposition \ref{propdinsi1} that the assumption in its third statement is equivalent to the convergence 
	\begin{align*}
	%\begin{eqnarray*}
		h^{-1}\sym\nabla_h {\vect u}^{\epsh}\chi_{\Omega^1_{\epsh}} &\strongdrtwoscale  
		 \iota\bigl(\simgrad_{\hat{x}} \vect \funcA(\hat{x})-x_3 \nabla_{\hat{x}}^2\funcB(\hat{x})\bigr)\chi_{I \times Y_1} +\funcC(x,y) \chi_{I \times Y_1}, \\[0.25em]
		 \epsh{h}^{-1}\sym  \nabla_h \mathring {\vect u}^{\epsh} & \strongdrtwoscale  % \sym\widetilde{\nabla}_{\delta}\, \mathring{\vect u}(x,y), 
		 0,\\[0.25em]
	\pi_{1/h}\vect{u}^{\epsh} & \strongdrtwoscale \bigl(
	\funcA_1(\hat{x}) - x_3 \partial_1 \funcB(\hat{x}), 
	 \funcA_2(\hat{x}) - x_3 \partial_2 \funcB(\hat{x}),\funcB(\hat{x})\bigr)^\top.
	%\end{eqnarray*} 
    \end{align*}
	Here $\vect \funcA=\vect \funcA^{\funcB}$ and $\funcC(x,\cdot)$ solves the minimization problem \eqref{tensorminimization} with $\vect A=\simgrad_{\hat{x}} \vect \funcA(\hat{x})$ and $\vect B=\nabla_{\hat{x}}^2\funcB(\hat{x})$.
	The analogous claim is valid in all other regimes. As we do not explicitly use it in what follows, we shall omit it.
\end{remark}	
The following theorem provides the limit resolvent equation. It can be seen that the limit equations do not couple $(\vect\funcA,\funcB)$ and $\mathring{\vect u}$. This is not usual in high-contrast analysis and is a consequence of setting $\tau=2$. This case is thus less interesting and we shall omit its analysis in other regimes ($\delta=0$ and $\delta=\infty$). However, we will study it here, as it resembles the standard model of a moderate-contrast plate (and so the corresponding evolution is obtained on a longer time scale).  
\begin{theorem} \label{thmivan111} 
	Under the notation of Proposition \ref{propdinsi1}, suppose that $\delta \in (0,\infty)$, $\mu_h=\epsh$,  $\tau=2$, and consider a sequence $(\vect  f^\epsh)_{h>0}$ of load densities  such that 
	\begin{equation} \label{rucak24} 
	\pi_{h}\vect  f^\epsh \drtwoscale \vect f(x, y) \in L^2(\Omega \times Y;\R^3).
	\end{equation} 
	Then the sequence of solutions  to the resolvent problem \eqref{resolventproblemstarting} converges in the sense of 
	%%Proposition \ref{propdinsi1} 
	\eqref{nakcomp1} to the unique solution of the following problem: Determine $\vect \funcA \in H_{\gamma_{\rm D}}^1(\omega;\R^2)$, $\funcB \in H_{\gamma_{\rm D}}^2(\omega)$, $\mathring{\vect u}\in V_{2,\delta}(\Omega \times Y_0)$, such that
	\begin{equation}
	\label{limitmodellongtime}
	\begin{split}
	&      \int\limits_{\omega}\C^{\rm hom}_{\delta}\left(\simgrad_{\hat{x}}\vect \funcA(\hat{x}),\nabla_{\hat{x}}^2\funcB(\hat{x})\right): \left(\simgrad_{\hat{x}}\vect \theta_* (\hat{x}),\nabla_{\hat{x}}^2{\theta_3}(\hat{x})\right) \,d\hat{x}
	+
	\lambda \int\limits_{\omega} \langle {\rho} \rangle
	\funcB(\hat{x})
	{\theta_3}(\hat{x})
	\,d\hat{x} \\[0.35em]
	& \hspace{+10pt}=
	\int\limits_\omega \langle \overline{ \vect f } \rangle(\hat{x}) \cdot \bigl(\vect\theta_*(\hat{x}), \theta_3(\hat{x})\bigr)\RRR^\top \BBB\,d\hat{x}
	-
	\int\limits_\omega \langle \overline{x_3 \vect f_{\!*}  }\rangle(\hat{x}) \cdot \nabla_{\hat{x}}{\theta_3}(\hat{x}) \,d\hat{x}\qquad\forall  \vect\theta_*\in  H^1_{\gamma_{\rm D}}(\omega,\R^2),\   {\theta_3} \in H^2_{\gamma_{\rm D}}(\omega), 
	\\[0.35em]
	& \int\limits_I \int\limits_{Y_0} \C_0(y) \sym\widetilde{\nabla}_{\delta}\, \mathring{\vect u}(x,y) : \sym\widetilde{\nabla}_{\delta}\, \mathring{\vect\xi}(x_3,y) \,dy\,dx_3\\[0.35em]
&	\hspace{+10pt}=\int\limits_I \int\limits_{Y_0} 
	\vect f(x,y)
	\cdot
	\bigl(\mathring{\xi}_1(x_3,y), 
	\mathring{\xi}_2(x_3,y),
	0\bigr)^\top	\,dy\,dx_3\qquad
	%% \\[0.4em]
	%%&\hspace{+150pt}
	\forall\mathring{\vect\xi} \in H_{00}^1(I\times Y_0;\R^3),\ \  \textrm{a.e.\,} \hat{x} \in \omega.
	\end{split}
	\end{equation}

If additionally one assumes the strong two-scale convergence in \eqref{rucak24}, then one has
%% the strong two-scale convergence 
$$
\pi_{1/h} \vect{u}^{\epsh} \strongdrtwoscale 
\bigl(\funcA_1(\hat{x}) - x_3 \partial_1 \funcB(\hat{x})+\mathring{u}_1, 
\funcA_2(\hat{x}) - x_3 \partial_2 \funcB(\hat{x})+\mathring{u}_2,
\funcB(\hat{x})\bigr)^\top.
$$	
\end{theorem}

\begin{remark} 
	Under Assumption \ref{assumivan1}~(1) the first identity in \eqref{limitmodellongtime} uncouples into two independent identities (see Remark \ref{remigor1})
	\begin{equation*}
	\begin{split}
	&      \int\limits_{\omega}\C^{\rm memb}_{\delta}\simgrad_{\hat{x}}\vect \funcA(\hat{x}): \simgrad_{\hat{x}}\vect \theta_* (\hat{x}) \,d\hat{x}
	=
	\int\limits_\omega \langle \overline{ \vect f_{\!*}  } \rangle(\hat{x}) \cdot \vect\theta_* (\hat{x}) \,d\hat{x}\qquad\forall  \vect\theta_*\in  H^1_{\gamma_{\rm D}}(\omega,\R^2),\\
	&      \int\limits_{\omega}\C^{\rm bend}_{\delta}\nabla_{\hat{x}}^2\funcB(\hat{x}): \nabla_{\hat{x}}^2{\theta_3}(\hat{x}) \,d\hat{x}
	+
	\lambda \int\limits_{\Omega} 
	\langle \rho \rangle \funcB(\hat{x})\,
	{\theta_3}(\hat{x})
	\,d\hat{x}\\
	&\hspace{+90pt}=\int\limits_\omega \langle \overline{ \vect f_3 } \rangle(\hat{x})\,{\theta_3}(\hat{x}) \,d\hat{x}
	-
	\int\limits_\omega \langle \overline{x_3 \vect f_{\!*}  }\rangle(\hat{x}) \cdot \nabla_{\hat{x}}{\theta_3}(\hat{x}) \,d\hat{x}\qquad   \forall{\theta_3} \in H^2_{\gamma_{\rm D}}(\omega).
	\end{split} 
	\end{equation*}	
\end{remark} 	
In connection with the limit problem, we consider a self-adjoint operator $\mathcal{A}^{\rm hom}_{\delta}$ defined on the $\langle \rho \rangle$-weighted space $L^2(\omega;\R^2)\times L^2(\omega)$ and corresponding to the differential expression\footnote{The repeated divergence $\div_{\hat{x}}\div_{\hat{x}}$ is applied to matrices and corresponds to the usual divergence applied row-wise and column-wise sequentially.}
\begin{equation*}
	%%\label{rucak90} 
\langle \rho \rangle^{-1}\left(-\div_{\hat{x}}, \div_{\hat{x}}\div_{\hat{x}} \right)\C^{\rm hom}_{\delta}\left(\simgrad_{\hat{x}},\nabla_{\hat{x}}^2\right).
\end{equation*}
More precisely, the operator $\mathcal{A}^{\rm hom}_{\delta}$ is defined via the bilinear form
\begin{equation*}
\begin{split}
a^{\rm hom}_{\delta} \bigl((\vect u,  v ),(\vect \zeta,  \xi)\bigr)  := \int\limits_{\omega} \C^{\rm hom}_{\delta} \left(\simgrad_{\hat{x}}\vect u(\hat{x}),\nabla_{\hat{x}}^2  v(\hat{x}) \right) : \left(\simgrad_{\hat{x}}\vect \zeta(\hat{x}),\nabla_{\hat{x}}^2  \xi(\hat{x}) \right) \,d\hat{x}, \\
a^{\rm hom}_{\delta} : \left( H_{\gamma_{\rm D}}^1(\omega;\R^2) \times H_{\gamma_{\rm D}}^2(\omega)\right)^2 \to \mathbb{R}.
\end{split}
\end{equation*}
%%%considered on the space $L^2(\omega;\R^3)$ with the constant weight $\langle \rho \rangle.$ 
%%%and thus in the expression \eqref{rucak90} we have $1/\langle\rho\rangle$ in front. 
We also make use of the following observation.
 Plugging $\vect \theta_3 = 0$ into the first equation in  \eqref{limitmodellongtime} and using linearity, we decompose $\vect \funcA = \vect \funcA^\funcB + \vect \funcA^{\vect f_{\!*} }$, where $\vect \funcA^\funcB, \vect \funcA^{\vect f_{\!*} }  \in H^1_{\gamma_{\rm D}}(\omega ; \R^2)$ are solutions to the integral identities 
\begin{equation}
\label{anketa1} 
\begin{aligned}
&\int\limits_{\omega}\C^{\rm memb}_{\delta}\simgrad_{\hat{x}}\vect \funcA^\funcB(\hat{x}): \simgrad_{\hat{x}}\vect \theta_* (\hat{x}) \,d\hat{x}\\[0.4em] 
&\hspace{2cm}= - \int\limits_{\omega}\C^{\rm hom}_{\delta}\left(0 ,\nabla_{\hat{x}}^2\funcB(\hat{x})\right): \left(\simgrad_{\hat{x}}\vect \theta_* (\hat{x}),0 \right) \,d\hat{x}\qquad \forall \vect\theta_*  \in H^1_{\gamma_{\rm D}} (\omega;\R^2), \\[0.4em]
&\int\limits_{\omega}\C^{\rm memb}_{\delta}\simgrad_{\hat{x}}\vect \funcA^{\vect f_{\!*} }(\hat{x}): \simgrad_{\hat{x}}\vect \theta_* (\hat{x}) \,d\hat{x} = \int\limits_\omega\bigl\langle \overline{ \vect f_{\!*}  }\bigr\rangle(\hat{x}) \cdot \vect\theta_* (\hat{x}) \,d\hat{x} \qquad \forall \vect\theta_*  \in H^1_{\gamma_{\rm D}} (\omega;\R^2).
\end{aligned}
\end{equation}
Notice that the in-plane deformation $\vect \funcA^\funcB$ can be calculated from the out-of-plane deformation $\funcB$ by solving the first identity alone. It is easily seen that the solutions $\vect\funcA^\funcB$, $\vect\funcA^{\vect f_{\!*} }$ satisfy the estimates
\begin{equation*}
\bigl\Vert \simgrad_{\hat{x}} \vect \funcA^\funcB \bigr\Vert_{L^2(\omega;\R^2)} \leq C\bigl\Vert \nabla^2_{\hat{x}}\funcB\bigr\Vert_{L^2(\omega)}, \quad \bigl\Vert \simgrad_{\hat{x}} \vect \funcA^{\vect f_{\!*} } \bigr\Vert_{L^2(\omega;\R^2)} \leq C  \bigl\Vert \overline{\vect f_{\!*} }\bigr\Vert_{L^2(\omega)}. 
\end{equation*}
The first identity in (\ref{anketa1})
%%expression $\vect \funcA^{\funcB}$ 
defines a bounded linear operator $\mathcal{A}^{ \vect \funcA; \funcB }_{\delta}  : H^2_{\gamma_{\rm D}}(\omega) \to H^1_{\gamma_{\rm D}}(\omega;\R^2)$ by the formula  
$
\mathcal{A}^{ \vect \funcA;\funcB }_{\delta} \funcB := \vect \funcA^\funcB
$.  
Furthermore, the bilinear form $a^\funcB_{\delta}: \bigl( H_{\gamma_{\rm D}}^2(\omega) \bigr)^2 \to \R$ given by
\begin{equation}
	\label{missaglic1}
\begin{aligned}
 a^{\funcB}_{\delta}(\funcB,\theta)&:=\int\limits_{\omega}\C^{\rm hom}_{\delta} \left(\simgrad_{\hat{x}}\mathcal{A}_{\delta}^{ \vect \funcA;\funcB } \funcB(\hat{x}),\nabla_{\hat{x}}^2\funcB(\hat{x}) \right): \left(0,\nabla_{\hat{x}}^2 \theta(\hat{x}) \right) \,d\hat{x}\\ 
&=\int\limits_{\omega}\C^{\rm hom}_{\delta} \left(\simgrad_{\hat{x}}\mathcal{A}_{\delta}^{ \vect \funcA;\funcB } \funcB(\hat{x}),\nabla_{\hat{x}}^2\funcB(\hat{x}) \right): \left(\simgrad_{\hat{x}}\mathcal{A}_{\delta}^{ \vect \funcA;\theta } \theta,\nabla_{\hat{x}}^2 \theta(\hat{x}) \right) \,d\hat{x},
\end{aligned}
\end{equation}
defines positive definite (as a consequence of Proposition \ref{propvecer}) self-adjoint operator on $L^2(\omega),$ which we denote by $\mathcal{A}^{\funcB, {\rm hom}}_{\delta}$.  
The first identity in \eqref{limitmodellongtime} can now be written as follows: 
\begin{equation}
	 \label{nadoknada11}
\begin{aligned}
 a^{\funcB}_{\delta}(\funcB,\theta)+\lambda  \int\limits_{\omega} 
\langle \rho \rangle \funcB(\hat{x})\theta(\hat{x})
\,d\hat{x} 
&= \int\limits_{\omega}\C^{\rm hom}_{\delta} \left(\simgrad_{\hat{x}}\vect{\funcA}^{\vect f_{\!*} },0 \right): \left(0,\nabla_{\hat{x}}^2 \theta(\hat{x}) \right) \,d\hat{x}
\\
\hspace{-10ex}&+
\int\limits_\omega \langle \overline{f_3} \rangle(\hat{x})\theta(\hat{x}) \,d\hat{x}
-
\int\limits_\omega \langle \overline{x_3 \vect f_{\!*}  }\rangle(\hat{x}) \cdot \nabla_{\hat{x}}\theta(\hat{x}) \,d\hat{x}\RRR=:\mathcal{F}_{\delta} (\vect f)(\theta), \BBB \qquad \forall \theta \in H^2_{\gamma_{\rm D}}(\omega). 
\end{aligned}
\end{equation} 
Notice that for $\vect f \in  L^2(\Omega;\R^3)$ the 
right-hand side of \eqref{nadoknada11}  can be interpreted as an element of $(H^2_{\gamma_{\rm D}}(\omega))^{*},$ which we \RRR denoted \BBB by $\mathcal{F}_{\delta} (\vect f).$  \RRR This reveals the resolvent structure of the limit problem \eqref{limitmodellongtime}. \BBB

\vspace{+2.5ex} 
\noindent\textbf{B. ``Membrane'' scaling: $\mu_h=\epsh,\tau=0$}
\vspace{1.3ex}

To formulate the convergence result for the present scaling, we consider a non-negative self-adjoint operator $\mathcal{A}_{\delta,\infty}$ defined by the bilinear form
\begin{equation} \label{formadelta}
	\begin{split} 
		a_{\delta,\infty} ((\vect{\funcA},\RRR \funcB \BBB)+\mathring{\vect u}, (\vect \theta, \varphi) + \mathring{\vect \xi}) &:= \int\limits_{\omega}\C^{\rm memb}_{\delta}\simgrad_{\hat{x}}\vect{\funcA}(\hat{x}): \simgrad_{\hat{x}}\vect\theta(\hat{x}) \,d\hat{x} \\ & + \int\limits_{\Omega}\int\limits_{Y_0} \C_0(y) \sym\widetilde{\nabla}_{\delta}\, \mathring{\vect u}(x,y) : \sym\widetilde{\nabla}_{\delta}\, \mathring{\vect \xi}(x,y) \,dy\,dx, \\[0.3em]
		a_{\delta,\infty}:&\ \left(V_{1,\delta,\infty}(\omega \times Y)+V_{2,\delta}(\Omega \times Y_0)\right)^2 \to  \R.
	\end{split} 
\end{equation}  
Notice that $\mathcal{A}_{\delta,\infty}$ is degenerate with an infinite-dimensional kernel: 
$$\mathcal{A}_{\delta,\infty}(0,0,v)=0 \quad\  \forall v \in L^2(\omega).  $$
However, the restriction of $\mathcal{A}_{\delta,\infty}$ on the space $H_{\delta,\infty}(\Omega \times Y) \cap L^{2, \rm memb}(\Omega \times Y;\R^3)$  does not exhibit such degeneracies (under Assumption \ref{assumivan1} (1)). 

  The following proposition gives a suitable compactness result, similar to Proposition \ref{propdinsi1}. 
\begin{proposition}
	\label{lemmaconvergencerealtime}
	Consider a sequence $\{(h, \eps_h)\}$ such that  $\delta=\lim_{h\to0} h/\eps_h\in (0,\infty),$
	%$\delta \in (0,\infty)$, 
	and suppose that $\mu_h=\epsh$,  $\tau=0$. The following statements hold: 
	\begin{enumerate}  
		\item 
		There exists $C>0$, independent of $h,$ such that for any sequence
		$(\vect f^{\epsh})_{h>0}\subset L^2(\Omega;\R^3)$ 
	and the corresponding solutions $\vect u^{\epsh}$ to the problem \eqref{resolventproblemstarting} one has 
	$$ a_{\epsh}(\vect u^{\epsh},\vect u^{\epsh}) +\|\vect u^{\epsh}\|^2_{L^2} \leq C\|\vect f^{\epsh}\|^2_{L^2}.$$ 
	\item If a sequence $(\vect u^{\epsh})_{h>0}\subset H^1_{\Gamma_{\rm D}}(\Omega;\R^3)$ is such that
	$$\limsup_{h  \to 0} \left(a_{\epsh}(\vect u^{\epsh},\vect u^{\epsh})+\|\vect u^{\epsh}\|^2_{L^2}\right)< \infty,$$ then
	 there exist functions $(\vect \funcA,\funcB)^\top \in V_{1,\delta,\infty}(\omega \times Y)$, $\funcC \in {\mathfrak C}_{\delta} (\Omega \times Y)$, $\mathring{\vect u}\in V_{2,\delta}(\Omega \times Y_0)$ such that, up to extracting a subsequence, one has
	\begin{equation}
	\begin{split}\label{ds2}
	\vect u^\epsh & = \Tilde{\vect u}^\epsh + \mathring{\vect u}^\epsh, \quad  \Tilde{\vect u}^{\epsh},\mathring{\vect u}  \in H^1_{\Gamma_{\rm D}}(\Omega;\R^3),\quad \mathring{\vect u}^\epsh |_{\Omega_1^{\varepsilon_h}} = 0, \\	
	\Tilde{\vect u}^\epsh_*   & {\,\xrightarrow{L^2}\,} \vect \funcA, \\ \Tilde{u}_3^\epsh &\drtwoscale \funcB(\hat{x}),\\
	\mathring{\vect u}^\epsh(x) & \drtwoscale  \mathring{\vect u}(x,y),
	\\
	\simgrad_h \Tilde{\vect u}^\epsh(x) & \drtwoscale \, \iota\left(\simgrad_{\hat{x}}\vect \funcA(\hat{x})\right)+ \funcC(x,y),
	\\
	\epsh\simgrad_h \mathring{\vect u}^\epsh(x) & \drtwoscale \sym\widetilde{\nabla}_{\delta}\, \mathring{\vect u}(x,y).
	\end{split}
	\end{equation}
\item  If, additionally to 2, one has $$ \lim_{h  \to 0} \left(a_{\epsh}(\vect u^{\epsh},\vect u^{\epsh})+\|\vect u^{\epsh}\|^2_{L^2}\right)=a_{\delta,\infty}\bigl((\RRR \vect \funcA \BBB,0)^\top+\mathring{\vect u},(\RRR\vect \funcA \BBB,0)^\top+\mathring{\vect u}\bigr)+\bigl\|(\vect \funcA,\funcB)^\top+\mathring{\vect u}\bigr\|^2_{L^2},$$ where $a_{\delta,\infty}$ is defined in \eqref{formadelta}, then the strong two-scale convergence
$$ \vect{u}^{\epsh} \strongdrtwoscale (\RRR \vect \funcA \BBB,\funcB)^\top+\mathring{\vect u}$$
holds.
\end{enumerate}	
\end{proposition}
The following theorem provides the limit resolvent equation. 
\begin{theorem} \label{thmivan11111}
	Under the notation of Proposition \ref{lemmaconvergencerealtime}, suppose that $\delta \in (0,\infty)$, $\mu_h=\epsh$,  $\tau=0,$ and consider a sequence $(\vect  f^\epsh)_{h>0}$ of load densities such that
	\begin{equation} 
		\label{rucak23} 
	\vect  f^\epsh(x) \drtwoscale \vect f(x, y) \in L^2(\Omega \times Y;\R^3).\end{equation}  
	Then the sequence of solutions $(\vect u^\epsh)_{h>0}$ to the resolvent problems \eqref{resolventproblemstarting} converges in the sense of \eqref{ds2}
	%% Proposition \ref{lemmaconvergencerealtime}   
	to the unique solution of the following problem: Determine $(\vect \funcA,\funcB)^\top \in V_{1,\delta,\infty}(\omega \times Y)$, $\mathring{\vect u}\in V_{2,\delta}(\Omega \times Y_0)$, such that
	\begin{equation}
	\label{realtimesystem}
	\begin{aligned}
	&\int\limits_\omega  \C^{\rm memb}_{\delta}\simgrad_{\hat{x}}\vect \funcA(\hat{x}) : \simgrad_{\hat{x}}\vect \theta(\hat{x})
	+\lambda \int\limits_\omega \langle {\rho} \rangle \vect \funcA(\hat{x}) \cdot \vect \theta(\hat{x}) \,d\hat{x}
	+\lambda \int\limits_\omega \langle \rho_0 \overline{ \mathring{\vect u}_*  }\rangle(\hat{x}) \cdot \vect \theta(\hat{x}) \,d\hat{x}\\[0.35em]
	&\hspace{60pt}
	=\int\limits_\omega \langle \overline{\vect f_{\!*}  }\rangle(\hat{x}) \cdot \vect \theta(\hat{x}) \,d\hat{x}\qquad\forall \vect \theta \in H_{\gamma_{\rm D}}^1(\omega;\R^2), \\[0.35em]
	&  \langle \rho \rangle\funcB (\hat x) +  \langle \rho_0\overline{ \mathring{u}_{3} }\rangle (\hat x) =\lambda^{-1} \langle\overline{ \vect f_3 }\rangle (\hat x) \quad  \textrm{a.e.\,} \hat{x} \in \omega. \\[0.4em]
	&\int\limits_I \int\limits_{Y_0} \C_0(y) \sym\widetilde{\nabla}_{\delta}\, \mathring{\vect u}(x,y) : \sym\widetilde{\nabla}_{\delta}\, \mathring{\vect \xi}(x_3,y) \,dy\,dx_3\\
	&\hspace{60pt}+\lambda \int\limits_I\int\limits_{Y_0}
	\rho_0(y)\bigl(\funcA_1(\hat{x}), 
	\funcA_2(\hat{x}),
	\funcB(\hat{x})\bigr)^\top
	\cdot \mathring{\vect\xi}(x_3,y) \,dy\,dx_3 \\
	&\hspace{60pt}+\lambda \int\limits_I \int\limits_{Y_0} \rho_0 (y)\mathring{\vect u}(x,y) \cdot \mathring{\vect\xi}(x_3,y) \,dy\,dx_3\\[0.4em]
	&\hspace{60pt}=\int\limits_I \int\limits_{Y_0} 
	\vect f(x,y)
	\cdot \mathring{\vect \xi} (x_3,y) \,dy\,dx_3 %%\\[0.4em]
	%%&\hspace{180pt}
	\qquad\forall  \mathring{\vect \xi} \in H_{00}^1(I\times Y;\R^3),\ \  \textrm{a.e.\,} \hat{x} \in \omega.  
	\end{aligned}
	\end{equation}
If, additionally, one assumes the strong two-scale convergence in \eqref{rucak23}, then one has
$$  \vect u^\epsh(x)   \strongdrtwoscale  (\vect \funcA,\funcB)^\top+ \mathring{\vect u}(x,y).$$
\end{theorem}
\begin{corollary} 
	 \label{remsim1} 
	Under Assumption \ref{assumivan1}~(1) and provided $(\vect f^{\epsh})_{h>0}\subset L^{2, \rm memb}(\Omega;\R^3)$, in addition to the convergences in Proposition \ref{lemmaconvergencerealtime} one has
	$$ \tilde{u}_3^{\epsh} \xrightharpoonup{H^1} 0,  $$
	and thus $\funcB=0$ in the limit equations \eqref{realtimesystem}.  
	
\end{corollary}	

\begin{remark} 
	The limit system \eqref{realtimesystem} can be written as a resolvent problem on the space  $H_{\delta,\infty}(\Omega \times Y)$, as follows:\footnote{Notice that this requires to take the inner product with the weight function $\langle \rho \rangle \chi_{Y_1}+\rho_0 \chi_{Y_0}.$} 
\begin{equation*}
\left( \mathcal{A}_{\delta,\infty} + \lambda \mathcal{I}\right)\vect u = \RRR P_{\delta,\infty} \BBB \vect f,\quad \RRR \vect u=(\vect \funcA,\funcB)^\top+ \mathring{\vect u}, \BBB
\end{equation*}
\RRR which is the usual resolvent structure for the limit problem in the high-contrast setting (see \cite{pastukhova} and Section \ref{apsechyp}). \BBB
\end{remark}

Next, the operator $\tilde{\mathcal{A}}_{\delta}$ on the space $H_{\delta,\infty}(\Omega \times Y) \cap L^{2, \rm memb}(\Omega \times Y;\R^3)$ is defined via the form $\tilde a_{\delta}$ given by the expression in \eqref{formadelta} with a different domain:
$$ \tilde a_{\delta}: \left(H^1_{\gamma_{\rm D}}(\omega;\R^2)\times \{0\} + \RRR V_{2,\delta} \BBB (\Omega \times Y_0)\right)^2 \cap \left( L^{2, \rm memb}(\Omega \times Y;\R^3) \right)^2 \to \R.                      $$  
This operator can only be defined under Assumption \ref{assumivan1}\,(1). 

In relation to the limit problem, we also define the following operators. The operator $\mathcal{A}_{00,\delta},$ referred to as 
the Bloch operator, corresponds to the differential expression\footnote{The differential expression $\widetilde{\div}_{2,\delta}$ stands for $\widetilde{\nabla}_{\delta}\cdot$ (applied row-wise), i.e., it is in the same relation to the gradient $\widetilde{\nabla}_{\delta}$ as the standard divergence is in the relation to the standard gradient.}
\begin{equation*}
-(\rho_0)^{-1}\widetilde{\div}_{2,\delta}\bigl(\sym \widetilde{\nabla}_{\delta}\bigr), 
\end{equation*}
and is defined via the bilinear form
\begin{equation*}
\begin{aligned}
a_{00,\delta}( \vect u, \vect v)&:=  \int\limits_{I \times Y_0} \C_0(y) \sym\widetilde{\nabla}_{\delta}\, \vect u(x_3,y) : \sym\widetilde{\nabla}_{\delta}\, \vect v(x_3,y)\,dx_3 \,dy,\qquad
%%\\[0.4em]
%%&
a_{00,\delta} : \left( H_{00}^1(I \times Y_0;\R^3)\right)^2  \to \R.
\end{aligned}
\end{equation*}
Similarly to the way the form $\tilde{a}_{\delta}$ and the associated operator $\tilde{\mathcal{A}}_{\delta}$ were defined by restricting the form $a_{\delta,\infty}$, we define a form 
\[
\tilde{a}_{00,\delta}:\left( H_{00}^1(I \times Y_0;\R^3)\right)^2 \cap \left(L^{2, \rm memb}(I \times Y_0;\R^3)\right)^2  \to \R,
\]
and the associated operator $\tilde{\mathcal{A}}_{00,\delta}$ by restricting the form $a_{00,\delta}.$
We also  define a positive self-adjoint
operator $\mathcal{A}^{\rm memb}_{\delta}$ on $L^2(\omega;\R^2)$ corresponding to the differential expression
\begin{equation*}
%\mathcal{A}^{\rm memb}_{\delta} \vect u := 
-\langle \rho \rangle^{-1}\div_{\hat{x}}\bigl(\mathbb{C}^{\rm memb}_{\delta} \simgrad_{\hat{x}}\bigr),
\end{equation*}
as the  one defined (on an appropriate weighted $L^2$ space) by the bilinear form
\begin{equation}
	\label{resf}
\begin{split}
a^{\rm memb}_{\delta}(\vect u, \vect v) &:= \int\limits_\omega \C^{\rm memb}_{\delta}\simgrad_{\hat{x}}\vect u(\hat{x}) : \simgrad_{\hat{x}} \vect v(\hat{x}) d\hat{x}, \qquad a^{\rm memb}_{\delta}: \bigl( H_{\gamma_{\rm D}}^1(\omega;\R^2) \bigr)^2\to \R.
\end{split}
\end{equation}
In order to simplify the system \eqref{realtimesystem}, one is led to first solve the equation (where we replace $\lambda$ with $-\lambda$) 
$$  (\mathcal{A}_{00,\delta}-\lambda \mathcal{I}) \mathring{\vect u}=\lambda \bigl(\vect \funcA(\hat x), \funcB(\hat x)\bigr)^\top
+(\rho_0)^{-1}\vect f (\hat x,\cdot), $$ 
where the variable $\hat x$ is treated as a parameter \RRR(see, e.g., \cite{zhikov2012}).\BBB When $\vect f\vert_{Y_0}=0$ and $\lambda\in{\mathbb C}\setminus{\mathbb R}_+,$ the equation (\ref{resf}) can be solved via the more basic problems  
 $$  (\mathcal{A}_{00,\delta}-\lambda \mathcal{I}) \vect b^{\lambda}_{i}=\vect e_{i}, \qquad
 % \textrm{for }
 i=1,2,3.$$
The following matrix-valued function $\beta_{\delta}$ taking values in $\R^{3 \times 3}$ will be important for characterizing the spectrum of the limit operator:
\begin{equation*}
\bigl(\beta_{\delta}(\lambda)\bigr)_{ij} = \lambda \langle  \rho \rangle  \delta_{ij}  + \lambda^2 \bigl\langle\rho_0\overline{(b^\lambda_i)_j}\bigr\rangle, 
%%\cdot  \bigl\langle\rho_0\overline{\vect b^\lambda_{j,i}} \bigr\rangle,
\qquad i,j=1,2,3.
\end{equation*}
We refer to $\beta_\delta$ as ``Zhikov function", to acknowledge its scalar version appearing in \cite{Zhikov2000}. Its significance will be clear in the next section.
We can obtain an alternative representation of the Zhikov function as follows. 

First, separate the spectrum of ${\mathcal{A}}_{00,\delta}$ into two parts:
\begin{equation}
\sigma({\mathcal{A}}_{00,\delta}) = \{{\eta}_1, {\eta}_2, ...\}\cup\{{\alpha}_1, {\alpha}_2, ...\},
\label{sig_spec}
\end{equation}
where the second subset corresponds to eigenvalues with all associated eigenfunctions having zero $\rho_0$-weighted mean 
%%(where weight is the function $\rho_0$) 
in all components. In each of the two subsets the eigenvalues are assumed to be arranged in the ascending order.
%%We assume that we have ordered $\{\eta_n\}_{n \in \N}$ and $\{\alpha_n\}_{n \in \N}$ in the ascending order.\\
Next, denote by $({\vect\varphi}_n)_{n\in \N}$ the set of orthonormal eigenfunctions corresponding to the eigenvalues from the set $\{\eta_1, \eta_2, ...\}$ in (\ref{sig_spec}), repeating every eigenvalue according to its multiplicity. Using the spectral decomposition, one obtains 
 \begin{equation} 
\label{razgovor1} 
	\bigl(\beta_{\delta}(\lambda)\bigr)_{ij}=\lambda \langle  \rho \rangle  \delta_{ij} + \sum_{n=1}^{\infty}\frac{\lambda^2}{{\eta}_n-\lambda}\bigl\langle\rho_0 \overline{({\varphi}_n)_i} \bigr\rangle \cdot \bigl\langle\rho_0 \overline{({\varphi}_n)_j}\bigr\rangle,\qquad i,j=1,2,3. 
\end{equation}
 
 Under Assumption \ref{assumivan1}~(1), one is actually only interested in the operator $\tilde{\mathcal{A}}_{00,\delta}$. We can then define a version of the Zhikov function, denoted by $\tilde{\beta}^{\rm memb}_{\delta}$  and taking values in $\R^{2 \times 2}$ (dropping the third row and the third column of $\beta_\delta,$ which necessarily vanish as a consequence of symmetries) as the one associated with this operator. Eliminating those values $\eta_n$ and $\alpha_n$ in (\ref{sig_spec}) whose eigenfunctions belong to the subspace $L^{2, \rm bend}(I \times Y_0,\R^3),$   we write
$$\sigma(\tilde{\mathcal{A}}_{00,\delta})=\{\tilde{\eta}_1,\tilde{\eta}_2,\dots \}\cup\{\tilde\alpha_1,\tilde\alpha_2,\dots\}.$$
Here, similarly to the above, the eigenvalues in the second set 
%%$\tilde{\alpha}_i$ 
are those whose all eigenfunctions have zero weighted mean in all of their components. (Note that due to symmetry the third component has zero weighted mean automatically.) 
We use the notation $\sigma(\tilde{\mathcal{A}}_{00,\delta})'$ for the set of such eigenvalues:
\begin{equation*}
\sigma(\tilde{\mathcal{A}}_{00,\delta})'=\{\tilde{\alpha}_1,\tilde{\alpha}_2,\dots\}.
%\label{sigmaprime}
\end{equation*}
Analogously to \eqref{razgovor1}, we can write a formula for the function $\tilde{\beta}^{\rm memb}_{\delta}$. Notice, in particular, that 
\begin{equation}
\bigl(\tilde{\beta}^{\rm memb}_{\delta}\bigr)_{\alpha \beta}=\bigl(\beta_{\delta}\bigr)_{\alpha \beta}, \quad \alpha,\beta=1,2.
\label{betamemb1} 
\end{equation}

\vspace{2.1ex} 
\noindent\textbf{C. ``Strong high-contrast bending" scaling: $\mu_h=\epsh h,$\  $\tau=2$}
%: ``strong contrast", ``bending"}\\
\vspace{1.4ex}

As was shown above, in the case $\mu_h=\epsh,$ $\tau=2$ one does not see effects of high-contrast inclusions in the limit equations (i.e the limit equations are not coupled). Here we consider an asymptotic regime of higher contrast, where the limit equations are coupled. Proposition \ref{compactnessanotherscaling} below provides the relevant compactness result. Before proceeding to its statement, we introduce some auxiliary objects.

In order to analyse the spectral problem, we will require a positive self-adjoint operator $\hat{\mathcal{A}}_{\delta}$ on the Hilbert space $\{0\}^2 \times L^2(\omega)+L^2(\Omega \times Y_0,\R^3),$ as the one defined by the bilinear form 
\begin{align} 
	\nonumber
	\RRR \hat{a}_{\delta}\bigl((0,0,\funcB)^\top+\mathring{\vect{u}}, (0,0,\theta)^\top+\mathring{\vect {\xi}}\bigr)&= a_{\delta}^{\funcB}(\RRR \funcB \BBB,\theta)+\int_{\omega} a_{00,\delta} (\mathring{\vect{u}},\mathring{\vect {\xi}})\,d \hat{x},\\  \label{missaglic2} & 
	\hat{a}_{\delta}:\left(\{0\}^2 \times H^2_{\gamma_{\rm D}}(\omega)+V_{2,\delta}(\omega \times Y_0)\right)^2 \to \R.  	
\end{align} 

We also define a scalar Zhikov function $\hat{\beta}_{\delta}$ associated with this problem. Namely, we eliminate the eigenvalues of $\mathcal{A}_{00,\delta}$ all of whose eigenfunctions have zero weighted mean in the third component and set
\begin{equation} 
\hat{\beta}_{\delta}:= \beta_{\delta,33}.
\label{beta_deriv}
\end{equation}
We also define $\hat{\sigma}(\mathcal{A}_{00,\delta})$ as the set of the eigenvalues of $\mathcal{A}_{00,\delta}$ all of whose eigenfunctions have zero weighted mean in the third component.
 
\begin{proposition}
	\label{compactnessanotherscaling}
	Consider a sequence $\{(h, \eps_h)\}$ such that  $\delta=\lim_{h\to0} h/\eps_h\in (0,\infty),$
	%$\delta \in (0,\infty)$, 
	and suppose that $\mu_h=\epsh h$,  $\tau=2$. 
	The following statements hold: 
	%%Let  $\mu_h=\epsh h$, $\delta \in (0,\infty)$, $\tau=2$.  The following statement is valid: 
	\begin{enumerate} 
		\item 
		There exists $C>0$, independent of $h$, 
		such that for any sequence $(\vect f^{\epsh})_{h>0}\subset L^2(\Omega;\R^3)$ of load densities and the corresponding solutions $\vect u^{\epsh}$ to the problem \eqref{resolventproblemstarting}  one has  
		$$
		h^{-2}a_{\epsh}(\vect u^{\epsh},\vect u^{\epsh})+\|\vect{u}^{\epsh}\|^2_{L^2}\leq C \|\vect f^{\epsh}\|^2_{L^2}.$$
		\item If $$\limsup_{h  \to 0}\left(h^{-2}a_{\epsh}(\vect u^{\epsh},\vect u^{\epsh})+\|\vect{u}^{\epsh}\|^2_{L^2}\right)< \infty, \qquad (\vect u^{\epsh})_{h>0}\subset H^1_{\Gamma_{\rm D}}(\Omega;\R^3),
		$$  then
	 there exist functions $\vect{\funcA} \in H_{\gamma_D}(\omega;\mathbb{R}^2)$, $\funcB \in H^2_{\gamma_{\rm D}}(\omega)$, $\funcC \in {\mathfrak C}_{\delta} (\Omega \times Y)$, $\mathring{\vect u}\in V_{2,\delta}(\Omega \times Y_0)$,  such that (up to a subsequence):
	\begin{equation}
	\begin{split} \label{ds1}
	\vect u^\epsh  &= \Tilde{\vect u}^\epsh + \mathring{\vect u}^\epsh,  \qquad \Tilde{\vect u}^{\epsh},\mathring{\vect u}  \in H^1_{\Gamma_{\rm D}}(\Omega;\R^3), \qquad \mathring{\vect u}^\epsh |_{\Omega_1^{\varepsilon_h}} = 0, \\[0.15em]
	\pi_{1/h} \Tilde{\vect u}^\epsh &{\,\xrightarrow{L^2}\,}\bigl(
	\funcA_1(\hat{x}) - x_3 \partial_1 \funcB(\hat{x}), 
	\funcA_2(\hat{x}) - x_3 \partial_2 \funcB(\hat{x}),
	\funcB(\hat{x})\bigr)^\top,\\[0.3em]
	\mathring{\vect u}^\epsh(x) & \drtwoscale  \mathring{\vect u}(x,y),
	\\[0.3em]
	h^{-1}\simgrad_h \Tilde{\vect u}^\epsh(x)  &\drtwoscale \, \iota\left(\simgrad_{\hat{x}}\vect \funcA(\hat{x})-x_3 \nabla_{\hat{x}}^2\funcB(\hat{x})\right) + \funcC(x,y),
	\\[0.3em]
	\epsh\simgrad_h \mathring{\vect u}^\epsh(x)  &\drtwoscale \sym\widetilde{\nabla}_{\delta}\, \mathring{\vect u}(x,y).
	\end{split}
	\end{equation}
	\item  If, additionally to 2, one assumes that 
	\[\lim_{h \to 0} \left(h^{-2}a_{\epsh} (\vect u^{\epsh}, \vect u^{\epsh})+\|\vect u^{\epsh} \|^2_{L^2}\right)=\hat{a}_{\delta} ((0,0,\funcB)^\top+\mathring{\vect u},(0,0,\funcB)^\top+\mathring{\vect u})+\|(0,0,\funcB)^\top+\mathring{\vect u}\|^2_{L^2},
	\] 
	where $\hat{a}_{\delta}$ is defined in \eqref{missaglic2}, then one has the strong two-scale convergence 
	$$
	\vect{u}^{\epsh} \strongdrtwoscale (0,0,\funcB)^\top+\mathring{\vect u}.  
	$$
\end{enumerate} 
\end{proposition}
The following theorem describes the limit resolvent equation. 
\begin{theorem}
	\label{thmresolventhighercontrast}
	Suppose that $\delta \in (0,\infty)$, $\mu_h=\epsh h,$ and $\tau=2,$ and consider a sequence $(\vect f^\epsh)_{h>0}$ of load densities such that 
	%satisfying 
	%the following convergence: 
	\begin{equation} 
		\label{konvsila00} 
	\vect  f^\epsh \drtwoscale \vect f \in L^2(\Omega \times Y;\R^3).
	\end{equation}
	Then the sequence of solutions to the resolvent problem \eqref{resolventproblemstarting} converges in the sense of \eqref{ds1} to the unique solution of the following problem: Determine $\vect \funcA \in H^1_{\gamma_{\rm D}}(\omega;\R^2)$, $\funcB \in H^2_{\gamma_{\rm D}}(\omega)$, $\mathring{\vect u}\in V_{2,\delta} (\Omega \times Y_0)$ such that
	\begin{equation}
	\label{realtimehigherordersystem}
	\begin{split}
	& \int\limits_{\omega}\C^{\rm hom}_{\delta}\left(\simgrad_{\hat{x}}\vect \funcA(\hat{x}),\nabla_{\hat{x}}^2\funcB(\hat{x})\right): \left(\simgrad_{\hat{x}}\vect \theta_* (\hat{x}),\nabla_{\hat{x}}^2{\theta_3}(\hat{x})\right) \,d\hat{x}
	+ 
	\lambda \int\limits_\omega\bigl(\langle \rho \rangle \funcB(\hat{x})+\langle \rho_0\overline{\mathring{u}_3} \rangle(\hat{x})\bigr) {\theta_3}(\hat{x}) \,d\hat{x}
	\\[0.35em]
	& 
	\hspace{+150pt}=\int\limits_\omega \langle\overline{\vect f_3} \rangle(\hat{x}) {\theta_3}(\hat{x}) \,d\hat{x}\quad\qquad \forall{\theta_3} \in H_{\gamma_{\rm D}}^2(\omega), \\[0.35em]
	& \int\limits_{I}\int\limits_{Y_0} \C_0(y) \sym\widetilde{\nabla}_{\delta}\, \mathring{\vect u}(x,y) : \sym\widetilde{\nabla}_{\delta}\, \vect\xi(x_3,y) \,dy\,dx_3
	+\lambda \int\limits_{I}\int\limits_{Y_0}
	\rho_0(y)\funcB(\hat{x})
	\cdot \mathring{\vect\xi}_3(x_3,y) \,dy\,dx_3 \\[0.35em]
	&\hspace{2cm} +\lambda \int\limits_{I}\int\limits_{Y_0}\rho_0(y) \mathring{\vect u}(x,y) \cdot \mathring{\vect \xi}(x_3,y) \,dy\,dx_3\\[0.4em]
	&\hspace{2cm}=
	\int\limits_{I}\int\limits_{Y_0} 
	\vect f(x,y)
	\cdot \mathring{\vect \xi}(x_3,y) \,dy\,dx_3\quad\qquad 
	%%\\
	%%&\hspace{10em}
	%%\forall
	 %%\vect \theta_*  \in H_{\gamma_{\rm D}}^1(\omega;\R^2), \ 
	 \forall\mathring{\vect \xi} \in H^1_{00}(I \times Y_0;\R^3),\ \  \textrm{a.e.\,} \hat{x} \in \omega.
	\end{split}
	\end{equation}
	 In the case when the strong two-scale convergence holds in \eqref{konvsila00}, one has 
	 %strong two-scale convergence 
	 \[
	 \vect{u}^{\epsh} \strongdrtwoscale (0,0,\funcB)^\top+\mathring{\vect u}.
	 \] 
	
\end{theorem}
\RRR 
\begin{remark} 
The limit problem \eqref{realtimehigherordersystem} can again be written as the resolvent problem on $\{0\}^2 \times L^2(\omega)+L^2(\Omega \times Y_0;\R^3)$: 
$$  (\hat{\mathcal{A}}_{\delta}+\lambda \mathcal{I})\vect u=\bigl( S_2 (P_{\delta,\infty} \vect f)_1,S_2 (P_{\delta,\infty} \vect f)_2, (P_{\delta,\infty} \vect f )_3  \bigr)^\top, \quad \vect u= (0,0,\funcB)^\top+\mathring{\vect u}$$
which is again the general desired structure.
\end{remark} 
\BBB
\begin{remark} 
	Under Assumption \ref{assumivan1}~(1), the first equation in \eqref{realtimehigherordersystem} decouples from the second (see Remark \ref{remigor1}) and one has 
	\begin{equation*} 
	\begin{aligned}	& 
	\vect \funcA=0,\\[0.2em]
	&
	\int\limits_{\omega}\C^{\rm bend}_{\delta}\nabla_{\hat{x}}^2\funcB(\hat{x}):\nabla_{\hat{x}}^2{\theta_3}(\hat{x}) \,d\hat{x}
	+ 
	\lambda \int\limits_\omega\bigl(\langle \rho \rangle \funcB(\hat{x})+\langle \rho_0\overline{\mathring{u}_3} \rangle(\hat{x})\bigr) {\theta_3}(\hat{x}) \,d\hat{x}=
	%%\\ &=
	\int\limits_\omega \langle\overline{\vect f_3} \rangle(\hat{x}) {\theta_3}(\hat{x}) \,d\hat{x}\qquad \forall{\theta_3} \in H_{\gamma_{\rm D}}^2(\omega). 
	\end{aligned} 
	\end{equation*}	
\end{remark} 		
	
In the following sections we will analyse only those two cases for each regime when there is a coupling between the deformations inside and outside the inclusions. 

\subsubsection{Asymptotic regime $h \ll \varepsilon_h:$ ``very thin" plate}
\label{hlleps}

Throughout this section, we additionally assume that the set $Y_0$ has $C^{1,1}$ boundary, to ensure the validity of  
%{\color{red}seemingly unavoidable}
some auxiliary extension results, see Appendix \ref{extension_app}.

\vspace{+2.4ex}
\noindent\textbf{A. ``Membrane" scaling: $\mu_h=\epsh,\tau=0$}
\vspace{1.4ex}

Similarly to Part B of Section \ref{sectionheps}, where the membrane scaling is discussed for the regime $h\sim\varepsilon_h,$ we define the following objects \RRR using the limit resolvent in Theorem \ref{rucak35} below  (expression  \eqref{jdbaprva} for the limit resolvent)   \BBB:
\begin{itemize}
	\item
	For each $\kappa \in [0,\infty]$, a form $a_{0,\kappa}:\bigl(V_{1,0,\kappa}(\omega\times Y)+V_{2,0}(\Omega \times Y_0)\bigr)^2 \to \R$ and the associated operator $\mathcal{A}_{0,\kappa}$ on the space $V_{0,\kappa}(\Omega \times Y),$ analogous to  $a_{\delta,\infty}$ and $\mathcal{A}_{\delta,\infty}$ of Part B, Section \ref{sectionheps}. 	\RRR In this way the limit problem \eqref{jdbaprva} can be written in the form
	$$(\mathcal{A}_{0,\kappa} +\lambda \mathcal{I})\vect u= P_{\delta,\kappa}\vect f,\quad \vect u=((\vect \funcA,\funcB)^\top+\mathring{\vect u}); $$
	\BBB
	\item
	A form 
	$$\tilde{a}_0: \bigl(H^1_{\gamma_{\rm D}} (\omega;\R^2)+L^2(\omega; H_0^1(Y_0;\R^2))\bigr)^2 \to \R$$ 
	and the associated operator $\tilde{\mathcal{A}}_{0}$ on  $L^2(\omega;\R^2)+L^2(\omega \times Y_0;\R^2)$
	(analogous to $\tilde{a}_{\delta}$ and  $\tilde{\mathcal{A}}_{\delta}$) --- these are correctly defined under Assumption \ref{assumivan1}(1);
	
	\item A bilinear form
	$a_0^{\rm memb}: \bigl(H_{\gamma_{\rm D}}^1(\omega;\R^2)\bigr)^2\to \R$ and the associated
	operator $\mathcal{A}_0^{\rm memb}$ on $L^2(\omega;\R^2)$ 
	(analogous to $a_{\delta}^{\rm memb}$ and $\mathcal{A}_{\delta}^{\rm memb}$);
	
	\item
	A bilinear form 
	$\tilde{a}_{00,0}: \bigl(H^1_0(Y_0,;\R^2) \bigr)^2 \to \R$ and the associated operator $\tilde{\mathcal{A}}_{00,0}$ on $L^2(Y_0;\R^2)$
	(analogous to $\tilde{a}_{00,\delta}$ and $\tilde{\mathcal{A}}_{00,\delta}$);
	
	\item
	Functions $\beta_0$, $\tilde{\beta}^{\rm memb}_0,$ by analogy with $\beta_{\delta}$, $\tilde{\beta}^{\rm memb}_{\delta};$
	\item
	A set $\sigma(\tilde{\mathcal{A}}_{00,0})',$ by analogy with $\sigma(\tilde{\mathcal{A}}_{00,\delta})'.$ 
\end{itemize}
%We leave the details to the interested reader. 
\RRR We do not write these definitions explicitly, since we assume their definition is natural. \BBB
The following proposition provides a compactness result for solutions to (\ref{resolventproblemstarting}). 
\begin{proposition} \label{propcompactregime11}
	Suppose that $\delta=0$, $\mu_h=\epsh$,  $\tau=0$. The following statements hold: 
	\begin{enumerate} 
		\item 
		There exists $C>0$, independent of $h$,
		such that for a sequence $(\vect f^{\epsh})_{h>0}\subset L^2(\Omega;\R^3)$ of load densities and the corresponding solutions $\vect u^{\epsh}$ to the problem \eqref{resolventproblemstarting} one has
		  $$a_{\epsh}(\vect u^{\epsh},\vect u^{\epsh})+\|\vect u^{\epsh}\|^2_{L^2}\leq C \|\vect f^{\epsh}\|^2_{L^2}.$$ 
		\item If $$ \limsup_{h  \to 0} \left(a_{\epsh}(\vect u^{\epsh},\vect u^{\epsh})+\|\vect u^{\epsh}\|^2_{L^2}\right)<\infty, \quad (\vect u^{\epsh})_{h>0} \subset H^1_{\Gamma_{\rm D}}(\Omega;\R^3),$$  then
		there exist  $({\vect \funcA},\funcB)^{\top}\in V_{1,0,\kappa}(\omega\times Y)$, $\vect{\varphi}_1 \in L^2(\omega;\dot{H}^1(\mathcal{Y};\R^2))$, $\varphi_2 \in L^2(\omega;\dot{H}^2(\mathcal{Y}))$, $\mathring{\vect u} \in V_{2,0}(\Omega \times Y_0)$, $\mathring{\vect g} \in L^2(\Omega \times Y;\R^3)$,  $\mathring{\vect g}_{|\Omega \times Y_1}=0$  such that  (up to a subsequence)
	\begin{equation}\label{nada100000}
	\begin{split}  
	\vect u^\epsh & = \Tilde{\vect u}^\epsh + \mathring{\vect u}^\epsh, \quad  \Tilde{\vect u}^{\epsh},\mathring{\vect u}  \in H^1_{\Gamma_{\rm D}}(\Omega;\R^3),\quad\mathring{\vect u}^\epsh |_{\Omega_1^{\varepsilon_h}} = 0, \\
	\tilde{\vect u}^\epsh_*   &{\,\xrightarrow{L^2}\,} 
	\vect{\funcA},\\
	\tilde{\vect u}^\epsh_3  &\drtwoscale  \begin{cases} \funcB(\hat{x}),&  
		%\textrm{ if } 
	\quad\kappa=\infty,\\[0.25em]
	\funcB(\hat{x},y), &
	% \textrm{ if }
	\quad \kappa \in [0,\infty), 
	\end{cases} 
	\end{split} 
	\end{equation} 
	\begin{equation*} 
	\begin{split}
	\simgrad_h \tilde{\vect u}^\epsh(x) & \drtwoscale  \begin{cases} \iota\left(\simgrad_{\hat{x}}\vect{\funcA}(\hat{x})\right)+\funcC_0(\vect{\varphi_1}, \varphi_2, \vect g)(x,y),&\quad
		 %\textrm{ if } 
		 \kappa=\infty, \\[0.25em]
	\iota\left(\simgrad_{\hat{x}}\vect{\funcA}(\hat{x})\right)+\funcC_0(\vect{\varphi_1}, \kappa \funcB, \vect g)(x,y),&\quad
	%\textrm{ if } 
	\kappa \in (0,\infty),\\[0.25em]
	\iota\left(\simgrad_{\hat{x}}\vect{\funcA}(\hat{x})\right)+
	\funcC_0(\vect{\varphi_1}, 0, \vect g)(x,y), &\quad
	 %\textrm{if } 
	 \kappa=0,
	\end{cases} \\[0.2em]
	\mathring{\vect u}^{\epsh} &\drtwoscale 
	\mathring{\vect u}(\hat x,y), 
	\\[0.2em]
	\epsh\simgrad_h \mathring{\vect u}^\epsh & \drtwoscale  
	\funcC_0(\mathring{\vect u}_* , 0, \mathring{\mat g})(x,y).
	\end{split} 
	\end{equation*}
\item 	If, additionally to 2, one assumes that $$ \lim_{h  \to 0} \left(a_{\epsh}(\vect u^{\epsh},\vect u^{\epsh})+\|\vect u^{\epsh}\|^2_{L^2}\right)=a_{0,\kappa} ((\RRR \vect \funcA\BBB,\funcB)^\top+\mathring{\vect u},(\RRR \vect \funcA \BBB,\funcB)^\top+\mathring{\vect u})+\|(\vect \funcA,\funcB)^\top+\mathring{\vect u}\|^2_{L^2}, $$ where the form $a_{0,\kappa}$ is defined above, then one has
%% we have strong two-scale convergence 
$$\vect u^{\epsh} \strongdrtwoscale (\RRR\vect \funcA\BBB,\funcB)^\top +\mathring{\vect{u}}.$$
\end{enumerate}
\end{proposition}
\begin{remark} \label{remzzzz1} 
	In the regimes $h \sim {\epshtwo}$ and $h \ll {\epshtwo}$ we are not able to identify the functions  $\funcB(\hat{x},y)$ and $\mathring{u}_3$ separately on $\omega \times Y_0$ (in the following theorem). However, we are able to identify their sum,  which is the only relevant object, since the third component of solution converges to their sum. Thus we artificially set $\funcB(\hat{x},y)=0$ on $\omega \times Y_0$, to have uniqueness of the solution of the limit problem.  In the case when $\funcB$ is a function of $\hat{x}$ only, the decomposition $\funcB(\hat x)+\mathring{u}_3(\hat{x},y)$ is unique in $L^2(\omega \times Y)$, since we know that $\mathring{u}_{3|\omega \times Y_1}=0$. 
\end{remark} 	
The limit resolvent problem for the model of homogenized plate is given by the following theorem. 
\begin{theorem}\label{rucak35}
	Let $\delta=0$, $\mu_h=\epsh$,  $\tau=0$ and suppose that the sequence of load densities converge as follows:
	\begin{equation}\label{rucak22} 
	\vect  f^\epsh \drtwoscale \vect f \in L^2(\Omega \times Y;\R^3).
	\end{equation}
Then the sequence of solutions   to the resolvent problem \eqref{resolventproblemstarting}  converges in the sense of 
%%Proposition \ref{propcompactregime11} 
\eqref{nada100000} to the  unique solution (see also Remark \ref{remzzzz1}) of the following problems:
 Find $(\vect \funcA,\funcB)^\top \in V_{1,0,\kappa}(\omega \times Y)$, $\mathring{\vect u} \in V_{2,0} (\Omega \times Y_0)$ such that
\begin{equation}
	\label{jdbaprva}
	\begin{aligned} 
		&\int\limits_{\omega}\C_{1}^{\rm memb,r}\simgrad_{\hat{x}} {\vect \funcA}(\hat{x}): \simgrad_{\hat{x}}{\vect \theta}_*(\hat{x}) \,d\hat{x}
		+ 
		\lambda \int\limits_\omega \bigl(\langle \overline{\rho} \rangle{\vect \funcA}(\hat{x})+\langle\rho_0 \mathring{\vect u}_*  \rangle\bigr) \cdot {\vect \theta}_*(\hat{x}) \,d\hat{x} \\		&\hspace{+120pt}=\int\limits_{\omega} 
		\langle\overline{ \vect{f}_* } \rangle (\hat{x}) 
		\cdot  {\vect\theta}_*(\hat{x})\,d\hat{x}\qquad \forall {\vect \theta}_*\in H_{\gamma_{\rm D}}^1(\omega;\R^2),\\ 
		&\int\limits_{Y_0} \C_{0}^{\rm memb,r}(y) \sym \nabla_y { \mathring{ \vect u}}_* (\hat{x},y): \sym \nabla_y{\mathring{\vect \xi}}_*(y)\,dy  +
		\lambda\int_{Y_0}\rho_0(y)\RRR \left(\vect \funcA(\hat{x})+\mathring{\vect u}_{*}(\hat x,y) \right)
		%%\cdot \mathring{\vect \xi}(y) \,dy +\lambda \int\limits_{Y_0} \rho_0(y){\mathring{\vect u}}_* (\hat{x},y) 
		\cdot \mathring{\vect{\xi}}_*(y)\,dy,\\
		&\hspace{+120pt}=\int\limits_{Y_0} 
		\overline{ \vect f_{\!*} } (\hat{x},y) 
		\cdot  \mathring{\vect\xi}_*  (y) \,dy\qquad\forall\mathring{\vect \xi}_* \in H_0^1(Y_0;\R^2), \  \textrm{a.e.\,} \hat{x} \in \omega,  \\
			&\langle \rho \rangle\funcB (\hat{x}) +  \rho_0(y)\mathring{u}_{3} (\hat{x},y) = \lambda^{-1}P^0\overline {f}_3 (\hat{x},y),
			\qquad 
			%%\textrm{ if }
			 \kappa=\infty, 
			 \\[0.4em]
		&\left.\!\!\!\!\!\begin{array}{lll}	\dfrac{\kappa^2}{12} \int\limits_{Y_1}\C_{1}^r(y)\nabla^2_y \funcB(\hat x,y): \nabla_y^2 v(y)  \,dy 
			+ 
			\lambda \int\limits_{Y_1}\rho_1(y) \funcB(\hat{x},y) v(y) \,dy\\[1.2em]
			%%&
			\hspace{+130pt}
			=\int\limits_{Y_1} 
			\overline{f}_3(\hat{x},y) v(y)  \, dy\qquad \forall v \in H^2(\mathcal{Y}),\  \textrm{a.e.\,} \hat{x} \in \omega. 
			%%\textrm { if } 
			%%\kappa \in (0,\infty), 
			\\[1.2em]
				\rho_0(y)\mathring{u}_3(\hat x,y)= \lambda^{-1}\overline{f}_3 (\hat x,y), \quad \funcB(\hat x,y)=0,\qquad 
				%%\textrm{ on }
				 y\in Y_0\end{array}\right\}\qquad 				
			%\textrm { if } 
				\kappa \in (0,\infty),
			\\[0.6em] 
			&\rho_1(y)\funcB(\hat{x},y)+\rho_0(y)\mathring{u}_3(\hat x, y)=\lambda^{-1}\overline{f}_3 (\hat{x},y);\quad \funcB(\hat x,y)=0,\quad 
			%\textrm{ on } 
			y\in Y_0,\qquad
			%%\textrm{ if } 
			\kappa=0.
			%%\\
	%%%	& \forall \vect \theta \in H_{\gamma_{\rm D}}^1(\omega;\R^2), \  \mathring{\vect \xi} \in H_0^1(Y_0;\R^2), \  v \in H^2(\mathcal{Y}).
		\end{aligned} 
		\end{equation} 	
If additionally we assume the strong two-scale convergence in \eqref{rucak22}, then we additionally have the strong two-scale convergence 
$$  \vect u^\epsh(x)   \strongdrtwoscale  (\vect \funcA,\funcB)^\top+ \mathring{\vect u}(x,y).$$
\end{theorem}
\begin{corollary}  
	\label{rucak45} 
	Under the Assumption \ref{assumivan1}~(1) and provided $(\vect f^{\epsh})_{h>0}\subset L^{2, \rm memb}(\Omega^h,\R^3)$, in addition to the convergences stated in Proposition \ref{propcompactregime11}, we have
	$$ \tilde{u}_3^{\epsh} \xrightharpoonup{H^1} 0, \quad \mathring{u}_3^{\epsh} \xrightarrow{L^2} 0,$$ 
	and thus we also have that $\funcB=\mathring{u}_3=0$ in the limit equations \eqref{jdbaprva}.
\end{corollary}

\vspace{+2.5ex}
\noindent\textbf{B. ``Bending" scaling: $\mu_h={\epshtwo},\tau=2$}
\vspace{+1.5ex}

In the regime $h\ll \epsh$ the band gap structure of the spectrum for the spectrum of order $h^2$ appears when we scale the coefficients in  inclusions with ${{\epshfour}}$. This can be seen from the a priori estimates obtained in Appendix. 

We define, for every $\vect f \in L^2(\Omega;\R^3),$ 
\begin{align}\label{pomoc5} \nonumber
	\mathcal{F}_0(\vect f)\bigl((0,0,\theta)^\top+\mathring{\vect \xi}\bigr)&:= \int_{\omega} \langle \overline{f}_3 \rangle\theta\, d \hat{x}+\int\limits_{\omega} \int\limits_{Y_0} 
	\overline{ \vect f_{\!*} } (\hat{x},y) 
	\cdot  \mathring{\vect\xi}_*   (x,y) \,dy \,d \hat{x},\\
	& + \int\limits_{\omega}\int\limits_{Y_0} \overline{f}_3 (\hat{x},y)\,\mathring{{\xi}}_3 (\hat{x},y) \, dy\,d\hat{x}
	-\int\limits_{\omega}\int\limits_{Y_0} \overline{ x_3 \vect f}_*  (\hat{x},y) \cdot \nabla_y{\mathring{\xi}}_3(\hat{x},y) \,dy\, d\hat{x},\\ \nonumber  & \theta \in  L^2(\omega),\quad \mathring{\vect \xi}=(\mathring{\vect \xi}_*,\xi_3) \in L^2\bigl(\omega; H_0^1(Y_0;\R^2)\times H_0^2(Y_0)\bigr).
\end{align}
 Furthermore, in connection with the limit problem described in Theorem \ref{thmivan113} (\RRR expression \eqref{nada1000} below)\BBB, we introduce several objects:
 \begin{itemize}
 	\item 
 A bilinear form 
 $$  
 a^{\rm hom}_0: \bigl(H^2_{\gamma_{\rm D}}(\omega)\bigr)^2 \to \R
 $$ 
 and the associated operator $\mathcal{A}_0^{\rm hom}$ on $L^2(\omega),$ analogous to $a_{\delta}^{\rm hom}$ and $\mathcal{A}_{\delta}^{\rm hom}$ of Part A, Section \ref{sectionheps} \RRR (notice that here the situation is simpler since necessarily $\vect \funcA=0$)\BBB;
\item 
The bilinear form
\begin{equation*}
	%\begin{split}
	\hat{a}_{00,0}( \mathring{u}, \mathring{\xi}) :=  \frac{1}{12}\int\limits_{Y_0} \C_{0}^{\rm bend,r}(y) \sym\nabla_y^2 \mathring{u} : \sym\nabla^2 \mathring{\xi}\,dy,  \qquad \hat{a}_{00,0} : \bigl( H_{0}^2(Y_0)\bigr)^2  \to \R
	%\end{split}
\end{equation*} 
and the associated ``Bloch operator" $\hat{\mathcal{A}}_{00,0}$ on $L^2(Y_0).$   
%%Using the operator  we can define 
\item
A scalar Zhikov function $\hat{\beta}_0$ defined via the operator $\hat{\mathcal{A}}_{00,0}$  (analogous to $\hat{\beta}_{\delta}$ defined via the operator $\mathcal{A}_{00,\delta},$ see Part C, Section \ref{sectionheps});
\item
A set $\hat{\sigma}(\hat{\mathcal{A}}_{00,0})$ (analogous to $\hat{\sigma}(\mathcal{A}_{00,\delta}));$ 
\item 
The bilinear form
%\begin{eqnarray*}
\begin{equation*} 
	\hat{a}_0\bigl(\funcB+\mathring{u},\theta+\mathring{\xi}\bigr)=a_0^{\rm hom}(\funcB,\theta)+\int_{\omega}\hat{a}_{00,0}(\mathring{u},\mathring{\xi}), \quad\qquad \hat{a}_{0}:\Bigl(H^2_{\gamma_{\rm D}}(\omega)+L^2(\omega;H_0^2(Y_0))\Bigr)^2 \to \R.   
	%\end{eqnarray*}
\end{equation*}
and the corresponding operator $\hat{\mathcal{A}}_0$ on $L^2(\omega)+L^2(\omega;H_0^2(Y_0)).$ 
\end{itemize}
 
The following proposition gives a suitable compactness result for the regime $h\ll\varepsilon_h.$ 
\begin{proposition} \label{propcompactregime}
	Let   $\delta=0$, $\mu_h={\epshtwo}$, $\tau=2$. The following statements hold: 
	\begin{enumerate} 
		
		\item  There exists $C>0$, independent of $h$, such that for any sequence  $(\vect f^{\eps_h})_{h>0} \subset  L^2(\Omega;\R^3)$ of load densities and the corresponding solutions $\vect u^{\epsh}$ to the problem \eqref{resolventproblemstarting}  one has $$h^{-2}a_{\epsh}(\vect u^{\epsh},\vect u^{\epsh})+\|\vect{u}^{\epsh}\|^2_{L^2}\leq C \bigl\|\pi_{h/\epsh}\vect f^{\epsh}\bigr\|^2_{L^2(\Omega;\R^3)}.$$

	\item If 
	$$
	\limsup_{h  \to 0}\left( h^{-2}a_{\epsh}(\vect u^{\epsh},\vect u^{\epsh})+\|\vect{u}^{\epsh}\|^2_{L^2}\right)< \infty, \quad (\vect u^{\epsh})\subset H^1_{\Gamma_{\rm D}}(\Omega;\R^3),
	$$  
	then there exist ${\vect \funcA} \in H_{\gamma_{\rm D}}^1(\omega;\R^2)$, $\funcB \in H_{\gamma_{\rm D}}^2(\omega)$, $\funcC \in {\mathfrak C}_0(\Omega \times Y)$, $\mathring{u}_{\alpha} \in L^2(\omega;H_0^1(Y_0)),$ $\alpha=1,2,$  $\mathring{u}_3 \in L^2(\omega; H_0^2(Y_0)) $,   $\mathring{\vect g} \in L^2(\Omega \times Y;\R^3)$, $\mathring{\vect g}_{|\Omega \times Y_1}=0$  such that (up to subsequence)
	\begin{equation}\label{nada100} 
	\begin{split}
	\vect u^\epsh & = \Tilde{\vect u}^\epsh + \mathring{\vect u}^\epsh, \quad  \Tilde{\vect u}^{\epsh},\mathring{\vect u}^{\epsh}  \in H^1_{\Gamma_{\rm D}}(\Omega;\R^3),\quad \mathring{\vect u}^\epsh |_{\Omega_1^{\varepsilon_h}} = 0, \\[0.3em]
	\pi_{1/h} \Tilde{\vect u}^\epsh & {\,\xrightarrow{L^2}\,}\bigl(
	{\funcA}_1(\hat{x})-x_3\partial_1 \funcB(\hat{x}), 
	{\funcA}_2(\hat{x})-x_3\partial_2 \funcB(\hat{x}),
	\funcB(\hat{x})\bigr)^\top, 
	\\[0.3em]
		\pi_{\epsh/h} \mathring{\vect u}^\epsh &\drtwoscale \bigl(
	\mathring{u}_1(\hat x,y)-x_3 \partial_{y_1} \mathring{u}_3 (\hat{x},y),  \mathring{u}_2(\hat{x},y)-x_3 \partial_{y_2} \mathring{u}_3(\hat{x},y), 
	\mathring{u}_3(\hat{x},y)\bigr)^\top
	\\[0.3em]
	h^{-1}\simgrad_h \Tilde{\vect u}^\epsh(x) & \drtwoscale \, \iota\bigl(\simgrad_{\hat{x}}\vect{\funcA}(\hat{x})-x_3 \nabla^2_{\hat{x}}\funcB(\hat{x})\bigr) + \funcC(x,y),
	\\[0.3em]
	\epshtwo{h}^{-1}\simgrad_h \mathring{\vect u}^\epsh & \drtwoscale  
	\funcC_0(\mathring{\vect u}_* , \mathring{u}_3, \mathring{\mat g})(x,y).
	\end{split}
	\end{equation} 
\item  If additionally to 2, one has  
\[
\lim_{h \to 0} \left(h^{-2} a_{\epsh} (\vect u^{\epsh}, \vect u^{\epsh})+\|\vect u^{\epsh} \|^2_{L^2}\right)=\hat{a}_{0} (\funcB+\mathring{u}_3, \funcB+\mathring{u}_3)+\|\funcB+\mathring{u}_3\|^2_{L^2},
\]
where $\hat{a}_{0}$ is defined below, then the strong two-scale convergence holds: 
$$\pi_{\epsh/h}\vect{u}^{\epsh} \strongdrtwoscale (0,0,\funcB+\mathring{u}_3)^\top.  $$
\end{enumerate} 
\end{proposition} 
The following theorem provides the limit resolvent equation. 
\begin{theorem}
	 \label{thmivan113}
	Let $\delta=0$, $\mu_h={\epshtwo}$, $\tau=2$ and let the sequence of load densities satisfy
	\begin{equation} \label{konvsila11} 
	\pi_{h/\epsh}\vect  f^\epsh \drtwoscale \vect f \in L^2(\Omega \times Y;\R^3).
	\end{equation}
	Then the sequence of solutions to the resolvent problem \eqref{resolventproblemstarting} converges in the sense of 
	%%Proposition \ref{propcompactregime} 
	\eqref{nada100}  to the unique solution of the following problem:
	Determine ${\vect \funcA} \in H^1_{\gamma_{\rm D}}(\omega;\R^2)$, $\funcB \in H^2_{\gamma_{\rm D}}(\omega)$, ${\mathring{u}}_{\alpha}\in L^2(\omega,H_0^1(Y_0)),$ $\alpha=1,2$, ${\mathring{u}}_3\in L^2(\omega,H_0^2(Y_0))$ such that
	\begin{equation}
		\label{nada1000}  
	\begin{aligned} 
	& \vect \funcA=0, \\
	& \frac{1}{12}\int\limits_{\omega}\C^{\rm bend,r}_{1}\nabla_{\hat{x}}^2{ \funcB}(\hat{x}): \nabla_{\hat{x}}^2{\theta}_3(\hat{x}) \,d\hat{x}
	+ 
	\lambda \int\limits_\omega \bigl(\langle \rho \rangle {\funcB}(\hat{x})
	%% \cdot {\vect \theta_3}(\hat{x}) \,d\hat{x} %%\\
	%%& +
	%%\lambda \int\limits_\omega
	 +\langle\rho_0 \mathring{u}_3 \rangle(\hat{x})\bigr) \,{\theta_3}(\hat{x}) \,d\hat{x}\\
	&\hspace{+80pt}=\int\limits_{\omega} \langle  \overline{{f}}_3 \rangle(\hat{x}){{\theta}}_3 (\hat{x}) \,d\hat{x}\qquad \forall\theta_3 \in H_{\gamma_{\rm D}}^2(\omega), \\
	& \int\limits_{Y_0} \C_{0}^{\rm memb,r}(y) \sym \nabla_y { \mathring{ \vect u}_* }(\hat{x},y): \nabla_y{\mathring{\vect \xi}_* }(y)\,dy  
	=
	\int\limits_{Y_0} 
	\overline{ \vect f_{\!*} } (\hat{x},y) 
	\cdot  \mathring{\vect\xi}_*   (y) \,dy\qquad
	\forall\mathring{\vect\xi}_* \in H_0^1(Y_0;{\mathbb R}^2), \ \textrm{a.e.\,} \hat{x} \in \omega.
	 %%\forall\mathring{\xi}_{\alpha} \in H_0^1(Y_0),\ \ \alpha=1,2, 
	\\
	& \frac{1}{12}\int\limits_{Y_0} \C_{0}^{\rm bend,r}(y) \nabla_y^2 {\mathring{u}}_3(\hat{x},y): \nabla_y^2{\mathring{{\xi}}}_3(y)\,dy  
	+
	\lambda \int\limits_{Y_0}\rho_0(y)
	\bigl(\funcB(\hat{x})+{\mathring{u}_3}(\hat{x},y)\bigr)\mathring{{\xi}_3}(y) \,dy\qquad
	\\ &
	\hspace{+80pt}=\int\limits_{Y_0} \overline{f}_3 (\hat{x},y)\mathring{{\xi}}_3 (y) \, dy
	-\int\limits_{Y_0} \overline{x_3 \vect f_*}  (\hat{x},y) \cdot \nabla_y{\mathring{\xi}}_3(y) \,dy
	%%\\ &
	%%%\forall \vect \theta_{\alpha} \in H_{\gamma_{\rm D}}^1(\omega), \textrm{ for } \alpha=1,2, \ 
  \qquad\forall\mathring{\xi}_3 \in H_0^2(Y_0), \  \textrm{a.e.\,} \hat{x} \in \omega.
	\end{aligned}
	\end{equation}
	 If the strong two-scale convergence in \eqref{konvsila11} holds, then additionally one has
	 $$ 
	 \pi_{\epsh/h} \mathring{\vect u}^\epsh \strongdrtwoscale\bigl(
	 \mathring{u}_1(\hat x,y)-x_3 \partial_{y_1} \mathring{u}_3 (\hat{x},y),  \mathring{u}_2(\hat{x},y)-x_3 \partial_{y_2} \mathring{u}_3(\hat{x},y), 
	 \mathring{u}_3(\hat{x},y)\bigr)^\top.
	 $$
\end{theorem}
The right-hand side of \eqref{nada1000} can be interpreted as the element of dual of 
\[
\{0\}^2 \times L^2(\omega)+L^2(\omega; H_0^1(Y_0;\R^2)\times H_0^2(Y_0)).
\] 

Notice that the second equation in \eqref{nada1000} is completely separated from the rest of the system.

\subsubsection{Asymptotic regime $\epsh \ll h:$ ``moderately thin" plate}
\label{epsllh}

\vspace{+2ex}
\noindent\textbf{A. ``Membrane'' scaling: $\mu_h=\epsh,$\ $\tau=0$}
\vspace{1.4ex}

Similarly to Section \ref{sectionheps}, we define the following objects \RRR using Theorem \ref{rucak50} (the expression for the limit resolvent \eqref{limjdba11}) \BBB:
\begin{itemize}
	\item 
	A bilinear form 
	$$
	a_{\infty,\infty}:\left(V_{1,\infty,\infty}(\omega \times Y)+ V_{2,\infty}(\Omega \times Y_0)\right)^2 \to \R 
	$$
	and the associated operator $\mathcal{A}_{\infty,\infty}$ on the space $H_{\infty,\infty}(\Omega \times Y)$, analogous to $a_{\delta,\infty}$ and  $\mathcal{A}_{\delta,\infty}$ of Part B, Section \ref{sectionheps}. \RRR In this way the limit problem \eqref{limjdba11} can be written in the form 
	$$ (\mathcal{A}_{\infty,\infty}+\lambda \mathcal{I}) \vect u=P_{\infty,\infty} \vect f, \quad \vect u=(\vect \funcA,\funcB)^\top+\mathring{\vect u};   $$
	\BBB
	\item   
	A bilinear form 
	$$
	\tilde{a}_{\infty}: \left(H^1_{\gamma_{\rm D}} (\omega;\R^2)\times \{0\}+V_{2,\infty}(\Omega \times Y_0)\right)^2 \cap  \left(L^{2, \rm memb}(\Omega\times Y,\R^3)\right)^2 \to \R,
	$$
 and the associated operator $\tilde{\mathcal{A}}_{\infty}$ on the space $$\left(L^2(\omega;\R^2)\times\{0\}+L^2(\Omega \times Y_0;\R^3)\right) \cap L^{2, \rm memb}(\Omega\times Y,\R^3)$$ (analogous to $\tilde{a}_{\delta}$ and $\tilde{\mathcal{A}}_{\delta}$) --- these are correctly defined under Assumption \ref{assumivan1}\,(1);
 \item
  A bilinear form 
  $a_{\infty}^{\rm memb}: (H_{\gamma_{\rm D}}^1(\omega;\R^2))^2\to \R$ and the associated operator $\mathcal{A}_{\infty}^{\rm memb}$ on $L^2(\omega;\R^2)$
(analogous to $a_{\delta}^{\rm memb}$ and $\mathcal{A}_{\delta}^{\rm memb}$);
\item 
A bilinear form  $\tilde{a}_{00,\infty}: (H^1_0(Y_0;\R^3))^2 \to \R$ and the associated operator $\tilde{\mathcal{A}}_{00,\infty}$ on $L^2(Y_0;\R^3)$
(analogous to $\tilde{a}_{00,\delta}$ and $\tilde{\mathcal{A}}_{00,\delta}$);
\item 
Functions $\beta_{\infty}$, $\tilde{\beta}^{\rm memb}_{\infty},$ by analogy with $\beta_{\delta},$ $\tilde{\beta}^{\rm memb}_{\delta};$
\item
  A set $\sigma(\tilde{\mathcal{A}}_{00,\infty})',$ by analogy with $\sigma(\tilde{\mathcal{A}}_{00,\delta})'$. 
  %%%We leave the details to the interested reader. \end{itemize}
\end{itemize}

As in the case of other regimes, we first prove an appropriate compactness result, as follows. 
\begin{proposition} \label{propcompactregime11111}
	Let $\delta=\infty$, $\mu_h=\epsh,\tau=0$. The following statements hold:
	\begin{enumerate}  
		\item 
		There exists $C>0$, independent of $h$, such that for any sequence
		 $(\vect f^{\epsh})_{h>0} \subset L^2(\Omega;\R^3)$ of load densities and the corresponding 
		solutions $\vect u^{\epsh}$ to the problem \eqref{resolventproblemstarting}  one has 
		$$ 
		a_{\epsh}(\vect u^{\epsh},\vect u^{\epsh})+\|\vect u^{\epsh}\|^2_{L^2}\leq C  \|\vect f^{\epsh}\|^2_{L^2}.$$ 
		\item If  $$\limsup_{h  \to 0} \left(a_{\epsh}(\vect u^{\epsh},\vect u^{\epsh})+\|\vect u^{\epsh}\|^2_{L^2}\right)< \infty, \quad (\vect u^{\epsh})_{h>0}\subset H^1_{\Gamma_{\rm D}}(\Omega;\R^3),
		$$ 
		there exist $(\vect \funcA,\funcB)^\top \in V_{1,\infty,\infty}(\omega \times Y)$, $\mathring{\vect u} \in H_{\infty,\infty}(\Omega \times Y)$, 
		$\funcC \in {\mathfrak C}_{\infty}(\Omega \times Y)$ such that (up to subsequence)  
		\begin{equation} 
		\begin{split} \label{ds100} 
		\vect u^\epsh & = \Tilde{\vect u}^\epsh + \mathring{\vect u}^\epsh, \quad  \Tilde{\vect u}^{\epsh},\mathring{\vect u}  \in H^1_{\Gamma_{\rm D}}(\Omega;\R^3),\quad\mathring{\vect u}^\epsh |_{\Omega_1^{\varepsilon_h}} = 0, \\
		\tilde{\vect u}^\epsh_*   &{\,\xrightarrow{L^2}\,} 
		\vect{\funcA}, 
		\\ 
		\tilde{\vect u}^\epsh_3  &\drtwoscale   \funcB(\hat{x}),
		\\ 
		\mathring{\vect u}^{\epsh}  &\drtwoscale   \mathring{\vect u},
				\\ 
		\simgrad_h \tilde{\vect u}^\epsh(x)  &\drtwoscale  
		\iota\left(\simgrad_{\hat{x}}\vect{\funcA}(\hat{x})\right)+\funcC(x,y),
		\\
		\epsh\simgrad_h \mathring{\vect u}^\epsh  &\drtwoscale  
		\sym \iota (\nabla_y \mathring{\vect u}).
		\end{split}
		\end{equation} 	
	\item If, additionally to 2, one has: 
	$$ 
	\lim_{h  \to 0} \left(a_{\epsh}(\vect u^{\epsh},\vect u^{\epsh})+\|\vect u^{\epsh}\|^2_{L^2}\right)=a_{\infty,\infty}\bigl((\RRR\vect \funcA \BBB,0)^\top+\mathring{\vect u},(\RRR \vect \funcA \BBB,0)^\top+\mathring{\vect u}\bigr)+\bigl\|(\vect \funcA,\funcB)^\top+\mathring{\vect u}\bigr\|^2_{L^2},
	$$ 
	where the form $a_{\infty,\infty}$ is defined above, then we have strong two-scale convergence
	$$ \vect{u}^{\epsh} \strongdrtwoscale (\RRR \vect \funcA \BBB,\funcB)^\top+\mathring{\vect u}.$$
	\end{enumerate} 
	\end{proposition}
The following theorem provides the limit resolvent equation. 
\begin{theorem} \label{rucak50} 
	Let $\delta=\infty$, $\mu_h=\epsh,\tau=0$ and let the sequence of load densities satisfy the following convergence: 
	\begin{equation} \label{rucak20} 
	\vect  f^\epsh \drtwoscale \vect f \in L^2(\Omega \times Y;\R^3).
	\end{equation}		
	The sequence of solutions  to the resolvent problem \eqref{resolventproblemstarting}  converges in the sense of 
	%%Proposition \ref{propcompactregime11111} 
	\eqref{ds100} to the unique  solution of the following problem: Find $(\vect \funcA,\funcB)^\top \in V_{1,\infty,\infty}(\omega \times Y)$, $\mathring{\vect u} \in V_{2,\infty}(\Omega \times Y_0)$ such that
	\begin{equation}
		\label{limjdba11} 
	\begin{aligned} 
	&\int\limits_{\omega}\C^{\rm memb,h}\simgrad_{\hat{x}} {\vect \funcA}(\hat{x}): \simgrad_{\hat{x}}{\vect \theta}_*(\hat{x}) \,d\hat{x}
	+ \lambda \int\limits_\omega \Bigl(\langle \rho \rangle{\vect \funcA}(\hat{x})+\langle \rho_0 \overline{\mathring{\vect u}_* }\rangle (\hat{x}) \Bigr) \cdot {\vect \theta}_*(\hat{x}) \,d\hat{x} 
	\\
	&\hspace{+180pt}=\int\limits_{\omega} 
	\langle\overline{ \vect{f}_* }  \rangle (\hat{x}) 
	\cdot  {\vect\theta}_*  (\hat{x})\,d\hat{x}\qquad \forall{\vect \theta}_*\in H_{\gamma_{\rm D}}^1(\omega;\R^2),
	\\ 
	&\langle \rho \rangle\funcB (\hat{x})+\bigl\langle \rho_0 \overline{\mathring{u}}_3\bigr\rangle (\hat{x})
	=\lambda^{-1}\langle \overline{f}_3\rangle(\hat{x}), 
	\\[0.6em]
	& \int\limits_{Y_0} \C_0 (y) \sym \iota \bigl(\nabla_y { \mathring{ \vect u}} (x,y)\bigr):\sym \iota \bigl( \nabla_y{\mathring{\vect \xi}(y)}\bigr) \,dy
	\\[0.4em]  
	&\hspace{+90pt}+\lambda\int_{Y_0}\rho_0(y)\bigl\{(\funcA_1(\hat{x}), \funcA_2(\hat{x}), \funcB(\hat{x}))^\top 
	%%%\cdot \mathring{\vect \xi}(y) \,dy\\ & +\lambda %%%\int\limits_{Y_0} \rho_0(y)
	+{\mathring{\vect u}}(x,y)\bigr\}\cdot \mathring{\vect{\xi}}(y) \,dy\\[0.4em]
	&\hspace{+180pt}=\int\limits_{Y_0}{\vect f} (x,y) 
	\cdot  \mathring{\vect\xi}  (y) \,dy \qquad\mathring{\vect \xi} \in H_0^1(Y_0;\R^3),\   \textrm{a.e.\,} x \in \Omega.
	\end{aligned}
	\end{equation} 
If we assume the strong two-scale convergence in \eqref{rucak20}, then the strong two-scale convergence 
\[
\vect u^\epsh(x)   \strongdrtwoscale  (\vect \funcA,\funcB)^\top+ \mathring{\vect u}(x,y)
\] 
holds.
	
\end{theorem}
\begin{corollary} 
	\label{corivan11} 
	Under Assumption \ref{assumivan1}~(1) and  provided $(\vect f^{\epsh})_{h>0}\subset L^{2, \rm memb}(\Omega;\R^3),$  in addition to the convergences in Proposition \eqref{propcompactregime11111} we have  
	\[
	\tilde{ u}^{\epsh}_3 \xrightarrow{L^2} 0, 
	\]	
	and thus $\funcB=0$ in the limit equations \eqref{limjdba11}.
\end{corollary} 
Notice that the variable $x_3$ is also just the parameter in the last equation in \eqref{limjdba11}.

\vspace{+2.7ex}
\noindent\textbf{B. ``Strong high-contrast bending" scaling $\mu_h=\epsh h,$\ $\tau=2$}
	%% ``strong contrast", "bending"} \\
\vspace{1.4ex}

Here we define the following objects \RRR using Theorem \ref{rucak60} (the expression \eqref{pomoc3} for the limit resolvent): \BBB
\begin{itemize}
	\item 
A bilinear form $a^{\rm hom}_{\infty}: (H^2_{\gamma_{\rm D}}(\omega))^2 \to \R$ and the associated operator $\mathcal{A}_{\infty}^{\rm hom}$ on $L^2(\omega),$ analogous to $a_{\delta}^{\rm hom}$ and $\mathcal{A}_{\delta}^{\rm hom}$ of Part A, Section \ref{sectionheps} \RRR (notice that here the situation is simpler since necessarily $\vect \funcA=0$)\BBB;
\item 
A scalar Zhikov function $\hat{\beta}_{\infty},$ analogous to $\hat{\beta}_{\delta}$ of Part C, Section \ref{sectionheps}, so similarly to (\ref{beta_deriv}) we have 
\[
\hat{\beta}_{\infty}=\beta_{\infty,33};
\]
\item
A set  $\hat{\sigma}(\mathcal{A}_{00,\infty}),$ analogous to $\hat{\sigma}(\mathcal{A}_{00,\delta});$ 
\item 
The operator $\hat{\mathcal{A}}_{\infty}$ on $\{0\}^2 \times L^2(\omega)+L^2(\Omega \times Y_0;\R^3)$ defined via the bilinear form 
\begin{equation}
\begin{aligned} 
	\hat{a}_{\infty}\bigl((0,0,\funcB)^\top+\mathring{\vect u},(0,0,\theta)^\top+\mathring{\vect \xi}\bigr)&=\RRR a_{\infty}^{\rm hom} \BBB(\funcB,\theta)+\int_{\Omega}\tilde{a}_{00,\infty}(\mathring{\vect u},\mathring{\vect \xi}), \\[0.4em] &\RRR \hat{a}_{\infty}\BBB:\left( \{0\}^2\times  H^2_{\gamma_{\rm D}}(\omega)+\RRR V_{2,\infty}\BBB(\Omega \times Y_0)\right)^2 \to \R.   
\end{aligned}
\label{a_infty_hat}
\end{equation}
\end{itemize}

Similarly to the regimes discussed above, a suitable compactness result is proved. 
\begin{proposition} \label{ds102}
	Let  $\delta=\infty$, $\mu_h=\epsh h$, $\tau=2$.
 The following statements hold: 
	\begin{enumerate} 
		\item 	There exists $C>0$, independent of $h$, such that for a sequence $(\vect f^{\epsh})_{h>0}\subset L^2(\Omega;\R^3)$ of load densities and 
		solutions $\vect u^{\epsh}$ to the problem \eqref{resolventproblemstarting} one has 
		$$
		h^{-2} a_{\epsh}(\vect u^{\epsh},\vect u^{\epsh})+\|\vect{u}^{\epsh}\|^2_{L^2} \leq C \|\vect f^{\epsh}\|^2_{L^2}.
		$$
		\item If 
		$$
		\limsup_{h  \to 0}\left( h^{-2} a_{\epsh}(\vect u^{\epsh},\vect u^{\epsh})+\|\vect{u}^{\epsh}\|^2_{L^2}\right)< \infty, \quad (\vect u^{\epsh})_{h>0}\subset H^1_{\Gamma_{\rm D}}(\Omega;\R^3),
		$$  
		then
		there exist functions  $\vect{\funcA}\in H^1_{\gamma_{\rm D}}(\omega;\R^2)$, $\funcB \in H^2_{\gamma_{\rm D}}(\omega)$, $\mathring{\vect u} \in V_{2,\infty}(\Omega \times Y_0)$, $\funcC \in \funcC_{\infty} (\Omega \times Y)$ such that (up to subsequence)  	
		\begin{equation}
		\begin{split}\label{ds101} 
		\vect u^\epsh & = \Tilde{\vect u}^\epsh + \mathring{\vect u}^\epsh, \quad  \Tilde{\vect u}^{\epsh},\mathring{\vect u}^{\epsh}  \in H^1_{\Gamma_{\rm D}}(\Omega;\R^3),\quad\mathring{\vect u}^\epsh |_{\Omega_1^{\varepsilon_h}} = 0, \\[0.1em]
		\pi_{1/h}\tilde{\vect u}^\epsh &{\,\xrightarrow{L^2}\,}\bigl(
		{\funcA}_1(\hat{x}) - x_3 \partial_1 \funcB(\hat{x}),
		{\funcA}_2(\hat{x}) - x_3 \partial_2 \funcB(\hat{x}), 
		\funcB(\hat{x})\bigr)^\top, \\[0.3em]
		\mathring{\vect u}^\epsh(x)  & \drtwoscale  \mathring{\vect u}(x,y),
		\\[0.3em]
		h^{-1}\simgrad_h \tilde{\vect u}^\epsh(x)  & \drtwoscale \, \iota\left(\simgrad_{\hat{x}}\vect{\funcA}(\hat{x})-x_3 \nabla_{\hat{x}}^2\vect{\funcB}(\hat{x})\right) + \funcC(x,y),
		\\[0.3em]
		\epsh\simgrad_h \mathring{\vect u}^\epsh(x)  & \drtwoscale \sym \iota\bigl(\nabla_y \mathring{\vect u}(x,y)\bigr).
		\end{split}
		\end{equation}	
		 \item  If, additionally to 2, one has $$\lim_{h \to 0} \left(h^{-2}a_{\epsh} (\vect u^{\epsh}, \vect u^{\epsh})+\|\vect u^{\epsh} \|^2_{L^2}\right)=\RRR \hat{a}_{\infty} ((0,0,\funcB)^\top+\mathring{\vect u}, (0,0,\funcB)^\top+\mathring{\vect u})\BBB+\|(0,0,\funcB)^\top+\mathring{\vect u}\|^2_{L^2},$$ where $\hat{a}_{\infty}$ is defined in (\ref{a_infty_hat}), then the strong two-scale convergence 
		$$
		\vect{u}^{\epsh} \strongdrtwoscale (0,0,\funcB)^\top+\mathring{\vect u}  
		$$
		holds.
		\end{enumerate} 	
\end{proposition} 	
The following theorem provides the limit resolvent equation.
\begin{theorem}\label{rucak60} 
	Let  $\delta=\infty$, $\mu_h=\epsh h$, $\tau=2$ 
	and let the sequence of load densities satisfy the following convergence: 
	\begin{equation} \label{konvsila22} 
	\vect  f^\epsh \drtwoscale \vect f \in L^2(\Omega \times Y;\R^3).
	\end{equation}
	Then the sequence of solutions to the resolvent problem \eqref{resolventproblemstarting} converges in the sense of
	%%of Proposition \ref{ds102} 
	\eqref{ds101} to the unique solution of the following problem:
	Determine ${\vect \funcA} \in H^1_{\gamma_{\rm D}}(\omega;\R^2)$, $\funcB \in H^2_{\gamma_{\rm D}}(\omega)$, ${\mathring{\vect u}}\in V_{2,\infty}(\Omega \times Y_0)$ such that
	\begin{equation}\label{pomoc3}
	\begin{split} 
	&\vect \funcA=0, \\
	& \frac{1}{12}\int\limits_{\omega}\C^{\rm bend,h}\nabla_{\hat{x}}^2{ \funcB}(\hat{x}): \nabla_{\hat{x}}^2{\theta}_3(\hat{x}) \,d\hat{x}
	+ 
	\lambda \int\limits_\omega \bigl(\rho_0(y) {\funcB}(\hat{x})+
	%% \cdot {\vect \theta_3}(\hat{x}) \,d\hat{x}+
	%%\lambda \int\limits_\omega
	 \langle \rho_0 \overline{\mathring{u}}_3 \rangle(\hat{x})\bigr)\,{\theta_3}(\hat{x}) \,d\hat{x} \\&\,
	\hspace{+100pt}=\int\limits_{\omega} \langle  \overline{ {f}}_3  \rangle(\hat{x})\,{{\theta}}_3 (\hat{x}) \,d\hat{x}\qquad \forall\theta_3 \in H_{\gamma_{\rm D}}^2(\omega),  
	\\
	& \int\limits_{Y_0} \C_0(y) \sym \iota \bigl( \nabla_y { \mathring{ \vect u}}(x,y)\bigr): \sym \iota \bigl(\nabla_y{\mathring{\vect \xi}}(y)\bigr)\,dy  
	+
	\lambda \int\limits_{Y_0}
	\rho_0(y)\funcB(\hat{x})\,\mathring{{\xi}}_3(y) \,dy + 
	\lambda \int\limits_{Y_0} \rho_0(y){\mathring{\vect u}}(x,y) \cdot \mathring{\vect{\xi}}(y) \,dy
	 \\ &
	\hspace{+100pt}=\int\limits_{Y_0} 
	\vect f (x,y) 
	\cdot  \mathring{\vect\xi}  (y) \,dy\qquad	
	%%\forall \vect \theta_{\alpha} \in H_{\gamma_{\rm %%D}}^1(\omega), \textrm{ for } \alpha=1,2,\ 
    \forall \mathring{\vect \xi} \in \RRR H_0^1(Y_0;\mathbb{R}^3)\BBB,  \  \textrm{a.e.\,} x \in \Omega.
	\end{split}
	\end{equation}
	If strong two-scale convergence takes place in \eqref{konvsila22}, then one additionally has
	%%we additionally have the strong two-scale convergence 
	$$  
	\vect u^\epsh(x)   \strongdrtwoscale  (0,0,\funcB)^\top+ \mathring{\vect u}(x,y).
	$$
\end{theorem}

\begin{remark} 
	The limit problem \eqref{pomoc3} can be written as a resolvent problem on $\{0\}^2 \times L^2(\omega)+L^2(\Omega \times Y_0;\R^3)$: 
	$$  (\hat{\mathcal{A}}_{\infty}+\lambda \mathcal{I})\vect u=\bigl( S_2 (P_{\infty,\infty} \vect f)_1,S_2 (P_{\infty,\infty} \vect f)_2, (P_{\infty,\infty} \vect f )_3  \bigr)^\top, \quad \vect u= (0,0,\funcB)^\top+\mathring{\vect u}.$$
	
\end{remark}

\begin{remark} 
	The limit resolvent equations exhibit several differences between the regimes discussed: beside different effective tensors (this also happens in the moderate-contrast setting, see, e.g., \cite{Neukamm13} in the case of nonlinear von K\'arm\'an plate theory), one has different kinds of behaviour on the inclusions: in the regime $h\sim \epsh$ the inclusions behave like three-dimensional objects, while for $\delta=0$ they can be seen as small plates. Furthermore, different scalings of load densities are required in different regimes, which does not happen in the case  moderate contrast. Finally, the case $\delta=0$ has an additional effective parameter $\kappa$; when $\kappa \in (0,\infty)$ the elastic energy only resists 
	to oscillatory (spatial) motion (i.e., oscillations on the level of periodicity cells) in the out-of-plane direction.   
\end{remark}	
	
\subsection{Limit spectrum} \label{limitspectrum} 
In this section we will use the above resolvent convergence results to infer convergence of spectra of the operators ${\mathcal A}_{\varepsilon_h}.$ 
%%the $\epsh$-problems. 
As we shall see below in the proofs of the spectral convergence, one does not need to apply different scalings to different components of external loads, and thus only simplified versions of the limit resolvent equations will be necessary. Also, the presence of a spectrum of order $h^2$ implies that any other scaling will cause the limit set to be the whole positive real line (see \cite{castro1}). Thus, for the case when $\mu_h=\epsh$ in (\ref{coeff_def}),
%% coefficients of order ${\epshtwo}$,
in order for the limit spectrum to have a \RRR ``band-gap" structure \BBB we are forced to restrict ourselves to the ``membrane" subspace $L^{2, \rm memb},$ which is possible under Assumption \ref{assumivan1}\,(1) concerning material symmetries. Otherwise, for the same case, the limit resolvent captures only the order-one part of the limit spectrum.
%%, since the limiting spectrum is whole non-negative real axis (full limit operator is in this case degenerate and has infinite dimensional kernel).
 This is consistent with the standard result that the strong resolvent convergence only implies that the spectrum of the limit operator is contained in limit spectrum for ${\mathcal A}_{\varepsilon_h},$ while an additional compactness argument  
 %%of $\epsh$-problems. 
 %%compactness of 
 %%%eigenfunctions 
%%%  ( finite domain), i.e., strong two-scale convergence of eigenfunctions, is needed to obtain the 
is necessary for the opposite inclusion (see, e.g, \cite{Zhikov2000}). In our setting, compactness of eigenfunctions is lost when passing from the spectrum of order $h^2$ (or order-one spectrum for the restriction to $L^{2, \rm memb}$) to the order-one spectrum for the full operator,
%% since under the scaling of spectrum of order one,
as the transversal component of an eigenfunction would converge only weakly two-scale.

Under Assumption \ref{assumivan1}, for the membrane scalings of Part B of Section \ref{sectionheps} and Parts A of Sections \ref{hlleps}, \ref{epsllh}, the resolvent equation can be restricted to the invariant subspace $L^{2, \rm memb},$ where the solutions happen to be compact in the strong topology, see Corollaries \ref{remsim1}, \ref{rucak45}, \ref{corivan11}. This compactness property enables one to prove the convergence of spectra of order one for this restriction. Notice that in the regime $h \ll \epsh$ there are different types of limit resolvents (distinguished by different values of the parameter $\kappa$) when Assumption \ref{assumivan1} is not satisfied. In this regime,  the convergence of the third component of the displacements is only weak two-scale, which is the reason why we do not invoke different resolvent limits in the analysis of the convergence of spectra in the mentioned regime. However, we will use this information in our study of the limit evolution equations. 

\begin{remark}
For the case of spectra of order $h^2,$ in order to be able to obtain limit spectra with band gaps, one needs to consider different scalings of the coefficients on high-contrast inclusions. This motivated us for the analysis of this situation in Part C of \ref{sectionheps} and Parts B of Sections \ref{hlleps}, \ref{epsllh}.
\end{remark}

\subsubsection{Preliminaries on spectral convergence}

The Lax-Milgram theorem (see \cite{OShY}) implies that for each $\vect f\in L^2(\Omega;\R^3)$ the equation 
$$
\mathcal{A}_{\epsh} \vect u = \vect f
$$ 
has a unique solution $\vect u\in H^1_{\Gamma_{\rm D}}(\Omega;\R^3)$ understood in the weak sense. The operator 
\begin{equation*}
\mathcal{T}_\epsh : L^2(\Omega;\R^3) \to H^1_{\Gamma_{\rm D}}(\Omega;\R^3), \quad \mathcal{T}_\epsh \vect f :=\vect u,
\end{equation*}
is compact due to the compact embedding $H^1_{\Gamma_{\rm D}}(\Omega;\R^3) \hookrightarrow L^2(\Omega;\R^3)$ (this compactness will be lost in the limit problem, except for the first case analysed). Therefore, $\mathcal{T}_\epsh$ has countably many eigenvalues forming a non-increasing sequence of positive numbers converging to zero, the only remaining element of the spectrum $\mathcal{T}_\epsh.$ 
Therefore, the spectrum of $\mathcal{A}_\epsh$ consists of eigenvalues ordered in a non-decreasing positive sequence $\lambda_n^\epsh$ that tends to infinity. 
%Also, we have the following characterization:
%\begin{equation*}
%\lambda_n^\epsh = \min_{V^n < V(\Omega), \dim V^n= n} 
%\max_{\vect v\in V^n} \frac{a_\epsh(\vect v,\vect v)}{||\vect %v||_{L^2(\Omega;\R^3)^3}}.
%\end{equation*} 
%The eigenvalue problem for the operator $\mathcal{A}_\epsh$ is the following: \\Find $\vect u^\epsh \in V(\Omega)$ , $\lambda^\epsh \in \R$, such that 
%\begin{equation*}\label{weakformeig}
%\int\limits_{\Omega} \C^\epsh\biggl(\dfrac{\hat{x}}{\epsh}\biggr) \simgrad_h \vect u^\epsh(x) : \simgrad_h \vect v(x) \,dx
%\,=\,
%\lambda^\epsh \int\limits_{\Omega} \vect u^\epsh(x) \cdot \vect v(x) \,dx,  \quad \forall \vect v \in V(\Omega).
%\end{equation*}         

In what follows, we are interested in understanding the relationship between the spectra of ${\mathcal A}_\epsh$ as $h\to0$ and eigenvalues of the limit operators discussed in Section \ref{limreseq}. To this end, the following standard notion of convergence will be referred to throughout.
\begin{definition}
	\label{rucak85}
	We say that a sequence of sets $S_h$ (e.g., $S_h=\sigma(\mathcal{A}_{\epsh})$) converges in the Hausdorff sense to the set $S$ if:
	\begin{itemize}
		\item ($H_1$) For any  $\lambda \in S$, there exists a sequence of $\lambda^h\in S_h$ convergent to $\lambda$ (as $h\to0.$)
		\item ($H_2$) The limit of any convergent sequence of $\lambda^h\in S_h$ is an element of $S.$ 
		%%is a sequence such that $\lambda^h \to \lambda$, then $\lambda \in S$.
	\end{itemize}
\end{definition}
For various scalings of Section \ref{limreseq}, we will discuss the convergence in the Hausdorff sense of $\sigma(\mathcal{A}_\epsh)$ to the spectrum of the corresponding limit operator.

The first property of Hausdorff convergence of spectra is normally a direct consequence of the strong resolvent convergence, while the second property requires the compactness of the sequence of eigenfunctions in an appropriate topology.

\subsubsection{Asymptotic regime $\delta \in (0,\infty)$, scaling $\tau=2$}
\label{section332}
In this section we will analyse the limit spectrum of order $h^2$ for ${\epshtwo}$-scaling of the coefficients in the inclusions. We will show that the high-contrast has no effect on the limit, in that the (scaled) limit spectrum is of the same type as for the ordinary plate (i.e., homogeneous or with moderate contrast), in particular the limit operator has compact resolvent. This is precisely the reason why we analyse this combination of scalings of the spectrum and the coefficients only for the asymptotic regime $h\sim\epsh$ (i.e., $\delta \in (0,\infty)$). 

On the one hand we would like to show that in the case of an ordinary plate the resolvent approach can also provide information about the convergence of spectra (alternatively to, say, using Rayleigh quotients), and on the other hand we aim at demonstrating that in the mentioned case the limit problem does not exhibit spectral gaps and thus a
 different scaling of the coefficients is required for them to appear. 
 
The following theorem provides the relevant result concerning spectral convergence. 
\begin{theorem}\label{rucak80} 
	Let $\lim_{h \to 0} h/\epsh=\delta  \in (0,\infty)$, $\mu_h=\epsh$. 
	The spectra $\sigma\bigl(h^{-2}\mathcal{A}_\epsh\bigr)$ converge in the Hausdorff sense to the spectrum of $\mathcal{A}_{\delta}^{\funcB, {\rm hom}},$ as $h \to 0$,
	%% The spectrum of $\sigma(\mathcal{A}_{\delta}^{\funcB, {\rm hom}})$
	 which is an increasing sequence of positive eigenvalues $(\lambda_{\delta,n})_{n \in \N}$ that tend to infinity, each of finite multiplicity. More precisely, if by $\lambda_n^\epsh$ we denote the $n$-th eigenvalue of $\mathcal{A}_{\epsh}$  (by repeating each eigenvalue according to its multiplicity), then 
	\begin{equation*}
	h^{-2}\lambda_n^\epsh \to \lambda_{\delta,n}, \quad h \to 0,
	\end{equation*}  
where $\lambda_{\delta,n}$ is $n$-th eigenvalue of $\mathcal{A}_{\delta}^{\funcB, {\rm hom}}$ (again repeated in accordance with multiplicity).  Furthermore, for any fixed $n$ and any choice of normalised eigenfunctions with eigenvalues $\lambda^{\epsh}_n,$ there is 
a ($h$-indexed) subsequence such that the corresponding eigenfunctions converge, as $h\to0,$ to an eigenfunction with the eigenvalue $\lambda_{\delta,n}$ of the limit problem. 
\end{theorem}
\subsubsection{Asymptotic regime $\delta \in [0,\infty),$ scaling $\mu_h=\epsh$, $\tau=0$} \label{secmembsp} 
In this section we analyse the operator $\mathcal{A}_{\epsh}$ in the space $L^{2, \rm memb}(\Omega;\R^3)$. In the regime $\delta=0$ we require that the component $Y_0$ has $C^{1,1}$ boundary.
%%For $\delta \in [0,\infty)$ we
We define the following generalized eigenvalue problem: Find $\lambda>0$ and $\vect{\funcA} \in H^1_{\gamma_{\rm D}}(\omega,\R^2)$ such that  
\begin{equation}
\label{generalizedeigenvaluemembrane}
\begin{split} 
\int_\omega \C_{\delta}^{\rm memb}\simgrad_{\hat{x}}\vect\funcA(\hat{x}): \simgrad_{\hat{x}}\vect \varphi (\hat{x}) d\hat{x} &=  \int_\omega\tilde{\beta}^{\rm memb}_{\delta}(\lambda) \vect{\funcA}(\hat{x})\cdot \vect \varphi(\hat{x}) d\hat{x}, \quad \forall \vect \varphi \in H^1_{\gamma_{\rm D}}(\omega;\R^2).  
\end{split}
\end{equation} 
(In the case $\delta=0$ we put $\C_{1}^{\rm memb,r}$ instead of $\C_{\delta}^{\rm memb}$, in the case when $\delta=\infty$ we put 
$\C^{{\rm memb},h}$ instead of $\C_{\delta}^{\rm memb},$ see Section \ref{effective} for the relevant definitions.) 
The following theorem contains the spectral convergence result for the regime considered here.
\begin{theorem} \label{rucak71}
Suppose $\lim_{h \to 0}h/\epsh=\delta  \in [0,\infty)$, $\mu_h =\epsh$ and let Assumption \ref{assumivan1}\,(1) be valid. 
The set of all $\lambda>0$ for which the problem \eqref{generalizedeigenvaluemembrane} has a non-trivial solution $\vect{\funcA} \in H^1_{\gamma_{\rm D}}(\omega;\R^2)$ is at most countable.  
  The spectra of the operators $\tilde{\mathcal{A}}_{\epsh}$  converge in the Hausdorff sense to the spectrum of $\tilde{\mathcal{A}}_{\delta},$ and one has
\begin{equation}\label{rucak96} 
\sigma(\tilde{\mathcal{A}}_{\delta}) = \sigma(\tilde{\mathcal{A}}_{00,\delta})' \cup \overline{\left\{ \lambda > 0: \mbox{ The generalized eigenvalue problem \eqref{generalizedeigenvaluemembrane} is solvable.} \right\}}.
\end{equation}
Additionally, under Assumption \ref{assumivan1}\,(2,3), the matrix  $\tilde{\beta}^{\rm memb}_{\delta}(\lambda)$ is \RRR scalar \BBB and 
\[
\sigma(\tilde{\mathcal{A}}_{\delta}) = \sigma(\tilde{\mathcal{A}}_{00,\delta})' \cup \overline{\left\{ \lambda > 0: \tilde{\beta}^{\rm memb}_{\delta,11}(\lambda)\RRR=\tilde{\beta}^{\rm memb}_{\delta,22}(\lambda)\BBB \in \sigma(\mathcal{A}^{\rm memb}_{\delta}) \right\}}.
\]	
\end{theorem}
\begin{remark} \label{remlsm}  
It was shown in \cite{Gri05} that each non-empty interval of the form $(\tilde{\omega}_n,\tilde{\omega}_{n+1}),$ $n\in \N,$ contains a subinterval $(\tilde{\omega}_n,\alpha)$, $\tilde{\omega}_n\leq \alpha < \tilde{\omega}_{n+1}$ in which both eigenvalues of the matrix $\tilde{\beta}^{\rm memb}_{\delta}$ are negative, a subinterval $(\alpha,\beta)$, $\alpha<\beta \leq \tilde{\omega}_{n+1}$ in which one of its eigenvalues 
%%of $\tilde{\beta}^{\rm memb}_{\delta}$ 
is negative while the other is positive, and the interval $(\beta,\tilde{\omega}_{n+1})$
%, $\beta\leq \tilde{\omega}_{n+1},$ where 
in which both its eigenvalues 
%%of $\tilde{\beta}^{\rm memb}_{\delta}$ 
are positive. It follows, as is explained in \cite{Gri05}, that in the interval  $(\tilde{\omega}_n,\alpha)$ there is no wave propagation in any direction, while in the interval  $(\alpha,\beta)$ one has evanescent solutions in the direction of the negative eigenvectors, and finally
%%, which means that there is no wave
%%propagation in certain directions corresponding to any linear combination of eigenvectors associated with the negative eigenvalues. 
in the intervals $(\beta,\tilde{\omega}_{n+1})$ all directions allow wave propagation. 

Under Assumption \ref{assumivan1}\,\RRR (1--3)\BBB, the above spectral structure can be quantified in a straightforward way and $\tilde{\omega}_n <\alpha=\beta<\tilde{\omega}_{n+1}$, see \cite{zhikov2005,Zhikov2000}.
In this case  the matrix $\tilde{\beta}^{\rm memb}_{\delta}$ is scalar and 
$$
\lim_{\lambda\to \tilde{\omega}_n^{+}}\tilde{\beta}^{\rm memb}_{\delta,11}(\lambda)=-\infty,\qquad  \lim_{\lambda\to \tilde{\omega}_{n+1}^{-}}\tilde{\beta}^{\rm memb}_{\delta,11}(\lambda)=+\infty,
$$
where $\tilde{\beta}^{\rm memb}_{\delta,11}$ is one of the two equal diagonal elements, and the limits are taken as $\lambda$ approaches $\tilde{\omega}_n$ on the right and on the left, respectively.
 
The above properties of the limit spectrum are relevant in a variety of applications, such as noise suppression. Being peculiar to wave propagation in high-contrast media, they are often referred to as ``high-contrast effects".

\end{remark} 
\subsubsection{Asymptotic regime $\delta \in (0,\infty),$ scaling $\mu_h=\epsh h$, $\tau=2$ and asymptotic regime $\delta=0,$ scaling $\mu_h={\epshtwo}$, $\tau=2$}
%\subsubsection{Scaling $\mu_h=\epsh h$, $\tau=2$ for regime $\delta \in (0,\infty)$ and $\mu_h={\epshtwo}$, $\tau=2$ for regime $\delta=0$}
\label{sectionbgl}
For the regimes considered here, we show that high-contrast effects occur in the limit as $h\to0.$ As before, when $\delta=0$ we assume that $Y_0$ has $C^{1,1}$ boundary. We have the following theorem.
\begin{theorem} \label{rucak72} 
\RRR Let $\lim_{h \to 0}h/\epsh=\delta  \in [0,\infty)$. \BBB In the cases $\delta=0,$ $\delta>0$ we assume that $\mu_h={\epshtwo}$ and $\mu_h=\epsh h,$ respectively. The spectrum of the operator $h^{-2}\mathcal{A}_{\epsh}$ converges in the Hausdorff sense to the spectrum of the operator $\hat{\mathcal{A}}_{\delta},$ given by
%% the following is valid:
\begin{eqnarray*} 
%%& &
\sigma(\hat{\mathcal{A}}_{\delta}) =\left\{
\begin{array}{ll}\hat{\sigma}(\mathcal{A}_{00,\delta}) \cup \overline{\left\{ \lambda > 0: \hat{\beta}_{\delta}(\lambda) \in \sigma(\mathcal{A}^{\funcB, {\rm hom}}_{\delta}) \right\}},\qquad \delta\in(0,\infty), 
	\\[0.45em]
%%& &\sigma(\hat{\mathcal{A}}_{0}) = 
\hat{\sigma}(\hat{\mathcal{A}}_{00,0}) \cup \overline{\left\{ \lambda > 0: \hat{\beta}_{0}(\lambda) \in \sigma(\mathcal{A}^{\rm hom}_{0}) \right\}},\qquad \delta=0.\end{array}\right.
\end{eqnarray*} 
\begin{remark} 
The operator $\mathcal{A}_{\delta}^{\funcB, {\rm hom}}$ is non-local when Assumption \ref{assumivan1} is not satisfied. It is not known to us whether this has been commented on in the existing literature, even in the case of a homogeneous plate.	
\end{remark} 	
\end{theorem} 	   
\subsubsection{Asymptotic regime $\delta=\infty$} \label{sectionspsur}

As we see below, in the case $\delta=\infty,$ the limit spectrum has points outside spectrum of the limit operator. From the intuitive point of view, the effective behaviour is similar to that of a cuboid with disjoint soft inclusions in the shape of long thin rods arranged parallel to each other and connecting two opposite sides of the body.
 
In order to formulate the result of this section, we define:
\begin{itemize}
	\item 
An operator as the one defined via the bilinear form 
$$ 
\mathring{a}_{\rm strip}=\int_{\R \times Y_0}\C_0(y) \sym \nabla \vect u:\sym \nabla \vect v dx_3 dy, \qquad \mathring{a}_{\rm strip}:\left( H_{00}^1(\R \times Y_0;\R^3)\right)^2\to \R;  
$$
\item
An operator $\mathring{\mathcal{A}}_{\rm strip}^+$  on $L^2(\R^{+}_0 \times Y_0;\R^3)$ as the one defined via the form 
$$ 
\mathring{a}_{\rm strip}^+=\int_{\R^{+}_0 \times Y_0}\C_0(y) \sym \nabla \vect u:\sym \nabla \vect v dx_3 dy, \qquad \mathring{a}_{\rm strip}^+:\left( H_{00}^1(\R^{+}_0 \times Y_0;\R^3) \right)^2\to \R;   
$$
\item
An operator $\mathring{\mathcal{A}}_{\rm strip}^-$  on $L^2(\R^{-}_0 \times Y_0;\R^3)$ as the one defined via the form 
$$ \mathring{a}_{\rm strip}^-=\int_{\R^{-}_0 \times Y_0}\C_0(y) \sym \nabla \vect u:\sym \nabla \vect v dx_3 dy, \qquad \mathring{a}_{\rm strip}^-: \left(H_{00}^1(\R^{-}_0 \times Y_0;\R^3)\right)^2 \to \R.   
$$
\item 
The restriction $\mathring{\tilde{\mathcal{A}}}_{\rm strip}$ of the operator $\mathring{\mathcal{A}}_{\rm strip}$ to the membrane subspace $L^{2,{\rm memb}}(\R \times Y_0;\R^3),$ whenever Assumption \ref{assumivan1}\,(1) holds. 
\end{itemize}

 First, we give characterisations of the limit spectra of $\mathring{\mathcal{A}}_{\epsh}$ and $\mathring{\tilde{\mathcal{A}}}_{\epsh}$, which in this regime play significant roles. 
\begin{theorem} 
	\label{thmivan301}  
	Suppose that $\epsh \ll h$. Then one has 
\begin{equation}
\lim_{h \to 0} \sigma(\mathring{\mathcal{A}}_{\epsh})= \lim_{h \to 0} \sigma(\mathring{\tilde{\mathcal{A}}}_{\epsh})=\sigma(\mathring{\mathcal{A}}_{\rm strip})\cup \sigma(\mathring{\mathcal{A}}_{\rm strip}^+) \cup  \sigma(\mathring{\mathcal{A}}_{\rm strip}^-). 
\label{characterise}
\end{equation}
Moreover, one has 
\begin{equation}
	\sigma_{\rm ess}(\mathring{\mathcal{A}}_{\rm strip}^{\pm})=\sigma(\mathring{\mathcal{A}}_{\rm strip}),
	\label{pmess}
\end{equation}
and there exists $m_0>0$ such that 
\begin{equation}
\sigma(\mathring{\mathcal{A}}_{\rm strip})=\sigma_{\rm ess}(\mathring{\mathcal{A}}_{\rm strip})= [m_0,+\infty). 
\label{sigma_ess}
\end{equation}
%%%an at most countable set $\sigma_{\rm disc}(\mathring{\mathcal{A}}_{\rm strip}^{\pm}) \subset \R_0^{+}$, which is the discrete spectrum of $\mathcal{A}_{\rm strip}^{\pm}$, such that 
%% \cup \sigma_{\rm disc}(\mathring{\mathcal{A}}_{\rm strip}^{\pm}).   $$
Under Assumption \ref{assumivan1}\,(1), one additionally has 
\begin{equation}
\sigma(\mathring{\mathcal{A}}^+_{\rm strip})=\sigma(\mathring{\mathcal{A}}^-_{\rm strip}) \supset \sigma(\mathring{\tilde{\mathcal{A}}}_{\rm strip})=\sigma_{\rm ess}(\mathring{\tilde{\mathcal{A}}}_{\rm strip})= \sigma\bigl(\mathring{{\mathcal{A}}}_{\rm strip}\bigr).
\label{assincl}
\end{equation} 	
\end{theorem} 
Next we provide a characterisation of the limit spectrum for ${\mathcal A}_\epsh.$ 
\begin{theorem} \label{thmivan302} 
Let	$\epsh \ll h$, $\mu_h=\epsh$ and $\tau=0$. 
The set of all $\lambda>0$ for which the problem \eqref{generalizedeigenvaluemembrane} obtains a nontrivial solution $\vect{\funcA} \in H^1_{\gamma_{\rm D}}(\omega;\R^2)$ is at most countable.  
The spectra of $\tilde{\mathcal{A}}_{\epsh}$  converge in the Hausdorff sense to 
$\sigma(\mathring{\mathcal{A}}^+_{\rm strip}) \cup \sigma(\tilde{\mathcal{A}}_{\infty})$, where
\begin{equation} \label{ispravak3}
\sigma(\tilde{\mathcal{A}}_{\infty}) = \sigma(\tilde{\mathcal{A}}_{00,\infty})' \cup \overline{\left\{ \lambda > 0: \mbox{ The generalized eigenvalue problem \eqref{generalizedeigenvaluemembrane} is solvable.} \right\}}.
\end{equation}
Under Assumption \ref{assumivan1}\,(2, 3), the matrix \RRR $\tilde{\beta}^{\rm memb}_{\infty}(\lambda)$ \BBB is \RRR scalar \BBB and 
\begin{equation} \label{ispravak4}
\sigma(\tilde{\mathcal{A}}_{\infty}) = \sigma(\tilde{\mathcal{A}}_{00,\infty})' \cup \overline{\left\{ \lambda > 0: \tilde{\beta}^{\rm memb}_{\infty,11}(\lambda)\RRR=\tilde{\beta}^{\rm memb}_{\infty,22}(\lambda) \BBB\in \sigma(\mathcal{A}^{\rm memb}_{\delta}) \right\}}.
\end{equation} 
Furthermore, one has $\sigma(\tilde{\mathcal{A}}_{00,\infty})\subset \sigma(\mathring{\mathcal{A}}^+_{\rm strip})$.  
\end{theorem} 	
\begin{theorem} \label{thmivan303} 
Suppose that $\epsh\ll h$, $\mu_h =\epsh h$, $\tau=2$.  
 The spectra of $h^{-2}\mathcal{A}_{\epsh}$ converge in the Hausdorff sense to $\sigma(\mathring{\mathcal{A}}^+_{\rm strip}) \cup \sigma(\mathring{\mathcal{A}}^-_{\rm strip}) \cup \sigma( \hat{\mathcal{A}}_{\RRR \infty \BBB})$ and 
 %%the following is valid:
 $$
\sigma(\hat{\mathcal{A}}_{\infty}) = \hat{\sigma}({\mathcal{A}}_{00,\infty}) \cup \overline{\left\{ \lambda > 0: \hat{\beta}_{\infty}(\lambda) \in \sigma(\mathcal{A}^{\rm hom}_{\infty}) \right\}}. 
$$
\end{theorem} 	
\begin{remark} 
As is shown in Lemma \ref{lemmaivan100}, the set $\lim_{h \to 0} \sigma(\mathring{\mathcal{A}}_{\epsh})$ (appropriately scaled) is always a subset of the limit spectrum. In the regime $\delta=\infty$, the operator has a scaling factor $\epsh/h$ in front of the derivative in $x_3.$ This allows eigenfunctions to oscillate in the out-of-plane direction (and thus weakly converge to zero). This is the reason for so-called ``spectral pollution'' (see, e.g., \cite{AllaireConca98}).	
	
\end{remark}	

\subsection{Limit evolution equations} \label{limeveqsec}

It is expected from the results of Section \ref{limreseq} concerning the resolvent convergence for the operators ${\mathcal A}_\epsh$ that the limit evolution equations have the form of a system that links the behaviour on the stiff matrix and the soft inclusions by means of coupled solution components, which can be viewed as macroscopic and microscopic variables. Representing the system in terms of the macroscopic component only leads to a non-trivial effective description exhibiting memory effects.  This is one of the reasons what makes high-contrast materials interesting in applications.

The present section aims at providing a detailed study of the consequences of the form of the limit resolvent equations obtained for different asymptotic regimes in Section \ref{limreseq} on the limit evolution equations in the corresponding regimes.  On the abstract level, this connection has been analysed in \cite{pastukhova}. A key fact used in that paper is that the resolvent is the Laplace transform of the exponential function of the operator of the wave equation, obtained from an equivalent 
 %%by using the exponential function of the  operator, one needs to write it 
 %%to write second order in time wave equation as
 system of equations of first order in time. In what follows, we adjust our analysis to these general results, in order to account for the particular features of our setup due to dimension reduction in linear elasticity. As we see below, in this context different scalings of spectra imply different scalings of time  (i.e., bending waves propagate on a slower time scale than in-plane, ``membrane", waves). As far as we know, the effect of considering different time scalings has not been addressed in the literature; see \cite{Raoult} for the analysis of limit evolution of isotropic homogeneous plates for the commonly considered ``long" time scaling of order $h^{-1}$.   
 
 %%in the case of three-dimensional plates of linear elasticity-isotropic homogeneous material). 
 
It should also be noted that some load density scalings prevent us from using the results of \cite{pastukhova}, in which case separate analysis is necessary to show weak convergence of solutions (see, e.g., the proof of Theorem \ref{connn1}). This happens for the case $\tau=2$ (i.e., for long times of order $h^{-1}$) in the regimes $\delta \in (0,\infty),$ $\mu_h=\epsh$ and $\delta=0,$ $\mu_h={\epshtwo}$. For these, to prove weak convergence we use the Laplace transform directly, following the same overall strategy as the one adopted in \cite{pastukhova} in the abstract setting, see Apprendix. However, due to the said load density scalings, a modification of the results of \cite{pastukhova} is required, in order to account for the specific structure of the right-hand side of the limit problem; this is also discussed in Appendix, see in particular Theorems \ref{existenceV*}, \ref{thmivan553}.

The starting point of this section is the family of Cauchy problems $(h>0)$ 
\begin{equation}
\label{evolutionepsilonproblem}
\left\{ \begin{array}{ll}
\partial_{tt}\vect u^\epsh (t) + {h^{-\tau}}\mathcal{A}_\epsh \vect u^\epsh(t) = \vect f^\epsh(t),\\[0.55em]
\vect u^\epsh(0)= \vect u_0^\epsh, \quad \partial_{t}\vect u^\epsh(0) = \vect u_1^\epsh
,
%\quad \forall\epsh>0. 
\end{array} \right.
\end{equation}
%The function $\vect u^\epsh$ is taken to solve this problem 
understood in the weak sense.   The term $\vect f^\epsh(t)$ represents the load density at time $t>0.$ For each $h,$ we suppose that $\vect f^\epsh$ is provided on the time interval $[0,T_h],$ $T_h>0.$ 
The functions $\vect u_0^\epsh, \vect u_1^\epsh$ are the initial data for the displacement and velocity fields, respectively.  We make the following assumptions: 
\begin{equation*}
%%\left\{ \begin{array}{ll}
\vect u_0^\epsh \in \mathcal{D}(\mathcal{A}_\epsh^{1/2})=H^1_{\Gamma_{\rm D}}(\Omega;\R^3), \quad  \vect u_1^\epsh \in L^2(\Omega;\R^3),\quad  
%%\\[0.55em]  
\vect f^\epsh \in \RRR L^2(0,T;L^2(\Omega;\R^3)), \BBB 
%%\\
%%\end{array} \right.
\end{equation*}
In what follows, we shall analyse the ``critical" cases $\tau=2$ and $\tau=0$ for the time scaling. Conditions for well-posedness of the problem (\ref{evolutionepsilonproblem}) can be found in Appendix, see Section \ref{apsechyp}.

 In conclusion of this section, we reiterate that there are two ways to interpret the scaling $h^{-\tau}$ of the differential expression in (\ref{evolutionepsilonproblem}): by scaling the density of the material (with $h^{\tau}$) or by introducing the new time scale $\tilde{t}=t/h^{\tau/2}$. We adopt the latter interpretation throughout.  Multiplying \eqref{evolutionepsilonproblem} by $h^{\tau}$ and replacing $t$ by $\tilde{t},$  we obtain the family of problems $(h>0)$
\begin{equation} \label{evscaled1} 
\left\{ \begin{array}{ll}
\partial_{\tilde{t} \tilde{t}}\vect u^\epsh (\tilde{t}) + \mathcal{A}_\epsh \vect u^\epsh(\tilde{t}) =\tilde{\vect f}^\epsh(\tilde{t}),\\[0.55em]
\vect u^\epsh(0)= \vect u_0^\epsh, \quad \partial_{\tilde{t}}\vect u^\epsh(0) = \tilde{\vect u}_1^\epsh,
%%, \quad \forall\epsh>0. 
\end{array} \right. 
\end{equation} 
where $\tilde {\vect f}^\epsh(\tilde{t}):=h^{\tau} {\vect f}^{\epsh} (h^{\tau}\tilde{t})$, $\tilde{\vect u}_1^{\epsh}:=h^{\tau/2}\vect{u}_1^{\epsh}$. Thus discussing the solution of \eqref{evolutionepsilonproblem} on a time interval $[0,T]$ (with an appropriate scaling of the load density) is equivalent to discussing the solution of \eqref{evscaled1} on the time interval $[0, T/h^{\tau/2}]$ (with the corresponding scaling of the loads). While from now on we shall work in the framework of the equation \eqref{evolutionepsilonproblem}, which is convenient from the mathematical point of view, it is the equation \eqref{evscaled1} that represents the actual physical wave motion, which thereby takes place on an appropriate time scale of order $h^{-\tau/2}.$ 
%%time $\tilde{t}$.

\subsubsection{Long-time behaviour for the regime $\delta \in (0,\infty)$, $\mu_h=\epsh$, $\tau=2$  }
\label{section341}
The case analysed here resembles the standard (moderate-contrast) plate model. The following convergence statement holds for the evolution problem.  
\begin{theorem} 
	\label{connn1}
Suppose that $\delta \in (0,\infty)$, $\mu_h=\eps_h$, $\tau=2$.
Let $(\vect{u}^{\epsh})_{h>0}$ be a sequence of solutions 
to \eqref{evolutionepsilonproblem} and assume that 
\begin{align}
%	 \label{ivan969} 
\left(h\partial_{t}\vect{f}^{\epsh}_{\alpha}\right)_{h>0} &\subset L^2([0,T];L^2(\Omega\times Y)) \textrm{\ \ bounded},\ \ \alpha=1,2, \label{loadstimebdd1} \\ 
\pi_{h}\vect  f^\epsh & \drtwoscalet  \vect f  \in L^2([0,T];L^2(\Omega \times Y;\R^3)),\label{ivan970}\\ 
\vect u_0^\epsh & \drtwoscale \vect u_0(\hat{x}) \in \{0\}^2 \times H^2_{\gamma_{\rm D}} (\omega), \label{loadsinitialpos1}\\
\vect u_1^\epsh & \drtwoscale \vect u_1(x,y) \in L^2(\Omega\times Y;\R^3),\label{loadsinitialvel1} 
\end{align}
 and assume additionally that 
 \[
 \limsup_{h \to 0}\left(h^{-2}a_{\epsh}(\vect u_0^{\epsh},\vect u_0^{\epsh})+\|\vect u^{\epsh}_0\|^2_{L^2} \right)<\infty.
 \] 
 Then one has 
\begin{align} 
\pi_{1/h}{\vect u}^\epsh & \drtwoscalet\left(\begin{array}{c}
	\funcA_1(t,\hat{x}) - x_3 \partial_1 \funcB(t,\hat{x})+\mathring{u}_1 (t,x,y) \\[0.3em]
	\funcA_2(t,\hat{x}) - x_3 \partial_2 \funcB(t,\hat{x})+\mathring{u}_2(t,x,y) \\[0.3em]
	\funcB(t,\hat{x})
\end{array}\right), \label{addu1}\\[0.3em] 
\partial_t {\vect u}^{\epsh}  & \drtwoscalet\bigl(0,0,
\partial_t \funcB(t,\hat{x})\bigr)^\top,\label{addu2}
\end{align}   
where $\vect{\funcA} \in C([0,T];H^1_{\gamma_{\rm D}}(\omega;\R^2))$, $\funcB \in C([0,T];H^2_{\gamma_{\rm D}}(\omega))$, $\mathring{\vect{u}}\in C([0,T];L^2(\omega;H^1_{00}(I \times Y_0;\R^3)))$ are determined uniquely by solving the problem 
\begin{align}
 	\label{limeq11}\partial_{tt}  \funcB(t)+\mathcal{A}^{\funcB, {\rm hom}}_{\delta} \funcB (t)&=\mathcal{F}_{\delta}\bigl({\vect f}(t) \bigr), \RRR \quad \textrm{(see \eqref{anketa1})} \BBB\\[0.3em]
 	\nonumber
 	\funcB(0)= u_{0,3} \in H^2_{\gamma_{\rm D}}(\omega),&\qquad \partial_t \funcB(0)= S_1 \RRR (P_{\delta,\infty} {\vect u}_{1})_3 \BBB \in L^2(\omega), \\[0.2em]
\label{limeq12}
 \vect{\funcA}(t)&= \vect{\funcA}^{\funcB(t)}+\vect{\funcA}^{\vect f_{\!*}  (t)},\quad  \RRR \textrm{(see \eqref{nadoknada11})} \BBB \\[0.3em]
\label{limeq13}
\mathcal{A}_{00,\delta} \mathring{\vect u}(t,\hat{x}, \cdot)&=(\vect{f}_* (t,\hat{x},\cdot),0)^\top,  
\end{align}	
so that $\partial_t \funcB \in C([0,T];L^2(\omega))$. 
One also has $$\limsup_{h \to 0} \int_0^T\left(h^{-2} a_{\epsh}(\vect{u}^{\epsh}(t),\vect{u}^{\epsh}(t))+\|\vect{u}^{\epsh}(t)\|^2_{L^2}\right)\, dt< \infty.$$  

If one assumes strong two-scale convergence of load densities 
	\begin{equation}
			\label{ivan1011} 
\begin{aligned}
\pi_{h}\vect{f}^{\epsh}&\strongdrtwoscalet (0,0,\vect{f})^\top \in L^2\left([0,T];L^2(\omega;\R^3)\right),\\[0.3em]  
h\partial_t \vect f_{\alpha}^{\epsh}&\longrightarrow
%%\xrightarrow{}
%%{L^2(0,T)}
 0\ \ \textrm{strongly in}\ L^2\left([0,T];L^2(\Omega)\right),\ \   \alpha=1,2,
\end{aligned} 
\end{equation}
strong two-scale convergence of the initial data in 
\eqref{loadsinitialpos1}, \eqref{loadsinitialvel1}, where $({\vect u}_1)_*=0$, ${u}_{1,3}\in L^2(\omega),$ and the condition $$\lim_{h \to 0} \left(h^{-2}a_{\epsh}(\vect u_0^{\epsh},\vect u_0^{\epsh})+\|\vect u_0^{\epsh} \|^2\right)=a_{\delta}^{\funcB}\bigl(\funcB(0),\funcB(0)\bigr)+\|\funcB(0)\|^2_{L^2},$$ then one has 
\begin{equation} 
\label{ivan1012} 
\pi_{1/h} \vect u^{\epsh} \strongdrtwoscalet (\funcA_1^{\funcB}-x_3 \partial_1 \funcB,\funcA_2^{\funcB}-x_3 \partial_2 \funcB, \funcB)^\top, \qquad \partial_t \vect u^{\epsh} \strongdrtwoscalet (0,0, \partial_t\funcB)^\top. 
\end{equation} 
Moreover, the following convergence of energy sequences holds for all $t\in[0, T]:$
$$
\lim_{h \to 0}\left(h^{-2}a_{\epsh} (\vect u^{\epsh}(t),\vect u^{\epsh}(t))+\|\vect  u^{\epsh}(t)\|^2_{L^2}\right)=a_{\delta}^{\funcB}\bigl(\funcB(t),\funcB(t)\bigr)+\|\funcB(t)\|^2_{L^2}. 
$$
%\quad \forall t \in [0,T].$$
\end{theorem}	
\begin{corollary} \label{kordinsi1} 
Suppose that for each $h>0,$ a surface load density 
$\mathcal{G}^{\epsh}\in L^2([0,T]; (H^1_{\Gamma_{\rm D}}(\Omega;\R^3)^*)$ is added to the right-hand side of \eqref{evolutionepsilonproblem}. We assume that $\mathcal{G}^{\epsh}$ is generated by an $L^2$-function $\vect g^{\epsh}$ (representing the ``true" surface load)
%% the surface forces
%, i.e., the element 
so that
$$\mathcal{G}^{\epsh}(\vect g^{\epsh})(\vect \theta)= \int_{\omega \times\{-1/2,1/2\}} \vect g^{\epsh} \vect \theta\, d\hat{x},\qquad \vect \theta \in H^1_{\Gamma_{\rm D}}(\Omega;\R^3),$$ where an obvious shorthand for a sum of two integrals over $\omega$ is used, 
and make the following additional assumptions on $\vect g^{\epsh}:$ 
%%we assume that
\begin{align*}
&\pi_h \partial_{t}\vect{g}^{\epsh}_{\alpha} \subset  L^2\Bigl([0,T]; L^2\bigl(\omega\times \bigl\{-1/2, 1/2\bigr\}\times Y;\R^3\bigr)\Bigr) \textrm{ is bounded },   \\
&\pi_{h}\vect  g^\epsh  \drtwoscalet  \vect g  \in L^2\Bigl([0,T];L^2\bigl(\omega\times\bigl\{-1/2, 1/2\bigr\} \times Y;\R^3\bigr)\Bigr).
 \end{align*} 
Then 
%%we have that the statement of 
a variant of Theorem \ref{connn1} holds, where in the limit equations \eqref{limeq11} the right-hand side has an additional term
$\mathcal{G}^{1}_{\delta}(\vect g) \in L^2([0,T]; (H^2_{\gamma_{\rm D}}(\omega))^*)$,  represented by a limiting surface load $L^2$ vector function $\vect g=(g_1, g_2, g_3)$ so that 
\begin{align*}
 \mathcal{G}_{\delta}^1(\vect g)(t) (\theta)&=\int_{\omega} \left(\bigl\langle g_3(t,\hat{x},-1/2,\cdot)+g_3(t,\hat{x},1/2, \cdot)\bigr\rangle\right) \theta(\hat{x}) \,d\hat{x} \\[0.45em]
 &+
   \frac{1}{2}\int\limits_\omega \left(\bigl\langle g_1(t,\hat{x},-1/2,\cdot)-g_1(t,\hat{x},1/2,\cdot)\bigr\rangle \right)  \partial_1 \theta(\hat{x}) \,d\hat{x}\\[0.35em]
   &+\frac{1}{2}\int\limits_\omega  \left(\bigl\langle g_2(t,\hat{x},-1/2,\cdot)-g_2(t,\hat{x},1/2,\cdot) \bigr\rangle\right) \partial_2 \theta(\hat{x})\,d\hat{x}\\[0.3em]
   &+\int\limits_{\omega}\C^{\rm hom}_{\delta} \bigl(\simgrad_{\hat{x}}\vect{\funcA}^{\vect g_{*}(t)},0 \bigr):\bigl(0,\nabla_{\hat{x}}^2 \theta(\hat{x})\bigr) \,d\hat{x}, \qquad \theta \in  \bigl(H^2_{\gamma_{\rm D}}(\omega)\bigr)^*.
\end{align*}
In the above formula, for every $t \in [0,T],$ the function $\vect \funcA^{\vect g_*(t)}\in H^1_{\gamma_{\rm D}}(\omega; {\mathbb R}^2)$ is the solution to the problem  
\[
\int\limits_{\omega}\C^{\rm memb}_{\delta}\simgrad_{\hat{x}}\vect \funcA^{\vect g_*(t)}(\hat{x}): \simgrad_{\hat{x}}\vect \theta_* (\hat{x}) \,d\hat{x} = \mathcal{G}_{\delta}^2(\vect g_*)(t)(\vect \theta_* )\qquad \forall \vect \theta_*  \in H^1_{\gamma_{\rm D}}(\omega;\R^2),
\]
where the functional $\mathcal{G}_{\delta}^2(\vect g_*)(t)$ is defined by the formula 
%% for each $\vect \theta_*  \in H^1_{\gamma_{\rm D}}(\omega;\R^2),$ $t\in(0, T),$ we define
\begin{align*}
\mathcal{G}_{\delta}^2(\vect g_*)(t)(\vect \theta_* )&=
\int\limits_\omega \left(\bigl\langle g_1(t,\hat{x},-1/2, \cdot)+g_1(t,\hat{x},1/2, \cdot)\bigr\rangle \right)  \vect \theta_1 \,d\hat{x}\\[0.3em]
&+\int\limits_\omega  \left(\bigl\langle g_2(t,\hat{x},-1/2, \cdot)+ g_2(t,\hat{x},1/2, \cdot)\bigr\rangle\right) \vect \theta_2 \,d\hat{x},\qquad \vect \theta_*  \in H^1_{\gamma_{\rm D}}(\omega;\R^2).
\end{align*}
Also, the right-hand side of \eqref{limeq12} contains
$\vect\funcA^{\vect g_*(t)}\in L^2([0,T]; H^1_{\gamma_{\rm D}}(\omega,\R^2))$ as an additional term, 
while on the right-hand side of \eqref{limeq13} one additionally has $\mathcal{G}^3 (\vect g_* )\in L^2([0,T];(H^1_{00}(I \times Y_0;\R^3)^*)$ defined by
$$ \mathcal{G}^3 (\vect g_* ) (t,\hat{x})(\vect \xi)=\int_{\{-1/2, 1/2\} \times Y_0} \vect g_* (t,\hat{x},\cdot)\cdot \vect \xi_*(\cdot) ,\qquad\vect \xi \in H^1_{00}(I \times Y_0;\R^3),\qquad  t\in [0,T],\quad \hat{x} \in \omega.    
$$

\end{corollary}	
\begin{remark} 
For each of the other regimes studied, a statement analogous to Corollary \ref{kordinsi1} is valid.
\end{remark}	
\begin{remark} 
The statement of Theorem \ref{connn1} can be strengthened as follows. The boundedness and convergence conditions (\ref{loadstimebdd1}) and (\ref{ivan970}) can be replaced by the requirement of boundedness and convergence, respectively, of the sequences $(\pi_h \vect f^{\epsh})_{h>0}$ and $(h\partial_t \vect f^{\epsh}_{\alpha} )_{h>0}$ in the corresponding spaces of $L^1$ functions on $[0, T].$ Under this weaker assumption, a still stronger version of (\ref{addu1}), (\ref{addu2}) holds, where the weak convergence in $L^2$ spaces on $[0, T]$ is replaced by a weak* convergence in the corresponding $L^{\infty}$ spaces on $[0, T],$ see the comment following Definition \ref{two_scale_def}. 

Similarly, the $L^2$ convergences \eqref{ivan1011} can be replaced by the weaker conditions 
\begin{equation*}
\begin{aligned}
\pi_h \vect f^{\epsh}&\xrightarrow{t,1,{\rm dr}-2} (0,0,\vect f)^\top \in L^1\left([0,T];L^2(\omega;\R^3)\right)\\[0.3em]  
h\partial_t \vect f_{\alpha}^{\epsh}&\longrightarrow
%%\xrightarrow{}
%%{L^2(0,T)}
0\ \ \textrm{strongly in}\ L^1\left([0,T];L^2(\Omega)\right),\ \   \alpha=1,2,
\end{aligned}
\end{equation*}
to obtain a strong two-scale convergence $\xrightarrow{t,\infty,{\rm dr}-2}$ for both sequences in \eqref{ivan1012}; see the same comment at the end of Section \ref{twoscale_conv} for the definition of $\xrightarrow{t,\infty,{\rm dr}-2}$. 

These stronger versions of the claims in Theorem \ref{connn1} 
follow immediately from a priori estimates, see also Remark \ref{ivan1013}, Remark \ref{ivan71}, however we choose to remain in the $L^2$ setting.

A version of the discussion within this remark applies also to Theorem \ref{connn11}, Theorem \ref{thmivan71}, and Theorem \ref{thmivan72}. 
\end{remark} 		
\begin{remark} 	
	\label{remev-1} 
The limit equations (\ref{limeq11})--(\ref{limeq13}) are obtained on a long time scale. The stiff component behaves like a perforated domain, and there is no coupling between its deformation and the deformation of the inclusions. The deformation of the inclusions and the even part of the in-plane deformation of the stiff component behave quasi-statically (i.e., without an inertia term), as a consequence of small forces slowly varying in time. (Recall that the physical equation is \eqref{evscaled1} with the right-hand side $\tilde{\vect f}^\epsh$ subject to an appropriate version of the condition \eqref{loadstimebdd1}.) Since there is no coupling in the limit between the inclusions and the stiff component, there are no memory effects in the time evolution. However, it is expected that high-contrast effects would be seen in higher-order terms (``correctors") of the deformation, which we do not pursue here.
%%% At this point, doing further asymptotics for the deformation seems to be a difficult task. 

Without making additional symmetry assumptions about the material properties, the limit operator for the evolution of the out-of-plane component is spatially non-local, due the coupling between the in-plane and out-of-plane components.

\end{remark} 	
\begin{remark}
We are not able to obtain pointwise in time convergence without additional assumptions on the load density. This is expected  (replacing weak two-scale convergence with strong two-scale convergence) also as \RRR a consequence \BBB of the analysis presented in \cite{pastukhova}. 	
\end{remark} 

\begin{remark} 
	The  influence of in-plane forces on the limit model is seen through their mean value across the plate, represented above by an integral over the interval $I=[-1/2, 1/2]$, as well as through the mean value of their moments over the same  interval $I$. In the case of planar symmetries, see Assumption 	\ref{assumivan1}\,(1), moments of in-plane forces have the same effect on the limit deformation as out-of-plane forces, i.e., they produce out-of-plane displacements. This is expected from the physical point of view and is standard for plate theories (see, e.g., \cite{ciarlet2}). 
\end{remark}

\begin{remark}
	\label{scalingrem}
Considering whether different components of the load density should be scaled differently is important from the modelling perspective. Indeed, if its in-plane and out-of-plane components had the same magnitude, one would not see the effects of the in-plane components in the (leading order of the) deformation. On the other hand, it is expected that sufficiently large in-plane loads do influence the limit deformation. 
%%Thus the obtained  models that do not take into account the possibility of different scalings of in-plane and out-of-plane component of loads  (in both static and evolution case) are in a way narrower in their scope.
However, for some of the asymptotic regimes analysed here the effects on the in-plane and out-of-plane loads on the limit deformation are similar, in which case these loads are set to have the same magnitude in the equations. This kind of situation also occurs in the context of linear elastic shells, see \cite{ciarlet1} for shells as compared to the case of linear elastic 
plates \cite{ciarlet2}.
\end{remark}

\subsubsection{Real-time behaviour for $\mu_h=\epsh$, $\tau=0$ in different regimes}
Here we discuss a class of evolution problems with ``non-standard" effective behaviour, which manifests itself, in particular, through time non-locality.   
\label{realtimeevol}
\begin{theorem} \label{connn11}
	Suppose that $\mu_h=\eps_h$, $\tau=0$, $\delta,\kappa \in [0,\infty],$ and
consider the sequence $(\vect{u}^{\epsh})_{h>0}$ of solutions to the problem \eqref{evolutionepsilonproblem}, assuming that 
	%%\begin{eqnarray}
	\begin{align}
\vect  f^\epsh & \drtwoscalet  \vect f  \in L^2([0,T];L^2(\Omega \times Y;\R^3)) \label{loadstimef21},\\ 
\nonumber
	\vect u_0^\epsh & \drtwoscale  \vect u_0(x,y) \in V_{1,\delta,\kappa} (\omega \times Y)+V_{2,\delta} (\Omega \times Y_0), \\ 
	\label{loadstimef23}
	\vect u_1^\epsh & \drtwoscale \vect u_1(x,y) \in L^2(\Omega \times Y;\R^3).
	%\end{eqnarray}
\end{align}
	Assume also that 
	\[
	\limsup_{h \to 0}\left(a_{\epsh}(\vect u_0^{\epsh},\vect u_0^{\epsh} )+\|\vect u_0^{\epsh}\|_{L^2}\right)<\infty. 
	\]
	Then one has 
%	\begin{eqnarray} 
\begin{align}
	\label{convergencefinal21}
	\vect u^\epsh & \drtwoscalet (\vect \funcA,\funcB)^\top+\mathring{\vect u}, \\ 
	\label{convergencefinal22}
\partial_t \vect u^\epsh & \drtwoscalet\partial_t\bigl( (\vect \funcA,\funcB)^\top+\mathring{\vect u} \bigr), 
%	\end{eqnarray}   
\end{align}
	where $(\vect{\funcA},\funcB)^\top +\mathring{\vect u} \in C([0,T];V_{1,\delta,\kappa}(\omega \times Y)+V_{2,\delta}(\Omega \times Y_0))$
	   is the unique weak solution of the problem
	\begin{align*}
	&\partial_{tt} \bigl((\vect \funcA,\funcB)^\top+\mathring{\vect u} \bigr) (t)+\mathcal{A}_{\delta,\kappa} \bigl((\vect \funcA,\funcB)^\top+\mathring{\vect u} \bigr) (t) = P_{\delta,\kappa} \vect f (t), \\[0.4em]
 &\bigl((\vect \funcA,\funcB)^\top+\mathring{\vect u}\bigr)(0)=\vect u_0(x,y),\qquad \partial_t \bigl((\vect \funcA,\funcB)^\top+\mathring{\vect u}\bigr)(0)=P_{\delta,\kappa}\vect u_1(x,y),
	\end{align*}	
such that $\partial_t \left((\vect \funcA,\funcB)^\top+\mathring{\vect u}\right) \in C([0,T];H_{\delta,\kappa}(\Omega \times Y))$. 
	Furthermore, the following limit energy bound holds:  
	$$\limsup_{h \to 0} \int_0^T \left(a_{\epsh}\bigl(\vect{u}^{\epsh}(t),\vect{u}^{\epsh}(t)\bigr)+\|\vect{u}^{\epsh}(t)\|^2_{L^2}\right)< \infty. $$  
	
	If strong two-scale convergence holds in \eqref{loadstimef21}--\eqref{loadstimef23} with $\vect f \in L^2([0,T];H_{\delta,\kappa}(\Omega \times Y))$, $\vect u_1 \in H_{\delta,\kappa}(\Omega \times Y),$ and 
	$$
	\lim_{h \to 0}\left(a_{\epsh} \bigl(\vect u_0^{\epsh}, \vect u_0^{\epsh}\bigr)+\|\vect u_0^{\epsh}\|^2_{L^2}\right) = a_{\delta,\kappa} \bigl(((\vect \funcA,0)^\top+\mathring{\vect u})(0), ((\vect \funcA,0)^\top+\mathring{\vect u})(0)\bigr)+\bigl\|\bigl((\vect \funcA,\funcB)^\top+\mathring{\vect u}\bigr)(0)\bigr\|^2_{L^2},
	$$ then strong two-scale convergence holds in \eqref{convergencefinal21}--\eqref{convergencefinal22}. Moreover, one has 
	$$\lim_{h \to 0}  \Bigl(a_{\epsh}\bigl(\vect u^{\epsh}(t), \vect u^{\epsh}(t)\bigr)+\|\vect u^{\epsh}(t)\|_{L^2}\Bigr) = a_{\delta,\kappa} \bigl(((\vect \funcA,0)^\top+\mathring{\vect u})(t), ((\vect \funcA,0)^\top+\mathring{\vect u})(t)\bigr)+\bigl\|\bigl((\vect \funcA,\funcB)^\top+\mathring{\vect u}\bigr)(t)\bigr\|^2_{L^2},
	$$
	for every $t \in[0,T]$.
\end{theorem}	
\begin{remark} \label{remev3} 
	The models obtained here are degenerate with respect to the out-of-plane component of the displacement. 
	%%, since no resistance is offered to out-of-plane motions. This degeneracy can be explained physically as follows. 
	Indeed, in the static case it is substantially easier for the plate to bend than to extend in-plane; however, in the dynamic case in real time, for the forces of magnitude one, there is no elastic resistance to out-of-plane motions, which are therefore entirely due to external loads. 
	%%thus in the out-of-plane the plate can only be accelerated out-of-plane loads. 
	
	It is also worthwhile noting that in the \RRR high-contrast \BBB    setting out-of-plane loads ${\vect f}=(0, 0, f_3)$ for which $\overline{f}_3=0$  do produce some in-plane motion in the case when $\delta \in (0,\infty]$, as a consequence of the coupling between the deformations on the stiff component and on the inclusions, which is not possible in the setting of homogenisation with moderate contrast. (In the regime $\delta=0,$ inclusions behave like small plates and thus only the effects of the average loads $\overline {\vect f}$ in the variable $x_3$ are seen in the limit.) On a related note, from the point of view of quantitative analysis, it is not expected that the effect elastic resistance to out-of-plane motions disappears entirely, as it may manifest itself 
	in  lower-order terms, see \cite{chered5} for a quantitative analysis of the resolvent equation for a thin infinite plate in moderate contrast.
	%%% with respect to different scalings of the operator). 
\end{remark} 
\begin{remark} \label{remev1} 
To the best of our knowledge, dynamic models representing ``real time behaviour" have not been discussed in the literature, even in the case of an ordinary plate. Certainly, these models are not as physically relevant as those in which elastic resistance to out-of-plane motions is observed. This might be due to the fact that for most materials mass density is much smaller than Lam\'{e} constants (in dimensionless terms). However, since these models exhibit high-contrast effects, which does not happen when the time is scaled (unless the coefficients on the inclusions are scaled in a non-standard way in relation to the coefficients on the stiff component), we find it is important to discuss them also. 	
\end{remark} 	
\begin{remark}
In the limit problem, due to the coupling of the deformation on the stiff component, given by $(\vect \funcA,\funcB)^\top,$ and the oscillatory part of the deformation on the soft component, given by $\mathring{\vect u},$ there are memory effects (under the assumption that the micro-variable $\mathring{\vect u}$ is unknown). The emergence of these memory effects can be seen as follows. If one would like to know the deformation on the stiff component at time $T,$ given by $(\vect \funcA,\funcB)^\top(T)$, one would not only require the initial data (deformation and speed) on the stiff component at an ``initial'' time $t_0<T$ and loads $\vect f$ on the time interval $[t_0,T],$ but also the value of the micro-variable $\mathring{\vect u}$ and its speed at time $t_0.$ It one cannot measure this micro-variable (which is a physically meaningful scenario), then the corresponding degree of freedom becomes ``hidden" internally, which results in a non-local time dependence macroscopically.
%%as one also requires information about the initial data at the moment when the  micro-variable $\mathring{\vect u}$ and its speed were both zero (or at least known) and the loads on the whole interval $[0,T]$.
\end{remark} 	

\subsubsection{Long-time behaviour for $\delta \in (0,\infty]$, $\mu_h=\epsh h$, $\tau=2$} \label{sectionltbhc1}
Here we demonstrate that by varying the contrast between material properties of the two components (``stiff" and ``soft"), the evolution problem may be shown to exhibit time non-locality also in the regime of long times.  
\begin{theorem} \label{thmivan71} 
	Suppose that $\delta \in (0,\infty]$, $\mu_h=\eps_h h$, $\tau=2$,  and
	let $(\vect{u}^{\epsh})_{h>0}$ be the sequence of solutions of the problem \eqref{evolutionepsilonproblem}, assuming that 
	%\begin{eqnarray*}
	\begin{align}
	\vect  f^\epsh & \drtwoscalet\vect f  \in L^2\bigl([0,T];L^2(\Omega \times Y;\R^3)\bigr), 
	\label{loadstimef31}
	\\ 
	\vect u_0^\epsh & \drtwoscale \vect u_0(\hat{x})+\RRR\mathring{\vect u}_0(x,y) \BBB \in \{0\}^2 \times H^2_{\gamma_{\rm D}} (\omega)+V_{2,\delta} (\Omega \times Y_0), \nonumber\\
	\vect u_1^\epsh & \drtwoscale \vect u_1(x,y) \in L^2(\Omega \times Y;\R^3) 
	\label{loadstimef33}. 
%	\end{eqnarray}
\end{align}
	Assume also that 
	\[
	\limsup_{h \to 0}\left(a_{\epsh}(\vect u_0^{\epsh},\vect u_0^{\epsh} )+\|\vect u_0^{\epsh}\|_{L^2}\right)<\infty. 
	\]
	Then one has 
	%\begin{eqnarray} 
	\begin{align}
	\label{convergencefinal31}\vect u^\epsh & \drtwoscalet (0,0,\funcB)^\top+\mathring{\vect u}, \\ 
	\label{convergencefinal32}
	\partial_t \vect u^\epsh & \drtwoscalet \partial_t\bigl( (0,0,\funcB)^\top+\mathring{\vect u} \bigr), 
	%\end{eqnarray}   
	\end{align}
	where $(0,\funcB)^\top +\mathring{\vect u} \in C([0,T];\{0\}^2 \times H^2_{\gamma_{\rm D}} (\omega)+V_{2,\delta}(\Omega \times Y_0))$
	is the unique weak solution to the problem 
	\begin{align*}
	%\begin{eqnarray*}
		\partial_{tt} 
		&\bigl( (0,0,\funcB)^\top+\mathring{\vect u} \bigr) (t)+\hat{\mathcal{A}}_{\delta} \bigl((0,0,\funcB)^\top+\mathring{\vect u} \bigr) (t) = \RRR \bigl( S_2 (P_{\delta,\infty} \vect f (t))_1,S_2 (P_{\delta,\infty} \vect f (t))_2, (P_{\delta,\infty} \vect f (t))_3  \bigr)^\top \BBB, \\[0.35em]
	&\bigl((0,0,\funcB)^\top+\mathring{\vect u}\bigr)(0)=\vect u_0(x,y),\qquad  \partial_t \bigl( \RRR(0,0,\funcB)^\top\BBB+\mathring{\vect u}\bigr)(0)=\RRR \bigl( S_2 (P_{\delta,\infty} \vect u_1)_1,S_2 (P_{\delta,\infty} \vect u_1 )_2, (P_{\delta,\infty} \vect u_1)_3  \bigr)^\top \BBB(x,y),
	%\end{eqnarray*}	
    \end{align*}
	such that $\partial_t \left((0,0,\funcB)^\top+\mathring{\vect u}\right) \in C([0,T]; H_{\delta,\infty}(\Omega \times Y))$. 
	Furthermore, the following limit energy bound holds: 	
	$$
	\limsup_{h \to 0} \int_0^T \left(h^{-2}a_{\epsh}(\vect{u}^{\epsh}(t),\vect{u}^{\epsh}(t))+\|\vect{u}^{\epsh}(t)\|^2_{L^2}\right)< \infty. 
	$$    
	
	If strong two-scale convergence holds in \eqref{loadstimef31}--\eqref{loadstimef33} with $\vect f \in L^2([0,T];H_{\delta,\infty}(\Omega \times Y))$, $\vect u_1 \in H_{\delta,\infty}(\Omega \times Y),$ 
	and 
	$$
	\lim_{h \to 0}  \left(h^{-2}a_{\epsh} (\vect u_0^{\epsh}, \vect u_0^{\epsh})+\|\vect u_0^{\epsh}\|^2_{L^2}\right) = \hat{a}_{\delta} \bigl(((0,0,\funcB)^\top+\mathring{\vect u})(0), ((0,0,\funcB)^\top+\mathring{\vect u})(0)\bigr)+\bigl\|\bigl((0,0,\funcB)^\top+\mathring{\vect u})(0)\bigr\|^2_{L^2},
	$$ 
	then strong two-scale convergence holds in \eqref{convergencefinal31}--\eqref{convergencefinal32}. Moreover, for every $t \in[0,T]$ one has
	 $$
	 \lim_{h \to 0}  \left(h^{-2}a_{\epsh} (\vect u^{\epsh}(t), \vect u^{\epsh}(t))+\|\vect u^{\epsh}(t)\|^2_{L^2}\right) = \hat{a}_{\delta} \bigl(((0,0,\funcB)^\top+\mathring{\vect u})(t), ((0,0,\funcB)^\top+\mathring{\vect u})(t)\bigr)+\bigl\|\bigl((0,0,\funcB)^\top+\mathring{\vect u})(t)\bigr\|^2_{L^2}.
	 $$	
	 	\end{theorem}			
\begin{remark} \label{remev4} 
The above limit model exhibits memory effects, due to the coupling of the deformations on the stiff component and on the inclusions, similarly to what happened in Section \ref{realtimeevol}. As before, see Remark \ref{remev-1},  in the case when $\delta \in (0,\infty)$ and no additional symmetries are imposed on the material properties, the limit macro-operator $\hat{\mathcal{A}}_{\delta}$ is spatially non-local. 
\end{remark} 	
\subsubsection{Long-time behaviour for $\delta=0$, $\mu_h={\epshtwo}$, $\tau=2$} \label{sectionltbhc2} 

Here we discuss an analogue of the result of the previous section for the case $\delta=0,$ in which we need to apply different scalings to the in-plane and out-of-plane loads. As already emphasized in Sections \ref{hlleps} (resolvent convergence), \ref{sectionbgl} (limit spectrum), in this regime we require that $Y_0$ have $C^{1,1}$ boundary. 
\begin{theorem} \label{thmivan72}
	Suppose that $\delta=0$, $\mu_h=\eps_h^2$, $\tau=2,$ and let $(\vect{u}^{\epsh})_{h>0}$ be the sequence of solutions to the problem \eqref{evolutionepsilonproblem}, assuming that 
%	\begin{eqnarray}
\begin{align}
	\left((h/\epsh)\partial_{t}\vect{f}^{\epsh}_{\alpha}\right)_{h>0} &\subset L^2([0,T];L^2(\Omega\times Y)) \textrm{\ \ is\ bounded},\qquad \alpha=1,2, 
	\label{fbound} \\
	\pi_{h/\epsh}\vect  f^\epsh &\drtwoscalet \vect f  \in L^2([0,T];L^2(\Omega \times Y;\R^3)),
	\nonumber
	\\ 
	\vect u_0^\epsh & \drtwoscale  \vect u_0 (\hat{x},y) \in \{0\}^2 \times H^2_{\gamma_{\rm D}} (\omega)+\{0\}^2 \times L^2(\omega \times Y_0), \label{loadsinitialpos4}\\
	\vect u_1^\epsh(x) & \drtwoscale  \vect u_1(x,y) \in L^2(\Omega\times Y;\R^3),\label{loadsinitialvel4} 
	%\end{eqnarray}
	\end{align}
	and assume additionally that 
	\[\limsup_{h \to 0}\left(h^{-2}a_{\epsh}(\vect u_0^{\epsh},\vect u_0^{\epsh})+\|\vect u^{\epsh}_0\|_{L^2} \right)<\infty.
	\] 
	Then one has 
%	\begin{eqnarray} 
\begin{align*}
	\pi_{\epsh/h}{\vect u}^\epsh & \drtwoscalet\left(\begin{array}{c}
	\mathring{u}_1 (t,\hat{x},y) -x_3\partial_{y_1}\mathring{u}_3 (t,\hat{x},y)  \\[0.3em]
\mathring{u}_2(t,\hat{x},y)-x_3\partial_{y_2}\mathring{u}_3 (t,\hat{x},y) \\[0.4em]
\funcB(t,\hat{x})+\mathring{u}_3 (t,\hat{x},y)
	\end{array}\right) \\[0.3em] 
	\partial_t {\vect u}^{\epsh}  & \drtwoscalet\bigl(0,0,
	\partial_t \left( \funcB(t,\hat{x})+\mathring{u}_3(t,\hat{x},y) \right) \bigr)^\top, 
	\end{align*}
%	\end{eqnarray}   
	where the pair $\funcB \in C([0,T];H^2_{\gamma_{\rm D}}(\omega))$, $\mathring{\vect{u}}\in C([0,T];L^2(\omega;H^1_{00}(I \times Y_0;\R^3))) $ form the unique weak solution of the problem 
	\begin{align}
%	\begin{eqnarray}
&	\partial_{tt}\left(  \funcB+\mathring{u}_3\right)(t)+\hat{\mathcal{A}}_0\left(\funcB+\mathring{u}_3\right)(t) =\mathcal{F}_0(\vect f),\quad \RRR \textrm{(see \eqref{pomoc5})}\BBB\nonumber
	\\[0.25em]
	&\left(\funcB+\mathring{u}_3\right)(0)= u_{0,3} \in H^2_{\gamma_{\rm D}}(\omega)+L^2(\omega \times Y_0),  \hspace{5ex}\partial_t (\funcB+\mathring{u}_3)(0)= P^0 u_{1,3} \in L^2(\omega)+L^2(\omega \times Y_0), 
	\nonumber\\[0.25em]
	\label{limeq43}
	&\RRR\tilde{\mathcal{A}}_{00,0} \BBB\mathring{\vect u}_* (t,\hat{x}, \cdot)=\vect{f}_* (t,\hat{x},\cdot),  
%	\end{eqnarray}	
\end{align}
	such that $\partial_t (\funcB(t,\hat{x}) +\mathring{u}_3(t,\hat{x},y))$ $\in C([0,T]; L^2(\omega)+L^2(\omega \times Y_0))$.
	Furthermore, the following limit energy bound holds:
	$$
	\limsup_{h \to 0} \int_0^T\left(h^{-2} a_{\epsh}(\vect{u}^{\epsh}(t),\vect{u}^{\epsh}(t))+\|\vect{u}^{\epsh}(t)\|^2_{L^2}\right)\, dt< \infty.
	$$ 
	 
	If one additionally assumes that 
	\begin{align*}
	&\pi_{h/\epsh}\vect{f}^{\epsh} \strongdrtwoscalet (0,0,\vect{f})^\top \in L^2\left([0,T];L^2(\omega;\R^3)+L^2(\omega\times Y_0;\R^3)\right),\\[0.4em] 
	&(h/\epsh)\partial_{t}\vect{f}^{\epsh}_{\alpha}\longrightarrow %%\xrightarrow{L^2}
	 0\ \ \textrm{strongly in}\ L^2\left([0,T];L^2(\Omega)\right),\qquad \alpha=1,2,
	\end{align*}
	the two-scale convergence in \eqref{loadsinitialpos4} and \eqref{loadsinitialvel4} holds in the strong sense with $(\vect u_{1})_*=0$, $\vect{u}_{1,3}\in L^2(\omega)+L^2(\omega \times Y_0),$ and that 
	\[
	\lim_{h \to 0} \left(h^{-2}a_{\epsh}(\vect u_0^{\epsh},\vect u_0^{\epsh})+\|\vect u_0^{\epsh} \|^2\right)=a_{\delta}^{\funcB}\bigl((\funcB+\mathring{u}_3)(0),(\funcB+\mathring{u}_3)(0) \bigr)+\bigl\|(\funcB+\mathring{u}_3)(0)\bigr\|^2_{L^2},
	\] 
then one has 
	\begin{equation*} 
	\vect u^{\epsh} \strongdrtwoscalet (0,0, \funcB(t,\hat{x})+\mathring{u}_3(t,\hat{x},y))^\top, \quad \partial_t \vect u^{\epsh} \strongdrtwoscalet (0,0, \partial_t(\funcB(t,\hat{x})++\mathring{u}_3(t,\hat{x},y) ))^\top. 
	\end{equation*} 
	Moreover, for every $t \in[0,T]$ the convergence
	 \[
	 \lim_{h \to 0}\left(h^{-2} a_{\epsh} (\vect u^{\epsh}(t),\vect u^{\epsh}(t))+\|\vect u^{\epsh}(t)\|^2_{L^2}\right)=a_{\delta}^{\funcB}\bigl((\funcB+\mathring{u}_3)(t),(\funcB+\mathring{u}_3)(t)\bigr)+\bigl\|(\funcB+\mathring{u}_3)(t)\bigr\|^2_{L^2}
	 \] 
	 holds.
	
\end{theorem}	
\begin{remark} \label{remev5}  
In the regime $\delta=0$ inclusions behave like small plates and thus the corresponding deformation satisfies a version of the classical Kirchhoff-Love ansatz. Using the rationale discussed in Remark \ref{scalingrem}, we argue that in order to see the effects of both in-plane and out-of-plane components of loads in the limit model, we should scale them differently to one another. 

Similarly to the regime $\delta\in(0,\infty)$, $\mu_h=\epsh$, $\tau=2,$ we impose a restriction on the time derivatives of in-plane forces, see (\ref{fbound}), which in terms of the ``physical" time corresponds to slowly acting loads. This results in a (partial) quasi-static evolution in the limit, see \eqref{limeq43}. 
%%%As in the case $\delta\in(0,\infty)$, $\mu_h=\epsh$, $\tau=2$ 
Furthermore, in order to obtain strong two-scale convergence of solutions, akin to (\ref{ivan1012}),  we impose a further restriction that properly scaled in-plane forces together with their time derivatives, in the spirit to (\ref{ivan1011}),  go to zero as $h\to0.$ 	
\end{remark} 	
\section{Proofs} 
\label{proofs}
\subsection{Proof of Proposition \ref{propvecer}}
%%	Proofs for Section \ref{effective}}
 \label{sectinproof3.1}  

%%\vskip 0.1cm

%%\noindent{\bf A. Proof of Proposition \ref{propvecer}}
\begin{proof}
	We provide the proof for the case $\delta \in (0,\infty)$; the other cases are dealt in a similar fashion, bearing in mind Remark \ref{nada10001} and Remark \ref{nada10001h}. 
	
	Consider the minimiser $\vect \varphi \in H^1(I \times \mathcal{Y};\R^3)$ in the variational formulation \eqref{tensorminimization}. Then for arbitrary symmetric matrices $\vect A, \vect B \in \R^{2\times 2}_{\rm sym}$ one has a lower bound for elastic stored energy density, as follows:
	\begin{equation}
	\label{coerc}
	\begin{split}
	\C^{\rm hom}_{\delta}(\vect A,\vect B):(\vect A,\vect B) & \geq C \int\limits_I \left\Vert \iota( \vect A- x_3 \vect B) + \sym\widetilde{\nabla}_{2,\gamma}\, \vect \varphi(x_3, \cdot)\right\Vert_{L^2( Y_1;\R^{3\times 3})}^2 dx_3 \\
	& \geq C \int\limits_I \left\Vert  \vect A- x_3 \vect B + \simgrad_{y}\, \vect \varphi_*(x_3, \cdot) \right\Vert_{L^2( Y_1;\R^{2\times 2})}^2 dx_3,
	\end{split}
	\end{equation}
	due to the coercivity of the tensor $\C_1$ representing the elastic properties on the stiff component. In order to eliminate the corrector $\vect \varphi_*$ from the bound (\ref{coerc}), we first construct an extension for it from $Y_1$ to the whole cell $Y$ for each $x_3\in I$. To this end, we first define the symmetric affine part of an arbitrary $H^1$ function, as follows. For $\vect \xi = (\vect \xi_1, \vect \xi_2 )^\top \in H^1(Y_1;\R^2),$ we consider the function $\hat{\vect \xi} \in H^1(Y; \R^2)$ defined by
	\begin{equation*}
	\begin{split}
	\hat{\vect \xi}(y) & := 
	%\begin{pmatrix}
	%\fint\limits_{Y_1} \vect \varphi_1 (y)dy + \fint\limits_{Y_1} %\partial_1 \vect \varphi_1(y)dy \left(y_1 - \fint\limits_{Y} y_1 dy  \right) + \fint\limits_{Y_1} \left( \frac{\partial_2 \vect \varphi_1(y) + \partial_1 \vect \varphi_2(y) }{2}  \right) dy \left(y_2 - \fint\limits_{Y} y_2 dy  \right)\\ 
	%\fint\limits_{Y_1} \vect \varphi_2 (y)dy + \fint\limits_{Y_1} \left( \frac{\partial_2 \vect \varphi_1(y) + \partial_1 \vect \varphi_2(y) }{2}  \right) dy \left(y_1 - \fint\limits_{Y} y_1 dy  \right) + \fint\limits_{Y_1} \partial_1 \vect \varphi_1(y)dy \left(y_2 - \fint\limits_{Y} y_2 dy  \right)
	%\end{pmatrix} \\
	 \fint\limits_{Y_1} \vect \xi(y) dy  + \fint\limits_{Y_1} \simgrad_y \vect \xi(y) dy\, \biggl( y - \fint\limits_{Y_1} y dy\biggr).
	\end{split}
	\end{equation*}
	Notice that the operator $\hat{\cdot}$ is linear and satisfies the following properties:
	\begin{equation*}
	\begin{split}
	\nabla_y \hat{\vect \xi} = \simgrad_y \hat{\vect \xi} = \fint\limits_{Y_1} \simgrad_y \vect \xi(y)dy, \qquad \fint\limits_{Y} \hat{\vect \xi}(y) dy = \fint\limits_{Y_1} \vect \xi(y) dy, \\
	\left\Vert \simgrad_y \hat{\vect \xi } \right\Vert_{L^2(Y;\R^{2\times 2})} \leq |Y|/|Y_1|\left\Vert \simgrad_y \vect \xi \right\Vert_{L^2(Y_1;\R^{2\times 2})}.
	\end{split}
	\end{equation*}
	Now we define the extension operator $\hat{E} : H^1(Y_1 ; \R^2) \to H^1(Y;\R^2),$ via 
	\begin{equation*}
	\hat{E} \vect \xi := E (\vect \xi - \hat{\vect \xi}) + \hat{\vect \xi },
	\end{equation*}
	where $E$ is the extension operator from \cite[Lemma 4.1]{OShY}, which satisfies the estimate
	 $$\left\Vert \simgrad_y(E \vect \xi) \right\Vert_{L^2(Y;\R^{2\times 2})} \leq C \left\Vert \simgrad_y \vect \xi \right\Vert_{L^2(Y_1;\R^{2\times 2})}.$$
	It is easy to see that 
	\begin{equation} \label{josipispravak1} 
	 \left\Vert \simgrad_y(\hat{E} \vect \xi) \right\Vert_{L^2(Y;\R^{2\times 2})}^2 \leq C\left\Vert \simgrad_y \vect \xi \right\Vert_{L^2(Y_1;\R^{2\times 2})}^2. 
	 \end{equation} 
	%\begin{equation*}
	%\begin{split}
	%\left\Vert \simgrad_y( \hat{E} \vect \varphi) \right\Vert_{L^2(Y;\R^{2\times 2})}^2 & =  \left\Vert \simgrad_y(E( \vect \varphi - \hat{\vect \varphi })) + \fint\limits_{Y_1} \simgrad_y \vect \varphi(y)dy \right\Vert_{L^2(Y;\R^{2\times 2})}^2 \\
	%& = \left\Vert \simgrad_y(E( \vect \varphi - \hat{\vect \varphi })) \right\Vert_{L^2(Y;\R^{2\times 2})}^2 + \left\Vert \fint\limits_{Y_1} \simgrad_y \vect \varphi(y)dy \right\Vert_{L^2(Y;\R^{2\times 2})}^2 \\
	%& \leq C\left\Vert \simgrad_y( \vect \varphi - \hat{\vect \varphi }) \right\Vert_{L^2(Y_1;\R^{2\times 2})}^2 + \leq C\left\Vert \simgrad_y \vect \varphi \right\Vert_{L^2(Y_1;\R^{2\times 2})}^2 \leq  C\left\Vert \simgrad_y \vect \varphi \right\Vert_{L^2(Y_1;\R^{2\times 2})}^2.
	%\end{split}
	%\end{equation*}
	
	Next, consider the function $$\vect \psi(y) := \left(\vect A - x_3 \vect B \right)y + \vect \varphi_* (y).$$ 
Clearly, one has 
	\begin{equation*}
	\hat{E}\vect \psi(y) = E(\vect \varphi_*  - \hat{\vect \varphi}_* )(y) + (\vect A-x_3 \vect B )y + \hat{\vect \varphi}_* (y) = \hat{E}\vect \varphi_* (y) + (\vect A-x_3 \vect B )y.
	\end{equation*}
	Furthermore, from \eqref{josipispravak1} one has
	\begin{equation} \label{josipispravak2} 
	\left\Vert \simgrad_y(\hat{E} \vect \psi) \right\Vert_{L^2(Y;\R^{2\times 2})}^2 \leq C\left\Vert \simgrad_y \vect \psi \right\Vert_{L^2(Y_1;\R^{2\times 2})}^2 = C\left\Vert\vect A - x_3 \vect B  + \simgrad_y \vect \varphi_*  \right\Vert_{L^2(Y_1;\R^{2\times 2})}^2.
	\end{equation}
	At the same time, the following bound holds:
	\begin{equation}\label{josipispravak3} 
	\begin{split}
	\left\Vert \simgrad_y(\hat{E} \vect \psi) \right\Vert_{L^2(Y;\R^{2\times 2})}^2 & = \left\Vert \simgrad_y(\hat{E}\vect \varphi_*)  + (\vect A-x_3 \vect B ) \right\Vert_{L^2(Y;\R^{2\times 2})}^2 \\[0.3em]
	& = \left\Vert \simgrad_y(\hat{E}\vect \varphi_*)  \right\Vert_{L^2(Y;\R^{2\times 2})}^2  + \left\Vert \vect A-x_3 \vect B   \right\Vert_{L^2(Y;\R^{2\times 2})}^2
	%% \\
	%%& 
	\geq |\vect A - x_3 \vect B |^2.
	\end{split}
	\end{equation}
	Integrating \eqref{josipispravak3} over $I$ and taking into account  \eqref{josipispravak2} and then \eqref{coerc}, the claim follows.  
\end{proof}
\subsection{Proofs for Section \ref{limreseq}} 
\vskip 0.3cm
\noindent{\bf A. Proof of Proposition \ref{propdinsi1}}
\begin{proof} 
Notice first that using $\vect u^\epsh$ as a test function in \eqref{resolventproblemstarting} immediately yields 
\begin{equation}\label{rucak2} 
h^{-2}a_{\epsh}(\vect u^{\epsh},\vect u^{\epsh}) + \|\vect u^\epsh \|_{L^2(\Omega;\R^3)}^2 \leq C\, \|\pi_h \vect  f^\epsh\|_{L^2(\Omega;\R^3)} \|\pi_{1/h} \vect u^\epsh  \|_{L^2(\Omega;\R^3)}.
\end{equation}	 
Next, we define $\Tilde{ \vect u}^\epsh$ by applying Theorem \ref{thmextension} to extend $\vect u^\epsh|_{\Omega_1^\epsh}$ to the whole domain $\Omega$ and set
$$
\mathring{\vect u}^\epsh: = \Tilde{\vect u}^\epsh - \vect u^\epsh.
$$ 
Furthermore, Theorem \ref{thmextension} and Lemma \ref{poincarekornonholes} imply
\begin{equation} \label{rucak1} 
\|\sym \nabla_h \Tilde { \vect u}^\epsh  \|^2_{L^2} \leq C a_{\epsh}(\vect u^{\epsh},\vect u^{\epsh}),\qquad h^{-2}\|\mathring{\vect u}^\epsh\|^2_{L^2}+\|\nabla_h \mathring{\vect u}^\epsh\|^2_{L^2} \leq Ch^{-2} a_{\epsh}(\vect u^{\epsh},\vect u^{\epsh}). 
\end{equation} 
Combining Corollary \ref{kornthincor} with \eqref{rucak1}, we obtain
\begin{equation} \label{rucak3} 
 \|\pi_{1/h} \vect u^{\epsh}\|^2_{L^2}\leq 2\|\pi_{1/h} \Tilde{\vect u}^{\epsh}\|^2_{L^2}+2\|\pi_{1/h}\mathring {\vect u}^{\epsh}\|^2_{L^2}\leq h^{-2} a_{\epsh}(\vect u^{\epsh},\vect u^{\epsh}). 
 \end{equation} 
 The claim in part 1 now follows directly from \eqref{rucak2} and \eqref{rucak3}. 
 
 Proceeding to the proof of part 2, we notice that the fourth convergence in \eqref{nakcomp1} is a direct consequence of \eqref{rucak1} and Theorem \ref{neukammresult}\,(1b).
 To prove the first and second convergence in \eqref{nakcomp1},  we use Lemma \ref{app:lem.limsup} and \eqref{rucak1}.  
  Lemma \ref{app:lem.limsup}\,(3) now yields the following  decomposition of the sequence $\Tilde{\vect u}^\epsh:$
  %% in the following way:
 \begin{equation*}
% \begin{split}
 \frac{1}{h}\Tilde{\vect u}^\epsh(x)=\left(\begin{array}{c} -x_3 \partial_1 \funcB \\[0.3em] -x_3 \partial_2 \funcB \\[0.25em] h^{-1}\funcB \end{array} \right)+      \left(\begin{array}{c} \funcA_1  \\[0.2em] \funcA_2 \\[0.2em] 0\end{array} \right) + \vect \psi^\epsh, \qquad\qquad
 \frac{1}{h}\simgrad_h \Tilde{\vect u}^\epsh = \iota\left(\simgrad_{\hat{x}}\vect \funcA-x_3\nabla_{\hat{x}}^2\funcB  \right) + \simgrad_h \vect\psi^\epsh,
 %\end{split}
 \end{equation*}
 where $\funcB \in H^2_{\gamma_{\rm D}}(\omega)$, $\vect \funcA \in H^1_{\gamma_{\rm D}}(\omega;\R^2)$, and  $(\vect\psi^\epsh)_{h>0} \subset H^1_{\Gamma_{\rm D}}(\Omega;\R^3)$ is such that $h\pi_{1/h} \vect\psi^\epsh\longrightarrow 0$ in $L^2.$ 
 
 To prove the third convergence in \eqref{nakcomp1},  we first assume that $\omega$ has $C^{1,1}$ boundary. By virtue of Lemma \ref{keydecompose}\,(3), there are sequences $(\varphi^\epsh)_{h>0} \subset H^2_{\gamma_{\rm D}}(\omega)$,  $(\tilde{\vect\psi}^\epsh)_{h>0} \subset H^1_{\Gamma_{\rm D}}(\Omega;\R^3)$, $({\vect o}^\epsh)_{h>0} \subset L^2(\Omega; \R^{3\times 3})$ such that
 \begin{equation*}
 \simgrad_h \vect\psi^\epsh =-x_3\iota\bigl(\nabla_{\hat{x}}^2 \varphi^\epsh \bigr) + \simgrad_h \tilde{\vect\psi}^\epsh + {\vect o}^\epsh,
 \end{equation*}
 where
 \begin{align*} 
 %\begin{equation*}
 %%\begin{split}
 \varphi^\epsh &\xrightarrow{L^2} 0, \quad \nabla_{\hat{x}} \varphi^\epsh \xrightarrow{L^2} 0, \quad \lVert \nabla_{\hat{x}}^2\varphi^\epsh  \rVert_{L^2} \leq C, \\[0.1em]
 \tilde{\vect \psi}^\epsh &\xrightarrow{L^2} 0, \quad \lVert \nabla_h \tilde{\vect \psi}^\epsh  \rVert_{L^2} \leq C, \\[0.1em]
 {\vect o}^\epsh &\xrightarrow{L^2} 0.
 %%\end{split}
 %\end{equation*}
 \end{align*}
 In view of Lemma \ref{lemmatwoscalecompact}\,(1) and Theorem \ref{neukammresult}\,(1a), 
 there exist
 $z \in L^2(\omega;H^2(\mathcal{Y}))$ and  $\tilde{\vect \psi} \in L^2(\omega;H^1(I \times \mathcal{Y};\R^3))$
 such that (up to extracting a subsequence)  
 \begin{align*}
%% \begin{eqnarray*}
 \nabla_{\hat{x}}^2 \varphi^\epsh(\hat{x}) &\drtwoscale  \nabla_y^2 z(\hat{x},y), \\[0.3em]
 \simgrad_h \tilde{\vect \psi}^\epsh(x) &\drtwoscale\sym\widetilde{\nabla}_{\delta}\, \tilde{\vect \psi}(x,y).
 %%\end{eqnarray*}
 \end{align*}
 Introducing the function 
 \begin{equation*}
 \vect \varphi(x,y):= \left(\begin{array}{c} -x_3 \partial_{y_1} z(\hat{x},y)   \\[0.3em] -x_3 \partial_{y_2} z(\hat{x},y) \\[0.3em] \delta^{-1}z(\hat{x},y) \end{array} \right) + \tilde{\vect \psi}(x,y),
 \end{equation*}
 we have 
 \begin{equation*}
 \sym\widetilde{\nabla}_{\delta}\, \vect \varphi(x,y) = -x_3 \iota\bigl( \nabla_y^2 z (\hat{x},y) \bigr) +  \sym\widetilde{\nabla}_{\delta}\, \tilde{\vect \psi}(x,y),
 \end{equation*}
 from which the third convergence in \eqref{nakcomp1} follows. 
 
 Next we can extend this result to the case of an arbitrary Lipschitz domain.
 In the general case we can only conclude that since $h^{-1}\simgrad_h \Tilde{\vect u}(x)$ is bounded in $L^2(\Omega;\R^3)$ there exists $\funcC \in L^2(\Omega \times Y;\R^{3\times 3})$ such that 
 \begin{equation*} 
h^{-1}\simgrad_h \Tilde{\vect u}^\epsh(x) \drtwoscale \, \iota\bigl(\simgrad_{\hat{x}} \vect \funcA(\hat{x})-x_3 \nabla_{\hat{x}}^2\funcB(\hat{x})\bigr) + \funcC(x,y).
\end{equation*}  
 Take a sequence $(\omega_n)_{n\in \N}$ of increasing domains with $C^{1,1}$ boundary such that $\omega_n \subset \omega$, $\cup_{n\in \N} \omega_n = \omega.$ By the preceding analysis we conclude that for every $n \in \N$ there exists $\vect \varphi^n \in  L^2(\omega_n;H^1(I \times \mathcal{Y};\R^3))$ such that
 \begin{equation*}
 \funcC(x,y) = \sym\widetilde{\nabla}_{2,\delta}\, \vect \varphi^n(x,y) \quad \mbox{a.e.}\  \hat{x}\in\omega_n,\ (x_3, y)\in I\times\mathcal{Y}.
 %%, \forall n\in \N.
 \end{equation*}
 Furthermore, notice that
 \begin{equation*}
 \bigl\Vert \sym\widetilde{\nabla}_{2,\delta}\, \vect \varphi^n \bigr\Vert_{L^2(\omega^n\times I \times Y;\R^{3\times 3})} \leq \lVert \funcC \rVert_{L^2(\Omega \times Y ;\R^{3\times 3})}, \quad \forall n\in \N.
 \end{equation*}
 Finally, we extend $\vect \varphi^n$ by zero outside $\omega_n \times I$. 
 The claim follows from the fact that $\mathcal{C}_{\delta} (\Omega \times I)$ is weakly closed, which in turn is a consequence of Korn's inequality for functions in $\dot{H}^1(I \times \mathcal{Y};\R^3)$ (see \cite[Theorem 6.3.8]{Neukamm10}). 

 To prove part 3, we first notice that
 $$
 \lim_{h \to 0} h^{-2}a_{\epsh}(\vect u^{\epsh},\vect u^{\epsh})= a_{\delta}^{\funcB}(\funcB,\funcB). 
 $$
 Using lower semicontinuity of convex functionals with respect to weak two-scale convergence and the definition of $a_{\delta}^{\funcB},$ we conclude that $\sym\widetilde{\nabla}_{\delta}\, \mathring{\vect u}(x,y)=0$, $\vect \funcA=\vect \funcA^{\funcB}$ and that $\funcC(x,\cdot)$ solves the minimisation problem \eqref{tensorminimization} with $\vect A=\simgrad_{\hat{x}} \vect \funcA(\hat{x})$ and $\vect B=\nabla_{\hat{x}}^2\funcB(\hat{x})$. 
 
 The strong two-scale convergence claim of part 3 as well as Remark \ref{rucak11} follow from the strict convexity of the tensors $\mathbb{C}_{\alpha}$, $\alpha=1,2,$  viewed as quadratic forms on symmetric matrices. 
 \end{proof}

\vskip 0.3cm

\noindent{\bf B. Proof of Theorem \ref{thmivan111}}  

\begin{proof} 
We choose the test function $\vect v$ in \eqref{resolventproblemstarting} to be of the form
\begin{equation*}
\vect v^\epsh(x) = 
\left(\begin{array}{c}
h \theta_1(\hat{x})-hx_3 \partial_1 \theta_3(\hat x)\\[0.3em]
h \theta_2(\hat{x})-hx_3 \partial_2 \theta_3(\hat x) \\[0.3em]
 \theta_3(\hat{x})
\end{array}\right)
+
h\epsh\, \vect\zeta\left(x,\frac{\hat{x}}{\epsh}\right)
+
h\mathring{\vect\xi}\left(x,\frac{\hat{x}}{\epsh}\right),
\end{equation*}
where $\vect \theta_*  \in C_{\rm c}^{1}(\omega;\R^2)$, $ \theta_3 \in C_{\rm c}^{2}(\omega)$,  $\vect\zeta \in C_{\rm c}^1(\Omega; C^{1} (I \times \mathcal{Y};\R^3))$, $\mathring{\vect\xi} \in C_{\rm c}^1(\omega;C_{00}^1(I \times Y_0;\R^3)).$
The arbitrary choice of $\vect \zeta$ and a density argument imply 
%%that $C(x,\cdot)$ satisfies pointwise
$$ \int\limits_I \int\limits_{Y_1} \C_1(y) \left[\iota\left( \nabla_{\hat{x}} \vect \funcA-x_3 \nabla^2_{\hat{x}}\funcB \right) + C(\hat{x},\cdot)\right] :  \sym\widetilde{\nabla}_{\delta}\, \vect \zeta \,dydx_3=0\qquad {\rm a.e.\hat{x}\in\omega}, 
$$
from which the effective tensor $\mathbb{C}_{\delta}^{\rm hom}$ is then obtained. Another density argument and Proposition \ref{propdinsi1} now provide the validity of the equations \eqref{limitmodellongtime}.  
The uniqueness of the solution to \eqref{limitmodellongtime} follows from Lax-Milgram and Proposition \ref{propvecer}, while the last claim follows by energy considerations or by duality arguments \cite[Proposition 2.8]{zhikov2005}, see also the proof of Theorem \ref{rucak72}. 
\end{proof}
\begin{remark} \label{remivan13}
It is not difficult to incorporate surface loads into the statement of Theorem \ref{thmivan111}. Namely, if one adds to the right-hand side of  \eqref{resolventproblemstarting} the term
$$ \int_{\omega \times\{-1/2,1/2\}} \vect g^{\epsh} \vect \theta\, d\hat{x}, \qquad \vect \theta \in H^1_{\Gamma_{\rm D}}(\Omega;\R^3),$$ 
where $\vect g^{\epsh} \in L^2(\omega \times \{-1/2, 1/2\};\R^3)$ and the integral over $\omega \times\{-1/2,1/2\}$ represents a sum of two integrals over $\omega,$ and assumes that
$$
	\pi_{h}\vect  g^\epsh  \drtwoscalet  \vect g  \in L^2(\omega \times \{-1/2, 1/2\} \times Y;\R^3),
$$
then using the proof of Theorem \ref{thmivan111} and Remark \ref{remivan102}, one concludes that the limit equations \eqref{limitmodellongtime} have an additional term 
$$ \int_{\omega\times \{-1/2,1/2\}} \langle\vect g\rangle(\hat x) \cdot \vect \theta (\hat x)\, d\hat{x}-\int_{\omega}\left(\RRR\langle\vect g_* (\hat x,1/2, \cdot)\rangle-\langle\vect g_* (\hat x,-1/2,\cdot)\BBB\rangle \right)\cdot\nabla_{\hat{x}} \vect \theta_3 (\hat x)\, d\hat{x}, $$
in the first equation and 
$$ \int_{Y_0}  \vect g(\hat x,-1/2,y)\cdot
\bigl(\mathring{\xi}_1(-1/2,y), 
\mathring{\xi}_2(-1/2,y),
0\bigr)^\top\RRR\,dy\BBB+\int_{Y_0}  \vect g(\hat x,1/2,y)\cdot\bigl(
\mathring{\xi}_1(1/2,y), 
\mathring{\xi}_2(1/2,y),
0\bigr)^\top \,dy,  
$$
in the second equation. 
\end{remark}	

\vskip 0.2cm

\noindent{\bf C. Proof of Proposition \ref{lemmaconvergencerealtime} and Corollary \ref{remsim1}}

\begin{proof}
	The proof partially follows the proof of Proposition \ref{propdinsi1}. Part 1 is obtained immediately by plugging ${ \vect v}^{\epsh}={\vect u}^{\eps_h}$ in \eqref{resolventproblemstarting}.

	 Proceeding to part 2, we perform an extension procedure similar to that undertaken in the proof of Proposition \ref{propdinsi1}. Using Theorem \ref{thmextension}, we define $\Tilde{ \vect u}^\epsh$ as the extension of $\vect u^\epsh|_{\Omega_1^\epsh}$ to the whole domain $\Omega$ and then set
	$$\mathring{\vect u}^\epsh:= \Tilde{\vect u}^\epsh - \vect u^\epsh.$$ Theorem \ref{thmextension} and Lemma \ref{poincarekornonholes} now imply the estimates \eqref{rucak1}.
	  
	Next, we characterise the behaviour of the sequence $\Tilde{\vect u}^\epsh.$  To this end, notice that Lemma \ref{app:lem.limsup} yields the following decomposition of the sequence $\Tilde{\vect u}^\epsh:$
	%% in the following way:
	\begin{equation*}
	%\begin{split}
	\Tilde{\vect u}^\epsh(x)=\left(\begin{array}{c} -x_3 \partial_1 \Tilde{\funcB} \\[0.3em] -x_3 \partial_2 \Tilde{\funcB} \\[0.3em] h^{-1}\Tilde{\funcB} \end{array} \right)+      \left(\begin{array}{c} \funcA_1  \\[0.3em] \funcA_2 \\[0.3em] 0\end{array} \right) + \vect \psi^\epsh,\qquad\quad 
	%%\\
	\simgrad_h \Tilde{\vect u}^\epsh = \iota\bigl(-x_3 \nabla_{\hat{x}} \Tilde{\funcB} + \simgrad_{\hat{x}} \vect \funcA \bigr) + \simgrad_h \vect\psi^\epsh,
	%\end{split}
	\end{equation*}
	where $\Tilde{\funcB} \in H^2_{\gamma_{\rm D}}(\omega)$, $\funcA \in H^1_{\gamma_{\rm D}}(\omega;\R^2)$, $(\vect\psi^\epsh)_{h>0} \subset H^1_{\Gamma_{\rm D}}(\Omega;\R^3),$ and  $h\pi_{1/h}\vect\psi^{\epsh}\longrightarrow 0$ in $L^2$. Since ${ u}_3^\epsh,$ and hence $\Tilde{ u}_3^\epsh$ as well, is bounded in $L^2(\Omega;\R^3)$  (see Lemma \ref{poincarekornonholes}), we infer that
	\begin{equation*}
	\Tilde{\funcB} = h \Tilde{u}_3^\epsh - h \psi_3^\epsh \xrightarrow{L^2} 0,
	\end{equation*}
	so consequently $\Tilde{\funcB}= 0$. By  Theorem \ref{aux:thm.griso}, we can decompose the third component as $\Tilde{u}_3^\epsh = \hat{\psi}^\epsh_3 + \bar{\psi}^\epsh_3$, where $ \hat{\psi}^\epsh_3=\int_I \Tilde{u}_3^\epsh$ and  $\lVert \bar{\psi}_3^\epsh \rVert_{L^2(\Omega)} \leq Ch$. Thus, by two-scale compactness, we conclude that there exists $\funcB \in L^2(\omega \times Y;\R^3)$ such that 
	$$ \Tilde{u}_3^\epsh(x) = \psi_3^\epsh \drtwoscale \funcB(\hat{x},y).$$  Furthermore, by invoking Remark \ref{remivan33} and applying Lemma \ref{lemmadependenceony}\,(1), we note that $\funcB(\hat{x},y) = \funcB(\hat{x})$. The rest of the proof is analogous to that of Proposition \ref{propdinsi1}. 
	
	To prove Corollary \ref{remsim1}, 
	we invoke Remark \ref{remivan}, Remark \ref{rucak30}, as well as the symmetries of the solution due to the assumption concerning the symmetries of the elasticity tensor. 
\end{proof}

\vskip 0.2cm
\noindent{\bf D. Proof of Theorem \ref{thmivan11111}}
\begin{proof} 
	We begin by plugging in \eqref{resolventproblemstarting} test functions of the form
	\begin{equation*}
	\vect v^\epsh(x) = 
	\left(\begin{array}{c}
	 \theta_1(\hat{x})-hx_3 \partial_1  \theta_3(\hat{x}))\\[0.3em]
	 \theta_2(\hat{x})-hx_3 \partial_2  \theta_3(\hat{x}) \\[0.3em]
	\theta_3(\hat{x})
	\end{array}\right)
	+
	\epsh\, \vect\zeta\left(x,\frac{\hat{x}}{\epsh}\right)
	+
	\mathring{\vect\xi}\left(x,\frac{\hat{x}}{\epsh}\right),
	\end{equation*}
	where $\vect \theta_*  \in C_{\rm c}^{1}(\omega;\R^2)$, $ \theta_3 \in C_{\rm c}^{2}(\omega)$,  $\vect\zeta \in C_{\rm c}^1(\Omega; C^{1} (I \times \mathcal{Y};\R^3))$, $\mathring{\vect\xi} \in C_{\rm c}^1(\omega;C_{00}^1(I \times Y_0;\R^3)),$
	and using the compactness result from Proposition \ref{lemmaconvergencerealtime}. The rest of the argument follows the proof of Theorem \ref{thmivan111}. 
	\end{proof} 

\vskip 0.2cm	
\noindent{\bf E. Proof of Proposition \ref{compactnessanotherscaling} and Theorem \ref{thmresolventhighercontrast} }
\begin{proof} 
To obtain part 1 of Proposition \ref{compactnessanotherscaling}, we plug $\vect u^{\epsh}$ in \eqref{resolventproblemstarting}. The rest of the proof of Proposition \ref{compactnessanotherscaling} and the proof of Theorem \ref{thmresolventhighercontrast} follow the steps of the proofs of Proposition \ref{propdinsi1} and Theorem \ref{thmivan111}, respectively. 
\end{proof} 

\vskip 0.2cm
\noindent{\bf F. Proof of Proposition \ref{propcompactregime11} }
\begin{proof} 
	
To prove part 1, we first plug in ${ \vect v}^{\epsh}={\vect u}^{\eps_h}$ in \eqref{resolventproblemstarting}. Next, using Theorem \ref{thmextreg1}, Corollary \ref{kornthincor}, and Remark \ref{remivan20}, we obtain the following a priori bounds:
\begin{align*}
\nonumber  
& \vect u^{\eps_h}= \tilde{\vect u}^{\eps_h}+\mathring{\vect u}^{\eps_h}, \quad \tilde{\vect u}^{\eps_h}=E^{\epsh} \vect u^{\epsh},\\[0.3em]
 & 
\nonumber	\bigl\|\sym \nabla_h\tilde{\vect u}^{\eps_h}\bigr\|_{L^2}+ h^2\bigl\| \pi_{1/h} \tilde{\vect u}^{\eps_h}\bigr\|^2_{H^1}+\| \tilde{\vect u}^{\eps_h}\|^2_{L^2}\leq  C \left(a_{\epsh} (\vect u^{\epsh},\vect u^{\epsh})+\|{\vect u}^{\eps_h}\|^2_{L^2}\right), \\[0.3em] 
& 
\nonumber	\mathring{\vect u}^{\eps_h} =\left(\begin{array}{c} -\epsh x_3  \partial_1 \mathring{v}^{\eps_h}  \\[0.25em] 
	-\epsh x_3  \partial_2 \mathring{v}^{\eps_h} \\[0.45em] 
	h^{-1}\epsh \mathring{v}^{\eps_h} \end{array} \right)+\mathring{\mat \psi}^{\eps_h}, \\[0.3em]
& 
\nonumber	\| \mathring{v}^{\eps_h} \|^2_{L^2}+{\epshtwo}\bigl\|\nabla \mathring{v}^{\eps_h}\bigr\|^2_{L^2}+ {{\epshfour}} 
\bigl\| \nabla^2 \mathring{v}^{\eps_h}\bigr\|^2_{L^2}+
\bigl\| \mathring{\vect \psi}^{\eps_h}\bigr\|^2_{L^2}
%%\\ 
%%\nonumber & &+ 
+{\epshtwo}\| \nabla_h \mathring{\mat \psi}^{\eps_h}\|^2_{L^2} \leq   C {\epshtwo} \| \sym \nabla_h \mathring{\vect u}^{\eps_h}\|^2_{L^2} \leq  Ca_{\epsh} (\vect u^{\epsh}, \vect u^{\epsh}),\\[0.4em]
\nonumber 
&  
h^{-1}\epsh\|\mathring{v}^{\epsh}\|^2_{L^2}  \leq  C\left( a_{\epsh} (\vect u^{\epsh}, \vect u^{\epsh})+\|\vect u^{\epsh}\|_{L^2} \right),
\end{align*}
where $\mathring{v}^{\epsh} \in H^2(\omega)$, $\mathring{\vect \psi} \in H^1(\Omega;\R^3)$, $\mathring{v}^{\epsh}=\mathring{\vect \psi}=0$ on $\Omega_1^{\epsh}$. 

Proceeding to part 2, we note that the first convergence in \eqref{nada100000} follows directly from Theorem \ref{aux:thm.griso} and Remark \ref{remivan33}.
To prove the remaining convergence statements, by analogy with the argument of Proposition \ref{propdinsi1} we first assume that $\omega$ has $C^{1,1}$ boundary. Using Lemma \ref{app:lem.limsup} and Lemma \ref{keydecompose},\,3 we have 
$$\tilde{u}_3^{\epsh}=h^{-1}\varphi^{\epsh}+w^{\epsh}+\tilde{\psi}_3^{\epsh},$$
where $(\varphi^{\epsh})_{h>0}$ is bounded in $H^2(\omega)$, $(w^{\epsh})_{h>0}$ is bounded in $H^1(\omega)$ and $\tilde{\psi}_3^{\epsh}\xrightarrow{L^2} 0$. Since $h^{-1}\varphi^{\epsh}$ is bounded in $L^2,$ the first part of the second convergence in \eqref{nada100000} follows from Lemma \ref{lemmadependenceony}\,(2). Furthermore, the first part of the third convergence in \eqref{nada100000} follows from Remark \ref{remivan33}, Remark \ref{remivan1111}, Theorem \ref{neukammresult}\,(2) and Lemma \ref{lemmatwoscalecompact}\,(1) (in addition to the more standard Lemma \ref{app:lem.limsup} and Lemma \ref{keydecompose}\,(3).) The second and third parts of the third convergence statement in \eqref{nada100000} need to be additionally combined with the second convergence statement in  \eqref{nada100000} through Lemma \ref{lemmatwoscalecompact}\,(2). Finally, the fourth and fifth convergence in \eqref{nada100000}   follow from Lemma \ref{lemmatwoscalecompact}\,(3) and Lemma \ref{nada2}\,(1) by noticing that $\mathring{v}^{\epsh}\longrightarrow 0$ in $L^2$ as a consequence of the fact $h \ll \epsh$. This concludes the proof of part 2 for the case when $\omega$ 
has $C^{1,1}$ boundary. For the general case of  $\omega$ with Lipschitz boundary, we now use Lemma \ref{keydecompose}\,(3), Theorem \ref{neukammresult}\,(2) and Lemma \ref{lemmatwoscalecompact}\,(1) in combination with the approach of the proof of Proposition \ref{propdinsi1}. 

The argument for part 3 is analogous to that for  Proposition \ref{propdinsi1}. 
\end{proof} 

\vskip 0.2cm
\noindent{\bf G. Proof of Theorem \ref{rucak35} }
\begin{proof} 
	The proof is carried out by taking appropriate vest functions $\vect v={\vect v}^\epsh$ in \ref{resolventproblemstarting} and then passing to the limit as $h\to0,$ for which we invoke a combination of Proposition \ref{propcompactregime11}, Remark \ref{nada10001}, and a density argument.
	
	Different equations in (\ref{jdbaprva}) are obtained by using different kinds of test functions. For the first equation, we use test functions of the form
%	\begin{eqnarray*}
\begin{align*}
	\vect v^{\epsh}(\hat{x})&=\bigl(\theta_1(\hat{x}), \theta_2(\hat{x}),0\bigr)^\top+\epsh\,\left( \zeta_1\biggl(\hat{x},\frac{\hat{x}}{\epsh}\biggr), \zeta_2\biggl(\hat{x},\frac{\hat{x}}{\epsh}\biggr),0\right)^\top+\epsh\left(-x_3\partial_{y_1}\psi\biggl(\hat{x},\frac{\hat{x}}{\epsh}\biggr),-x_3\partial_{y_2}\psi\biggl(\hat{x},\frac{\hat{x}}{\epsh}\biggr),\frac{1}{h}\psi\biggl(\hat{x},\frac{\hat{x}}{\epsh}\biggr)\right)\\[0.1em] 
	&+h\int_0^{x_3}{\vect r}\biggl(x,\frac{\hat{x}}{\epsh}\biggr)\, dx_3,
	%\end{eqnarray*}  
	\end{align*}
	where $\vect \theta \in C_{\rm c}^1(\omega;\R^2)$, $\vect \zeta \in C_{\rm c}^1(\omega;C^2(\mathcal{Y};\R^2))$, $\psi \in C_{\rm c}^1(\omega;C^1(\mathcal{Y}))$, ${\vect r} \in C_{\rm c}^1(\Omega;C^1(\mathcal{Y};\R^3))$. 
	Next, for the second  equation we use test function of the form $$v^{\epsh}(x)=\left( \mathring{\vect \xi}_1\biggl(\hat{x},\frac{\hat{x}}{\epsh}\biggr),  \mathring{\vect \xi}_2\biggl(\hat{x},\frac{\hat{x}}{\epsh}\biggr),0 \right)^\top+\frac{h}{\epsh}\int_0^{x_3} \mathring{\vect r}\biggl(x,\frac{\hat{x}}{\epsh}\biggr)\,dx_3,$$ where  $\mathring{\vect \xi} \in C_{\rm c}^1(\omega;C_{\rm c}^1(Y_0;\R^2))$, $\mathring{\vect r} \in C_{\rm c}^1(\Omega;C_{\rm c}^1(Y_0;\R^3))$ .
	Further, for the third equation ($\kappa=\infty$) we use test equation of the form 
	$$\vect v^{\epsh}(x)=(-hx_3 \partial_1 \theta(\hat{x}),-hx_3 \partial_2\theta(\hat{x}), \theta(\hat{x}))^\top+  \left(-\frac{h}{\epsh}x_3 \partial_{y_1} \mathring{ \xi}\biggl(\hat{x},\frac{\hat{x}}{\epsh}\biggr) ,-\frac{h}{\epsh}x_3 \partial_{y_1} \mathring{ \xi}\biggl(\hat{x},\frac{\hat{x}}{\epsh}\biggr) ,\mathring{ \xi}\biggl(\hat{x},\frac{\hat{x}}{\epsh}\biggr)\right)^\top,$$ where    $\theta \in C_{\rm c}^2(\omega)$, $\mathring{\xi} \in C_{\rm c}^1(\omega;C_{\rm c}^2(Y_0))$. For the fourth equation ($\kappa\in(0,\infty)$) we use test functions of the form $$\vect v^{\epsh}(x)=\left(-\frac{h}{\epsh}x_3\partial_{y_1}v\biggl(\hat{x},\frac{\hat{x}}{\epsh}\biggr), -\frac{h}{\epsh}x_3\partial_{y_2}v\biggl(\hat{x},\frac{\hat{x}}{\epsh}\biggr),v\biggl(\hat{x},\frac{\hat{x}}{\epsh}\biggr) \right)^\top,$$ where  $v \in C_{\rm c}^1(\omega,C^2(\mathcal{Y}))$. For the fifth equation ($\kappa\in(0,\infty)$) we use test functions of the form 
	$$\vect v^{\epsh}(x)= \left(-\frac{h}{\epsh}x_3 \partial_{y_1} \mathring{\xi}\biggl(\hat{x},\frac{\hat{x}}{\epsh}\biggr) ,-\frac{h}{\epsh}x_3 \partial_{y_1} \mathring{ \xi}\biggl(\hat{x},\frac{\hat{x}}{\epsh}\biggr), \mathring{ \xi}\biggl(\hat{x},\frac{\hat{x}}{\epsh}\biggr)\right)^\top,$$ $\mathring{\xi} \in C_{\rm c}^1(\omega,C_{\rm c}^2(Y_0))$. Finally, for the sixth equation ($\kappa=0$) we use test functions of the form
	 \begin{align*} 
	 \vect v^{\epsh}(x)=\Biggl(-\frac{h}{\epsh}x_3\partial_{y_1}v\biggl(\hat{x},\frac{\hat{x}}{\epsh}\biggr),&-\frac{h}{\epsh}x_3 \partial_{y_1}  v\biggl(\hat{x},\frac{\hat{x}}{\epsh}\biggr), v\biggl(\hat{x},\frac{\hat{x}}{\epsh}\biggr) \Biggr)^\top
	 \\[0.4em] 
	 & +\left(-\frac{h}{\epsh}x_3 \partial_{y_1} \mathring{\xi}\biggl(\hat{x},\frac{\hat{x}}{\epsh}\biggr) ,-\frac{h}{\epsh}x_3 \partial_{y_1} \mathring{ \xi}\biggl(\hat{x},\frac{\hat{x}}{\epsh}\biggr), \mathring{ \xi}\biggl(\hat{x},\frac{\hat{x}}{\epsh}\biggr)\right)^\top,
	 \end{align*} 
	  where  $v \in C_{\rm c}^1(\omega;C^2(\mathcal{Y}))$, $\mathring{\xi}\in C_{\rm c}^1(\omega;C^2_{\rm c}(Y_0))$. 
	 The proof of the remaining claims follow an analogous part of the proof of Theorem \ref{thmivan111}. 
\end{proof}	 

\vskip 0.2cm	
\noindent{\bf H. Proof of Corollary \ref{rucak45}}
\begin{proof} 
The proof follows easily from Remark \ref{remivan}, Remark \ref{remivan10}, and Remark \ref{remivan11}. 		
\end{proof}	

\vskip 0.2cm
\noindent{\bf I. Proof of Proposition \ref{propcompactregime}}
\begin{proof} 
	Part 1 follows easily from Theorem \ref{thmextreg1} (in particular, \eqref{rucak36}--\eqref{est1000}) and Corollary \ref{kornthincor}, after plugging ${ \vect v}^{\epsh}={\vect u}^{\eps_h}$ into \eqref{resolventproblemstarting}. 
To justify the scaling, notice that as a consequence of the above mentioned statements we have  
$$ 
\|\pi_{\epsh/h}\vect u^{\epsh}\|_{L^2} \leq Ch^{-2}a_{\epsh} (\vect u^{\epsh},\vect u^{\epsh}), 
$$
see also the last expression in \eqref{nada1} below.  
To prove part 2 we use Theorem \ref{thmextreg1} and Corollary \ref{kornthincor} again and obtain
\begin{equation}
	\label{nada1}
	\begin{aligned}
&\vect u^{\eps_h}= \tilde{\vect u}^{\eps_h}+\mathring{\vect u}^{\eps_h}, \quad \tilde{\vect u}^{\epsh}=E^{\epsh} \vect u^{\epsh}, \\[0.3em]
&\left\| \pi_{1/h} \tilde{\vect u}^{\eps_h} \right\|^2_{H^1(\omega;\R^3)}\leq {C}h^{-2}\bigl\|\sym \tilde{\vect u}^{\eps_h}\bigr\|^2_{L^2(\Omega;\R^3)} \leq {C}h^{-2} a_{\epsh} (\vect u^{\epsh},\vect u^{\epsh}), \\[0.35em]
& \mathring{\vect u}^{\eps_h} =\left(\begin{array}{c} -\epsh x_3  \partial_1 \mathring{v}^{\eps_h}  \\[0.25em] -\epsh x_3  \partial_2 \mathring{v}^{\eps_h} \\[0.25em] h^{-1}\epsh \mathring{v}^{\eps_h} \end{array} \right)+\mathring{\mat \psi}^{\eps_h}, \\[0.35em]
& 
\bigl\|h^{-1}\epsh\mathring{v}^{\eps_h} \bigr\|^2_{L^2(\omega)}+{\epshtwo}\bigl\|h^{-1}\epsh \nabla \mathring{v}^{\eps_h}\bigr\|^2_{L^2(\omega;\R^2)}+ \eps_h^4 \bigl\|h^{-1}\epsh\nabla^2 \mathring{v}^{\eps_h}\bigr\|^2_{L^2(\omega;\R^2)}+\bigl\|h^{-1}\epsh\mathring{\vect \psi}^{\eps_h}\bigr\|^2_{L^2(\Omega;\R^3)}\\[0.35em] 
& \hspace{+60pt}+ {\epshtwo}\bigl\|h^{-1}\epsh\nabla_h \mathring{\mat \psi}^{\eps_h}\bigr\|^2_{L^2(\Omega;\R^{3\times 3})} \leq   Ch^{-2}{{\epshfour}}\bigl\| \sym \nabla_h \mathring{\vect u}^{\eps_h}\bigr\|^2_{L^2(\Omega;\R^{3\times 3})} \leq Ch^{-2}a_{\epsh} (\vect u^{\epsh},\vect u^{\epsh}),  
\end{aligned}
\end{equation}
where $\mathring{v}^{\epsh} \in H^2(\omega)$, $\mathring{\vect \psi} \in H^1(\Omega;\R^3)$, $\mathring{v}^{\epsh}=\mathring{\vect \psi}=0$ on $\Omega_1^{\epsh}$. 
Assuming first that $\omega$ has $C^{1,1}$ boundary, part 2 follows by using Lemma \ref{app:lem.limsup}, Lemma \ref{keydecompose}\,(3),  Theorem \ref{neukammresult}\,(2), Lemma \ref{lemmatwoscalecompact}\,(1,3), Lemma \ref{nada2}\,(1), and Theorem \ref{thmextreg1}. 
For general Lipschitz domains we follow the approach of Proposition \ref{propdinsi1} and Proposition \ref{propcompactregime11}. Finally, part 3 is obtained  the same way as part 3 of Proposition \ref{propdinsi1}. 	
\end{proof} 

\vskip 0.2cm	
\noindent{\bf J. Proof of Theorem \ref{thmivan113}}
\begin{proof} 
	Proof follows the approach of the proof of Theorem \ref{thmivan111}, by using Proposition \ref{propcompactregime}, Remark \ref{nada10001}, and test functions of the form 
	\begin{align*} 
	\vect v^{\epsh}(x)&=\bigl(h \theta_1(\hat{x})-hx_3 \partial_1 \theta_3 (\hat{x}),h \theta_2(\hat{x})-hx_3\partial_2  \theta_3(\hat{x}),\vect \theta_3(\hat{x})\bigr)^\top+ h\epsh\, \left( \zeta_1\biggl(\hat{x},\frac{\hat{x}}{\epsh}\biggr),\zeta_2\biggl(\hat{x},\frac{\hat{x}}{\epsh}\biggr),0 \right)^\top\\[0.4em] 
	&+\epsh\left(-hx_3\partial_{y_1}\psi\biggl(\hat{x},\frac{\hat{x}}{\epsh}\biggr),-hx_3\partial_{y_2}\psi\biggl(\hat{x},\frac{\hat{x}}{\epsh}\biggr),\psi\biggl(\hat{x},\frac{\hat{x}}{\epsh}\biggr)\right)^\top+h^2\int_0^{x_3}{\vect r}\biggl(x,\frac{\hat{x}}{\epsh}\biggr)\, dx_3\\[0.3em]
	&+ \left(\frac{h}{\epsh}\left(\mathring{ \xi}_1 -x_3 \partial_{y_1} \mathring{ \xi}_3\biggl(\hat{x},\frac{\hat{x}}{\epsh}\biggr)\right) , \frac{h}{\epsh}\left(\mathring{ \xi}_2-x_3 \partial_{y_2} \mathring{ \xi}_3\biggl(\hat{x},\frac{\hat{x}}{\epsh}\biggr)\right), \mathring{\vect \xi}_3\biggl(\hat{x},\frac{\hat{x}}{\epsh}\biggr)\right)^\top\\[0.3em] 
	&+\frac{h}{\epsh}\int_0^{x_3} \mathring{\vect r}\biggl(x,\frac{\hat{x}}{\epsh}\biggr)\,dx_3,
	\end{align*} 
where $\vect \theta_*  \in C_{\rm c}^1(\omega;\R^2)$, $\vect \theta_3 \in C_{\rm c}^2(\omega)$, $\vect \zeta \in C_{\rm c}^1(\omega;C^1(\mathcal{Y};\R^2))$, $\psi \in C_{\rm c}^1(\omega;C^1(\mathcal{Y}))$, $\vect r \in C_{\rm c}^1(\omega;C^1(\mathcal{Y};\R^3))$, $\mathring{\vect \xi}_*  \in C_{\rm c}^1(\omega;C_{\rm c}^1(Y_0;\R^2))$, $\mathring{\vect \xi}_{3}\in C_{\rm c}^1(\omega;C_{\rm c}^2(Y_0))$, $\mathring{\vect r} \in C_{\rm c}^1(\Omega;C_{\rm c}^1(Y_0;\R^3))$. 
\end{proof} 

\vskip 0.2cm	
\noindent{\bf K. Proof of Proposition \ref{propcompactregime11111} and Corollary \ref{corivan11}}
\begin{proof} 
	The proof proceeds in the same way as the proofs of Proposition \ref{lemmaconvergencerealtime} or Proposition \ref{propdinsi1} by invoking additionally Theorem \ref{neukammresult}\,(3), Theorem \ref{thmivann111}, and Lemma \ref{nada2}\,(2). In order to conclude the form of $\mathcal{C}_{\infty}$ from Lemma \ref{app:lem.limsup}, Lemma \ref{keydecompose}\,(3), and Theorem \ref{neukammresult}\,(3),
	it is also important to see that the following simple identity holds:	
	$$ 
	x_3 \nabla^2_y \varphi(\hat{x},y)=\nabla_y (x_3 \partial_{y_1}\varphi,x_3 \partial_{y_2}\varphi,0)^\top, \quad \forall \varphi \in L^2(\omega; H^2(\mathcal{Y})).
	$$
	The proof of Corollary \ref{corivan11}  uses Remark \ref{remivan}, Remark \ref{rucak31}, and symmetries of the solution, as a consequence of the assumption on symmetries of the elasticity tensor. 
\end{proof} 

\vskip 0.2cm
\noindent{\bf L. Proof of Theorem \ref{rucak50}}
\begin{proof} 
The proof is similar to the proof of Theorem \ref{thmivan11111}, by invoking Remark \ref{nada10001h} and plugging in \eqref{resolventproblemstarting} test functions of the form 
\begin{eqnarray*}
\vect v^\epsh(x) &=& 
\left(\begin{array}{c}
 \theta_1(\hat{x})-hx_3 \partial_1  \theta_3 (\hat{x}) \\[0.3em]
 \theta_2(\hat{x})-hx_3 \partial_1  \theta_3 (\hat{x}) \\[0.3em]
 \theta_3(\hat{x})
\end{array}\right)
+
\epsh\, \vect\zeta\biggl(x,\frac{\hat{x}}{\epsh}\biggr)
+ h\int_0^{x_3} \vect r(x) \,dx_3 
+\mathring{\vect\xi}\biggl(x,\frac{\hat{x}}{\epsh}\biggr),
\end{eqnarray*}
where $\vect \theta \in C_{\rm c}^{1}(\omega;\R^2)$, $ \theta_3 \in C_{\rm c}^{2}(\omega)$,  $\vect\zeta \in C_{\rm c}^1(\Omega; C^{1} (I \times \mathcal{Y};\R^3))$, $ \vect r \in C_{\rm c}^1(\Omega)$, $\mathring{\vect\xi} \in C_{\rm c}^1(\omega;C^1_{00}(I \times Y_0;\R^3)).$
\end{proof} 

\vskip 0.2cm	
\noindent{\bf M. Proof of Proposition \ref{ds102}}
\begin{proof} 
The proof is carried out similarly to the proof of  Proposition \ref{lemmaconvergencerealtime} and Proposition \ref{propdinsi1}, where we additionally use Theorem \ref{neukammresult}\,(3), Theorem \ref{thmivann111}, and Lemma \ref{nada2}\,(2).  	
\end{proof} 

\vskip 0.2cm
\noindent{\bf N. Proof of Theorem \ref{rucak60}}
\begin{proof} 
The proof follows the proof of Theorem \ref{thmivan11111}, by using Remark \ref{nada10001h} and by plugging in \eqref{resolventproblemstarting} test functions of the form 
\begin{eqnarray*}
	\vect v^\epsh(x) &=& 
	\left(\begin{array}{c}
		h \theta_1(\hat{x})-hx_3 \partial_1  \theta_3 (\hat{x}) \\[0.3em]
		h \theta_2(\hat{x})-hx_3 \partial_1  \theta_3 (\hat{x}) \\[0.3em]
		\theta_3(\hat{x})
	\end{array}\right)
	+
	h\epsh\, \vect\zeta\biggl(x,\frac{\hat{x}}{\epsh}\biggr)
	+ h^2\int_0^{x_3} \vect r(x) \,dx_3 
	+\mathring{\vect\xi}\biggl(x,\frac{\hat{x}}{\epsh}\biggr),
\end{eqnarray*}
where $\vect \theta \in C_{\rm c}^{1}(\omega;\R^2)$, $\vect \theta_3 \in C_{\rm c}^{2}(\omega)$,  $\vect\zeta \in C_{\rm c}^1(\Omega; C^{1} (I \times \mathcal{Y};\R^3))$, $ \vect r \in C_{\rm c}^1(\Omega)$, $\mathring{\vect\xi} \in C_{\rm c}^1(\omega;C_{00}^1(I \times Y_0;\R^3)).$
\end{proof} 		
 \subsection{Proofs for Section \ref{limitspectrum}} 
 \label{proofs_sec}
 
 \vskip 0.1cm
 
\noindent{\bf A. Proof of Theorem \ref{rucak80}}
 \begin{proof} 
It is easy to see from Proposition \ref{propvecer} that the operator $\mathcal{A}_{\delta}^{\funcB, {\rm hom}}$ is positive definite, coercive, and has compact inverse. This, in particular, allows one to obtain immediately a characterization of its spectrum, which we omit.  

From Proposition \ref{propdinsi1} and Theorem \ref{thmivan111} we infer that if $\vect f^{\epsh}\longrightarrow \RRR \vect f \BBB$ in $L^2,$ then the sequence of solutions $(\vect u^{\epsh})_{h>0}$ of \eqref{resolventproblemstarting} for $\lambda=1$ satisfies $\vect u^{\epsh}\to (0,0,\funcB)^\top$, where $\funcB \in H_{\gamma_{\rm D}}^2(\omega)$ solves 
$$
\bigl(\mathcal{A}_{\delta}^{\funcB, {\rm hom}}+\mathcal{I}\bigr)\funcB=\langle \rho \rangle^{-1}\vect f_3.
$$ 
Using the proof of \cite[Proposition 2.2]{zhikov2005}, we show that the property $(H_1)$ in Definition \ref{rucak85} holds. To prove the property $(H_2),$ we take a sequence $\lambda^{\epsh}$ of eigenvalues of the operator $h^{-2}\mathcal{A}_{\epsh}$ converging to $\lambda>0$. Next, consider the sequence  $(\vect u^{\epsh})_{h>0}$ of the corresponding eigenfunctions
$$
h^{-2}\mathcal{A}_{\epsh} \vect u^{\epsh}=\lambda^{\epsh} \vect u^{\epsh}, \quad \|\vect u^{\epsh}\|_{L^2}=1.  
$$
Multiplying the above equation by $\vect u^{\epsh},$  using the compactness result from Proposition \ref{propdinsi1}, and invoking an argument similar to that of Theorem \ref{thmivan111},  we conclude that $\vect u^{\epsh}\longrightarrow(0,0,\funcB)^\top$ in $L^2$, where $\funcB \in H^2_{\gamma_{\rm D}}(\omega)$ solves 
$$ \mathcal{A}_{\delta}^{\funcB, {\rm hom}}\funcB=\lambda \funcB, \quad \| \funcB\|_{L^2}=1,$$
which completes the proof of $(H_2)$. This also proves the convergence of eigenfunctions.  

To prove a refined version of the Hausdorff convergence concerning the convergence of eigenvalues ordered in the increasing order, we take an arbitrary closed curve $\Gamma \subset \C$, intersecting an interval in $( 0, \infty )$ and not passing through any of the eigenvalues $\lambda_{\delta,n}$ and define the following projection operators: 
\begin{equation*}
P_\Gamma^\epsh =-\frac{1}{2\pi{\rm i}}\oint_{\Gamma }\left(\frac{1}{h^2}\mathcal{A}_\epsh - z \RRR \mathcal{I} \BBB \right)^{-1} dz, \quad P_\Gamma = -\frac{1}{2\pi{\rm i}}\oint_{\Gamma }\left(\mathcal{A}_{\delta}^{\funcB, {\rm hom}} - z \RRR \mathcal{I} \BBB \right)^{-1} dz.
\end{equation*}
We claim that for small enough $\epsh> 0$ the dimensions of the ranges $\mathcal{R}(P_\Gamma^\epsh)$ and $\mathcal{R}(P_\Gamma)$  coincide. (Note that they are finite by the compactness of the resolvent.) Indeed, from the compactness result in Proposition \ref{propdinsi1} and Lebesgue theorem on dominated convergence it follows that if $(\vect f^{\epsh})_{h>0} \subset L^2(\Omega;\R^3)$, $(\vect v^{\epsh})_{h>0} \subset L^2(\Omega;\R^3)$ are such that $\vect f^{\epsh}\rightharpoonup \vect f$, $\vect v^{\epsh}\rightharpoonup \vect v$ weakly in $L^2,$ then one has
$$
\bigl(P_{\Gamma}^{\epsh} \vect f^{\epsh}, \vect v^{\epsh}\bigr) \to (P_{\Gamma} \vect f, \vect v). 
$$
It follows that $P_{\Gamma}^{\epsh} \vect f^{\epsh}\longrightarrow P_{\Gamma} \vect f$ in $L^2.$ This immediately implies that the dimensions of $\mathcal{R}(P_\Gamma^\epsh)$ and $\mathcal{R}(P_\Gamma)$  coincide for sufficiently small $\epsh.$

Next,  fix a closed curve $\Gamma_{\delta,n} \subset \C$ containing in its interior the eigenvalue $\lambda_{\delta,n}$ and no other eigenvalues, intersecting the real line at $w_1$ and $w_2$ such that $\lambda_{\delta,n-1} < w_1 < \lambda_{\delta,n}< w_2 < \lambda_{\delta,n+1}$, where we set $\lambda_{\delta,0}=0$. The multiplicity $k_{\delta,n}$ of this eigenvalue equals $\dim \mathcal{R}(P_{\Gamma_{\delta,n}})$. By using the above claim, we know that for small enough $\epsh$ exactly $k_{\delta,n}$ eigenvalues  of $h^{-2}\mathcal{A}_\epsh$ (including their multiplicities) are contained in the interval $( w_1, w_2).$ \end{proof} 	
Before giving the rest of the proofs we will state and prove one helpful lemma: 
\begin{lemma} \label{lemmaivan100}
	\begin{enumerate}
	\item If $\mu_h=\epsh,$ one has  
	$$ \lim_{h \to 0} \sigma(\mathring{\tilde{\mathcal{A}}}_{\epsh}) \subset \lim_{h \to 0} \sigma(\tilde{\mathcal{A}}_{\epsh}).$$
	\item If $\delta \in (0,\infty]$, $\mu_h=\epsh h,$ one has
	$$\lim_{h \to 0} \sigma(\mathring{\mathcal{A}}_{\epsh}) \subset \lim_{h \to 0}h^{-2}\sigma({\mathcal{A}}_{\epsh}).$$ 
	\item If $\delta=0$, $\mu_h={\epshtwo},$ one has 
	 	$$\lim_{h \to 0} h^{-2}{\epshtwo}\sigma(\mathring{\mathcal{A}}_{\epsh}) \subset \lim_{h \to 0} h^{-2}\sigma({\mathcal{A}}_{\epsh}).$$ 
\end{enumerate} 
\end{lemma} 	
\begin{proof} 
We prove part 1 for the case $\delta \in (0,\infty)$ only, as the cases $\delta=0$ and $\delta=\infty$ are dealt with by similar arguments. We take $\lambda^{\epsh} \in \sigma(\mathring{\tilde{\mathcal{A}}}_{\epsh})$ such that $\lambda^{\epsh} \to \lambda$  and $\mathring{\vect u}^{\epsh}_r \in H^1_{00}(I \times Y_0;\R^3)$ such that $\|\mathring{\vect u}^{\epsh}_r\|_{L^2}=1$ and  $\mathring{\tilde{\mathcal{A}}}_{\epsh} \mathring{\vect u}^{\epsh}_r=\lambda^{\epsh} \mathring{\vect u}^{\epsh}_r$. 
The convergence properties of $\lambda^{\epsh}$ and $\mathring{\vect u}^{\epsh}_r$ immediately imply that the sequence 
\[
\Bigl(\|\sym \nabla_{\frac{h}{\epsh}} \mathring{\vect u}^{\epsh}_r\|_{L^2}\Bigr)_{h>0}
\] 
is bounded.\footnote{Actually it can be concluded that the sequence 
$
(\| \mathring{\vect u}^{\epsh}_r\|_{H^1})_{h>0}
$
is bounded, see \cite[Section 7]{chered5}.}  
For each $h$ we take a cube $Q^h=q^h \times I$ such that $q^h \subset \omega$ has vertices in $\epsh \Z^2$ and side length $2n^h\epsh$, where $n^h$ is an integer. Furthermore, we assume that $n^h\epsh$ converge to some positive number as $h\to0$. We define $\vect u^{\epsh}$ as follows. Consider $z \in \Z^2$ such that the cube of size $\epsh$ whose left corner is at $\epsh z$ is contained in $q^h.$ On the inclusion $\epsh (Y_0+z) \times I,$ we set $\vect u^{\epsh}$ to be equal to $\mathring{\vect u}^{\epsh}_r(\hat{x}/\epsh-z,x_3)$ if $z_1$ (the first coordinate of $z$) is even and to $- \mathring{\vect u}^{\epsh}_r(\hat{x}/\epsh-z,x_3)$ if $z_1$ is odd. We then extend  $\vect u^{\epsh}$ by zero outside $\epsh (Y_0+z) \times I.$ This procedure is repeated for all $z\in\Z^2$ with the above property, and finally $\vect u^{\epsh}$ is set to zero on $\Omega\setminus Q^h.$ 

 It can be easily checked that for $\vect \xi \in H_{\gamma_{\rm D}}^1(\Omega;\R^3)\cap L^{2, \rm memb}(\Omega;\R^3)$  one has
%\begin{eqnarray} 
\begin{align}
	\label{ivan401} 
& \int\limits_{\Omega} \C^{\mu_h}\biggl(\frac{\hat{x}}{\epsh}\biggr) \simgrad_{h} \vect u^{\epsh}(x) : \simgrad_{h} \vect \xi(x) \,dx-\lambda^{\epsh}\int_{\Omega} \rho \vect u^{\epsh}\cdot\vect \xi\\ &=\nonumber \int_{Q^h\cap\Omega_0^{\epsh}} \C^{\mu_h}\biggl(\frac{\hat{x}}{\epsh}\biggr)   \simgrad_{h} \vect u^{\epsh}(x): \simgrad_{h} \tilde{\vect \xi}(x)\, dx-\lambda^{\epsh}
\int_{Q^h\cap\Omega_0^{\epsh}} \rho_0 \vect u^{\epsh}\cdot\tilde{ \vect \xi}\,dx, 
%\end{eqnarray}
\end{align}
where $\vect \xi=\tilde{\vect \xi}+\mathring{\vect \xi}$, with $\tilde{\vect \xi}$ being the extension provided by Theorem \ref{thmextension}. Recall that, as a consequence of Corollary \ref{kornthincor}, 
\[
\bigl\|\tilde{\vect \xi}\bigr\|_{H^1} \leq C\bigl\|\sym \nabla_h \tilde{\vect \xi}\bigr\|_{L^2},
\] 
where $C>0$ does not depend on $h.$ Using this fact and the definition of $\vect u^{\epsh}$ (noting that the mean value of $\vect u^{\epsh}$ is zero on each two neighbouring small cubes of size $\epsh$ in the $x_1$ direction) it can be easily seen that the right hand side of \eqref{ivan401} can be written in the form  
\begin{equation} 
\int_{\Omega} \vect f_1^{\epsh}: \sym \nabla_h \tilde{\vect \xi}\, dx+ \int_{\Omega} \vect f_2^{\epsh}\cdot\tilde{\vect \xi} \, dx,
\label{fexpr}
\end{equation}	 
where $\vect f_1^{\epsh} \in L^2(Q^h;\R^{3 \times 3})$, $\vect f_2^{\epsh} \in L^2(Q^h;\R^3),$ and $\|\vect f_1^{\epsh}\|_{L^2} \to 0$, 	$\|\vect f_2^{\epsh}\|_{L^2} \to 0$ as $h \to 0$. 
To see this, we divide the domain into small rectangles  containing two neighbouring cubes, where the first coordinate of the left corner is even and odd respectively, and apply the Poincar\'{e} inequality. This yields an estimate for the right-hand side of \eqref{ivan401} by the expression $C \epsh (\|\nabla_h\tilde {\vect \xi}\|_{L^2}+ \|\tilde {\vect \xi}\|_{L^2} ),$ where $C>0$ is $h$-independent. 
By using the Riesz representation theorem (applied first on the physical domain and then moved on the canonical domain) and the fact that on $Q^h$ the norm 
$\|\cdot\|_{L^2}+\|\nabla_h(\cdot)\|_{L^2}$ is equivalent to the norm  
$
\| \cdot\|_{L^2}+\|\sym\nabla_h(\cdot)\|_{L^2},   
$
we conclude that the right-hand side of \eqref{ivan401} can be written in the form
$$\epsh \left(\int_{Q^h} \simgrad_h \vect r^{\epsh}(x) :\simgrad_{h} \tilde{\vect \xi}(x)\, dx+\int_{Q^h}  \vect r^{\epsh} \cdot \tilde{\vect \xi}\, dx \right),   $$
where  $ \|\simgrad_h \vect r^{\epsh}\|_{L^2}+\|\vect r^{\epsh}\|_{L^2}$ is bounded independently of $h$. The claim follows by taking $\vect f_1^{\epsh}= \epsh \simgrad_h \vect r^{\epsh}$ and $\vect f_2^{\epsh}= \epsh \vect r^{\epsh}$ in (\ref{fexpr}).

To conclude the proof of part 1, we note that there exists $C>0$ such that $\|\vect u^{\epsh}\|_{L^2} \geq C$ and hence, by applying a suitable version of Lemma  \ref{nenad100}  (see also Remark \ref{ivan501}), one has 
\begin{equation}
\dist\bigl(\lambda^{\epsh}, \sigma(\tilde{\mathcal{A}}_{\epsh})\bigr)\to 0\ \ {\rm as}\ \ h \to 0.
\label{ref1}
\end{equation}

The proof of part 2 proceeds in a similar way. Part 3 requires an additional explanation while following the same kind of argument. We again take 
  $\lambda^{\epsh} \in \sigma({\epshtwo}h^{-2}\mathring{\tilde{\mathcal{A}}}_{\epsh})$ such that $\lambda^{\epsh} \to \lambda$  and $\mathring{\vect u}^{\epsh}_r \in H^1_{00}(I \times Y_0;\R^3)$ such that $\|\mathring{\vect u}^{\epsh}_r\|_{L^2}=1$ and  ${\epshtwo}h^{-2}\mathring{\tilde{\mathcal{A}}}_{\epsh} \mathring{\vect u}^{\epsh}_r=\lambda^{\epsh} \mathring{\vect u}^{\epsh}_r$. 
Using the same argument as in the proof of Theorem \ref{rucak80} (notice that here the $n$-th eigenvalue is of order $h^{-2}\epshtwo$),
 we infer immediately that $\bigl(\epsh h^{-1}\|\sym \nabla_{\frac{h}{\epsh}} \mathring{\vect u}^{\epsh}_r\|_{L^2}\bigr)_{h>0}$ is bounded. Furthermore, invoking Corollary 
\ref{kornthincor}, we obtain 
\[
\biggl\|\biggl(\frac{\epsh}{h}\mathring{ u}^{\epsh}_{r,1}, \frac{\epsh}{h}\mathring{ u}^{\epsh}_{r,2},\mathring{ u}^{\epsh}_{r,3}\biggr)\biggr\|_{H^1}\leq C\frac{\epsh}{h}\Bigl\|\sym \nabla_{\frac{h}{\epsh}} \mathring{\vect u}^{\epsh}_r\Bigr\|_{L^2}
\] 
for some $C>0$ independent of $h$. The rest of the proof follows the proof of part 1. \end{proof} 
\begin{remark} \label{remivan101} 
Using a standard approach (resolvent convergence and compactness of eigenfunctions), it can be easily shown that when $\delta \in (0,\infty)$ one  has $\lim_{h \to 0}\sigma (\mathring{\tilde{\mathcal{A}}}_{\epsh}) =\sigma(\tilde {\mathcal{A}}_{00,\delta})$ and   $\lim_{h \to 0}\sigma (\mathring{{\mathcal{A}}}_{\epsh}) =\sigma({\mathcal{A}}_{00,\delta})$. To obtain this result one needs to use uniform (in $\delta$) Korn inequality (see, e.g., \cite[Section 7]{chered5}). 

In the case $\delta=0$ one can prove (similarly to the proof of Theorem \ref{rucak80}) that $\lim_{h \to 0}\sigma (\mathring{\tilde{\mathcal{A}}}_{\epsh}) =\sigma(\tilde {\mathcal{A}}_{00,0})$ and 	$\lim_{h \to 0}{\epshtwo}h^{-2}\sigma (\mathring{{\mathcal{A}}}_{\epsh}) =\sigma( {\mathcal{A}}_{00,0})$. 

The analogous claim is not valid for $\delta=\infty.$ This is the main reason why in this regime the limiting spectrum is different than the spectrum of the limit operator.  Here, due to the fact that only resolvent convergence (as in Theorem \ref{rucak50}) holds and no compactness of eigenfunctions is available, one only has $\sigma(\tilde{\mathcal{A}}_{00,\infty}) \subset \lim_{h \to 0}\sigma (\mathring{\tilde{\mathcal{A}}}_{\epsh})$, $\sigma({\mathcal{A}}_{00,\infty}) \subset \lim_{h \to 0}\sigma (\mathring{{\mathcal{A}}}_{\epsh})$.
\end{remark}

\noindent{\bf B. Proof of Theorem \ref{rucak71}}
\begin{proof} 
The countability of the solutions of \eqref{generalizedeigenvaluemembrane} is proved in Proposition \ref{generalizedbetatheorem}. 	

The equality \eqref{rucak96} is  proved in the same way as in \cite[Section 8]{Zhikov2000}, by analysing the resolvent equation for the limit operator. 

The proof of the Hausdorff convergence consists of two parts: the statement $(H_1)$ is the direct consequence of the strong resolvent convergence established in Theorem \ref{thmivan11111} and Theorem \ref{rucak35}.  The statement $(H_2)$ is proved by following the strategy of Theorem \ref{rucak80}: taking the sequence of the solutions to
\begin{equation} 
\tilde{\mathcal{A}}_{\epsh} \vect u^{\epsh}=\lambda^{\epsh} \vect u^{\epsh},\quad \| \vect u^{\epsh}\|_{L^2}=1, 
\label{apeigeq}
\end{equation}
where $\lambda^{\epsh} \to \lambda$. One only needs to establish that the $L^2$ weak limit of $\vect u^{\epsh}$ is not zero, so $(H_2)$ then follows by letting $\epsh \to 0$ in (\ref{apeigeq}). This claim is verified by proving that the sequence $(\vect u^{\epsh})_{h>0}$ converges strongly two-scale to the limit $\vect u$, i.e., 
\begin{equation}
\vect u^{\epsh} \strongdrtwoscale \vect u.
\label{uuconv}
\end{equation}
 Note that, due to Lemma \ref{lemmaivan100}, one can assume without loss of generality that $\lambda \notin \lim_{h \to 0} \sigma (\mathring{\tilde{\mathcal{A}}}_{\epsh})= \sigma(\tilde{\mathcal{A}}_{00,\delta})$.  
%%(under the assumption that $\lambda \notin \lim_{h \to 0} \sigma(\mathring{\tilde{\mathcal{A}}}_{\epsh})$). 
One can then prove (\ref{uuconv}) in the same way as in \cite[Lemma 8.2]{Zhikov2000}, see also \cite[Theorem 6.2]{cherd1} for an analogous proof in the stochastic setting as well as the proof of Theorem \ref{rucak72} below. It is important to emphasize that the proof requires strong convergence in $L^2$ of the sequence of extensions $(\tilde{\vect u}^{\epsh})_{h>0}$, which can be ensured by imposing Assumption \ref{assumivan1}\,(1) and using Corollary \ref{remsim1} and Corollary \ref{rucak45}. 

The claim about the symmetry of $\tilde \beta^{\rm memb}_{\delta}$ is a direct consequence of Assumption \ref{assumivan1}.
  %% about the symmetry  which transfers to the symmetry properties of $\tilde \beta^{\rm memb}_{\delta}$. 
\end{proof} 
\noindent{\bf C. Proof of Theorem \ref{rucak72}}
\begin{proof} 
The proof follows the lines of the proof of Theorem \ref{rucak71}. The analysis of the spectrum of limit operator is carried out as in \cite[Section 8]{Zhikov2000}, by studying the limit resolvent equations in Theorem \ref{thmresolventhighercontrast} and Theorem \ref{thmivan113}. Furthermore, in Theorem \ref{thmivan113} we take $\vect f_{\!*} =0$, which implies $\mathring{\vect u}_* =0$.  Strong resolvent convergence is then obtained as the last statement in the mentioned theorem, and compactness of an appropriate sequence of eigenfunctions can be proved by invoking \cite[Lemma 8.2]{Zhikov2000}. The only fact we will additionally comment on is the strong two-scale convergence of the eigenfunctions in the regime $\delta=0.$ We take $\lambda^{\epsh} \in \sigma(h^{-2}\mathcal{A}_{\epsh})$ such that $\liminf_{h \to 0}\dist(\lambda^{\epsh}, h^{-2}{\epshtwo}\sigma(\mathring{\mathcal{A}}_{\epsh}))>0$ (this is again the only situation that requires special analysis, due to Lemma \ref{lemmaivan100}) and $\lambda^{\epsh} \to \lambda$. Next, we take $\vect u^{\epsh}\in \mathcal{D}(\mathcal{A}_{\epsh})$, such that $\|\vect u^{\epsh}\|_{L^2}=1$ and $h^{-2}\mathcal{A}_{\epsh} \vect u^{\epsh}=\lambda^{\epsh} \vect u^{\epsh}$. 
In order to prove that $\lambda$ is in the spectrum of the limit operator, we show that the sequence $\vect u^{\epsh}$ is compact in the sense of strong two-scale convergence. 
We decompose $\vect u^{\epsh}=\tilde{\vect u}^{\epsh}+\mathring{\vect u}^{\epsh}$, where $\tilde{\vect u}^{\epsh}=E^{\epsh} \vect u^{\epsh} $, where $E^{\epsh}$ is an extension given in Theorem \ref{thmextreg1}. In the same way as in Proposition \ref{propcompactregime}, we infer that \eqref{nada1} holds. Taking test function $\mathring{\vect \xi} \in H^1_{\Gamma_{\rm D}}(\Omega;\R^3)$ that vanish on $\Omega_1^{\epsh},$ we conclude that 
\begin{equation} 
	\label{ivan402} 
	\begin{aligned}
\frac{1}{h^2}\int\limits_{\Omega_0^{\epsh}} \C^{\mu_h}&\biggl(\dfrac{\hat{x}}{\epsh}\biggr) \simgrad_h \mathring{\vect u}^{\epsh}(x) : \simgrad_h \mathring{\vect \xi}^{\epsh}(x) \,dx-\lambda^{\epsh}\int_{\Omega_0^{\epsh}} \rho \mathring{\vect u}^{\epsh} \cdot\mathring{\vect \xi}\,dx =\\[0.2em]
&\frac{1}{h^2}\int_{\Omega_0^{\epsh}} \C^{\mu_h}\biggl(\dfrac{\hat{x}}{\epsh}\biggr)   \simgrad_h \tilde{\vect u}^{\epsh}(x): \simgrad_h \mathring{\vect \xi}^{\epsh}(x)\, dx-\lambda^{\epsh}
\int_{\Omega_0^{\epsh}} \rho \tilde{\vect u}^{\epsh}\cdot\mathring{ \vect \xi}\,dx.
\end{aligned}
\end{equation}
To prove the strong two-scale convergence, we shall use a duality argument. To this end, consider the identity 
\begin{equation} 
	\label{ivan601}
	\begin{aligned} 
\frac{1}{h^2}\int\limits_{\Omega_0^{\epsh}} \C^{\mu_h}\biggl(\dfrac{\hat{x}}{\epsh}\biggr) \simgrad_h \mathring{\vect z}^{\epsh}(x) : \simgrad_h \mathring{\vect \xi}(x) \,dx&-\lambda^{\epsh}\int_{\Omega_0^{\epsh}} \rho \mathring{\vect z}^{\epsh} \cdot\mathring{\vect \xi}\,dx = 
\int_{\Omega_0^{\epsh}} \mathring{\vect f}^{\epsh}\cdot\mathring{ \vect \xi}\,dx,\\[0.2em] 
&\forall \mathring{\vect \xi} \in H^1_{\Gamma_{\rm D}}(\Omega;\R^3),\quad \mathring{\vect \xi}=0 \textrm{ on } \Omega_1^{\epsh}, 
\end{aligned} 
\end{equation}
where $\mathring{\vect f}^{\epsh}\in L^2(\Omega;\R^3)$ and $\mathring{\vect z}^{\epsh} \in H^1_{\Gamma_{\rm D}}(\Omega;\R^3)$, $\mathring{\vect z}^{\epsh}=0$ on  $\Omega_1^{\epsh}.$
Denoting by $\mathring{\vect u}^{\epsh}_{\rm c}$ the solution of \eqref{ivan601} with $\mathring{\vect f}^{\epsh}=-\lambda^{\epsh}\rho\tilde{\vect u}^{\epsh},$ subtracting \eqref{ivan601} from \eqref{ivan402}, and using an appropriate version of Lemma  \ref{nenad100} (see Remark \ref{ivan501}), we obtain that
\begin{equation}
	\eps_h\bigl\|\sym \nabla_h(\mathring{\vect u}^{\epsh}-\mathring{\vect u}^{\epsh}_c)\bigr\|_{L^2}  +\bigl\|\mathring{\vect u}^{\epsh}-\mathring{\vect u}^{\epsh}_{\rm c}\bigr\|_{L^2} \to 0\ \ \textrm{ as }\  h \to 0. 
\label{ref2}
\end{equation}
Notice also that $\mathring{\vect u}_* ^{\epsh} \stackrel{L^2}{\to} 0$ as the consequence of apriori estimates, see also \eqref{nada1}. 
We now take $\mathring{ \vect g}^{\epsh} \in L^2(\Omega;\R^3)$ such that $\mathring{\vect g}^{\epsh} \drtwoscale \mathring{\vect g} \in L^2(\Omega \times Y;\R^3)$.  Furthermore, we take $\mathring{\vect s}^{\epsh}$ as the solution of \eqref{ivan601} with $\mathring{\vect f}^{\epsh}=\mathring{\vect g}^{\epsh}$. Substituting  $\mathring{\vect s}^{\epsh}$ as a test function in the equation for $\mathring{\vect u}^{\epsh}_{\rm c}$ and $\mathring{\vect u}^{\epsh}_{\rm c}$ as a test function in the equation for $\vect s^{\epsh},$ we obtain by the same argument as in the proof of Theorem \ref{thmivan113} that 
\begin{eqnarray*}
 & &\frac{1}{12}\int\limits_{\omega \times Y_0} \C_{0}^{\rm bend,r}(y) \nabla_y^2 {\mathring{u}}_3(\hat{x},y): \nabla_y^2{\mathring{{s}}}_3(\hat{x},y)\,d\hat{x}dy  -
\lambda \int\limits_{\omega \times Y_0} \rho_0(y) {\mathring{u}}_3(\hat{x},y) \cdot \mathring{{s}}_3(\hat{x},y) \,d\hat{x}dy
\\ & & \hspace{+2ex} =
-\lambda \int\limits_{\omega \times Y_0}\rho_0(y) \tilde{u}_3 (\hat{x},y) \cdot  \mathring{ {s}}_3 (\hat{x},y) \, d\hat{x} dy
%%\\
%%& & \hspace{+2ex} 
=\int\limits_{\omega \times Y_0} \overline{\mathring{  g}}_3 (\hat{x},y) \cdot  \mathring{ {u}}_3 (\hat{x},y) \, d\hat{x} dy,
 \end{eqnarray*} 
where $\mathring{u}_3$, $\mathring{ s}_3 \in L^2(\omega;H_0^2(Y_0))$ are weak two-scale limits of $\mathring{ u}^{\epsh}_3,$ $\mathring{ s}^{\epsh}_3$ while $\mathring{\vect s}_* ^{\epsh}\longrightarrow 0$ in $L^2,$ and $\tilde{u}_3\in H^2_{\gamma_D}(\omega)$ is the strong limit of $\tilde{ u}^{\epsh}_3$ while $\tilde{\vect u}^{\epsh}_*  \longrightarrow 0$ in $L^2$. It follows that $$
\lim_{h \to 0} \int_{\Omega} \mathring{\vect g}^{\epsh}\cdot \mathring{\vect u}^{\epsh}\,dx=-\lambda \lim_{h \to 0} \int_{\Omega} \rho\tilde{\vect u}^{\epsh}\cdot \mathring{\vect s}^{\epsh}\,dx=-\lambda \int\limits_{\Omega \times Y}\rho \tilde{u}_3 (\hat{x},y) \cdot  \mathring{ {s}}_3 =\int_{\Omega \times Y} g_3(x,y) \mathring{u}_3(\hat{x},y)d \hat{x} dy.
$$ 
Therefore, the sequence $\mathring{\vect u}^{\epsh},$ and consequently ${\vect u}^{\epsh},$ converges strongly two-scale. Passing to the limit in the (weak formulation of the) equation $h^{-2}\mathcal{A}_{\epsh} \vect u^{\epsh}=\lambda^{\epsh} \vect u^{\epsh},$ we immediately obtain $\hat{\mathcal{A}_0} \vect   u=\lambda \vect u$, where $\vect u \neq 0$ is the two-scale limit of $\vect u^{\epsh}$. 
\end{proof} 
\noindent{\bf D. Proof of Theorem \ref{thmivan301}}
\begin{proof} 
	The proof uses some ideas given in \cite{AllaireConca98} adapted to the present, simpler, setup. \\
\RRR	{\bf Step 1.} 	We   prove (\ref{sigma_ess}). \BBB\\

	By applying the Fourier transform, it is easily seen that the generalised eigenfunctions of $\mathring{\mathcal{A}}_{\rm strip}$ are of the form 
	\begin{equation}
	\vect u^{\eta}_{\rm strip}(y_1,y_2,x_3)=e^{i\eta x_3} \vect u^{\eta}(y_1,y_2), \quad \eta \in \R,  
	\label{eta_gen}
\end{equation}
	where $\vect u^{\eta} \in H^1_{0}(Y_0;\C^3)$
	is an eigenfunction of the self-adjoint operator $\mathring{\mathcal{A}}_{\rm strip}^{\eta}$ on $L^2(Y_0;\C^3)$ defined via the bilinear form 
	\begin{eqnarray*} 
		& &\mathring{a}_{\rm strip}^{\eta}(\vect u, \vect v)=\int_{Y_0} \C_0(y) \sym\bigl(\partial_{y_1} \vect u\,|\,\partial_{y_2} \vect u\,|\,{\rm i}\eta\vect u\bigr): \sym \bigl(\partial_{y_1} \overline{\vect v}\,|\,\partial_{y_2}\overline{\vect v}\,|\,\overline{ {\rm i}\eta\vect v}\bigr) dy, \\ & & \hspace{+10ex} \mathring{a}_{\rm strip}^{\eta}: H_{0}^1(Y_0;\C^3) \times H_{0}^1(Y_0;\C^3)\to \C.  
	\end{eqnarray*} 
	It is easily seen that for each $\eta \in \R$ the operator $\mathring{\mathcal{A}}^{\eta}_{\rm strip}$ is positive definite and has compact resolvent, and thus it has an increasing sequence of eigenvalues $\{\alpha_1^{\eta},\alpha_2^{\eta}, \dots\}$ diverging to $+\infty$. It follows that
	$$ \sigma(\mathring{\mathcal{A}}_{\rm strip})=\bigcup_{\eta \in \R}\bigl\{\alpha_1^{\eta}, \alpha_2^{\eta}, \dots\bigr\}. $$ 
	By using a suitable Korn's inequality  on the on $I \times Y_0$ (applied to the function $(x_3,y_1,y_2) \mapsto e^{i\eta x_3} \vect u(y_1,y_2)$) and \eqref{coercivity}, we obtain that there exists a constant $C>0$, which is independent of $\eta$, such that
	\begin{equation}	
		\|\vect u\|_{L^2}^2+\left\|\bigl(\partial_{y_1}\vect u\,|\,\partial_{y_2}\vect u\,|\,{\rm i}\eta\vect u \bigr)\right\|^2_{L^2} \leq C \mathring{a}_{\rm strip}^{\eta} (\vect u, \vect u) \qquad  \forall \vect u \in H^1_{0}(Y_0;\C^3).  
		\label{Korns}
	\end{equation}
	Furthermore, using the characterisation of eigenvalues through a Rayleigh quotient, we obtain 
	\begin{equation}
		\alpha_1^{\eta}=\min_{\vect u \in H^1_{0}(Y_0;\C^3)} \frac{\mathring{a}^{\eta}_{\rm strip}(\vect u,\vect u)}{\|\vect u\|^2_{L^2}}.  
		\label{uRayleigh}
	\end{equation}
	Combining this with (\ref{Korns}), we infer that there exists $c>0$, independent of $\eta$, such that  
	$$ 
	\alpha_1^{\eta} \geq c \min_{\vect u \in H^1_{0}(Y_0;\C^3)}\frac{\left\|\bigl(\partial_{y_1}\vect u\,|\,\partial_{y_2}\vect u\,|\,{\rm i}\eta\vect u \bigr)\right\|^2_{L^2}}{\| \vect u \|^2_{L^2}}.  
	$$
	
	Finally, using Poincar\'{e}'s inequality on $Y_0,$ we obtain the existence of $c>0$ such that $\alpha_1^{\eta} \geq c+\eta^2$. The continuity of $\alpha_1^{\eta}$ with respect to $\eta$ (which can also be inferred from (\ref{uRayleigh})) implies that the range of the mapping $\eta \mapsto \alpha_1^{\eta}$ is $[m_0,+\infty)$ for some $m_0>0.$ This concludes the characterisation of the set $\sigma(\mathring{\mathcal{A}}_{\rm strip}),$ provided by (\ref{sigma_ess}). 
	%%%Note, in particular, that the spectrum of $\mathring{\mathcal{A}}_{\rm strip}$ is purely essential
	\\
		\RRR {\bf Step 2.} \BBB 	We   prove \eqref{pmess}.
		\\
	Proceeding to the discussion of the sets $\sigma_{\rm ess}(\mathring{\mathcal{A}}_{\rm strip}^{\pm}),$ we show that they in fact coincide 
	%%%with  show that the essential spectrum of $\mathring{\mathcal{A}}_{\rm strip}^{\pm}$ actually coincides
	 with $\sigma(\mathring{\mathcal{A}}_{\rm strip}).$ The proof of this claim, for which we just provide a sketch, is similar to the argument of \cite[Proposition 7.5]{AllaireConca98}. Consider a Weyl sequence associated to $\lambda \in \sigma_{\rm ess}(\mathring{\mathcal{A}}_{\rm strip}^{+})$, i.e.,  
	$(\vect u^{+,n})_{n \in \N} \in \mathcal{D} (\mathring{\mathcal{A}}_{\rm strip}^{+})$ such that
	\begin{equation} \label{ispravak1} 
		\|\vect u^{+,n}\|_{L^2}=1, \quad \vect u^{+,n} \weakL 0, \quad \bigl\|\mathring{\mathcal{A}}_{\rm strip}^+\vect u^{+,n}-\lambda \vect u^{+,n}\bigr\|_{L^2}\to 0.
	\end{equation}  
	The properties (\ref{ispravak1}) imply that $(\vect u^{+,n})_{n \in \N}$ is bounded in $H^1$. Next, take a smooth positive function $\psi:\R^{+}_0 \to \R$ that takes zero values on $(-\infty,1]$ and is equal to unity on $[2,+\infty)$ and show that 
	for all $v \in H^1_{00}(\R^+_0 \times Y_0;\R)$ one has 
	\begin{equation}
		\label{ivan201} 
		\begin{aligned}
			\int_{\R^+_0 \times Y_0} \C_0(y) \nabla (\psi\vect u^{+,n}):\nabla \vect v dx_3 dy&-\lambda\int_{\R^+_0 \times Y_0}   \rho_0 (\psi \vect u^{+,n})\cdot \vect v dx_3 dy\\[0.35em]
			& 
			=\int_{\R \times Y_0} \C_0(y) \nabla ( \psi \vect u^{+,n}):\nabla \vect v dx_3 dy -\lambda \int_{\R \times Y_0} \rho_0 (\psi\vect u^{+,n})\cdot\vect v dx_3 dy\\[0.35em]
			& 
			=\int_{[1,2] \times Y_0} \C_0(y) \sym\bigl(0\,|\,0\,|\partial_{x_3} \psi \vect u^{+,n}\bigr)  :\sym \nabla \vect v dx_3 dy\\[0.25em]
			&\qquad\qquad\qquad+\int_{[1,2] \times Y_0} \C_0(y) \sym \nabla \vect u^{+,n} : \sym\bigl(0\,|\,0\,|  \partial_{x_3} \psi \vect v\bigr)\,dx_3dy.    		
		\end{aligned}
	\end{equation}	
	Combining \eqref{ispravak1} with compact embedding of $H^1$ into $L^2$ on bounded domains, we conclude that for all bounded sets $A$ one has $\|\vect u^{+,n}\|_{L^2(A)}\to 0$.
	Furthermore, considering a smooth non-negative compactly supported function $\psi_A$ that is equal to one on $A$ and noting that by virtue of (\ref{ispravak1}) one has  
	$$
	\int_{{\mathbb R}_0^+\times Y_0}\bigl(
	\mathring{\mathcal{A}}_{\rm strip}^+\vect u^{+,n}-\lambda \vect u^{+,n}\bigr)\psi_A \vect u^{+,n}\to0,
$$	
%%%	 By testing the last expression in \eqref{ispravak1} with $\psi_A \vect u^{+,n}$, where 
 we obtain that actually $\|\vect u^{+,n}\|_{H^1(A)}\to 0$.  
	
	Thus we conclude that the right-hand side of \eqref{ivan201} can be written in the form
	$$
	\int_{[1,2] \times Y_0} \vect f_1^n:\sym \nabla \vect v dx_3dy+ \int_{[1,2] \times Y_0} \vect f_2^n\cdot\vect v dx_3dy, 
	$$
	where $\|\vect f_1^n\|_{L^2}\to 0$ and $\|\vect f_2^n\|_{L^2}\to 0$ as $n \to \infty$. By combining a suitable version of Lemma \ref{nenad100} with \eqref{ivan201}, we conclude that $\lambda \in \sigma(\mathring{\mathcal{A}}_{\rm strip})$. In a similar fashion, starting from the generalised eigenfunction (\ref{eta_gen}), we conclude that $\sigma(\mathring{\mathcal{A}}_{\rm strip}) \subset \sigma(\mathring{\mathcal{A}}_{\rm strip}^+) $. 
	
By repeating the above argument for ${\mathcal{A}}_{\rm strip}^{-},$ we also obtain
	\[
	\sigma_{\rm ess}\bigl(\mathring{\mathcal{A}}_{\rm strip}^{-}\bigr)=\sigma\bigl(\mathring{\mathcal{A}}_{\rm strip}\bigr).
	%%%\qquad \sigma \bigl(\mathring{\mathcal{A}}_{\rm strip}^{-}\bigr)\subset \lim_{h\to 0} \sigma\bigl(\mathring{\mathcal{A}}_{\epsh}\bigr).
	\] 
This establishes the property \eqref{pmess}. 
\\
\RRR {\bf Step 3.} 	We   prove  \eqref{assincl}. \BBB
\\

		First, by virtue of the symmetries of the elastic tensor (and considering appropriate Weyl sequences), we easily obtain the equality $\sigma(\mathring{\mathcal{A}}_{\rm strip}^+)=\sigma(\mathring{\mathcal{A}}_{\rm strip}^-)$. 
	Next we show that
	\begin{equation} \label{ispravak222} 
		\sigma_{\rm ess}\bigl(\mathring{\mathcal{A}}_{\rm strip}^+\bigr) \subset \sigma\bigl(\mathring{\tilde{\mathcal{A}}}_{\rm strip}\bigr), \qquad \sigma\bigl(\mathring{\tilde{\mathcal{A}}}_{\rm strip}\bigr)=\sigma\bigl(\mathring{{\mathcal{A}}}_{\rm strip}\bigr) 
	\end{equation} 
	To show the first inclusion in \eqref{ispravak222}, we take a Weyl sequence associated to the $\lambda \in \sigma_{\rm ess}(\mathring{\mathcal{A}}_{\rm strip}^{+})$, i.e.,
	$(\vect u^{+,n})_{n \in \N} \subset \mathcal{D} (\mathring{\mathcal{A}}_{\rm strip}^{+})$ such that 
	$$
	\|\vect u^{+,n}\|_{L^2}=1,\qquad\vect u^{+,n} \weakL 0,\qquad\bigl\|\mathring{\mathcal{A}}_{\rm strip}^+\vect u^{+,n}-\lambda\vect u^{+,n}\bigr\|_{L^2}\to 0.
	$$ 
	Using the elastic symmetries once again, we infer that for the functions
	\[
	\vect u_* ^{-,n}(x_3, y):=\vect u^{+,n}_* (-x_3, y),\qquad u_{3}^{-,n}(x_3, y):=-u^{+,n}_3(-x_3, y), \qquad (x_3, y)\in{\mathbb R}_0^+\times Y_0,
	\] 
	one has
	\[
	\|\vect u^{-,n}\|_{L^2}=1,\qquad \vect u^{-,n} \weakL 0,\qquad \bigl\|\mathring{\mathcal{A}}_{\rm strip}^-\vect u^{-,n}-\lambda \vect u^{-,n}\bigr\|_{L^2}\to 0.
	\] 
	We also note that the sequences $(\vect u^{\pm,n})_{n \in \N}$ are bounded in $H^1$. We now define 
	\[
	\vect u^{n}(x_3, y):=\psi(x_3)\vect u^{+,n}(x_3,y)+\psi(-x_3)\vect u^{-,n}(x_3,y),\qquad (x_3, y)\in{\mathbb R}\times Y_0.
	\]
	  In the same way as in \eqref{ivan201}, we conclude that for every $v \in H^1_{00}(\R^+_0 \times Y_0;\R)$ one has
	\begin{equation*}
		\begin{aligned}
			%\begin{eqnarray*}
			\int_{\R \times Y_0} \C_0(y) \nabla \vect u^{n}:\nabla \vect v dx_3 dy&-\lambda\int_{\R^ \times Y_0} \rho_0 \vect  u^{n} \vect v dx_3 dy\\&=
			%%\\
			%%& &
			\int_{([1,2]\cup[-2,-1]) \times Y_0} \vect f_1^n:\sym \nabla \vect v dx_3dy+ \int_{([1,2]\cup[-2,-1]) \times Y_0} \vect f_2^n \vect v dx_3dy,
			%\end{eqnarray*} 
		\end{aligned}
	\end{equation*}
	where $\|\vect f_1^n\|_{L^2}\to 0$ and $\|\vect f_2^n\|_{L^2}\to 0$ as $n \to \infty,$ from which it follows that $\lambda \in \sigma(\mathring{\tilde{\mathcal{A}}}_{\rm strip})$.

	 For the last equality in \eqref{ispravak222} it suffices to argue that  
	$\sigma(\mathring{{\mathcal{A}}}_{\rm strip}) \subset \sigma(\mathring{\tilde{\mathcal{A}}}_{\rm strip})$. To this end, we apply the Fourier transform and for $\eta \in \R$ we consider  generalised eigenfunctions of the operator $\mathring{{\mathcal{A}}}_{\rm strip}$ of the form	
	\[
	\vect u^{\eta}_{\rm strip}(x_3, y)=e^{i\eta x_3} \vect u^{\eta}(y),\qquad (x_3, y)\in {\mathbb R}\times Y_0,
	\] 
	where  $\vect u^{\eta} \in H^1_{00}(Y_0;\C^3)$
	is an eigenfunction of the operator $\mathring{\mathcal{A}}_{\rm strip}^{\eta}$, i.e., $\mathring{\mathcal{A}}_{\rm strip}^{\eta}\vect u^{\eta}(y_1,y_2)=\alpha_i^{\eta} \vect u^{\eta}(y_1,y_2)$, $\vect u^{\eta} \neq 0$, for some $i \in \N$. Invoking the symmetries, we infer that for each $\eta\in{\mathbb R}$
	\[
	\vect (\vect u^{-\eta}_{\rm strip})_*(x_3,y):=(\vect u^{\eta}_{\rm strip})_* (-x_3, y),\qquad u_{\rm strip,3}^{-\eta}(x_3, y):=-u^{\eta}_{\rm strip,3}(-x_3,y), \qquad (x_3, y)\in{\mathbb R}\times Y_0,
	\] 	
	is also a generalised eigenfunction of the operator $\mathring{{\mathcal{A}}}_{\rm strip}$ associated with the same eigenvalue $\alpha_i^{\eta}$. Therefore, the function  $\bigl( \vect u^{\eta}_{\rm strip}+ \vect u^{-\eta}_{\rm strip}\bigr)/2$ 
	is a generalised eigenfunction of the operator $\mathring{\tilde{\mathcal{A}}}_{\rm strip}$ (and hence the operator $\mathring{{\mathcal{A}}}_{\rm strip}$) associated with the same eigenvalue. %%$\alpha_i^{\eta}$. 
	Since every element of the spectrum of the operator $\mathring{{\mathcal{A}}}_{\rm strip}$ coincides with $\alpha_i^{\eta}$ for some $\eta \in \R$ and $i \in \mathbb{N}$, we conclude that $\sigma(\mathring{{\mathcal{A}}}_{\rm strip}) \subset \sigma(\mathring{\tilde{\mathcal{A}}}_{\rm strip})$. 
	This construction also proves that  $\sigma_{\rm ess}(\mathring{\tilde{\mathcal{A}}}_{\rm strip})= \sigma\bigl(\mathring{{\mathcal{A}}}_{\rm strip}\bigr)$. 
	Since we have already established that $\sigma\bigl(\mathring{{\mathcal{A}}}_{\rm strip}\bigr) \subset \sigma\bigl(\mathring{\mathcal{A}}_{\rm strip}^+\bigr),$ the property \eqref{assincl} follows.
\\
\RRR {\bf Step 4.} 	We   prove   \eqref{characterise}. \BBB
\\
 We start by proving the inclusion 
	\begin{equation}
	\lim_{h \to 0} \sigma(\mathring{\mathcal{A}}_{\epsh}) \subset \sigma(\mathring{\mathcal{A}}_{\rm strip})\cup \sigma(\mathcal{A}_{\rm strip}^+) \cup \sigma(\mathcal{A}_{\rm strip}^-).
	\label{sigma_incl}
    \end{equation}
	Let us take $\lambda^{\epsh} \in \sigma (\mathring{\mathcal{A}}_{\epsh})$ and $\vect u^{\epsh} \in \mathcal{D} (\mathring{\mathcal{A}}_{\epsh}) $ such that $\lambda^{\epsh} \to \lambda$ and 
	$$\mathring{\mathcal{A}}^{\epsh} \vect u^{\epsh} =\lambda^{\epsh} \vect u^{\epsh}, \quad \|\vect u^{\epsh}\|_{L^2}=1. $$  
	Consider smooth positive functions  $\psi_i,$ $i=1,2,3$ on ${\mathbb R}$ such that 
	$\psi_1+\psi_2+\psi_3=1$, $\textrm{supp}\ \psi_2\subset [-1/4,1/4]$, $\textrm{supp}\ \psi_1\subset (-\infty,-1/8]$, $\textrm{supp}\ \psi_3\subset [1/8, \infty),$ and $\psi_3(x_3)=\psi_1(-x_3)$. 
	Then  there exists $i \in \{1,2,3\}$ such that (up to a subsequence) 
	$$
	\bigl\|\psi_i \vect u^{\epsh}\bigr\|_{L^2} \geq \frac{1}{3} \quad \forall h. 
	$$
	If $i=2$, we extend $\psi_2 \vect u^{\epsh}$ 
	by zero on $\R \times Y_0$ and, by scaling the variable $x_3,$ define  
	$$
	\vect u_{\rm strip}^{\epsh}(x_3, y)= \sqrt{\frac{h}{\epsh}}\psi_2\biggl(\frac{\epsh}{h}x_3\biggr) \vect u^{\epsh}\biggl(\frac{\epsh}{h}x_3, y\biggr), \qquad (x_3, y)\in{\mathbb R}\times Y_0.
	$$
	It is straightforward to see that
	\[
	\bigl\|\vect u_{\rm strip}^{\epsh}\bigr\|_{L^2}\geq \frac{1}{3}
	\] 
	and that for all $\vect v \in H^1_{00} (\R \times Y_0;\R^3)$ one has 
	$$ \int_{\R \times Y_0} \C_0(y) \nabla \vect u_{\rm strip}^\epsh:\nabla \vect v dx_3 dy -\lambda^{\epsh} \int_{\R \times Y_0} \rho_0  \vect u_{\rm strip}^\epsh \vect v dx_3 dy= \int_{\R \times Y_0} \vect f_1^{\epsh}:\sym \nabla \vect v dx_3 dy+ \int_{\R \times Y_0} \vect f_2^{\epsh} \vect v,    $$ 
	where $\|\vect f_1^{\epsh}\|_{L^2}\to 0$, $\| \vect f_2^{\epsh}\|_{L^2} \to 0$ as $ h \to 0$.  
	By using an appropriate analogue of Lemma \ref{nenad100} (see also Remark \ref{ivan501}) adapted to the operator $\mathring{\mathcal{A}}_{\epsh},$  we conclude  that 
	\begin{equation}
	\lambda \in \sigma\bigl(\mathring{\mathcal{A}}_{\rm strip}\bigr).
	\label{ref3}
	\end{equation}  
If $i=1$ or $i=3$ we argue similarly that $\lambda \in \mathring{\mathcal{A}}_{\rm strip}^+$, i.e., $\lambda \in \mathring{\mathcal{A}}_{\rm strip}^-$ respectively.

	Next, we prove that $\sigma (\mathring{\mathcal{A}}_{\rm strip}) \subset \lim_{h \to 0} \sigma(\mathring{\mathcal{A}}_{\epsh})$. Considering $\alpha_i^{\eta}$ and $\vect u^{\eta}\in \mathcal{D}(\mathring{\mathcal{A}}^{\eta}_{\rm strip})$ such that 
	$$
	\mathring{\mathcal{A}}^{\eta}_{\rm strip}\vect u^{\eta}=\alpha_i^{\eta} \vect u^{\eta}, \qquad \|\vect u^{\eta}\|_{L^2}=1
	$$ 
	we set 
	\[
	\vect u^{\eta}_{\rm strip}(x_3, y)=e^{i\eta x_3}\vect u^{\eta}(y),\qquad (x_3, y)\in {\mathbb R}\times Y_0.
	\] 
	It is easily seen that $\mathring{\mathcal{A}}_{\rm strip} \vect u^{\eta}_{\rm strip}=\alpha_i^{\eta}\vect u^{\eta}_{\rm strip}$. 
	We define 
	$$\vect u^{\epsh}(x_3, y)={\biggl\|\psi_2(x_3) \vect u^{\eta}_{\rm strip}\biggl(\dfrac{h}{\epsh}x_3, y\biggr) \biggr\|_{L^2}}^{-1}{\psi_2(x_3) \vect u^{\eta}_{\rm strip}\biggl(\dfrac{h}{\epsh}x_3, y\biggr)}, \qquad (x_3, y)\in I\times Y_0.
	 $$
	It then follows easily that for every $h>0$ and $\vect v\in H^1_{00}(I \times Y_0;\C^3)$ one has
	\begin{equation} 
		\label{ivan202} 
		\begin{aligned}
	\int_{I \times Y_0} \C_0(y) \nabla_{\frac{h}{\epsh}} \vect u^\epsh:\nabla_{\frac{h}{\epsh}} \vect v dx_3 dy&-\alpha_i^{\eta}\int_{I \times Y_0} \rho_0 \vect u^\epsh\cdot\vect v dx_3 dy\\[0.4em]
	&=\int_{I \times Y_0} \vect f_1^{\epsh}:\sym \nabla_{\frac{h}{\epsh}} \vect v dx_3 dy+ \int_{I \times Y_0} \vect f_2^{\epsh}\cdot\vect v,  
	\end{aligned}  
	\end{equation} 
	where $\|\vect f_1^{\epsh}\|_{L^2}\to 0$, $\| \vect f_2^{\epsh}\|_{L^2} \to 0$ as $ h \to 0$.  By using a result analogous to Lemma \ref{nenad100} (see also Remark \ref{ivan501}) we obtain  
	\begin{equation}
	\dist\bigl(\alpha_i^{\eta}, \sigma(\mathring{\mathcal{A}}_{\epsh})\bigr)\to 0\quad{\rm as}\ \ \ h \to 0.
	\label{ref4}
	\end{equation}

	It can be also easily deduced that $\sigma_{\rm disc}(\mathring{\mathcal{A}}_{\rm strip}^+)\subset \lim_{h \to 0} \sigma(\mathring{\mathcal{A}}_{\epsh})$. Namely, for an eigenvalue $\alpha_{\rm strip}^+$ of $\mathring{\mathcal{A}}_{\rm strip}^+$ and associated eigenfunction $\vect u^{\alpha^+}_{\rm strip} \in \mathcal{D}(\mathring{\mathcal{A}}_{\rm strip}^+)$, $\|\vect u^{\alpha^+}_{\rm strip}\|_{L^2}=1,$ i.e.,  
	$\mathring{\mathcal{A}}_{\rm strip}^+ \vect u^{\alpha^+}_{\rm strip}=\alpha_{\rm strip}^+ \vect u^{\alpha^+}_{\rm strip},$ it can be easily shown that the sequence 
	$$
	\vect u^{\epsh}(x_3, y)={\biggl\|\psi_1(x_3) \vect u^{\alpha^+}_{\rm strip}\biggl(\dfrac{h}{\epsh}\biggl(x_3+\dfrac{1}{2}\biggr), y\biggr)\biggr\|_{L^2}}^{-1}{\psi_1(x_3)\,\vect u^{\alpha^+}_{\rm strip}\biggl(\dfrac{h}{\epsh}\biggl(x_3+\dfrac{1}{2}\biggr), y\biggr)}, \qquad (x_3,y)\in I\times Y_0,
	$$
	satisfies \eqref{ivan202} with $\|\vect f_1^{\epsh}\|_{L^2}\to 0$, $\| \vect f_2^{\epsh}\|_{L^2} \to 0$ as $ h \to 0$ and with $\alpha_i^{\eta}$ replaced by $\lambda$. It follows that $\lambda \in \lim_{h \to 0} \sigma(\mathring{\mathcal{A}}_{\epsh})$. In view of \eqref{pmess}, we obtain the opposite inclusion in \eqref{sigma_incl}.

%%%	It remains to prove, under the Assumption \ref{assumivan1}\,(1), the characterization of $\lim_{h \to 0} \sigma(\mathring{\tilde{\mathcal{A}}}_{\epsh})$. 
It remains to prove, under the Assumption \ref{assumivan1}~(1), the characterisation of $\lim_{h \to 0} \sigma(\mathring{\tilde{\mathcal{A}}}_{\epsh})$ provided by (\ref{characterise}). 
By the same argument as in the case without planar symmetries, we obtain  
\begin{equation} \label{ispravak2222} 
\lim_{h \to 0} \sigma\bigl(\mathring{\tilde{\mathcal{A}}}_{\epsh}\bigr)\subset \sigma\bigl(\mathring{\tilde{\mathcal{A}}}_{\rm strip}\bigr)\cup \sigma\bigl(\mathring{\mathcal{A}}_{\rm strip}^+\bigr) \cup  \sigma\bigl(\mathring{\mathcal{A}}_{\rm strip}^-\bigr), 
\end{equation} 
and $\sigma(\mathring{\tilde{\mathcal{A}}}_{\rm strip}) \subset  \lim_{h \to 0} \sigma(\mathring{\tilde{\mathcal{A}}}_{\epsh})$. By virtue of \eqref{pmess} and \eqref{assincl}, it remains to prove the inclusion
$$
\sigma_{\rm disc}\bigl(\mathring{\mathcal{A}}_{\rm strip}^\pm\bigr) \subset \lim_{h \to 0} \sigma\bigl(\mathring{\tilde{\mathcal{A}}}_{\epsh}\bigr).
$$
This will be done by a slightly different argument, as follows. 
For $\alpha\in \sigma_{\rm disc}(\mathring{\mathcal{A}}_{\rm strip}^+)$ we take the associated eigenfunction  $\vect u^{\alpha,+}_{\rm strip} \in \mathcal{D}(\mathring{\mathcal{A}}_{\rm strip}^+)$, $\|\vect u^{\alpha,+}_{\rm strip}\|_{L^2}=1,$ of the operator $\mathring{\mathcal{A}}_{\rm strip}^+,$ i.e., $\mathring{\mathcal{A}}_{\rm strip}^+ \vect u^{\alpha,+}_{\rm strip}=\alpha\vect u^{\alpha,+}_{\rm strip}$. Using the elastic symmetries, we infer that the functions $\vect u^{\alpha,-}_{\rm strip}$ defined by 
\[
\vect u_* ^{\alpha,-}(x_3, y):=\vect u^{\alpha,+}_* (-x_3, y),\qquad\vect u_{3}^{\alpha,-}(x_3, y):=-\vect u^{\alpha,+}_3(-x_3, y), \qquad (x_3, y)\in {\mathbb R}_0^+\times Y_0,
\] 
satisfy  $\vect u^{\alpha,-}_{\rm strip} \in \mathcal{D}(\mathring{\mathcal{A}}_{\rm strip}^-)$, $\|\vect u^{\alpha,-}_{\rm strip}\|_{L^2}=1$ and $\mathring{\mathcal{A}}_{\rm strip}^- \vect u^{\alpha,-}_{\rm strip}=\alpha\vect u^{\alpha,-}_{\rm strip}$. 
Finally, we define
$$
\vect u^{\epsh}(x_3, y)=\frac{\psi_1(x_3) \vect u^{\alpha,+}_{\rm strip}\biggl(\dfrac{h}{\epsh}\biggl(x_3+\dfrac{1}{2}\biggr), y\biggr)+\psi_3(x_3) \vect u^{\alpha,-}_{\rm strip}\biggl(\dfrac{h}{\epsh}\biggl(x_3-\dfrac{1}{2}\biggr), y\biggr)}{\biggl\|\psi_1(x_3) \vect u^{\alpha,+}_{\rm strip}\biggl(\dfrac{h}{\epsh}\biggl(x_3+\dfrac{1}{2}\biggr), y\biggr)+\psi_3(x_3) \vect u^{\alpha,-}_{\rm strip}\biggl(\dfrac{h}{\epsh}\biggl(x_3-\dfrac{1}{2}\biggr), y\biggr)\biggr\|_{L^2}},\qquad (x_3, y)\in I\times Y_0,
$$
and use an argument similar to that employed for showing that $\alpha\in\lim_{h \to 0} \sigma(\mathring{\tilde{\mathcal{A}}}_{\epsh})$ under no symmetry assumptions. 

Similarly, we demonstrate that 
\[
\sigma_{\rm disc}(\mathring{\mathcal{A}}_{\rm strip}^-)\subset\lim_{h \to 0} \sigma(\mathring{\tilde{\mathcal{A}}}_{\epsh}),
\] 
which concludes the proof of the opposite inclusion in \eqref{ispravak2222}. 	
\end{proof} 		
\begin{remark} 
	In the same way as in \cite[Proposition 7.5]{AllaireConca98}, it can be shown that eigenfunctions associated with eigenvalues in $\sigma(\mathring{\mathcal{A}}_{\rm strip}^{\pm})$ have exponential decay at infinity. 	
\end{remark} 	

\noindent{\bf E.  Proof of Theorem \ref{thmivan302} and Theorem \ref{thmivan303}}
\begin{proof} 
The equality \eqref{ispravak3} is proved in the same way as in \cite[Section 8]{Zhikov2000}. The inclusion $\sigma(\tilde{\mathcal{A}}_{\infty}) \subset \lim_{h\to 0} \sigma(\tilde{\mathcal{A}}_{\epsh})$ follows from resolvent convergence provided by Theorem \ref{rucak50} and Corollary \ref{corivan11}, while the inclusion $\sigma(\mathring{\mathcal{A}}^+_{\rm strip})=\lim_{h \to 0} \sigma(\mathring{\tilde{\mathcal{A}}}_{\epsh}) \subset \lim_{h\to 0} \sigma(\tilde{\mathcal{A}}_{\epsh})$ follows from Theorem \ref{thmivan301} and Lemma \ref{lemmaivan100}. 

It remains to show that $\lim_{h \to 0} \sigma({\tilde{\mathcal{A}}}_{\epsh}) \subset \sigma(\mathring{\mathcal{A}}^+_{\rm strip}) \cup \sigma(\tilde{\mathcal{A}}_{\infty})$.  To this end, consider $\lambda^{\epsh} \in \sigma(\tilde{\mathcal{A}}_{\epsh})$ such that 
\begin{equation}
\liminf_{h \to 0}\dist\bigl(\lambda^{\epsh}, \sigma(\mathring{\tilde{\mathcal{A}}}_{\epsh})\bigr)>0
\label{liminfin}
\end{equation}
(which is the only case that requires analysis, due to Lemma \ref{lemmaivan100}) and $\lambda^{\epsh} \to \lambda$. Furthermore, consider $\vect u^{\epsh}\in \mathcal{D}(\tilde{\mathcal{A}}_{\epsh})$ such that
$\|\vect u^{\epsh}\|_{L^2}=1$ and $\tilde{\mathcal{A}}_{\epsh} \vect u^{\epsh}=\lambda^{\epsh} \vect u^{\epsh}$. 
The strong two-scale compactness of $\vect u^{\epsh}$ is proved in the same way as in the proof of Theorem \ref{rucak72} by combining (\ref{liminfin}) with
%%from the fact $\liminf_{h \to 0}\dist(\lambda^{\epsh}, \sigma(\mathring{\tilde{\mathcal{A}}}_{\epsh}))>0$ by using
Lemma \ref{nenad100}, see also Remark \ref{ivan501}.   
The equation \eqref{ispravak4} is a direct consequence of the symmetry assumptions.

 The proof of Theorem \ref{thmivan303} is carried out in a similar fashion. 
\end{proof} 
 \subsection{Proofs for Section \ref{limeveqsec}} 

\vskip 0.2cm

\noindent{\bf A. Proof of Theorem \ref{connn1}}
\begin{proof} 
It is not possible to put the first claim in the framework of Theorem \ref{thmivan551} or Theorem \ref{thmivan553} directly (i.e., using Proposition \ref{propdinsi1} and Theorem \ref{thmivan111}) and we will provide a direct proof instead, using Laplace transform similarly to how it was done in the proofs of these theorems. 
\RRR The reason why we cannot put the first claim  in the framework of Theorem \ref{thmivan551} or Theorem \ref{thmivan553} directly comes from the fact that $\vect f_{*}\neq 0$ and they influence the (quasistatic) behavior of the part of in-plane deformation. \BBB

 For every $\epsh>0,$ we write the system \eqref{evolutionepsilonproblem} for $\mu_h=\epsh$, $\tau=2$, using the formula \eqref{abstractevolutionproblemrephrased}, where $\mathbb{A}=\mathbb{A}_{\epsh}$ 
%%(we add to it index $\epsh$ for $\epsh$ problem) 
is given by formula \eqref{matrixoperator} and the associated operator $\mathcal{A}$ is given by $h^{-2}\mathcal{A}_{\epsh}$. Furthermore, we set $H_{\epsh}=L^2(\Omega;\R^3)$, $V_{\epsh}=\mathcal{D}(\mathcal{A}_{\epsh}^{1/2})=H^1_{\Gamma_{\rm D}}(\Omega;\R^3)$,  $H=L^2(\Omega \times Y;\R^3)$, $H_0=\{0\}^2 \times  L^2(\omega)$, $V=\{0\}^2 \times \mathcal{D}((\mathcal{A}_{\delta}^{\funcB, {\rm hom}})^{1/2})=\{0 \}^2 \times H_{\gamma_{\rm D}}^2(\omega)$. The space $H_{\epsh}$ is equipped with the $L^2$ inner product with weight $\rho^h$, while the space $H$ is equipped with the $L^2$ inner product with weight $\rho.$ 

In accordance with the abstract approach of Section \ref{apsechyp}, for $\vect v \in V_{\epsh}$ we set $\|\vect v\|_{V_{\epsh}}:=\|(h^{-2}\mathcal{A}_{\epsh}+\RRR \mathcal{I} \BBB)^{1/2} \vect v\|_{L^2}$ and, similarly, for  $\vect v \in V$ we set $\|\vect v\|_V:=\|(\mathcal{A}_{\delta}^{\funcB, {\rm hom}}+\RRR \mathcal{I} \BBB)^{1/2} \vect v\|_{L^2}$. Furthermore, the  convergence $\xrightharpoonup{H_{\epsh}}$ is given by two-scale convergence. Next, for $\vect f \in \R^3$, we define the vectors $\vect f_v:=(0,0, f_3)^\top$, $\vect f_h:=( \vect f_*,0)^\top$. 
We apply the estimate \eqref{ivan55} to the case of the loads $\vect f^{\epsh}_v$ and initial conditions $\vect u_0^{\epsh}$, $\vect u_1^{\epsh}$ and the estimate \eqref{ivan992} to the case of the loads $\vect f^{\epsh}_h$ and zero initial conditions. This yields
\begin{equation}
\begin{aligned} 
%%\nonumber& &
&\left\|(h^{-2}\mathcal{A}_{\epsh} +\mathcal{I} )^{1/2}\vect u^{\epsh}\right\|_{L^\infty([0,T];H_{\epsh})}+\|\partial_t \vect u^{\epsh} \|_{L^{\infty} ([0,T];H_{\epsh})}\leq
\\[0.3em]
%&\hspace{+5ex} \nonumber 
&\hspace{+5ex}Ce^\top\Big(\left\|(h^{-2}\mathcal{A}_{\epsh}+\mathcal{I} )^{1/2}  \vect u_0^{\epsh}\right\|_{H_{\epsh}}+\|\vect u^{\epsh}_1 \|_{H_{\epsh}} +\|\vect f_v^{\epsh}\|_{L^1([0,T];H_{\epsh})} \\[0.3em]
%&\hspace{+5ex} 
&\hspace{+8ex}+\left\|(h^{-2}\mathcal{A}_{\epsh} +\mathcal{I})^{-1/2} \vect f_h^{\epsh}(0)\right\|_{H_{\epsh}} +  
 \left\|(h^{-2}\mathcal{A}_{\epsh} +\mathcal{I})^{-1/2} \partial_t \vect  f_h^{\epsh}\right\|_{L^1([0,T];H_{\epsh})}
\Big). 
\end{aligned}
\label{ivan968}
\end{equation} 
In order to obtain the boundedness of the last two terms in \eqref{ivan968}, notice that for $\vect l^{\epsh} \in V_{\epsh}^*$ one has 
\begin{equation} \label{ivan966} 
 \left\|(h^{-2}\mathcal{A}_{\epsh} +\mathcal{I})^{-1/2} \vect l^{\epsh}\right\|^2_{H_{\epsh}}=h^{-2}a_{\epsh}(\vect s^{\epsh}, \vect s^{\epsh})+(\vect s^{\epsh}, \vect s^{\epsh}), 
 \end{equation} 
where $\vect s^{\epsh} \in \mathcal{D}(\mathcal{A}^{1/2}_{\epsh})$ is the solution of the problem 
\begin{equation} \label{ivan9667} 
 h^{-2}a_{\epsh}({\vect s}^{\epsh}, \vect v)+({\vect s}^{\epsh}, \vect v)_{H_{\epsh}}= {\vect l}^{\epsh}(\vect v), \quad \forall \vect v \in V_{\epsh}.  
  \end{equation} 
Combining the result of Proposition \ref{propdinsi1}\,(1) with \eqref{ivan966} and \eqref{ivan9667}, we obtain the existence of $C>0$, independent of $h,$ such that
%% for $\vect{l}^{\epsh} \in H_{\epsh}$
 \begin{equation} \label{ivan967} 
 \left\|(h^{-2}\mathcal{A}_{\epsh} +\RRR \mathcal{I} \BBB)^{-1/2} {\vect l}^{\epsh}\right\|^2_{H_{\epsh}} \leq C\|\pi_h {\vect l}^{\epsh}\|^2_{H_{\eps}},\qquad  \vect{l}^{\epsh} \in H_{\epsh}.
\end{equation} 
 Taking into account \eqref{loadstimebdd1} and \eqref{ivan970}, this implies the stated boundedness property.  Also, a consequence of \eqref{loadstimebdd1} and \eqref{ivan970}, we have
\begin{equation}
\pi_h \vect f_h^{\epsh} \xrightarrow{t,\infty,{\rm dr}-2} \vect f_h,\quad \pi_h \partial_t \vect f^{\epsh}\xrightharpoonup{t,{\rm dr}-2}\partial_t \vect f_h. 
\label{f_convergence}
\end{equation}  
From \eqref{ivan968} and Corollary \ref{kornthincor} we conclude  that $\pi_{1/h}\vect u^{\epsh}$ is bounded in $L^{\infty}([0,T];V_{\epsh})$ and $\partial_t \vect u^{\epsh}$ is bounded in $L^{\infty}([0,T];H_{\epsh}),$ and hence there exists $\vect u_l \in L^{\infty}([0,T];V)$, $\partial_t \vect u_l \in L^{\infty}([0,T];H)$ such that
\[
\pi_{1/h} \vect u^{\epsh} \xrightharpoonup{t,\infty,{\rm dr}-2} \vect u_l,\qquad\partial_t \vect u^{\epsh} \xrightharpoonup{t,\infty,{\rm dr}-2} \partial_t \vect u_l.
\] 
As in Section \ref{apsechyp}, we use the notation $\Vec{\vect u}^{\epsh}:=(\vect u^{\epsh},\partial_t \vect u^{\epsh})$. Similarly, we introduce $\Vec{\vect u}_0^{\epsh}$, $\Vec{\vect u}_0$, $\Vec{\vect u}_l$, $\Vec{\vect u}$, as well as 
\[
\pi_{1/h}\Vec{\vect u}^{\epsh}:=\bigl(\pi_{1/h} \vect u^{\epsh},\partial_t \vect u^{\epsh}\bigr)^\top,\quad \Vec{\vect f}^{\epsh}:=\bigl(0,0,0,(\vect f^{\epsh})^\top\bigr)^\top,\quad \Vec{\vect f}_v^{\epsh}:=\bigl(0,0,0, (\vect f_v^{\epsh})^\top\bigr)^\top,\quad \Vec{\vect f}_h^{\epsh}:=\bigl(0,0,0,\vect f_h^{\epsh}\bigr)^\top.
\] 
We then follow the proof of Theorem \ref{thmivan551} or Theorem \ref{thmivan553}. On the one hand, for every $\lambda>1,$ we have 
 $$
 \pi_{1/h}\mathcal{L}(\Vec{\vect u}^{\epsh})(\lambda)=\mathcal{L}(\pi_{1/h}\Vec{\vect u}^{\epsh})(\lambda) \RRR \drtwoscale \BBB \mathcal{L}(\Vec{\vect u}_l)(\lambda)\quad {\rm as}\quad h\to0,
 $$ 
 where $\mathcal{L}$ denotes the Laplace transform. On the other hand, by combining
 \begin{equation*} 
 \mathcal{L}(\Vec{\vect u}^{\epsh})(\lambda)=(\mathbb{A}_{\epsh}+\lambda\mathbb{I})^{-1}\mathcal{L}(\Vec{\vect f}^{\epsh})(\lambda)+(\mathbb{A}_{\epsh}+\lambda\mathbb{I})^{-1} \Vec{\vect u}^{\eps}_0 \qquad \forall\lambda>1,	
 \end{equation*} 
the representation \eqref{ivan103}, and Theorem \ref{thmivan111}, we obtain
$$
 \pi_{1/h}\mathcal{L}(\Vec{\vect u}^{\epsh})(\lambda) =  \mathcal{L}(\pi_{1/h}\Vec{\vect u}^{\epsh})(\lambda)\RRR \drtwoscale \BBB\mathcal{L} (\Vec{\vect u})(\lambda) \qquad \forall \lambda>1,  
$$
where $\Vec{\vect u}=({\funcA}_1-x_3 \partial_1 \funcB+\mathring{u}_1,{\funcA}_2-x_3 \partial_2 \funcB+\mathring{u}_2, \funcB,0,0,\partial_t \funcB )^\top$, with the functions 
$\vect{\funcA}$, $\funcB$, $\mathring{\vect u}$ being the solutions of the equations \eqref{limeq11}--\eqref{limeq13} for the loads $\vect f$. It follows that $\Vec{\vect u}_l=\Vec{\vect u}$. 

The existence and uniqueness of the solution of the limit problem follows from Theorem \ref{existenceH} and Theorem \ref{existenceV*}. Note that one can split the limit problem into two: the one with initial conditions $ u_{0,3},$ $S_1P_{\delta,\infty} u_{1,3}$ and out-of-plane loads, given by the part of $\mathcal{F}_{\delta}(\vect f)$ depending on $\vect f_3$ (where we apply Theorem \ref{existenceH}), and the one with zero initial conditions and in-plane loads, given by the part of $\mathcal{F}_{\delta}(\vect f)$ depending on $\vect f_{\!*} $ (where we apply Theorem \ref{existenceV*}.) The last claim of the theorem follows by combining Theorem \ref{thmivan552} applied to initial conditions $\vect u_0^{\epsh}$, $\vect u_1^{\epsh}$ and loads $\vect f^{\epsh}_v$ and the second claim of Theorem \ref{thmivan553} applied to initial conditions equal to zero and loads $\vect f^{\epsh}_h$ (using the resolvent compactness and convergence proved in Proposition \ref{propdinsi1} and Theorem \ref{thmivan111}). The conditions \eqref{ivan561} follow by applying \eqref{ivan967} to $\vect{l}^{\epsh}=\vect f_h^{\epsh}(0)$ and $\vect{l}^{\epsh}(t)=\partial_t\vect f_h^{\epsh}(t)$  and integrating over the interval $[0,T]$. 

\end{proof} 

\vskip 0.2cm

\noindent{\bf B. Proof of Corollary \ref{kordinsi1}}
\begin{proof} 
The proof follows from the first part of Theorem \ref{thmivan553} for the weak convergence and from the second part of the same theorem for the strong two-scale convergence. We will just briefly outline the proof of the weak convergence.  From \eqref{ivan992} we obtain the estimate
\begin{eqnarray*} 
& &\left\|(h^{-2}\mathcal{A}_{\epsh} +\RRR \mathcal{I} \BBB)^{1/2}  \vect u^{\epsh}\right\|_{L^\infty([0,T];H_{\epsh})}+\|\partial_t \vect u^{\epsh} \|_{L^{\infty} ([0,T];H_{\epsh})}  \nonumber \\ & & \nonumber \hspace{+5ex}  \leq Ce^\top\Big(\left\|(h^{-2}\mathcal{A}_{\epsh} +\RRR \mathcal{I} \BBB)^{-1/2} \mathcal{G}^{\epsh}( \vect g^{\epsh})(0)\right\|_{H_{\epsh}}+  
\left\|(h^{-2}\mathcal{A}_{\epsh} +\RRR \mathcal{I} \BBB)^{-1/2}  \partial_t \mathcal{G}(\vect g^{\epsh})\right\|_{L^1([0,T];H_{\epsh})}
\Big). 
\end{eqnarray*} 
Similarly to the argument of Section A above (see (\ref{f_convergence})), we have 
\begin{equation*}
\pi_h \vect g^{\epsh} \xrightarrow{t,\infty,{\rm dr}-2} \vect g,\quad \pi_h \partial_t \vect g^{\epsh}\xrightharpoonup{t,{\rm dr}-2}\partial_t \vect g. 
\end{equation*}
Furthermore, using \eqref{ivan966} and \eqref{ivan9667}
 %%from the proof of Theorem \ref{connn1} 
 we infer by Theorem \ref{thmextension}, Remark \ref{remivan102}, and  Corollary \ref{kornthincor} that for $\vect l^{\epsh} \in L^2(\omega \times \{-1/2, 1/2\};\R^3)$ one has
 \begin{equation*}
\left\|(h^{-2}\mathcal{A}_{\epsh} +\RRR \mathcal{I} \BBB)^{-1/2} {\vect l}^{\epsh}\right\|^2_{H_{\epsh}} \leq \|\pi_h {\vect l}^{\epsh}\|^2_{L^2(\omega \times \{-1/2, 1/2\};\R^3)}.
\end{equation*} 
The remainder of the argument follows the proof of Theorem \ref{connn1}, using Remark \ref{remivan13}.  
\end{proof} 	

\vskip 0.2cm
\noindent{\bf C. Proof of Theorem \ref{connn11} and Theorem \ref{thmivan71}}

\begin{proof} 
The claims are established directly by applying  Theorem \ref{thmivan551}, Theorem \ref{thmivan552}, and the results of Section \ref{limreseq} concerning resolvent convergence. For example, in the case $\delta \in (0,\infty)$, $\mu_h=\epsh,$ $\tau=0$ we set $H_{\epsh}=L^2(\Omega;\R^3)$, $\mathcal{A}_{\eps}=\mathcal{A}_{\epsh}$, $\mathcal{A}=\mathcal{A}_{\delta,\infty}$, $H=L^2(\Omega \times Y;\R^3)$, $H_0=V_{\delta,\infty}(\Omega \times Y)$, and the convergence $\xrightharpoonup{H_{\epsh}}$ is the two-scale convergence.    
\end{proof} 		

\vskip 0.2cm

\noindent{\bf D. Proof of Theorem \ref{thmivan72}}
\begin{proof}
The argument follows the proof of Theorem \ref{connn1}. The first part of the statement, which concerns weak two-scale convergence, is proved separately, by using the Laplace transform, Proposition \ref{propcompactregime}, and Theorem \ref{thmivan113} while separating out-of-plane and horizontal forces. The proof of the second part is carried out using Theorem \ref{thmivan552} and the second part of Theorem \ref{thmivan553}. We leave the details to the interested reader.  
\end{proof}

\section*{Acknowledgements} M.~Bu\v{z}an\v{c}i\'{c}, I.~Vel\v{c}i\'{c}, and J.~\v{Z}ubrini\'{c} were supported by Croatian Science Foundation under
Grant Agreement no. IP-2018-01-8904 (Homdirestroptcm). 
K.~Cherednichenko is grateful for the financial support of the Engineering and Physical Sciences Research Council: Grant
EP/L018802/2 "Mathematical foundations of metamaterials: homogenisation, dissipation and operator theory".
We are grateful to M.~Cherdantsev for the idea in the proof of Lemma \ref{lemmaivan100}.

\appendix

\section{Appendix}

\subsection{Decompositions of symmetrised scaled gradients}
\begin{theorem}[Korn inequalities]  \cite{Hor95}
	Let $p>1$, $\Omega\subset\R^n$ and suppose that $\Gamma\subset\partial\Omega$ has a positive measure. There exist positive constants $C_{\rm K}^1$, $C_{\rm K}^2$ and $C_{\rm K}^{\Gamma},$ which depend on $p$, $\Omega,$ and $\Gamma$ only, such that then the following inequalities hold for all $\mat{\psi}\in W^{1,p}(\Omega;\R^n):$
	\begin{align}
		\label{prvaaKorn}
		\|\mat{\psi}\|_{W^{1,p}}^p &\leq C_{\rm K}^1\left(\|\mat{\psi}\|_{L^p}^p + \|\sym\nabla \mat\psi\|_{L^p}^p\right),
		%%\,,
		%%\quad \forall \mat{\psi}\in W^{1,p}(\Omega;\R^n),
		\\[0.5em] 
		\label{Kornwithantisym} 
		\inf_{\vect A \in \R^{n\times n}_{\skeww},\vect b\in \R^n} \|\vect \psi-\vect Ax-\vect b\|^p_{W^{1,p}} &\leq  C_{\rm K}^2 \|\sym \nabla \vect \psi\|^p_{L^p},
		%% \qquad \forall \mat{\psi}\in W^{1,p}(\Omega;\R^n),
		\\[-0.1em]
		\label{Kornwithbc}
		\|\mat{\psi}\|_{W^{1,p}(\Omega;\R^n)}^p &\leq C_{\rm K}^{\Gamma}\left(\|\vect\psi\|^p_{L^p(\Gamma)}+\|\sym\nabla \mat\psi\|_{L^p}^p\right).
		%%\qquad \forall \mat{\psi}\in W^{1,p}(\Omega;\R^n).
	\end{align}
	%%where positive constants $C_{\rm K}^1$, $C_{\rm K}^2$ and $C_{\rm K}^{\Gamma}$ depend only on $p$, $\Omega$ and $\Gamma$.
\end{theorem}
Within this appendix, we will also use  the following version of Korn's inequality.
\begin{proposition}\label{propgriso} 
	Suppose that $\omega \subset \R^2$ has Lipschitz boundary. 
	Then for every $\vect \psi \in H^1(\omega;\R^2)$ one has
	$$ 
	\biggl\|\vect \psi-\fint_{\omega} \vect \psi\biggr\|_{L^2}\leq C\bigl( \|\sym \nabla \vect \psi\|_{L^2}+\dist (\psi,\mathcal{G})\bigr), $$
	where $C>0$ depends only on $\omega,$
	$$
	\mathcal{G}:=\bigl\{\nabla \phi: \phi \in H^1(\omega)\bigr\},
	$$
	and the distance is understood in the sense of the $L^2$ metric.
\end{proposition}
\begin{proof}
	The proof follows a standard contradiction argument. Suppose the claim is false, i.e., there exists a sequence $(\vect \psi^n)_{n \in \N} \subset H^1(\omega;\R^2)$ such that
	\begin{align}
		\int_{\omega}\vect \psi^n&=0,
		\label{meanpsi}\\[0.4em]
		\|\vect \psi^n\|_{L^2}&\geq n \bigl(\|\sym \nabla \vect \psi^n\|_{L^2}+\dist(\vect \psi^n,\mathcal{G}) \bigr)\qquad \forall n\in{\mathbb N}. 
		\label{griso100}
	\end{align}
	Without loss of generality, 
	\begin{equation}
	\|\vect \psi^n\|_{L^2}=1.
	\label{norm_ass}
    \end{equation}
 and (\ref{griso100}) can be written as
	$$ 
	\|\sym \nabla \vect \psi^n\|_{L^2}+\dist(\vect \psi^n,\mathcal{G}) \leq n^{-1} \qquad \forall n\in{\mathbb N}. 
	$$ 
	By Korn's inequality, it follows that $\vect\psi^n \weak \vect \psi$, weakly in $H^1(\omega;\R^2)$. Combining this with \eqref{griso100}, we infer that $\sym \nabla \vect\psi=0$ and $\vect \psi \in \mathcal{G}$. From $\sym \nabla \vect\psi=0$ we obtain 
	$$ 
	\vect \psi=  (-ax_2,ax_1)\RRR^\top \BBB+\vect b, \quad a \in \R,\vect b \in \R^2.
	$$
	Together with (\ref{meanpsi}) 
	 %%$\int_{\omega}\vect \psi=0$ 
	 and $\vect \psi \in \mathcal{G}$ this implies $\vect \psi=0$, which contradicts (\ref{norm_ass}).	 
	%%% $\| \vect \psi\|_{L^2}=1$. 
\end{proof}
\begin{theorem}[Griso's decomposition, \cite{Gri05}]\label{aux:thm.griso}
	Let $ \omega\subset\R^2$ with Lipschitz boundary and $\mat{\psi}\in H^1(\Omega;\R^3).$ Then one has
	\begin{equation}\label{griso1}
		\mat{\psi} = \hat{\mat{\psi}}(x') + \vect{r}(x')\wedge x_3\vect{e}_3 + \bar{\mat{\psi}}(x)
		= \left\{   \begin{array}{l}
			\hat{ \psi}_1(x') + \ r_2(x')x_3 + \bar{ \psi}_1(x),\\[0.35em]
			\hat{ \psi}_2(x') -  r_1(x')x_3 + \bar{ \psi}_2(x),\\[0.35em]
			\hat{ \psi}_3(x') + \bar{ \psi}_3(x),
		\end{array}
		\right.
	\end{equation}
	where
	\begin{equation} \label{griso2}
		\hat{\mat\psi}(x') = \int_I \mat\psi(x',x_3)\dd x_3\,,\quad \vect{r}(x') 
		= \frac32\int_Ix_3\vect{e}_3\wedge\mat\psi(x',x_3)\dd x_3,
	\end{equation}
	 the following inequality holds  for arbitrary $h>0,$ with a constant $C>0$ that depends on $\omega$ only:
	\begin{equation}\label{prvaKorn}
		\bigl\|\sym\nabla_h(\hat{\mat\psi} + \vect{r}\wedge x_3\vect{e}_3)\bigr\|_{L^2}^2 
		+\bigl\|\nabla_h\bar{\mat\psi}\bigr\|_{L^2(\Omega;\R^{3\times 3})}^2 + h^{-2}\|\RRR \bar{\mat\psi}\BBB\|_{L^2}^2
		\leq C\bigl\|\sym\nabla_h\mat\psi\bigr\|_{L^2}^2.
	\end{equation}
\end{theorem}
\begin{remark}
	Notice that
	\begin{equation}
		\begin{aligned}
			%%\begin{eqnarray}
			%%& &
			\bigl\| \sym \nabla_h (\hat{\vect \psi}+\vect r\wedge x_3 e_3)\bigr\|^2_{L^2}&=
			%% \\
			%%& & \nonumber
			\bigl\|\sym \nabla_{\hat{x}} (\hat{ \psi}_1, \hat{ \psi}_2)\RRR^\top \BBB\bigr\|^2_{L^2}+\frac{1}{12}\bigl\|\sym \nabla_{\hat{x}} ( r_2,-  r_1)\RRR^\top \BBB\bigr\|^2_{L^2}\\[0.4em]
			%%\\ \nonumber & &
			&+h^{-2}\bigl\|\partial_1 (h\hat{ \psi}_3)+ r_2\bigr\|^2_{L^2}+h^{-2}\bigl\|\partial_2 (h\hat{ \psi}_3)- r_1\bigr\|^2_{L^2}.
			%%\end{eqnarray}
		\end{aligned}
		\label{teichmann}
	\end{equation}
	Thus from Korn's inequality it follows
	
	\begin{equation}
		\label{lukas1}
		\begin{aligned}
			h^2\bigl\|\pi_{1/h}\hat{\vect \psi} \bigr\|^2_{H^1}&+\bigl\|( r_1, r_2)\RRR^\top \BBB\bigr\|^2_{H^1}+h^{-2}\bigl\|\partial_1 (h\hat{ \psi}_3)+ r_2\bigr\|^2_{L^2}+h^{-2}\bigl\|\partial_2 (h\hat{ \psi}_3)- r_1\bigr\|^2_{L^2}\\[0.35em] 
			&\leq C \left(\bigl\| \sym \nabla_h (\hat{\vect \psi}+\vect r\wedge x_3 \vect e_3)\bigr\|^2_{L^2}+\|\vect r\|^2_{L^2}+h^2\|\pi_{1/h}\hat{\vect \psi} \|^2_{L^2}\right)\\[0.35em]
			&\leq C \left(\bigl\| \sym \nabla_h (\hat{\vect \psi}+\vect r\wedge x_3 \vect e_3)\bigr\|^2_{L^2}+h^2\|\pi_{1/h}{\vect \psi} \|^2_{L^2}\right).
		\end{aligned}
	\end{equation}
	
\end{remark}

The following corollary is the direct consequence of \eqref{prvaKorn}, \eqref{lukas1}, and Korn inequalities.
\begin{corollary}[Korn's inequality for thin domains] 
	\label{kornthincor}
 Suppose that $\omega \subset \R^2$ is such that $\gamma \subset \partial \omega$ has positive measure.
	Then there exist constants  $C_{\rm T}, C_{\rm T}^{\gamma}>0$ that depend on  $\omega$ and $\gamma$ only, such that the following inequalities hold for all $\mat{\psi}\in H^1(\Omega;\R^3):$
	\begin{align}
		\|\pi_{1/h}{\vect \psi} \|^2_{H^1} &\leq C_{\rm T}\left(\|\pi_{1/h}{\vect \psi} \|^2_{L^2} + h^{-2}\|\sym\nabla_h \mat\psi\|_{L^2}^2\right),\nonumber
		%%\,,	\quad \forall \mat{\psi}\in H^1(\Omega;\R^3);\,
		\\[0.35em]
		\|\pi_{1/h}{\vect \psi} \|^2_{H^1} &\leq C_{\rm T}^{\gamma}\RRR \left(\|\pi_{1/h}{\vect \psi} \|^2_{L^2(\Gamma;\R^3)}+h^{-2}\left\|\sym\nabla_h \mat\psi\right\|_{L^2}^2\right).\BBB\label{CTgamma}
		%%		\,,\quad \forall \mat{\psi}\in H^1(\Omega;\R^3).
	\end{align}
	%%%where positive constants $C_{\rm T}$ and $C_{\rm T}^{\gamma}$ depend only on  $\omega$ and $\Gamma.$
\end{corollary}
\begin{remark} \label{remivan} 
	If it is known that the components $\psi_{\alpha},$ $\alpha=1,2,$ are even in $x_3$ and $ \psi_3$ is odd in $x_3$, then additionally $\vect r=0$, $\hat { \psi}_3=0$. 
\end{remark}
\begin{remark}\label{remivan33} 
	If a sequence of deformations $(\vect \psi^h)_{h>0}$ is such that $(\vect \psi^h)_{h>0}$ and $(\sym \nabla_h \vect \psi^h)_{h>0}$ are bounded in $L^2$  and if $\vect r^h$, $\hat{\vect{\psi}}^h$ and $\bar{\vect{\psi}}^h$ are the terms in the decomposition  \eqref{griso1}, then the relations \eqref{prvaKorn}--\eqref{lukas1} imply that $\vect r^h \,\overset{H^1}\weak\, 0$  and $h \hat{\vect \psi}^h \, \overset{H^1}\to\, 0$.  
\end{remark} 

The following lemma provides additional information on the weak limit of sequences with bounded symmetrised scaled gradients and is proved in \RRR \cite[Lemma A.4]{BukalVel}\BBB as a direct consequence of Griso's decomposition. 
\begin{lemma}\label{app:lem.limsup}
	Consider a bounded set $\omega\subset \R^2$  with Lipschitz boundary. Suppose that  a sequence $(\mat{\psi}^h)_{h > 0}\subset H_{\gamma_{\rm D}}^1(\Omega;\R^3)$  is such that
	\begin{equation*}
		\limsup_{n\to\infty}\bigl\|\sym \nabla_h\mat{\psi}^h\bigr\|_{L^2} < \infty.
	\end{equation*}
	Then there exists a subsequence (still labelled by $h > 0$) for which
	\begin{equation*}
		\mat \psi^h = \left(\begin{array}{c} \funcA_1 -x_3 \partial_1 \funcB  \\[0.25em] \funcA_2 -x_3 \partial_2 \funcB \\[0.25em] h^{-1}\funcB \end{array} \right)+\tilde{\mat \psi}^h,
		% \textrm{ i.e.,}
	\end{equation*}
in particular
	\begin{equation*}
		\sym\nabla_h\mat\psi^h = \imath\bigl(-x_3\nabla_{\hat{x}}^2 \funcB + \sym\nabla_{\hat{x}} \vect{\funcA}\bigr) + \sym\nabla_h\tilde{\mat\psi}^h,
	\end{equation*}
	for some 
	$\funcB\in H_{\gamma_{\rm D}}^2(\omega)$, $\vect{\funcA}\in H_{\gamma_{\rm D}}^1(\omega;\R^2)$, and the sequence $(\tilde{\mat\psi}^h)_{h>0}\subset H_{\gamma_{\rm D}}^1(\Omega;\R^3)$
	satisfies $h\pi_{1/h}\tilde{\vect \psi}^h \overset{L^2}\to\, 0$. 
	%Furthermore,
	%\begin{equation*}
	%\|v\|_{L^2}^2 + \|\vect{w}\|_{L^2}^2 \leq %\limsup_{n\to\infty}\|(\vect\psi_1^h,\vect \psi_2^h,h_n\vect \psi_3^h)\|_{L^2}^2\,.
	%\end{equation*}
\end{lemma}
\begin{remark}\label{remivan1111}
	It can be easily seen that $\vect \funcA$ is the weak limit of $(\hat{ \psi}_1^h,\hat{ \psi}_2^h)\RRR^\top \BBB$ in $H^1(\omega;\R^2)$ 
	and $\funcB$ is the weak limit on $h\hat{ \psi}^h_3$ in $H^1(\omega).$ (More precisely, in this case $(-r^h_2,  r^h_1)\RRR^\top \BBB \,\overset{H^1}\weak\, x_3 \nabla \funcB$, where $\vect r^h$ and $\hat{\vect \psi}^h$ come from the decomposition \eqref{griso1} of $\vect \psi^h$.)
\end{remark}
We will now establish a decomposition that can be viewed as a consequence of Griso's decomposition. Note that another proof of the first part of Lemma \ref{keydecompose} is given in \cite{casado}. The proof provided here uses the strategy of \cite{Vel14a}. 
\begin{lemma} \label{keydecompose}
	Suppose that $\omega \subset \R^2$ is a connected set with $C^{1,1}$ boundary and  $\vect\psi \in H^1(\Omega;\R^3)$. 
	\begin{enumerate}
		\item  There exist $\vect a\in \R^3$, $\vect B\in \R^{3\times 3}_{\skeww}$,  $v \in H^2(\omega)$, $\tilde{\mat\psi} \in H^1(\Omega;\R^3)$ such that 
		\begin{equation} \label{moddek} 
			\vect \psi=\vect a+\vect B \left(\begin{array}{c}x_1\\[0.25em]x_2\\[0.25em]hx_3\end{array}\right)+\left(\begin{array}{c} -x_3 \partial_1 v  \\[0.25em] 
				-x_3 \partial_2 v \\[0.25em] 
				h^{-1}v \end{array} \right)+\tilde{\mat \psi},
		\end{equation}
		and the estimate 
		\begin{equation*}
			\| v\|^2_{H^2}+\|\tilde{\vect \psi}\|_{L^2}^2+ \|\nabla_h \tilde{\mat \psi}\|^2_{L^2} \leq C(\omega) \| \sym \nabla_h \vect \psi\|^2_{L^2}      
		\end{equation*}
	holds for some $C(\omega)>0.$
		\item If $\vect \psi \in H^1 (\Omega, \R^3)$, $\vect \psi=0$ on $\partial \omega \times I$, then in (\ref{moddek}) one can take $\vect a=\vect B=0$. In addition, $v,$ $\tilde{\vect\psi}$ can be chosen so that $v=\nabla v=0$ on $\partial \omega$ and $\tilde {\vect \psi}=0$ on $\partial \omega \times I$. 
		\item 	If a sequence $(\vect\psi^h)_{h>0} \subset H^1(\Omega;\R^3)$ is such that
		\begin{equation*}
			%%\begin{eqnarray*}
			%%&& 
			h\pi_{1/h}\vect \psi^{h} \,\overset{L^2}\to\, 0,\qquad\qquad 
			%% \\ 
			%%&& 
		\limsup_{n \to \infty}\bigl\|\sym \nabla_h \vect \psi^h\bigr\|_{L^2}<\infty,
			%%\end{eqnarray*}
		\end{equation*}
		then there exist sequences $(\varphi^h)_{h>0} \subset H^2(\omega)$, $(\tilde{\vect \psi}^h)_{h>0} \subset H^1(\Omega ;\R^3)$
		such that
		$$ \sym \nabla_h\vect \psi^h=-x_3\iota( \nabla_{\hat{x}}^2 \varphi^h)+\sym \nabla_h \tilde{\vect \psi}^h+o^h,$$
		where $(o^h)_{h>0}\subset L^2(\Omega;\R^{3 \times 3})$ is such that $o^h\, \overset{L^2} \to\, 0$, and the following properties hold:
		\begin{equation*}
			%	\begin{eqnarray*}
			%&& 
			\lim_{h \to 0} \left(\|\varphi^h\|_{H^1}+\|\tilde{\vect \psi}^h\|_{L^2} \right)=0,\qquad\qquad
			%&& 
			\limsup_{h \to 0} \left( \| \varphi^h \|_{H^2}+\|\nabla_h \tilde{\vect \psi}^h\|_{L^2} \right)\leq C \limsup_{n \to \infty}\bigl\|\sym \nabla_h \vect \psi^h\bigr\|_{L^2},
			%	\end{eqnarray*}
		\end{equation*}
		where $C>0$ depends on $\omega$ only. 
		Moreover, one has
		$$ \psi^h_3 =h^{-1}\varphi^h+w^h+\tilde{ \psi}_3^h, $$
		where $w^h \in  H^1(\omega)$ with
		$$ \|w^h\|_{H^1} \leq C\left(\bigl\|\sym \nabla_h\vect \psi^h\bigr\|_{L^2}+\bigl\|h\pi_{1/h}\vect \psi^h\bigr\|_{L^2}\right), $$
		for some $C>0$ that depends on $\omega$ only. 
	\end{enumerate}
\end{lemma}
\begin{proof}
	We decompose $\vect \psi$ as in \eqref{griso1}. From \eqref{Kornwithantisym}, \eqref{prvaKorn} and Proposition \ref{propgriso} we conclude that there exist $c\in \R$, $\vect d \in \R^2$ such that
	\begin{align*}
		\biggl\|\vect r-\fint_{\omega} \vect r\biggr\|^2_{H^1}&+\bigl\|(\hat { \psi}_1, \hat { \psi}_2)\RRR^\top \BBB-(-cx_2,cx_1)^\top-\vect d\bigr\|_{H^1}+h^{-2}\bigl\|\bar{\mat\psi}^h\bigr\|_{L^2}^2+\bigl\|\nabla_h\bar{\mat\psi}\bigr\|_{L^2}^2\\ 
		&
		%\hspace{+10ex}
		+h^{-2}\bigl\|\partial_1 (h\hat{ \psi}_3)+ r_2\bigr\|^2_{L^2}+h^{-2}\bigl\|\partial_2 (h\hat{ \psi}_3)- r_1\bigr\|^2_{L^2}
		\leq C\bigl\|\sym\nabla_h\mat\psi\bigr\|_{L^2}^2. 
	\end{align*}
	We do the regularization of $\vect r$, i.e.,
	we  look for the solution of the problem
	\begin{equation} \label{elias1}
		\min_{\varphi\in H^1(\omega) , \ \int_{\omega} \varphi=h\int_{\omega} \hat{\vect \psi_3} }\int_{\omega}\bigl|\nabla_{\hat{x}} \varphi+( r_2,- r_1)\RRR^\top \BBB\bigr|^2\,dx'.
	\end{equation}
	The Euler-Lagrange equation and the associated boundary conditions for the problem (\ref{elias1}) read
	\begin{equation*}
\begin{aligned}
			-\Delta'\varphi=\nabla_{\hat{x}}\cdot ( r_2,- r_1)\RRR^\top \BBB\ \ \text{in}\  \omega,\quad\qquad
			\partial_{\nu}\varphi=-(r_2,-r_1)\RRR^\top \BBB\cdot \nu\ \ \text{on}\ \partial \omega.
		\end{aligned} 
	\end{equation*}
	Since $\nabla_{\hat{x}}\cdot ( r_2,- r_1)\in L^2$, by
	standard  regularity estimates we obtain the inclusion  $\varphi\in H^2(\omega)$ and the estimate
	\begin{equation}
	\biggl\|\varphi-h\fint_{\Omega} \hat{ \psi}_3\biggr\|_{H^2(\omega)}\leq C(\omega)  \|\vect r\|_{H^1(\omega;\R^2)},
	\label{phir}
	\end{equation}
	for which we require the $C^{1,1}$ regularity of $\partial \omega$. In particular, one has
	\begin{equation*}
		\left\|\varphi+\fint_{\omega}  r_2 x_1-\fint_{\omega} r_1 x_2-h\fint_{\omega} \hat{ \psi}_3\right\|_{H^2} \leq C \left\| \vect r-\fint_{\omega} \vect r\right\|_{H^1}. 
	\end{equation*} 
	Furthermore, from (\ref{elias1}) we have the following inequalities:
	\begin{align*}
		\|\partial_1 \varphi+  r_2 \|^2_{L^2}+\|\partial_2 \varphi-  r_1 \|^2_{L^2(\omega)}  &\leq \|\partial_1 (h\hat{\psi}_3)+ r_2 \|^2_{L^2}+\|\partial_2 (h\hat{ \psi}_3)- r_1 \|^2_{L^2}  
		\\[0.2em]
		\left\| \nabla_{\hat{x}} \bigl(\hat{ \psi}_3-h^{-1}\varphi\bigr)\right\|^2_{L^2} &\leq h^{-2}\|\partial_1 (h\hat{ \psi}_3)+ r_2\|^2_{L^2}+h^{-2}\|\partial_1 \varphi+  r_2\|^2_{L^2} 
		\\[0.3em]
		&\nonumber
		+h^{-2}\|\partial_2 (h\hat{  \psi}_3)- r_1\|^2_{L^2}+h^{-2}\|\partial_2 \varphi-  r_1\|^2_{L^2} 
			\\[0.3em]
		&\nonumber
		\leq 2(  \|\partial_1 (h\hat{\psi}_3)+ r_2 \|^2_{L^2}+h^{-2}\|\partial_2 (h\hat{ \psi}_3)-  r_1\|^2_{L^2}).
	\end{align*}
	The claim follows by taking 
	\begin{align*}
		\vect a&=\biggl(d_1,d_2, \fint_{\omega} \hat{ \psi}_3\biggr)^\top, \quad \vect B=\left(\begin{array}{ccc} 0 &-c &h^{-1}
			{\fint  r_2} \\[0.3em]
			c & 0 & -h^{-1}{\fint  r_1}
			\\[0.3em]
			-h^{-1}{\fint  r_2}& h^{-1}\fint  r_1 & 0\end{array} \right), \\[0.4em] 
	&v=\varphi+\fint_{\omega}  r_2 x_1-\fint_{\omega}  r_1 x_2-h\fint_{\omega} \hat{\vect \psi}_3,\\[0.4em]
	&\tilde{\vect \psi}=\bar{ \psi}+\left(\hat { \psi}_1,\hat { \psi}_2,\hat{ \psi}_3-h^{-1}\varphi\right)^\top+\left(x_3( r_2+\partial_1 \varphi),x_3(- r_1+\partial_2 \varphi),0\right)^\top- (-cx_2,cx_1,0)^\top-(\vect d,0)^\top.  
	\end{align*}
	This finishes the proof of part 1 of the lemma. To prove part 2, we only need to note  that if $\vect \psi=0$ on $\partial \omega$, then $\vect r=0$, $\hat {\vect \psi}=0$ on $\partial \omega$. 
	Combining this \eqref{teichmann}, we infer
	\begin{align*}
	 \|\vect r\|^2_{H^1}+\|(\hat { \psi}_1, \hat { \psi}_2)\RRR^\top \BBB\|_{H^1}&+ h^{-2}\|\bar{\mat\psi}^h\|_{L^2(\Omega;\R^3)}^2+ \|\nabla_h\bar{\mat\psi}\|_{L^2}^2\\[0.3em] 
		&\hspace{+10ex}+h^{-2}\|\partial_1 (h\hat{ \psi}_3)+ r_2\|^2_{L^2}+h^{-2}\|\partial_2 (h\hat{\psi}_3)- r_1\|^2_{L^2}
		\leq C\|\sym\nabla_h\mat\psi\|_{L^2}^2. 
	\end{align*}
	Furthermore, due to the condition $\vect\psi=0$ on $\partial\omega,$ in the ``regularisation'' of $({ r_2}, -{ r}_1)\RRR^\top \BBB$ provided by the variational problem \eqref{elias1} we can minimise over $\varphi \in H^1_0(\omega)$ and we immediately obtain that 
	\[
	\|\varphi\|_{H^2}\leq C(\omega)  \|\vect r\|_{H^1},
	\]
	which replaces (\ref{phir}).
	
	 Part 3 is proved in \cite{Vel14a}; alternatively, one can follow the argument used for part 1, taking into account \eqref{lukas1} and noting that $h \pi_{1/h}\vect \psi^h \to 0 $ in $L^2(\Omega;\R^3)$ implies the convergence $\vect r^h \weak 0$ in $H^1(\omega;\R^2)$, $h \pi_{1/h}\hat{\vect \psi}^h \weak 0 $ in $H^1(\omega;\R^3),$ where $r^h$ and $\hat{\vect \psi}^h$ are from the decomposition \eqref{griso1} applied to ${\vect \psi}^h.$ (Note that within the described argument here one can set to zero
	 %%$\vect a^h=\vect B^h=0$ for 
	 the vectors $\vect a^h,$ $\vect B^h$ in the decomposition (\ref{moddek}) for $\vect\psi^h.$)

\end{proof}
\begin{remark} \label{remivan10} 
	Following Remark \ref{remivan}, we note that if $ \psi_{\alpha},$ $\alpha=1,2$, are even in the variable and $ \psi_3$ is odd in $x_3$ variable, then on has $v=0$, $ a_3=0$, $ B_{13}= B_{23}=0$ in \eqref{moddek}.  Moreover, the estimate 
	$$
	\| \psi_3\|_{L^2}=\|\tilde{\psi}_3\|_{L^2}=\biggl\| \tilde { \psi}_3-\fint_I \tilde { \psi}_3\biggr\|_{L^2} \leq C\| \partial_3\tilde { \psi}_3\|_{L^2} \leq C h\|\sym \nabla_h \vect \psi\|_{L^2}.   
	$$
	holds with $C>0$ that depends on $\omega$ only. 
\end{remark}
\subsection{Two-scale convergence}
\label{twoscale_conv} 
In this chapter we assume that $\Omega \subset \R^n$, if not otherwise stated, is a bounded open set with Lipschitz boundary. 
As before, we use the notation $I=(-{1/2}, {1/2})$. 
For $x\in\R^n,$ we denote by $\hat{x}$ the first $n-1$ coordinates, thus $x=(\hat{x},x_n)$. Depending on the context, the unit cell is $Y=[0,1)^n$ or $Y=[0,1)^{n-1},$ while $\mathcal{Y}$ denotes the unit flat torus in $\R^n$ or $\R^{n-1},$ respectively.

\begin{definition}\RRR(Dimension-reduction two-scale convergence). \BBB
	\label{withouttdef}
	Let $ (u^\varepsilon)_{\varepsilon>0}$ be a bounded sequence in $L^2(\Omega)$. We say that $u^\varepsilon$ weakly two-scale converges to $u \in L^2(\Omega \times Y)$ with respect to $Y$ if (in settings where $Y=[0,1)^{n-1}$)
	$$
	\int\limits_{\Omega}  u^\varepsilon (x)  \phi\biggl(x,\frac{\hat{x}}{\varepsilon}\biggr)\,dx
	\,\xrightarrow\,
	\int\limits_{\Omega}\int\limits_Y  u(x, y) \phi\left(x,y\right)\,dydx\qquad \forall\phi \in C_{\rm c}^\infty\bigl(\Omega;C(\mathcal{Y})\bigr),
	$$ or (in settings where $Y=[0,1)^{n}$) 
	$$
	\int\limits_{\Omega}  u^\varepsilon (x)  \phi\biggl(x,\frac{x}{\varepsilon}\biggr)\,dx
	\,\xrightarrow\,
	\int\limits_{\Omega}\int\limits_Y  u(x, y) \phi\left(x,y\right)\,dydx\qquad \forall\phi \in C_{\rm c}^\infty\bigl(\Omega;C(\mathcal{Y})\bigr).
	$$ 
	%%for every $\phi \in C_{\rm c}^\infty(\Omega;C(\mathcal{Y}))$. 
	We write
	\[
	u^\varepsilon \drtwoscale  u(x,y).
	\]
	Furthermore, we say that $(u^\varepsilon)_{\eps>0}$ strongly two-scale converges to $ u \in L^2(\Omega \times Y)$ if
	\[
	\int\limits_{\Omega}  u^\varepsilon(x)  \phi^\varepsilon(x)\,dx
	\to
	\int\limits_{\Omega}\int\limits_Y  u(x, y) \phi\left(x,y\right)\,dydx,
	\]
	for every weakly two-scale convergent sequence $ \phi^\varepsilon(x) \drtwoscale  \phi(x,y)$. We write  	
	\[\RRR
	u^\varepsilon \strongdrtwoscale  u(x,y).
	\BBB\]
\end{definition}
The following theorem is given in \cite[Theorem 6.3.3]{Neukamm10}. 
\begin{theorem}
	\label{neukammresult}
	Let $\Omega=\omega\times I$, where $\omega \subset \R^2$ is bounded and has Lipschitz boundary, and
	let $\epsh>0$ be a sequence such that $\epsh \to 0$ as $h\to 0$ so that $\lim_{h\to 0}h/\epsh=\delta \in [0,\infty]$.  Let $(\vect u^\epsh)_{h>0}$ be a weakly convergent sequence in $H^1(\Omega;\R^3)$ with limit $\vect u$ and suppose that
	\begin{equation} \label{neukammuvjet}
		\limsup\limits_{h \rightarrow 0} ||\nabla_h \vect u^\epsh ||_{L^2(\Omega;\R^3)} < \infty.
	\end{equation}
	\begin{enumerate} 
		\item \begin{enumerate} \item  If $\delta \in ( 0, \infty)$ then there exists a function $\vect w \in L^2(\omega;H^1(I \times \mathcal{Y};\R^3))$ and a subsequence (not relabelled) such that
			\[
			\nabla_h \vect u^\epsh(x) \drtwoscale \bigl(\nabla_{\hat{x}} \vect u(\hat{x})\,|\, 0\bigr) + \widetilde{\nabla}_{\delta}\vect w(x,y).
			\]
			\item If $\delta \in (0,\infty)$ and in addition \eqref{neukammuvjet} we assume that
			$$ 
			\limsup\limits_{h \rightarrow 0} h^{-1}\bigl\|\vect u^\epsh \bigr\|_{L^2(\Omega;\R^3)} < \infty, 
			$$  
			then there exists a function $\vect w \in L^2(\omega;H^1(I \times \mathcal{Y};\R^3))$ and a subsequence (not relabeled) such that
			\[ 
			h^{-1}\vect u^\epsh(x) \drtwoscale \vect w(x,y), \quad
			\nabla_h \vect u^\epsh(x) \drtwoscale  \widetilde{\nabla}_{\delta}\vect w(x,y).
			\]
		\end{enumerate} 
		\item If $\delta=0$ then there exits 
		$\vect w \in L^2(\omega; H^1(\mathcal{Y};\R^3))$ and $\vect g \in L^2(\Omega \times Y;\R^3)$ such that
		$$\nabla_h \vect u^\epsh(x) \drtwoscale \bigl(\nabla_{\hat{x}}\vect u(\hat{x})\, |\,0\bigr) +\bigl(\nabla_y \vect w\,|\,\vect g\bigr). $$
		\item If $\delta=\infty$ then there exists 
		$\vect w \in L^2(\Omega; H^1(\mathcal{Y};\R^3))$, $\vect g \in L^2(\Omega;\R^3)$ such that
		$$
		\nabla_h \vect u^\epsh(x) \drtwoscale \bigl(\nabla_{\hat{x}} \vect u(\hat{x})\,|\, 0\bigr) +\bigl(\nabla_y \vect w\,|\,\vect g\bigr). 
		$$
	\end{enumerate}
\end{theorem}
\begin{comment}
A simple consequence of the above theorem is that for all sequences $\{u^h\}_{h>0}$ that satisfy \eqref{neukammuvjet} such that $u^h \weak 0$ in $H^1(\Omega)$, there exists a function $w \in L^2(\omega;H^1(I \times \mathcal{Y};\R^3))$ and a subsequence such that
\[
\simgrad_h u^h(x) \drtwoscale \sym\widetilde{\nabla}_{2,\gamma}w(x,y).
\]
\end{comment} 
We will need the following helpful lemma. 
\begin{lemma}
	\label{lemmadependenceony}
	\begin{enumerate} 
		\item Suppose that $(\varphi^{\varepsilon})_{\eps>0} \subset H^1(\Omega)$ be a bounded sequence in $L^2(\Omega)$ such that $\varphi^{\eps} \drtwoscale \varphi(x,y) \in L^2(\Omega\times \mathcal{Y})$. Suppose additionally that $\eps \varphi^{\eps} \to 0 $ strongly in $H^1(\Omega)$. Then $\varphi(x,y)$ depends on $x$ only. 
		\item Suppose $(\varphi^{\eps} )_{\eps>0} \subset H^2(\Omega)$ be a bounded sequence in $L^2(\Omega)$ such that $\varphi^{\eps} \drtwoscale \varphi(x,y) \in L^2(\Omega\times \mathcal{Y})$. Suppose additionally that $\eps^2 \varphi^{\eps} \to 0$ strongly  in $H^2(\Omega)$. Then $\varphi(x,y)$ depends on $x$ only. 
	\end{enumerate} 
\end{lemma}
\begin{proof}
	We write 
	$$ \varphi(x,y)=\sum_{k \in \Z^n} a_k(x)\exp  \bigl(2\pi{\rm i}(  k,y)\bigr), \qquad a_k \in L^2(\Omega,\C^n),\quad  \sum_{k\in \Z^n} \int |a_k(x)|^2 < \infty. $$
	We want to show that for $k \neq 0$ we have that $a_k=0$. We take an arbitrary $b\in C_0^{\infty} (\Omega)$ and $i \in \{1,\dots,n\}$ such that $k_i \neq 0$ and calculate
	\begin{align*}
		\int_{\Omega} a_k(x)b(x) dx  &=\int_{\Omega \times Y} \varphi(x,y) b(x) \exp\bigl(2\pi {\rm i}( k,y)\bigr)dx dy\\[0.35em] 
		&= \lim_{\eps\to 0} \int_{\Omega} \varphi^{\eps}(x) b(x) \exp\biggl(2\pi {\rm i}\biggl(  k,\frac{x}{\varepsilon}\biggr)\biggr)dx \\[0.35em] 
		&= \lim_{\eps \to 0} \int_{\Omega} \frac{\varepsilon}{2\pi {\rm i} k_i} \varphi^\eps(x) b(x) \partial_{x_i} \left(\exp\biggl(2\pi {\rm i}\biggl(  k,\frac{x}{\varepsilon}\biggr)\biggr)\right)dx\\[0.35em]  &=-\lim_{\eps \to 0} \left\{\int_{\Omega} \frac{\varepsilon}{2 \pi {\rm i}k_i} \partial_{x_i} \varphi^{\eps} (x) b(x) \exp\biggl(2\pi {\rm i}\biggl(  k,\frac{x}{\varepsilon}\biggr)\biggr)+\int_{\Omega} \frac{\varepsilon}{2 \pi {\rm i}k_i}  \varphi^{\eps } \partial_{x_i} b(x)  \exp\biggl(2 \pi {\rm i}\biggl( k,\frac{x}{\varepsilon}\biggr)\biggr) \right\}=0.   
	\end{align*}
	From this we infer that for $k \neq 0$, $a_k=0$, and the claim follows. The proof of the second claim is similar. 
\end{proof}
The following claim can be proved directly by integration by parts.  
\begin{lemma} \label{lemmatwoscalecompact} 
	
	\begin{enumerate} 
		\item 
  Let $(\varphi^{\varepsilon})_{\eps>0} \subset H^2(\Omega) $ be a bounded sequence. Suppose that 
		$\varphi^{\eps}  \to \varphi_0$ strongly in $L^2(\Omega)$ and $\nabla^2 \varphi^{\eps} \drtwoscale \vect \psi$, where $\vect \psi \in L^2(\Omega \times Y;\R^{n \times n}) $. Then there exists $\varphi_1 \in L^2(\Omega; H^2(\mathcal{Y})) $ such that 
		$$\nabla^2 \varphi^{\eps} \drtwoscale \nabla^2 \varphi_0(x)+\nabla^2_y \varphi_1 (x,y). $$
		\item Suppose that $(\varphi^{\varepsilon_h})_{h>0} \subset H^2(\Omega) $ is a bounded sequence such that   $h^{-1}\varphi^{\epsh}\drtwoscale \varphi(x,y)$ and $\lim_{h\to 0} {\epsh}^{-2}h=\kappa\in[0,\infty).$ 
		%%or $h \ll {\epshtwo}$. 
		Then $\nabla^2 \varphi^{\epsh} \drtwoscale \kappa \nabla^2_y \varphi(x,y).$ 
		%%%in the first case and $\nabla^2 \varphi^{\epsh} \drtwoscale 0$ in the second case.   
		\item 
		\begin{enumerate} 
			\item Let $(\varphi^{\varepsilon})_{\eps>0} \subset H^2(\omega) $ be such that the sequences $(\varphi^{\varepsilon})_{\eps>0}$, $(\varepsilon \nabla \varphi^{\varepsilon})_{\eps>0}$ are bounded in the corresponding $L^2$ spaces. Suppose that $\varphi^{\varepsilon} \drtwoscale \varphi(x,y) \in L^2(\Omega \times Y)$. Then $\varphi \in L^2(\Omega;H^1(\mathcal{Y}))$ and $\eps \nabla \varphi^\varepsilon \drtwoscale \nabla_y \varphi(x,y)$.
			\item 
			Let $(\varphi^{\varepsilon})_{\eps>0} \subset H^2(\omega) $ be such that the sequences $(\varphi^{\varepsilon})_{\eps>0}$, $(\varepsilon \nabla \varphi^{\varepsilon})_{\eps>0}$, $(\eps^2 \nabla^2 \varphi^{\varepsilon})_{\eps>0}$ are bounded in $L^2.$ Suppose that $\varphi^{\varepsilon} \drtwoscale \varphi(x,y) \in L^2(\Omega \times Y)$. Then $\varphi \in L^2(\Omega;H^2(\mathcal{Y}))$ and $\eps \nabla \varphi^{\varepsilon} \drtwoscale \nabla_y \varphi(x,y)$, $\eps^2 \nabla^2 \varphi^{\varepsilon} \drtwoscale \nabla^2_y \varphi(x,y)$. 
		\end{enumerate} 
	\end{enumerate} 
\end{lemma}
We will prove the following lemma.
\begin{lemma}\label{nada2}
	Let $\Omega=\omega \times I$, where $\omega\subset \R^2$ a bounded set with Lipschitz boundary and let $( \psi^{\eps_h})_{h>0} \subset H^1(\Omega)$ be such that there exists $C>0$ such that
	\begin{equation}\label{lemascaled1}
		\| \psi^{\eps_h}\|_{L^2(\Omega)}^2+ \eps_h^2\| \nabla_h  \psi^{\eps_h}\|_{L^2(\Omega;\R^3)}^2 \leq C. 
	\end{equation}
	\begin{enumerate} 
		\item If $h \ll \eps_h$ then  
		there exist $\psi_1\in H^1(\omega\times \mathcal{Y})$, $\psi_2\in L^2(\Omega\times \mathcal{Y})$ such that (on a subsequence)
		\begin{equation} \label{lemascaled2}
			\psi^{\eps_h} \drtwoscale  \psi_1,\quad \eps_h \nabla_h  \psi^{\eps_h} \drtwoscale (\partial_{y_1}  \psi_1, \partial_{y_2}  \psi_1, \psi_2). 
		\end{equation}
		The opposite claim is also valid, i.e., for every  $ \psi_1\in H^1(\omega\times \mathcal{Y})$, $ \psi_2\in H^1(\Omega\times \mathcal{Y})$ we have that 
		there exists  $( \psi^{\eps_h})_{h>0} \subset H^1(\Omega)$ such that \eqref{lemascaled1} and \eqref{lemascaled2} are satisfied.   
		\item If $\epsh \ll h$ then there exists $\psi\in L^2(\Omega,H^1(\mathcal{Y}))$ such that
		\begin{equation} \label{lemascaled3}
			\psi^{\eps_h} \drtwoscale  \psi,\quad \eps_h \nabla_h  \psi^{\eps_h} \drtwoscale (\partial_{y_1}  \psi, \partial_{y_2}  \psi, 0). 
		\end{equation}
		The opposite claim is also valid, i.e., for every  $ \psi\in L^2(\Omega; H^1(\mathcal{Y};\mathbb{R}^2))$ there exists $( \psi^{\eps_h})_{h>0} \subset H^1(\Omega)$ such that \eqref{lemascaled1} and \eqref{lemascaled3} hold.
	\end{enumerate}
\end{lemma}
\begin{proof}
	To prove the first part of the lemma,
	we take $ \psi_1\in L^2(\Omega \times Y;\R^3)$ such that $\psi^{\eps_h} \drtwoscale  \psi_1$ on a subsequence. Since, by assumption,
	$$\|\partial_{x_3}  \psi^{\eps_h} \|_{L^2(\Omega)} \leq C \frac{h}{\eps_h}, $$
	we immediately conclude that $ \psi_1$
	does not depend on $x_3$. Denote the two-scale limit of 
	$h^{-1}\eps_h\partial_{x_3}\psi^{\eps_h}$ 
	by $\psi_2.$ Invoking integration by parts in a standard fashion, it is easy to check that 
	$$ \eps_h(\partial_1  \psi^{\eps_h}, \partial_2  \psi^{\epsh}) \drtwoscale (\partial_{y_1}  \psi_1, \partial_{y_2}  \psi_1). $$
	In order to prove the second claim of part 1, it suffices to consider the case $\psi_1\in C^1(\omega; C^{1} (\mathcal{Y}))$, $ \psi_2 \in C^1(\Omega; C^{1}(\mathcal{Y}))$. We can then take 
	$$  \psi^{\epsh}:= \psi_1 \biggl(x_1,x_2,\frac{x_1}{\eps_h},\frac{x_2}{\eps_h}\biggr)+\frac{h}{\epsh} \int_{-1/2}^{x_3}  \psi_2 \biggl(x_1,x_2,s,\frac{x_1}{\eps_h},\frac{x_2}{\eps_h}\biggr)\,ds. 
	$$
	This completes the proof of part 1.
	
	To prove part 2, we take $ \psi\in L^2(\Omega \times Y;\R^3)$ such that $\psi^{\eps_h} \drtwoscale  \psi$ on a subsequence. Again, using integration  by parts, we obtain 
	$$ \eps_h(\partial_1  \psi^{\eps_h}, \partial_2  \psi^{\epsh}) \drtwoscale (\partial_{y_1}  \psi, \partial_{y_2}  \psi). $$
	Next, for $b\in C_0^1(\Omega), v \in C^1(\mathcal{Y})$ we have
	$$
	\int_{\Omega} \frac{\epsh}{h} \partial_{x_3} \psi^{\epsh} b(x) v\biggl(\dfrac{\hat{x}}{\epsh}\biggr)\,dx=-\int_{\Omega} \frac{\epsh}{h} \psi^{\epsh} \partial_{x_3}b(x) v\biggl(\dfrac{\hat{x}}{\epsh}\biggr)\,dx \to 0. 
	$$
	It follows that $\partial_{x_3}\psi^{\eps_h} \drtwoscale 0$. To prove the last claim, we set
	\[
	\psi^{\epsh}:=\psi\biggl(x,\frac{x_1}{\epsh},\frac{x_2}{\epsh}\biggr)
	\]
	for $\psi \in C^1(\Omega,C^1(\mathcal{Y}))$ and pass to the limit as $h\to0.$
\end{proof}	

The definition of two-scale convergence (Definition \ref{withouttdef}) naturally extends to time dependent spaces. 
\begin{definition}
	\label{two_scale_def} 
	Let $( u^\varepsilon)_{\varepsilon>0}$ be a bounded sequence in $L^2([0,T];L^2(\Omega))$. We say that $ (u^\varepsilon)_{\eps>0}$ weakly two-scale converges to $ u \in L^2([0,T];L^2(\Omega \times Y)),$ and write 
	\[
	u^\varepsilon \drtwoscalet  u(t,x,y),
	\] 
	if
	\[
	\int\limits_0^T \int\limits_{\Omega}  u^\varepsilon (t, x)  \phi\biggl(x,\frac{\hat{x}}{\varepsilon}\biggr)\,\varphi(t)\,dxdt
	\,\xrightarrow[ ]\,
	\int\limits_0^T\int\limits_{\Omega}\int\limits_Y u(t,x, y) \phi\left(x,y\right)\varphi(t)\,dydxdt,
	\]	
	i.e., 
	\[
	\int\limits_0^T \int\limits_{\Omega}  u^\varepsilon(t,x) \phi\biggl(x,\frac{x}{\varepsilon}\biggr)\,\varphi(t)\,dxdt
	\,\xrightarrow[ ]\,
	\int\limits_0^T\int\limits_{\Omega}\int\limits_Y u(t,x, y) \phi\left(x,y\right)\varphi(t)\,dydxdt,
	\]
	for every $\phi \in C_{\rm c}^\infty(\Omega;C(\mathcal{Y}))$, $\varphi \in C(0,T)$. 
	If in addition one has 
	\begin{equation}
		u^{\eps}(t,x) \strongdrtwoscale u(t,x,y)\ \ \ \ \textrm{a.e.}\ \ t \in [0,T]
		\label{pointwiseconv}
	\end{equation}	
	and 
	\begin{equation*}
		\int_0^T \bigl\|u^{\eps}(t,\cdot)\bigr\|^2_{L^2(\Omega)}\, dt \to \int_0^T \bigl\|u(t,\cdot)\bigr\|^2_{L^2(\Omega)}\,dt, 
	\end{equation*} 
	then we will say that $(u^\eps)_{\eps>0}$ strongly two-scale converges to $u$ and write 
	\[
	u^\varepsilon \strongdrtwoscalet  u(t,x,y).
	\]
\end{definition} 	
%\begin{definition} 
%Let $\Omega \subset \R^3$ and $Y = [0, 1)^2$. Let $\{ u^\varepsilon\}_{\varepsilon>0}$ be a bounded sequence in $L^2([0,T];L^2(\Omega))$. We say that $\{ u^\varepsilon\}$  two-scale converges to $ u \in L^2([0,T];L^2(\Omega \times Y))$  if
%If we have that $u^{\eps}(t) \drtwoscale u(t,x,y)$, i.e.,  $u^{\eps}(t) \strongdrtwoscale u(t,x,y)$  for every $t \in[0,T]$ then we will say that $u^{\eps}$ converges weakly (respectively strongly) two-scale pointwise in time. 
%\end{definition}   
Similarly, we define the notions of weak two-scale convergence and strong two-scale convergence of sequences in $L^p([0,T];L^2(\Omega))$, for any $1\leq p<\infty,$ denoted by $\xrightharpoonup{t,p,{\rm dr}-2}$ and $\xrightarrow{t,p,{\rm dr}-2},$ respectively. The convergence  $\xrightharpoonup{t,\infty,{\rm dr}-2}$ is understood in the weak* sense with respect to the time variable $t$, while $\xrightarrow{t,\infty,{\rm dr}-2}$ is understood in the sense of simultaneous pointwise convergence (\ref{pointwiseconv}) and boundedness of $\|u^{\eps}(t, \cdot)\|_{L^2(\Omega)},$ $t\in[0,T],$ in the space $L^{\infty}(0,T)$.
The following lemma is standard (see, e.g., the proof of \cite[Lemma 4.7]{pastukhova}. 
\begin{lemma} 
	If $(u^{\eps})_{\eps>0}$ is a bounded sequence in $L^p([0,T];L^2(\Omega))$, $p>1,$ then it has a subsequence that converges weakly two-scale in the sense of Definition \ref{two_scale_def}. 	
\end{lemma}

\subsection{Extension theorems} 
\label{extension_app}
\subsubsection{Asymptotic regime $h\sim \varepsilon_h$}

We use the extension theory   in order to decompose the sequence of displacements  into two functions for which we can get enough compactness for passing to the limit as $h\to0$ in the original equations for displacements. 
We use the following result, which can be found in \cite{OShY}:
\begin{theorem} 
	\label{thmextension} 
	For every $h>0$ there exists a linear extension operator $\tilde{\cdot} : H^1(\Omega_1^{\varepsilon_h};\R^3) \to H^1(\Omega;\R^3)$ such that $\tilde{ \vect u} =\vect u$ on $\Omega_1^{\varepsilon_h}$ and
	\begin{align*}
		\left\| \simgrad_h \tilde{\vect u} \right\|_{L^2(\Omega;\R^{3 \times 3})} &\leq C \left\| \simgrad_h \vect u \right\|_{L^2(\Omega_1^{\varepsilon_h};\R^{3 \times 3})}.
	\end{align*}
	
\end{theorem}
\begin{proof}
	The way in which we introduce the extensions here is to look at every single cell inside the  thin domain $\Omega^h.$ The extension of the function $\vect u$ (defined on $\Omega^h$) is constructed as follows. First we ``inflate" the cell (with the scaling factors $\varepsilon_h^{-1}$ and $h^{-1}$ in the in-plane and out-of plane directions, respectively) and translate it to the reference cell $Y \times I$. On this reference cell we apply a linear extension operator $ E:H^1(Y_1 \times I)\to H^1(Y \times I)$ (see \cite[Lemma 4.1]{OShY} to extend $\vect u$ to the function $E\vect u=:\Tilde{\vect u}$. By passing to the original coordinates and concatenating the extensions we construct functions $\Tilde{\vect u} \in H^1(\Omega^h;\R^3)$ which is the extension of $\vect u$ from $\Omega^{h,\epsh}_1$ to $\Omega^h$ and satisfy the estimate
	\[
	\left\| \sym\nabla \Tilde{\vect u} \right\|_{L^2(\Omega^h;\R^{3 \times 3})} \leq C \left\| \sym\nabla\vect u \right\|_{L^2(\Omega^{h,\epsh}_1;\R^{3 \times 3})},
	\]
	where the constant $C$ does not depend on the thickness $h$. We finish the proof by rescaling the estimates back to $\Omega$.
\end{proof}
\begin{remark} 
 It is not difficult to see that if we have a sequence $\{(h,\epsh)\}$ such that $0<\alpha<h/\epsh<\beta<\infty$, where $\alpha,\beta$ do not depend of $h$, then the constant in Theorem \ref{thmextension} depends on $\alpha,\beta$ only. 	
\end{remark} 	
\begin{remark} \label{rucak30} 
	By inspecting the construction of the extension operator in \cite[Lemma 4.1]{OShY}, it can be easily seen that if $\vect u \in L^{2, \rm bend}(\Omega;\R^3)$ or $\vect u \in L^{2, \rm memb}(\Omega;\R^3)$, then the same inclusion holds for $\tilde{\vect u}$.  	
\end{remark} 	
Using the result above we conclude the following lemma.
\begin{lemma}
	\label{poincarekornonholes}
	For all $\mathring{\vect u}\in H^1(\Omega;\R^3)$ such that $\mathring{\vect u} |_{\Omega_1^{\varepsilon_h}}=0,$
	%Then $\mathring{\vect u}$ satisfies 
	the following Poincar\'{e} and Korn inequalities hold:
	\begin{align}
		\label{rucak37} 
		\left\| \mathring{\vect u} \right\|_{L^2} &\leq C\epsh \left\| \nabla_h \mathring{\vect u} \right\|_{L^2}, \\[0.3em]
		  \label{rucak38}   
		\left\| \nabla_h \mathring{\vect u} \right\|_{L^2} &\leq  C \left\| \sym\nabla_h \mathring{\vect u} \right\|_{L^2},
	\end{align}
where the constant $C$ is independent of $h.$ 
\end{lemma}
\begin{proof}
	We again look the problem on the physical domain $\Omega^h$. The function $\vect u$ (scaled properly), when restricted to a single cell within the domain and then rescaled and translated to $Y\times I,$ satisfies the Poincare inequality, as well as Korn's inequality \eqref{Kornwithbc}  with a constant determined by the domain $Y\times I$. Scaling back to the physical domain $\Omega^h$ and summing up the norms over all cells, we obtain a version of the estimates \eqref{rucak37} and \eqref{rucak38} for $\Omega^h$. Finally, rescaling to $\Omega,$ we obtain \eqref{rucak37} and \eqref{rucak38}. 
	%\[
	%\left\|(\mathring{u}_1^\epsh,\mathring{u}_2^\epsh,h\mathring{u}_3^\epsh ) \right\|_{L^2(\Omega;\R^3)} \leq C \epsh
	%\left\|\nabla(\mathring{u}_1^\epsh,\mathring{u}_2^\epsh,h\mathring{u}_3^\epsh ) \right\|_{L^2(\Omega;\R^3)}.
	%\]
\end{proof}
\begin{remark} \label{remivan102}  
	Using continuity of embeddings into spaces on the boundary, it follows immediately that
	$$ \| \mathring{\vect u}\|_{L^2(\Gamma \cap \overline{\Omega}_0)} \leq C\epsh\bigl\|\sym \nabla_h \mathring{\vect u}\bigr\|_{L^2(\Omega;\R^{3 \times 3})},	$$
	where $\Gamma:=\omega \times \{-{1/2},{1/2}\}$. 
\end{remark}	
\subsubsection{Asymptotic regime $h \ll \varepsilon_h$}
In this section we assume that $h \ll \varepsilon_h$. 
%%We again take the reference cell $Y$. 
We will assume that $Y_0 \subset Y$ does not touch the boundary of $Y$ and is of class $C^{1,1}$.
First, we provide an extension property, in the spirit of Theorem \ref{thmextension}. We denote by $\Omega_{\alpha}^\epsh$, $\alpha=1,2$, the same sets as in Section \ref{notation}. 
We have the following theorem.
\begin{theorem}\label{thmextreg1}
 There exists a linear extension  $E^{\varepsilon_h}:H^1(\Omega_1^{\varepsilon_h};\R^3) \to H^1(\Omega;\R^3)$ such that for every $\vect u\in H^1(\Omega, \R^3) $, $E^{\varepsilon_h}\vect u= \vect u$ on $\Omega_1^{\varepsilon_h}$ and 
	\begin{equation} \label{rucak36} 
		\bigl\|\sym \nabla_h E^{\varepsilon_h}\vect u \bigr\|_{L^2(\Omega;\R^{3 \times 3})} \leq  C\bigl\|\sym\nabla_h \vect u\bigr\|_{L^2(\Omega_1^{\epsh},\mathbb{R}^{3 \times 3})}.
	\end{equation} 
	for some $C>0$ independent of $h$.
	Moreover, there exist $\mathring{v} \in H^2(\omega)$ and $\mathring {\vect \psi} \in H^1(\Omega;\R^3)$ such that $\mathring{v}=\mathring{\vect \psi}=0$ on $ \Omega_1^{\varepsilon_h}$ and 
	\begin{equation}\label{eqivan5} 
		\mathring{\vect u}:=  \vect u-E^{\varepsilon_h} \vect u=\left(\begin{array}{c} -\epsh x_3  \partial_1 \mathring{v}  \\[0.25em] 
			-\epsh x_3  \partial_2 \mathring{v} \\[0.25em]
			h^{-1}\epsh \mathring{v} \end{array} \right)+\mathring{\mat \psi},
	\end{equation}
	with the estimate 
	\begin{equation}\label{est1000}
		\epsh^{-2}\|\mathring{v}\|^2_{L^2}+\|\nabla \mathring{v}\|^2_{L^2}+ \varepsilon_h^2
		\bigl\| \nabla^2 \mathring{v}\bigr\|^2_{L^2}+\epsh^{-2}
		\bigl\| \mathring{\vect \psi}\bigr\|^2_{L^2}+ \bigl\|\nabla_h \mathring{\mat \psi}\bigr\|^2_{L^2} \leq C\bigl\|\sym \nabla_h \mathring{\vect u}\bigr\|^2_{L^2},    
	\end{equation}
	where $C$ depends on $Y_0$ only. 
\end{theorem}
\begin{proof}
	We consider a domain $\tilde Y_0 \times (h/\varepsilon_h)I$, such that $\tilde Y_0$ has $C^{1,1}$ boundary, $\overline{Y}_0 \subset \tilde Y_0$, and $\tilde{Y}_0 \backslash \overline{Y}_0$ is connected.
		%% We consider the domain $(\tilde Y_0 \backslash \overline{Y}_0) \times I$ and
		 For the extension $\vect \psi \in H^1((\tilde Y_0 \backslash \overline{Y}_0) \times I;\R^3),$ we apply the decomposition of Part 1 of Lemma \ref{keydecompose} to obtain
	\begin{equation} \label{moddek111} 
		\vect \psi=\vect a+\vect B (x_1,x_2,\varepsilon_h^{-1}hx_3)^\top+\left(\begin{array}{c} -x_3 \partial_1 v  \\[0.25em] 
			-x_3 \partial_2 v \\[0.25em] h^{-1}\varepsilon_h v \end{array} \right)+\tilde{\mat \psi},
	\end{equation}
	where $\vect a\in \R^3$, $\vect B\in \R^{3\times 3}_{\skeww}$,  $v \in H^2(\tilde Y_0 \backslash \overline{Y}_0)$, $\tilde{\mat\psi} \in H^1((\tilde Y_0 \backslash \overline{Y}_0) \times I;\R^3),$ and the following estimate holds:
	\begin{equation}\label{estmoddek111}
		\| v\|^2_{H^2(\tilde Y_0 \backslash \overline{Y}_0)}+\|\tilde{\vect \psi}\|_{L^2((\tilde Y_0 \backslash \overline{Y}_0) \times I;\R^3)}+ \|\nabla_{h/\varepsilon_h} \tilde{\mat \psi}\|^2_{L^2((\tilde Y_0 \backslash \overline{Y}_0) \times I;\R^{3\times 3})} \leq C \| \sym \nabla_{h/\varepsilon_h} \vect \psi\|^2_{L^2((\tilde Y_0 \backslash \overline{Y}_0) \times I;\R^{3\times 3})},      
	\end{equation}
	where $C$ depends on $Y_0$ only.
	It is not difficult to construct extension operators 
	\[
	E_1:H^2(\tilde Y_0 \backslash \overline{Y}_0) \to H^2(\tilde Y_0),\qquad E_2:H^1\bigl((\tilde Y_0 \backslash \overline{Y}_0) \times I\bigr) \to H^1(\tilde Y_0  \times I)
	\]
	such that
	$E_1 \varphi=\varphi $ on $\tilde Y_0 \backslash Y_0$ and $E_2  w= w $ on $(\tilde Y_0 \backslash Y_0) \times I$ and
	\begin{equation*}
		\|E_1\varphi\|_{L^2(\tilde Y_0)} \leq C \|\varphi\|_{L^2(\tilde Y_0 \backslash \overline{Y}_0)}, \quad  \|E_1 \varphi\|_{H^2(\tilde Y_0)}\leq C\|\varphi\|_{H^2(\tilde Y_0 \backslash \overline{Y}_0)}\qquad \forall \varphi \in H^2(\tilde Y_0 \backslash \overline{Y}_0),
	\end{equation*}
	\begin{equation}
		\label{svojstvoeks2}
		\begin{aligned}
			&\|E_2  w\|_{L^2(\tilde Y_0) \times I)} \leq C\| w\|_{L^2((\tilde Y_0 \backslash \overline{Y}_0)) \times I)},\\[0.35em] 
			&\hspace{80pt}\| \nabla_{h/\varepsilon_h} E_2  w\|_{L^2(\tilde Y_0) \times I;\R^{3})} \leq C \| \nabla_{h/\varepsilon_h}   w\|_{L^2((\tilde Y_0 \backslash \overline{Y}_0)) \times I;\R^{3})}\qquad
			%%,\\[0.35em] 
			%%& \hspace{+10ex} 
			\forall  w \in H^1(\tilde Y_0 \backslash \overline{Y}_0 \times I), 
		\end{aligned}
	\end{equation}
	for some $C>0$. Indeed, $E_1$ is constructed by using the standard reflection principle. Also using the reflection principle, we first construct  $ \tilde E_2: H^1(\tilde Y_0 \backslash \overline{Y}_0) \to H^1(\tilde Y_0)$ such that 
	%% For the operator $\tilde E_2$ it is satisfied
	\begin{eqnarray*}
		& &  \| \tilde  E_2 \varphi\|_{L^2(\tilde Y_0)} \leq C \| \varphi\|_{L^2(\tilde Y_0 \backslash \overline{Y}_0)},  \quad \ \| \tilde  E_2\varphi\|_{H^1(\tilde Y_0)} \leq C \| \varphi\|_{H^1(\tilde Y_0 \backslash \overline{Y}_0)}\qquad
		\forall \varphi \in H^1(\tilde Y_0 \backslash \overline{Y}_0),
	\end{eqnarray*}
	for some $C>0$.
	On the basis of $\tilde E_2,$ we construct $E_2$ as follows. For $w \in  C^2(\overline{\tilde Y_0 \backslash \overline{Y}_0 \times I})$ we set $E_2 w(\cdot,x_3)=\tilde E_2 w(\cdot, x_3)$ for all $x_3 \in I.$ It is easy to check that for  $w \in C^2(\overline{\tilde Y_0 \backslash \overline{Y}_0 \times I})$ one has
	$\partial _{x_3} E_2 w=E_2 (\partial_{x_3} w)$, from which we infer the property  \eqref{svojstvoeks2} for $w \in C^2(\overline{\tilde Y_0 \backslash \overline{Y}_0 \times I})$. We then extend $E_2$ to the whole of $H^1(\tilde Y_0 \backslash \overline{Y}_0 \times I)$ by density, which concludes the construction. %%%$E_2$ then trivially satisfies \eqref{svojstvoeks2}. 
	
	For $\vect \psi \in H^1(\tilde Y_0 \backslash \overline{Y}_0 \times I;\R^3),$ using the expression \eqref{moddek111},
	%%% and the estimate \eqref{estmoddek111}
	  %%%which is decomposed according to  \eqref{moddek111} 
	  we define $\tilde E^\epsh \vect \psi \in H^1(\tilde Y_0\times I;\R^3)$ as follows:
	\begin{equation} \label{missaglic10} 
		\tilde E^\epsh \vect \psi =\vect a+\vect B (x_1,x_2,\varepsilon_h^{-1}hx_3)^\top+\left(\begin{array}{c} -x_3 \partial_1 E_1 v  \\[0.25em]
			-x_3 \partial_2 E_1 v \\[0.25em] 
			\varepsilon_h h^{-1} E_1 v \end{array} \right)+E_2\tilde{\mat \psi}.
	\end{equation} 
Recalling \eqref{estmoddek111}, we obtain the estimate
	$$
	\|E_1 v\|^2_{H^2(\tilde{Y}_0)}+\|E_2\tilde{\vect \psi}\|^2_{L^2(\tilde{Y}_0;\R^3)}+ \|\nabla_{h/\varepsilon_h} E_2\tilde{\mat \psi}\|^2_{L^2(\tilde{Y}_0;\R^{3\times 3})} \leq C \| \sym \nabla_{h/\varepsilon_h} \vect \psi\|^2_{L^2((\tilde Y_0 \backslash \overline{Y}_0) \times I;\R^{3\times 3})}. $$
	We construct $E^{\varepsilon_h} \vect u$ by considering $z \in \Z^2$ such that $\varepsilon_h (Y+z) \subset \omega$ and  applying the extension $\tilde E^\epsh$ to the function $x\mapsto {\vect u} (\epsh \hat x+\epsh z,x_3)$. In this way we obtain
	\begin{equation}
		\label{rucak40} 
		E^{\varepsilon_h} \vect u\big|_{\varepsilon_h(\tilde Y_0+z)\times I}= \vect a^\epsh_z +\vect B^\epsh_z\left(x_1, x_2, hx_3\right)^\top
		+\left(\begin{array}{c} - x_3 \partial_1  v_z  \\[0.25em] - x_3 \partial_2  v_z \\[0.25em] h^{-1}v_z \end{array} \right)+\mat{ \psi}_z,
	\end{equation}
	with the estimate
	\begin{align*} 
		\varepsilon_h^{-4}\|v_z\|^2_{L^2(\epsh(\tilde Y_0+z))}&+\varepsilon_h^{-2}\|\nabla v_z\|^2_{L^2(\epsh(\tilde Y_0 +z);\R^2)}+
		\bigl\|\nabla^2 v_z\bigr\|^2_{L^2(\epsh(\tilde Y_0+z);\R^{2\times 2})}\\[0.4em]
		& 
		%&\hspace{+10ex} 
		+{\epsh}^{-2}\|{\vect \psi}_z\|^2_{L^2(\epsh(\tilde Y_0+z) \times I;\R^3)}+\| \nabla_h{\vect \psi}_z\|^2_{L^2(\epsh(\tilde Y_0+z) \times I;\R^{3\times 3})} \leq C\|\sym \nabla_h \vect u \|^2_{L^2(\epsh(\tilde Y_0\backslash \bar{Y}_0+z) \times I;\R^{3 \times 3})}.  
	\end{align*}
	This concludes the proof of \eqref{rucak36}.
	To prove \eqref{eqivan5}, for each $z\in \Z^2$ consider the deformation
	\[
	\vect u-E^{\epsh} \vect u|_{\varepsilon(\tilde Y_0+z)\times I}
	\] 
	and apply the above rescaling as well as the second claim of Lemma \ref{keydecompose}. The linearity of $E^{\epsh}$ follows from the decompositions in Theorem \ref{aux:thm.griso} and Lemma \ref{keydecompose}. 
\end{proof}
\begin{remark} Notice that as a consequence of 
	Corollary \ref{kornthincor} and the estimate \eqref{rucak36}, if $\vect u=0$ on $\gamma_{\rm D} \times I$, where $\gamma_{\rm D} \subset \partial \omega$ is the set of positive measure, then 
	\begin{eqnarray*}
		\left\|\pi_{1/h}(E^{\epsh} \vect u)\right\|_{H^1(\Omega;\R^3)}^2 &\leq C\Bigl(\|\pi_{1/h}\vect u)\|^2_{L^2(\Gamma_D;\R^3)}+h^{-2}\|\sym\nabla_h \vect u \|_{L^2(\Omega_1^{\epsh},\mathbb{R}^{3 \times 3})}^2\Bigr),
	\end{eqnarray*}
where $C$ is obtained by combining $C_{\rm T}^\gamma$ in (\ref{CTgamma}) and the constant in the extension inequality (\ref{rucak36}). 
\end{remark}
\begin{remark} \label{remivan20} 
	%%In the decomposition \eqref{moddek111} 
	We can assume, without loss of generality, that in the above proof $E_1$ maps affine expressions $a_1x_1+a_2x_2+a_3$, for $a_1,a_2,a_3 \in \R,$ to themselves. (Indeed, as in the proof of Proposition \ref{propvecer},  on the orthogonal complement of affine maps in $L^2$ the extension is constructed by reflection.)  
	From \eqref{rucak40}, recalling (\ref{moddek111}) and (\ref{missaglic10}), we also have the estimate  
	$$ 
	\|E^{\varepsilon_h} \vect u \|_{L^2(\Omega;\R^3)} \leq  C \left(\|\vect u \|_{L^2(\Omega_1^{\epsh};\R^3)}+\|\sym \nabla_h \vect u\|_{L^2(\Omega_1^{\epsh};\R^{3\times 3} )}\right),  
	$$
	for some $C>0$. From  \eqref{eqivan5}, \eqref{missaglic10} we then additionally obtain that
	$$\epsh h^{-1}\| \mathring{v} \|_{L^2(\omega)} \leq C \left(\|\vect u \|_{L^2(\Omega;\R^3)}+\|\sym \nabla_h \vect u\|_{L^2(\Omega_1^{\epsh};\R^{3\times 3} )}+\epsh\|\sym \nabla_h \vect u\|_{L^2(\Omega_0^{\epsh};\R^{3\times 3} )}\right). $$
\end{remark} 
\begin{remark} \label{remivan11} 
	Following Remark \ref{remivan} and Remark \ref{remivan10}  we infer that if $ u_{\alpha}$ is even in the variable $x_3$ for $\alpha=1,2$ and $u_3$ is odd in the variable $x_3,$ the extension $E^{\epsh} \vect u$ has the same properties.
	%%also satisfies that it is even in the first two components and odd in the third component in $x_3$ variable. 
	Noting that by Lemma \ref{keydecompose} and Theorem \ref{aux:thm.griso}
	 one has $a_{z,3}^\epsh= B_{z,13}^\epsh= B_{z,23}^\epsh=v_z^\epsh=\mathring{v}^\epsh= 0$ in \eqref{rucak40}, we also infer that 
		\begin{align*} 
			\bigl\| (E^{\epsh} \vect u)_3\bigr\|_{L^2(\Omega)} &\leq Ch \|\sym \nabla_h \vect u\|_{L^2(\Omega_1^{\epsh};\R^{3\times 3})}, \\[0.3em]
			\|\vect u_3 \|_{L^2(\Omega)} &\leq  \| \mathring{\vect \psi}_3\|_{L^2(\Omega;\R^3)} \leq C\epsh^{-1}h\epsh \|\sym  \nabla_h \vect u\|_{L^2(\Omega;\R^{3\times 3} )}=Ch\|\sym  \nabla_h \vect u\|_{L^2(\Omega;\R^{3\times 3} )}, 
		\end{align*}
	for some $C>0$.  
\end{remark}
\subsubsection{Asymptotic regime $\varepsilon_h \ll h$}
In this regime the extension theorem is analogous to Theorem \ref{thmextension}. 
\begin{theorem} \label{thmivann111} 
	For every $\varepsilon_h>0$ there exists a linear extension operator $u\mapsto\tilde{u}: H^1(\Omega_1^{\varepsilon_h};\R^3) \to H^1(\Omega;\R^3)$ such that $\tilde{\vect u} =\vect  u$ on $\Omega_1^{\varepsilon_h}$ and
	\begin{align*}
		\left\| \simgrad_h \tilde{\vect u } \right\|_{L^2(\Omega;\R^{3\times 3})} &\leq C \left\| \simgrad_h  \vect u \right\|_{L^2(\Omega_1^{\varepsilon_h};\R^{3\times 3})}.
	\end{align*}
	
\end{theorem}
\begin{proof}
	Since $\epsh \ll h,$ the way in which we introduce the extensions here is to look at every single cell inside the thin domain $\Omega^h$. 
	The cells are $\varepsilon_h$-cubes 
	%%of size  $\eps_h,$ given by  
	$
	\epsh(Y+z)\times[k\epsh,(k+1)\epsh],
	$ 
	where $z \in \Z^2$, $k \in \Z,$	are chosen so that each cube is entirely inside $\Omega$. We use the extension operator on the cube $Y\times I$ (see \cite[Lemma 4.1]{OShY}), followed by a scaling argument.
	%% obtain the extensions on the cubes.
	We label the resulting extension by $E_1.$
	The problem is that we can have a mismatch at the lines $x_3=k\epsh$, where $k \in \Z,$ and there are possible ``boundary layers" at the sides $x_3=\pm h/2$ where the extension is not defined (due to the fact that $h/\epsh$ is not necessarily an integer).  We deal with this by introducing another series of extensions to cubes 
	\[
	\varepsilon_h(Y+z)\times \biggl[\biggl(k+\frac{1}{2}\biggr)\epsh,\biggl(k+\frac{3}{2}\biggr)\epsh\biggr],\qquad z \in \Z^2,\quad k \in \Z,
	\] 
	from the complements of the corresponding ``perforations"
	\[
	\epsh(Y_0+z)\times \biggl[\biggl(k+\frac{3}{4}\biggr)\epsh,\biggl(k+\frac{5}{4}\biggr)\epsh\biggr].
	\] 
	(On the parts 
%	{\color{red}
		\[
		\epsh(Y_0+z) \times \biggl[\biggl(k+\frac{1}{2}\biggr)\epsh,\biggl(k+\frac{3}{4}\biggr)\epsh\biggr],\quad 
		%\bigcup 
		\epsh(Y_0+z) \times \biggl[\biggl(k+\frac{5}{4}\biggr)\epsh,\biggl(k+\frac{3}{2}\biggr)\epsh\biggr] 
		\] 
%	}
	we continue using the extension $E_1.$) In this way we correct the first extension and eliminate the mismatch.
	We denote the resulting extension by $E_2$.
	We deal with the upper layers at $x_3=\pm h/2$ in a different way, namely, we first consider the extension on the cubes 
	\[
	\epsh(Y+z)\times\left(\biggl[\frac{h}{2}-\epsh,\frac{h}{2}\biggr]\cup \biggl[-\frac{h}{2}\,-\frac{h}{2}+\epsh\biggr]\right)
	\]
	(referring to this as $E_3$), and then we correct the possible mismatch between $E_2$ and $E_3$ by performing extensions to the cubes 
	\[
	\epsh(Y+z)\times\left(\biggl[\frac{h}{2}-\frac{3}{2}\epsh,\frac{h}{2}-\frac{1}{2}\epsh\biggr]\cup \biggl[-\frac{h}{2}+\frac{1}{2}\epsh,-\frac{h}{2}+\frac{3}{2}\epsh\biggr]\right) 
	\]
	from the complements of the corresponding perforations
	\[
	\epsh(Y_0+z)\times\left(\biggl[\frac{h}{2}-\frac{5}{4}\epsh,\frac{h}{2}-\frac{3}{4}\epsh\biggr]\cup \biggl[-\frac{h}{2}+\frac{3}{4}\epsh,-\frac{h}{2}+\frac{5}{4}\epsh\biggr]\right) 
	\]
	(On the part 
	\[
	\eps(Y_0+z)\times\left(\biggl[\frac{h}{2}-\frac{3}{2}\epsh,\frac{h}{2}-\frac{5}{4}\epsh\biggr]\cup\biggl[-\frac{h}{2}+\frac{5}{4}\epsh,-\frac{h}{2}+\frac{3}{2}\epsh\biggr]\right) 
	\]
	we take the above extension $E_2$, while on the part 
	\[
	\epsh(Y_0+z)\times\left(\biggl[\frac{h}{2}-\frac{3}{4}\epsh,\frac{h}{2}-\frac{1}{2}\epsh\biggr]\cup \biggl[-\frac{h}{2}+\frac{1}{2}\epsh,-\frac{h}{2}+\frac{3}{4}\epsh\biggr]\right) 
	\]
	we take the extension $E_3$.) 
	We refer to this extension on the cube 
	\[
	\epsh(Y +z)\times\left(\biggl[\frac{h}{2}-\frac{3}{2}\epsh,\frac{h}{2}\biggr]\cup \biggl[-\frac{h}{2},-\frac{h}{2}+\frac{3}{2}\epsh\biggr]\right) 
	\] 
	as $E_4$.
	The final extension is given
	by $E_4$ on the layer 
	\[
	\Biggl\{(x_1, x_2, x_3)\in\Omega^h: x_3 \in\biggl[\frac{h}{2}-\frac{3}{2}\epsh,\frac{h}{2}\biggr]\cup \biggl[-\frac{h}{2},-\frac{h}{2}+\frac{3}{2}\epsh\biggr]\Biggr\}
	\]
	and by $E_2$ on the remaining part of $\Omega^h$. The required extension on $\Omega$ is now then by scaling in $x_3.$ (A procedure analogous to this has been described in \cite[Chapter 4]{OShY} for some specific domains.)
\end{proof}   
\begin{remark} \label{rucak31} 
	It is easy to see that if $\vect u \in L^{2, \rm bend}(\Omega;\R^3)$ or $\vect u \in L^{2, \rm memb}(\Omega;\R^3)$ then the same is valid for $\tilde{\vect u}$ (see the extension operator in \cite[Lemma 4.1]{OShY}).  	
\end{remark} 

The following statement is proved analogously to Theorem \ref{thmivann111}, see also the proof of Lemma \ref{poincarekornonholes}.
\begin{lemma}
	For $\mathring{\vect u}\in V(\Omega)$ be such that $\mathring{\vect u }^\epsh |_{\Omega_1^{\varepsilon_h}} = 0,$
	%%Then $\mathring{\vect u}$ satisfies 
	the following Poincar\'{e} and Korn inequalities hold:
	%	\begin{align*}
	\begin{equation}
		\left\| \mathring{ \vect u} \right\|_{L^2} \leq \epsh C \left\| \nabla_h \mathring{ \vect u} \right\|_{L^2},\qquad\qquad 
		%%\\    
		\left\| \nabla_h \mathring{ \vect u} \right\|_{L^2} \leq C \left\| \sym\nabla_h \mathring{\vect  u} \right\|_{L^2},
		\label{PK_ineq}
		%	\end{align*}
	\end{equation}
where the constant $C$ does not depend on $h.$
\end{lemma}
\begin{proof} 
	As in the case of Theorem \ref{thmivann111}, we work on the physical domain $\Omega^h$. 
	The restrictions of the inequalities (\ref{PK_ineq}) to those cylinders 
	$\epsh(Y_0+z)\times[-h/2+k\epsh,-h/2+(k+1)\epsh]$, $z \in \Z^2$, $k \in \N_0,$ that are contained in $\Omega$ are obtained by combining a scaling argument, the Korn inequality, and the Poincare inequality on the unit cube. 
	We similarly obtain the restrictions of the inequalities (\ref{PK_ineq}) to the cylinders $\epsh(Y_0+z)\times[h/2-\epsh, h/2]$, $z \in \Z^2$. The argument for the physical domain $\Omega^h$ is now completed by  summing up all above inequalities. To obtain the inequality on the canonical domain $\Omega$ (where gradients are replaced by scaled gradients), we simply perform the corresponding rescaling.   
\end{proof} 	
\subsection{Hyperbolic evolution problems of second order}\label{apsechyp}
Let $\mathcal{A}$ be a non-negative, self-adjoint operator with domain $\mathcal{D}(\mathcal{A})$ in a separable Hilbert space $H,$ and let $\mathcal{A}^{1/2}$ its unique non-negative self-adjoint root in $H$. We define the following norm on \RRR $V:=\mathcal{D}(\mathcal{A}^{1/2})$\BBB:
\begin{equation*}
	\|\vect u\|^2_{V}:=\bigl((\mathcal{A}+ \RRR \mathcal{I} \BBB)^{1/2}\vect u, (\mathcal{A}+ \RRR \mathcal{I} \BBB)^{1/2}\vect u\bigr)_{H}, \quad \vect u \in \mathcal{D}(\mathcal{A}^{1/2}).
\end{equation*}
\RRR  $V$ is a Hilbert space
 and the inclusion $V  \geq \mathcal{D}(\mathcal{A})$ \BBB is dense. 
Let $V^*$ be the dual of $V$ making $(V,H,V^*)$ the Gelfand triple \cite{Berezansky, dautraylions}. 
	Due to the density argument, the operators $\mathcal{A}$, $\mathcal{A}^{1/2}$ can be uniquely extended to bounded linear operators:
	\begin{equation*}
		\mathcal{A}: V \to V^*, \quad \mathcal{A}^{1/2}: H \to V^*.
	\end{equation*}
	Moreover, $\mathcal{A}+\mathcal{I}:V \to V^{*}$ as well as $(\mathcal{A}+\mathcal{I})^{1/2},$  $\mathcal{A}^{1/2}+\mathcal{I},$ viewed both as operators from $V$ to $H$ and form $H$ to $V^*,$ %%$(\mathcal{A}+\mathcal{I})^{1/2}:H \to V^* $, i.e.,  $\mathcal{A}^{1/2}+\mathcal{I}:V \to H$, $\mathcal{A}^{1/2}+\mathcal{I}:H \to V^*$ 
	are isomorphisms. 
	Consider also the following evolution problem:
	\begin{equation}
		\label{abstractevolutionproblem}
		\begin{split}
			&  \partial_{tt} \vect u(t) + \mathcal{A} \vect u(t) = \vect f(t), \\[0.2em]
			&  \vect u(0) = \vect u_0, \quad \partial_t \vect u(0) = \vect u_1, \\[0.1em]
			&   \vect u_0 \in V, \quad \vect u_1 \in H, \quad \vect f  \in L^2([0,T];V^*).
		\end{split}
	\end{equation}
\begin{definition}
	\label{weaksolutionofabstractevolutionproblem}
	We say that $\vect u \in L^2([0,T];V)$ is a weak solution of the problem \eqref{abstractevolutionproblem} if it satisfies: 
	\begin{equation}\label{ivan902} 
		\begin{split}
			& \vect u \in C([0,T];V), \quad \partial_t \vect u\in C([0,T];H), \\[0.35em]
			& \partial_{t}( \partial_t \vect u(t), \vect v )_{H} + a(\vect u(t), \vect v) = (\vect f(t), \vect v )_{V^*,V} \quad \forall \vect v \in V 
			\mbox{ in the sense of distributions on }  (0,T), \\[0.35em]
			&  \vect u(0) = \vect u_0, \quad \partial_t \vect u(0) = \vect u_1.
		\end{split}
	\end{equation}
\end{definition}
The problem \eqref{abstractevolutionproblem} can be restated as a first-order form, as follows. Consider the product space $E = V \times H$ endowed with the inner product 
\begin{equation*}
	( \Vec{\vect v}, \Vec{\vect w} )_E =  \bigl( (\vect v_1,  \vect v_2)^\top, (\vect w_1, \vect w_2)^\top\bigr)_E:= 
	\bigl((\mathcal{A} + \mathcal{I})^{1/2} \vect v_1,(\mathcal{A} + \mathcal{I})^{1/2}\vect w_1\bigr)_H + (\vect v_2, \vect w_2)_H,
\end{equation*}
and set
\begin{equation}
	\label{matrixoperator}
	\mathbb{A}:= \begin{bmatrix}0 & -\mathcal{I} \\[0.2em] \mathcal{A} & 0 \end{bmatrix}, \quad \mathcal{D}(\mathbb{A})= \mathcal{D}(\mathcal{A}) \times \mathcal{D}(\mathcal{A}^{1/2}),
\end{equation}
with the embedding $\mathcal{D}(\mathcal{\mathbb{A}}) \hookrightarrow E$ being dense. It is easily seen that
\begin{equation}\label{ivan701}
	( (\mathbb{A} + \mathbb{I})\Vec{\vect v}, \Vec{\vect v} )_E=(\mathcal{A} \vect v_1, \vect v_1 )+( \vect v_1,  \vect v_1)+( \vect v_2,  \vect v_2)-( \vect v_1,  \vect v_2) \geq (( \vect v_1, \vect v_1)+( \vect v_2, \vect v_2))/2,
\end{equation}
where by $\mathbb{I}$ we have denoted the identity operator on $E$. 
Moreover for $\lambda \in \R\backslash \{0\}$ one has 
\begin{equation} \label{ivan702} 
	(\mathbb{A} + \lambda\mathbb{I})\Vec{\vect v}=\Vec {\vect f}\quad\Longleftrightarrow\quad  \vect v_1=(\mathcal{A}+\lambda^2 \mathcal{I})^{-1}(\RRR \lambda \BBB \vect f_1+\vect f_2), \quad   \vect v_2=\lambda  \vect v_1-\vect f_1, 
\end{equation} 
which implies 
\begin{equation} \label{ivan103} 
	(\mathbb{A}+\lambda\mathbb{I})^{-1}=\begin{bmatrix} \lambda (\mathcal{A}+\lambda^2 \mathcal{I})^{-1} & (\mathcal{A}+\lambda^2 \mathcal{I})^{-1} \\[0.35em] \lambda^2(\mathcal{A}+\lambda^2 \mathcal{I})^{-1}-\mathcal{I} & \lambda(\mathcal{A}+\lambda^2 \mathcal{I})^{-1} \end{bmatrix}, \qquad\lambda \neq 0. 
\end{equation}
As a consequence of \eqref{ivan701} we can conclude that $-(\mathbb{A}+\mathbb{I})$ is a dissipative operator in the sense that 
%%for every $\lambda>0$  
\begin{equation}\label{ivan901} 
	\bigl\| (\lambda \mathbb{I}+\mathbb{I}+\mathbb{A})\Vec {\vect v}\bigr\|_{E} \geq \sqrt{\lambda^2+\lambda} \|\Vec {\vect v}\|_E \geq \lambda  \|\Vec {\vect v}\|_E, \qquad \forall\lambda>0. 
\end{equation}  
The property \eqref{ivan702} implies that $\mathbb{A}+\mathbb{I}$ is a closed operator. (Note that a dissipative  operator $\mathcal{S}$ is closed if there exists $\lambda>0$ such that the range of $\lambda \mathcal{I}-\mathcal{S}$ is closed.) 
From \eqref{ivan702} and \eqref{ivan901} we conclude that $(\lambda+1) \mathbb{I}+\mathbb{A}$ is a bijection
\begin{equation*} 
	\bigl\|((\lambda+1)\mathbb{I}+\mathbb{A})^{-1}\bigr\|\leq\lambda^{-1}, \quad \RRR \forall \BBB \lambda>0. 
\end{equation*}
It follows that $-(\mathbb{A}+\mathbb{I})$ generates a contraction semigroup (by Hille-Yosida Theorem) and 
\begin{equation} 
	\label{ivan910} 
	e^{-t\mathbb{A}}=e^t e^{-t(\mathbb{A}+\mathbb{I})},\quad \bigl\|e^{-t(\mathbb{A}+\mathbb{I})}\bigr\|\leq 1,\quad \bigl\|e^{-t\mathbb{A}}\bigr\|\leq e^{t}.
\end{equation}
The problem  \eqref{abstractevolutionproblem} can be formally written in the form 
\begin{equation}
	\label{abstractevolutionproblemrephrased}
	%%\begin{split}
	\partial_t \Vec{\vect u}(t) + \mathbb{A}\Vec{\vect u}(t) = \Vec{\vect f}(t),\qquad\ \ \Vec{\vect u}(0) = \Vec{\vect u}_0,
	%%\end{split}
\end{equation}
with $\Vec{\vect u}=(\vect u, \partial_t \vect u)^\top$, $\Vec{\vect u}_0 := (\vect u_0, \vect u_1)^\top$, $\Vec{\vect f}=(0,\vect f)^\top \in E $. 
The following two theorems establish sufficient conditions for the problem (\ref{abstractevolutionproblem}) to be well posed. 
\begin{theorem}
	\label{existenceH}
	Under the additional assumption $\vect f \in L^2([0,T];H)$, there exists a unique weak solution of the problem \eqref{abstractevolutionproblem}, understood in the sense of Definition \ref{weaksolutionofabstractevolutionproblem}.
\end{theorem}
\begin{proof} 
	The existence can be obtained by the variation of constants formula
	\begin{equation} \label{ivan903}
		\Vec{\vect u}(t) =  e^{-t\mathbb{A}}\Vec{\vect u}_0 + \int_0^T  e^{-(t-s)\mathbb{A}}\Vec{ \vect f}(s)ds. 
	\end{equation} 
	The uniqueness is given by parabolic regularisation and can be found in \cite[Theorem 3, p.\,572]{dautraylions} in  a more general setting, where also a proof of existence is obtained by the same method while working directly with the problem \eqref{abstractevolutionproblem}.
\end{proof} 
\begin{remark} 
	It can be easily seen from \eqref{ivan902} that in Theorem \ref{existenceH} one additionally has $\partial_{tt} \vect u \in L^2([0,T];V^*).$ Furthermore, if $\vect f \in L^{\infty}([0,T];H)$ then  $\partial_{tt} \vect u \in L^\infty([0,T];V^*)$. 		
\end{remark} 
\begin{remark} 
	It follows from \eqref{ivan910} and \eqref{ivan903}	that there exists $C>0$ such that 
	\begin{equation*} 
		%\label{ivan911} 
		\| \Vec{\vect u}\|_{L^{\infty}([0,T];E)} \leq Ce^T\left(\|\Vec{\vect u}_0\|+\|\Vec{\vect f}\|_{L^1([0,T];E)}\right),
	\end{equation*}
	from which one directly concludes that 
	\begin{equation} 
		\label{ivan55}  
		\| \vect u\|_{L^{\infty}([0,T];V)}+\|\partial_t \vect u\|_{L^{\infty}([0,T];H)} \leq Ce^{T}\left(\|\vect u_0\|_V+\|\vect u_1\|_H+ \|\vect f\|_{L^1([0,T];H)}\right).
	\end{equation} 
\end{remark} 
Note that as a consequence of \eqref{ivan702} the operator $\mathbb{A}+\lambda\mathbb{I}$ has bounded inverse for every $\lambda \neq 0$. 
\begin{theorem}
	\label{existenceV*}
	Assume that $\vect f, \partial_t \vect f\in L^2([0,T];V^*)$  Then, there exists a unique weak solution in the sense of Definition \ref{weaksolutionofabstractevolutionproblem} of the problem \eqref{abstractevolutionproblem}. 
\end{theorem}
\begin{proof}
	Notice that we actually have $\vect f \in C([0,T];V^*)$. 
	The existence of solution follows from the formula 
	\begin{equation}
		\label{exponentialsemigroupsolution}
		%%\begin{split}
		\Vec{\vect u}(t) =  e^{-t\mathbb{A}} \Vec{\vect {u}}_0 - \int_0^T e^{(s-t)\mathbb{A}} (\mathbb{A} +\RRR  \mathbb{I} \BBB)^{-1}\Bigl(\partial_s \Vec{ \vect f}(s)- \Vec{ \vect f}(s)\Bigr)ds %%\\
		%%& 
		+(\mathbb{A}+\mathbb{I})^{-1}\Vec{\vect f}(t) - e^{-t\mathbb{A}} (\mathbb{A} + \mathbb{I})^{-1}  \Vec{ \vect f}(0),
		%%\end{split}	
	\end{equation} 
	which can be obtained formally from \eqref{ivan903} by using integration by parts. Here we also use the fact that (see \eqref{ivan103}):
	\begin{equation} \label{ivan955} 
		(\mathbb{A}+\mathbb{I})^{-1}(0,\vect f)^\top=\bigl((\mathcal{A}+\mathcal{I})^{-1}\vect f,(\mathcal{A}+\mathcal{I})^{-1}\vect f\bigr), 
	\end{equation} 
	from which it follows that if $f, \partial_t \vect f \in L^2([0,T];V^*)$ then $(\mathbb{A}+\mathbb{I})^{-1}(0,\vect f)^\top$, i.e., $(\mathbb{A}+\mathbb{I})^{-1}(0,\partial_t \vect f)^\top$ is an element of $L^2([0,T];V)$. 
	As in the proof of Theorem \ref{existenceH}, the  uniqueness follows from \cite[Theorem 3, 
	p.\,572]{dautraylions}. (Notice that considering the difference of the solutions associated with two different load densities gives us a solution to the problem with zero load density, which is then necessarily zero by the cited theorem.)
\end{proof}
\begin{remark} 
	It can be easily seen that in Theorem \ref{existenceV*} one additionally has $\partial_{tt} \vect u \in L^{\infty}([0,T];V^*)$. This follows from \eqref{ivan902} and the fact that $\vect f \in L^{\infty} ([0,T];V^*)$. 		
\end{remark}
\begin{remark} 
	It follows from \eqref{ivan910} and \eqref{exponentialsemigroupsolution} that there exists $C>0$ such that 
	\begin{equation*}
		\| \Vec{\vect u}\|_{L^{\infty}([0,T];E)} \leq Ce^T\left(\|\Vec{\vect u}_0\|+\|\Vec{\vect f}(0)\|_{V^*}+\|\partial_t \Vec {\vect f} \|_{L^1([0,T];V^*)}\right),
	\end{equation*}
	from which one directly concludes that
	\begin{equation} \label{ivan992}  
		\| \vect u\|_{L^{\infty}([0,T];V)}+\|\partial_t \vect u\|_{L^{\infty}([0,T];H)} \leq Ce^{T}\left(\|\vect u_0\|_V+\|\vect u_1\|_H+\|\vect f(0)\|_{V^*}+ \|\partial_t \vect f\|_{L^1([0,T];V^*)}\right).
	\end{equation} 
\end{remark} 	
We will now give an overview of the results of \cite{pastukhova}, which we will then extend with the concept of solution discussed in Theorem \ref{existenceV*}. %%which 
While this extension is not considered in \cite{pastukhova}, its validity follows from the formula (\ref{exponentialsemigroupsolution}).
%%analysis inTheorem \ref{existenceV*}. 

Suppose that we are given a sequence of Hilbert spaces $(H_\varepsilon)_{\varepsilon >0}$ endowed with norms $||\cdot||_{H_\varepsilon}$ and some type of weak convergence $\vect u^\varepsilon \xrightharpoonup{H_\varepsilon} \vect u \in H$ of sequences $(\vect u^\varepsilon)\subset H_\varepsilon$. Our assumption on this type of weak convergence is that every weakly convergent sequence $({\vect u}^\varepsilon)_{\varepsilon>0}$ is bounded, i.e., $\limsup_{\varepsilon \to 0} ||\vect u^\varepsilon||_{H_\varepsilon} < \infty. $ We additionally assume that the ``limit space" $H$ is separable. 
\begin{definition}
	We say that a sequence $(\vect u^{\varepsilon})_{\varepsilon>0}\subset H_\varepsilon$ strongly converges to $\vect u \in H$ and write $\vect u^\varepsilon \xrightarrow{H_\varepsilon} \vect u$ if $\vect u^\eps \xrightharpoonup{H_{\eps}} \vect u$ and 
	\begin{equation}\label{ivan35} 
		\lim_{\varepsilon \to 0 } ( \vect u^\varepsilon, \vect v^\varepsilon )_{H_\varepsilon} = ( \vect u, \vect v )_H,
	\end{equation}
	for every weakly convergent sequence $\vect v^\varepsilon \xrightharpoonup{H_\varepsilon} \vect v \in H$, $\vect v^\varepsilon \in H_\varepsilon$. 
\end{definition}
Additionally, we assume the following properties of this abstract weak convergence: 
\begin{eqnarray*}
	&\mbox{(compactness principle)}&  \mbox{ Every bounded sequence contains a weakly convergent subsequence. } \\
	&\mbox{(approximation principle)}&  \mbox{ For every $\vect u \in H$, there exists a sequence $(\vect u^\varepsilon)_{\varepsilon>0}\subset H_\varepsilon$, such that $\vect u^\varepsilon \xrightarrow{H_\varepsilon} \vect u$. } \\
	&\mbox{(norm convergence)}& \mbox{If }\vect u^{\eps} \xrightarrow{H_\varepsilon} \vect u\mbox{, then }\lim_{\eps \to 0} \|\vect u^{\eps}\|_{H_{\eps}}=\|\vect u\|_{H}.  
\end{eqnarray*}
As a consequence of these properties, to guarantee strong convergence $\vect u^{\varepsilon} \xrightarrow{H_{\varepsilon}} \vect u$ it suffices to have the property \eqref{ivan35} or, alternatively, to have $\vect u^{\varepsilon} \xrightharpoonup{H_{\eps}} \vect u$ and norm convergence (see the proof of Lemma \cite[Lemma 6.3]{pastukhova}). 
The following kind of operator convergence is convenient in the analysis of parameter-dependent problems.
\begin{definition}\label{ivan933} 
	Let $(\mathcal{A}_\varepsilon)_{\varepsilon>0}$ be a sequence of non-negative self-adjoint operators acting on the respective spaces $H_\varepsilon.$ Suppose that $\mathcal{A}$ is a non-negative self-adjoint operator on some closed subspace $H_0$ of $H$, and consider the orthogonal projection $P:H\to H_0.$ We say that $\mathcal{A}_\varepsilon$ converge to $\mathcal{A}$ in the weak resolvent sense if
	\begin{equation}
		\label{weakres}
		\forall \lambda>0, \quad \left( \mathcal{A}_\varepsilon + \lambda \RRR \mathcal{I} \BBB \right)^{-1}\vect f^\varepsilon \xrightharpoonup{H_\varepsilon} \left( \mathcal{A} + \lambda \RRR \mathcal{I} \BBB \right)^{-1} P\vect f, \quad \forall (\vect f^\varepsilon)_{\varepsilon>0}, \quad \vect f^\varepsilon \in H_\varepsilon, \quad \vect f^\varepsilon \xrightharpoonup{H_\varepsilon} \vect  f \in H.
	\end{equation}
	Similarly, we say that $\mathcal{A}_\varepsilon$ converge to $\mathcal{A}$ in the strong resolvent sense if
	\begin{equation}
		\label{strongres}
		\forall \lambda>0, \quad \left( \mathcal{A}_\varepsilon + \lambda \mathcal{I} \right)^{-1}\vect f^\varepsilon \xrightarrow{H_\varepsilon} \left( \mathcal{A} + \lambda \mathcal{I} \right)^{-1} P\vect f, \quad \forall (\vect f^\varepsilon)_{\varepsilon>0}, \quad \vect f^\varepsilon \in H_\varepsilon, \quad \vect f^\varepsilon \xrightarrow{H_\varepsilon} \vect  f \in H.
	\end{equation}
\end{definition}
\begin{lemma}
	The convergence \eqref{weakres} is equivalent to the convergence \eqref{strongres}.
\end{lemma}
\begin{proof} 
	The proof is based on a duality argument. Take $\lambda>0$ and consider $(\vect f^{\eps})_{\eps>0}$ such that $\vect f^{\eps}  \xrightharpoonup{H_\varepsilon} \vect f$ and $(\vect g^{\eps})_{\eps>0}$ such that $\vect g^{\eps}  \xrightarrow {H_\varepsilon} \vect g$ . Then one has
	\begin{equation*} 
		\lim_{\eps \to 0}  \left(\vect f^\varepsilon,\left( \mathcal{A}_\varepsilon + \lambda \mathcal{I} \right)^{-1}\vect g^{\eps}\right)_{H_{\eps}}=\lim_{\eps \to 0}  \left(\left( \mathcal{A}_\varepsilon + \lambda \mathcal{I} \right)^{-1}\vect f^\varepsilon,\vect g^{\eps}\right)_{H_{\eps}}=  \left(\left( \mathcal{A} + \lambda I \right)^{-1}P\vect f,\vect g\right)_{H}=\left( \vect f,\left(\mathcal{A} + \lambda \mathcal{I} \right)^{-1}P\vect g\right)_{H}.
	\end{equation*} 
	These equalities show that \eqref{weakres} implies \eqref{strongres}. In a similar fashion, one shows that \eqref{strongres} implies \eqref{weakres}. 
\end{proof} 	
Henceforth we work within the framework of Definition \ref{ivan933}. 
For the sequence  $(\mathcal{A}_\varepsilon)_{\varepsilon>0}$ we construct the associated Hilbert spaces $V_\varepsilon$ endowed with norms $||\cdot||_{V_\varepsilon},$ defined as follows:
\begin{equation*}
	\|\vect u^\varepsilon\|^2_{V_\varepsilon} := ( (\mathcal{A}_\varepsilon+ \mathcal{I})^{1/2}\vect u^\varepsilon,(\mathcal{A}_\varepsilon+ \mathcal{I})^{1/2} \vect u^\varepsilon )_{H_\varepsilon}.
\end{equation*}
Similarly, we define $\|\vect u\|_{V}$, where $V=\mathcal{D}(\mathcal{A}^{1/2})$.  It is easily seen that, since $H$ is separable, the spaces $H_0$ and $V$ are also separable. (Indeed, $H_0$ is a subspace of $H,$ and if $\{h_n\}_{n \in \N}$ is a dense subset of $H$, then $\{(\mathcal{A}+I)^{-1/2}h_n\}_{n \in \N}$ is a dense subset of $V.$)    
\begin{definition}
	Suppose that $(\vect u^\varepsilon)_{\varepsilon>0}\subset V_\varepsilon$,    $\vect u \in V$. We say that $\vect u^\varepsilon$ converge weakly to $\vect u$ in $V_\varepsilon,$ and write $\vect u^\varepsilon \xrightharpoonup{V_\varepsilon} \vect u,$  if 
	\begin{equation}\label{ivan34} 
		\vect u^\varepsilon \xrightharpoonup{H_\varepsilon} \vect u  \quad \mbox{ and } \quad  \limsup_{\varepsilon> 0} ||\vect u^\varepsilon||_{V_\varepsilon} < \infty.
	\end{equation}
	Additionally, we say that \RRR$(\vect u^\varepsilon)_{\eps>0}$  converges \BBB strongly to $\vect u$ in $V_\varepsilon,$ and write $\vect u^\varepsilon \xrightarrow{V_\varepsilon} \vect u,$  if 
	\begin{equation*}
		(\vect u^\varepsilon, \vect v^\varepsilon )_{V_\varepsilon} \to ( \vect u, \vect v)_{V}  \quad \mbox{ for all } \quad \vect v^\varepsilon \xrightharpoonup{V_\varepsilon} \vect v.
	\end{equation*}
\end{definition}

The following statement is \cite[Lemma 6.2, Lemma 6.3]{pastukhova}. 
\begin{lemma}
	If $\vect u^\varepsilon \xrightarrow{V_\varepsilon} \vect u$, then $\vect u^\varepsilon \xrightarrow{H_\varepsilon} \vect u$. Moreover, one has
	$$ \vect u^\varepsilon \xrightarrow{V_\varepsilon} \vect u\  \Longleftrightarrow\ \vect u^\varepsilon \xrightharpoonup{V_\varepsilon} \vect u\ \textrm{ and }\ \|\vect u^{\eps}\|_{V_{\eps}} \to \|\vect u\|_V.  $$
\end{lemma}
It can be shown that there exists a dense subset $S \subset V$, such that for every $\vect z \in S$ there exists a subsequence $(\vect z^{\eps})_{\eps>0}$ such that $\vect z^{\eps} \xrightarrow{V_{\eps}} \vect z$ (see \cite[Lemma 6.5]{pastukhova}).

We also introduce convergence notions convenient for the analysis of time-dependent problems.   
\begin{definition}
	Suppose that a sequence $(\vect u^\varepsilon)_{\varepsilon>0}\subset L^2([0,T];H_\varepsilon)$ is bounded. We say that \RRR $(\vect u^\varepsilon)_{\eps>0}$ weakly converges \BBB to $\vect u \in L^2([0,T];H)$, and write $\vect u^\varepsilon \xrightharpoonup{t,H_\varepsilon} \vect u,$ if 
	\begin{equation*}
		\int_0^T ( \vect u^\varepsilon(t), \vect v^\varepsilon )_{H_\varepsilon} \varphi(t)dt \to \int_0^T ( \vect u(t), \vect v ) \varphi(t)dt, 
	\end{equation*}
	for all $\vect v^\varepsilon \xrightarrow{H_\varepsilon} \vect v$ and $\varphi \in L^2(0,T)$.
\end{definition}
\begin{definition}
	Suppose that $(\vect u^\varepsilon)_{\varepsilon>0}\subset L^2([0,T];V_\varepsilon)$. We say that \RRR$(\vect u^\varepsilon)_{\varepsilon>0} $ \BBB weakly converges to $\vect u \in L^2([0,T];V),$ and write $\vect u^\varepsilon \xrightharpoonup{t,V_\varepsilon} \vect u,$ if 
	\begin{equation*}
		\vect u^\varepsilon \xrightharpoonup{t,H_\varepsilon} \vect u  \quad \mbox{ and } \quad  \limsup_{\varepsilon\to 0} \int_0^T (\mathcal{A}_\varepsilon^{1/2} \vect u^\varepsilon, \mathcal{A}_\varepsilon^{1/2}\vect u^\varepsilon )_{H_\varepsilon}dt < \infty.
	\end{equation*}
\end{definition}
In the same way we can define the weak convergence in $L^p([0,T];H_{\eps})$ and $L^p([0,T];V_{\eps})$, $1\leq p \leq \infty$ (which denote by $\xrightharpoonup{t,p,H_{\eps}}$ and $\xrightharpoonup{t,p,V_{\eps}}$, respectively.)
The following lemma is stated in \cite[Lemma 4.7]{pastukhova}.
\begin{lemma} \label{ivan56} The spaces $L^2([0,T];H_{\varepsilon})$ and $L^2([0,T];V_{\varepsilon})$ satisfy the weak compactness principle. 
\end{lemma} 
The next lemma can also be easily established, see \cite[ Lemma 2.3, Lemma 4.5]{pastukhova}.
\begin{lemma} \label{lmivan401} 
	If $\vect u^{\eps} \xrightharpoonup{H_{\eps}} \vect u$, then $\liminf_{\eps \to 0} \|\vect u^{\eps}\|_{H_{\eps}} \geq \| \vect u\|_{H}$. The same is valid for the weak convergence in $V_{\eps}$. If  $\vect u^{\eps} \xrightharpoonup{t,H_{\eps}} \vect u$, then $\liminf_{\eps \to 0} \|\vect u^{\eps}\|_{L^2([0,T];H_{\eps})} \geq \|\vect u\|_{L^2([0,T];H)}$. The same is valid for the weak convergence in $L^2([0,T];V_{\eps})$. 
\end{lemma} 	
In the natural way, by components, we define the weak and strong convergence in $E_{\varepsilon}=V_\varepsilon\times H_\varepsilon$ as well as the weak convergence in $L^2([0,T];E_{\varepsilon})$. Also the space $E_0=V \times H_0$ and the projection $P$ onto $E_0$ are defined in a  natural way, the latter being given by $P\Vec{\vect v}=( v_1,P v_2)^\top$. In an obvious way we also define operators $\mathbb{A}_{\eps}$ and $\mathbb{A}$. 
The following theorem is a basic tool for proving weak or strong convergence of solutions to $\varepsilon$-parametrised evolution problems, understood in the sense of Definition \ref{weaksolutionofabstractevolutionproblem}. The theorem can be found in \cite[Theorem 5.2, Theorem 7.1]{pastukhova}. The first part is easily proved by combining the Laplace transform and a compactness result as in Theorem \ref{thmivan551}. 
\begin{theorem} \label{thmivan104} 
	\RRR Let $(\mathcal{A}_\varepsilon)_{\varepsilon>0}$ be a sequence of non-negative self-adjoint operators in $H_\varepsilon$ that converge to a non-negative self-adjoint operator $\mathcal{A}$ in some subspace $H_0 \leq H$ in the sense of weak resolvent convergence.
	\begin{enumerate} 
\item 	If $\Vec{\vect f}^{\varepsilon} \xrightharpoonup{E_{\eps}} \Vec{\vect f} $, then 
$$ (\mathbb{A}_{\eps}+\lambda\mathbb{I})^{-1} \Vec{\vect f}^{\varepsilon} \xrightharpoonup{E_{\eps}} (\mathbb{A}+\lambda\mathbb{I})^{-1} \Vec{\vect f}, \quad \forall \lambda>1.  $$

\item 	If $\Vec{\vect f}^{\varepsilon} \xrightharpoonup{E_{\eps}} \Vec{\vect f} $, then $e^{-t\mathbb{A}_{\eps}}\Vec{\vect f}^{\eps} \xrightharpoonup{t,E_{\eps}} e^{-t\mathbb{A}}P\vect f$ for every $T>0$. If $\Vec{\vect f}^{\eps} \xrightarrow{E_{\eps}}\Vec{ \vect f} \in E_0$, then $e^{-t\mathbb{A}_{\eps}}\Vec{ \vect f}^{\eps} \to e^{-t\mathbb{A}}\Vec{ \vect f} $ for every $t \geq 0$. 
	\end{enumerate} 
		\BBB
\end{theorem} 
\begin{remark} 
	The pointwise convergence $e^{-t\mathbb{A}_{\eps}}\Vec{\vect f}^{\eps} \xrightarrow{E_{\eps}} e^{-t\mathbb{A}}P\vect f$ in Theorem \ref{thmivan104} does not necessarily hold if we only assume that $\Vec{\vect f}^{\eps} \xrightarrow{E_{\eps}}\Vec{ \vect f} \in E,$ see \cite[p.\,2267]{pastukhova}. 
\end{remark} 		
A version of the following theorem can be found in \cite[Theorem 5.2, Theorem 5.3]{pastukhova}. 
\begin{theorem}\label{thmivan551} 
	Let $(\mathcal{A}_\varepsilon)_{\varepsilon>0}$ be a sequence of non-negative self-adjoint operators in $H_\varepsilon$ that converge to a non-negative self-adjoint operator $\mathcal{A}$ in some subspace $H_0 \leq H$ in the sense of weak resolvent convergence. Let $T>0$ and $(\vect u^\varepsilon)_{\varepsilon>0}$ be a sequence of weak solutions of the evolution problems \eqref{abstractevolutionproblem} where ${\mathcal A}$ is replaced by ${\mathcal A_\varepsilon},$ with initial data $\vect u_0^\varepsilon \in V_\varepsilon$, $\vect u_1^\varepsilon \in H_\varepsilon$ and right-hand sides $ \vect f^\varepsilon \in L^2([0,T];H_\varepsilon)$ such that
	\begin{equation} \label{ivan91111} 
		\vect u_0^\varepsilon \xrightharpoonup{V_\varepsilon} \vect u_0 \in V, \quad \vect u_1^\varepsilon \xrightharpoonup{H_\varepsilon} \vect u_1 \in H, \quad \vect f^\varepsilon \xrightharpoonup{t,H_\varepsilon} \vect f \in L^2([0,T];H).
	\end{equation}
	Then one has 
	\begin{equation}\label{ivan102}
		\vect u^\varepsilon \xrightharpoonup{t,V_\varepsilon} \vect u \in L^2([0,T];V), \quad \partial_t \vect u^\varepsilon \xrightharpoonup{t,H_\varepsilon} \partial_t \vect u \in L^2([0,T];H_0), 
	\end{equation}
	where $\vect u$ is the weak solution of the evolution problem \eqref{abstractevolutionproblem} for the limit operator $\mathcal{A}$, with the initial data $\vect u_0 \in V$, $P \vect u_1 \in H_0$ and the right-hand side $P \vect f \in L^2([0,T];H_0)$.
\end{theorem}
\begin{proof} 
	We write the problem \eqref{abstractevolutionproblem} in the form \eqref{abstractevolutionproblemrephrased}. As a consequence of \eqref{ivan91111}, \eqref{ivan55}, and Lemma \ref{ivan56}, there exist $\Vec {\vect u}_l=(\vect u_l, \partial_t \vect u_l) \in L^2([0,T];E_0)$ such that  the convergence \eqref{ivan102} holds. 
	Due to above mentioned bounds and weak convergence, we have 
	\begin{equation} \label{ivan110} 
		\mathcal{L}(\Vec{\vect u}^{\eps})(\lambda) \RRR \xrightharpoonup{E_{\eps}} \BBB \mathcal{L}(\Vec{\vect u}_l)(\lambda),\qquad \lambda>1, 
	\end{equation} 
	where $\mathcal{L}$ denotes the Laplace transform (where extend $\vect f^{\eps}$ and $\vect f$ by zero on $(T,\infty)$). 
	On the one hand, the Laplace transform $\mathcal{L}$ of the solution of the 
	$\eps$-parametrised equation is then given by  
	\begin{equation} \label{ivan111} 
		\mathcal{L}(\Vec{\vect u}^{\eps})(\lambda)=(\mathbb{A}_{\eps}+\lambda\mathbb{I})^{-1}\mathcal{L}(\Vec{\vect f}^{\eps})(\lambda)+(\mathbb{A}_{\eps}+\lambda\mathbb{I})^{-1} \Vec{\vect u}^{\eps}_0, \qquad\lambda>1.	
	\end{equation} 
	On the other hand, the Laplace transform of the solution of the limit problem is given by 
	\begin{equation} \label{ivan112} 
		\mathcal{L}(\Vec{\vect u})(\lambda)=(\mathbb{A}+\lambda\mathbb{I})^{-1}\mathcal{L}(P\Vec{\vect f})(\lambda)+(\mathbb{A}+\lambda\mathbb{I})^{-1}P \Vec{\vect u}_0,\qquad \lambda>1. 	
	\end{equation} 
Using \eqref{ivan103}, we infer that for every sequence $(\Vec{\vect f}^{\eps})_{\eps>0}\subset E_{\eps}$ such that \RRR$\Vec{\vect f}^{\eps} \xrightharpoonup{E_{\eps}} \Vec{\vect f} \in E$\BBB, and every $\lambda \neq 0$ we have (see \eqref{weakres})
	$$(\mathbb{A}+\lambda\mathbb{I})^{-1} \Vec{\vect f}^{\eps} \RRR \xrightharpoonup{E_{\eps}} \BBB (\mathbb{A}+\lambda\mathbb{I})^{-1} P \Vec{\vect f}. $$ 
	Using \eqref{ivan110}, \eqref{ivan111},  \eqref{ivan112}, and the fact that for $\lambda>1$ one has \RRR $\mathcal{L}(\Vec{\vect f}^{\eps})(\lambda) \xrightharpoonup{E_{\eps}} \mathcal{L}(\Vec{\vect f})(\lambda)$ \BBB and $\mathcal{L}(P\Vec{\vect f})(\lambda)=P\mathcal{L}(\Vec{\vect f})(\lambda),$ we infer that for every $\lambda>1$ one has \RRR $\mathcal{L}(\Vec{\vect u}_l)=\mathcal{L}(\Vec{\vect u}),$ \BBB and the claim follows. 
\end{proof} 

We proceed to the strong convergence analogue of Theorem \ref{thmivan551}, see \cite[Theorem 7.2]{pastukhova}. 
%\begin{definition}
%	Let $(\vect u_\varepsilon)_{\varepsilon>0}$ be a sequence of elements in $L^2([0,T];H_\varepsilon)$. We say that $\vect u_\varepsilon$ strongly converges to $\vect u \in L^2([0,T];H)$ in $L^2([0,T];H_\varepsilon)$, and write $\vect u_\varepsilon \xrightarrow{L^2([0,T];H_\varepsilon)} \vect u$ if 
%	\begin{equation*}
%	\int_0^T (\vect u_\varepsilon(t),  v_\varepsilon(t) )_{H_\varepsilon} dt \to \int_0^T (\vect u(t), \vect v(t) )_H dt, 
%	\end{equation*}
%	for every $ v_\varepsilon \xrightharpoonup{H_\varepsilon} \vect v$.
%\end{definition}
\begin{theorem}\label{thmivan552} 
	Let $(\mathcal{A}_\varepsilon)$ be a sequence of non-negative self-adjoint operators in $H_\varepsilon$ that converges to a non-negative self-adjoint operator $\mathcal{A}$ in some subspace $H_0 \leq H$ in the sense of strong resolvent convergence. Let $T>0$ and $(\vect u^\varepsilon)_{\varepsilon>0}$ be a sequence of weak solutions of the evolution problems \eqref{abstractevolutionproblem} where ${\mathcal A}$ is replaced by ${\mathcal A}_\varepsilon,$ with initial data $\vect u_0^\varepsilon \in V_\varepsilon$, $\vect u_1^\varepsilon \in H_\varepsilon$ and right-hand sides $ \vect f^\varepsilon \in L^2([0,T];H_\varepsilon)$ such that
	\begin{equation}
		\label{ivan912} 
		\begin{split}
			&\vect u_0^\varepsilon \xrightarrow{V_\varepsilon} \vect u_0 \in V, \quad \vect u_1^\varepsilon \xrightarrow{H_\varepsilon} \vect u_1 \in H_0, \\   &\vect f_\varepsilon(t) \xrightarrow{H_\varepsilon} \vect f(t) \in H_0\quad  \mbox{a.e. $t \in [0,T],$} \quad 
			 \int_0^T\bigl\|\vect f^\varepsilon(s)\bigr\|^2_{H_\varepsilon}ds \to \int_0^T \bigl\|\vect f(s)\bigr\|^2_{H}ds.
		\end{split}
	\end{equation}
	Then for every $t \in [0,T]$ one has 
	\begin{equation}\label{ivan913} 
		\vect u_\varepsilon(t) \xrightarrow{V_\varepsilon} \vect u(t) \in V, \quad \partial_t \vect u_\varepsilon(t) \xrightarrow{H_\varepsilon} \partial_t \vect u(t) \in H_0, 
	\end{equation}
	where $\vect u$ is the weak solution of the evolution problem of \eqref{abstractevolutionproblem} for the operator $\mathcal{A}$, with the initial data $\vect u_0 \in V$, $\vect u_1 \in H_0$ and the right-hand side $\vect f \in L^2([0,T];H_0)$.
\end{theorem}
\begin{proof} 
	Again we use \eqref{abstractevolutionproblemrephrased}, formula
	\eqref{ivan903} and Theorem \ref{thmivan104}. The proof follows by using Lebesgue theorem on {\rm Dom}inated convergence and the fact that as a consequence of Lemma \ref{lmivan401} we have $\int_0^t \|\vect f^\varepsilon(s)\|^2_{H_\varepsilon}ds \to \int_0^t \|\vect f(s)\|^2_{H}ds$, for every $t\leq T$. 
\end{proof} 	
Theorem \ref{thmivan552} can be generalized as follows. %%%This generalization is straightforward, using formula \eqref{exponentialsemigroupsolution}, but is not given in \cite{pastukhova}. 
\begin{theorem}
	\label{thmivan553} 
	Let $(\mathcal{A}_\varepsilon)_{\eps>0}$ be a sequence of non-negative self-adjoint operators in $H_\varepsilon$ that converge to a non-negative self-adjoint operator $\mathcal{A}$ in some subspace $H_0 \leq H$ in the sense of strong resolvent convergence. Let $T>0$ and $(\vect u^\varepsilon)_{\varepsilon>0}$ be a sequence of weak solutions of the evolution problems \eqref{abstractevolutionproblem} where ${\mathcal A}$ is replaced by ${\mathcal A}_\varepsilon,$ with initial data $\vect u_0^\varepsilon \in V_\varepsilon$, $\vect u_1^\varepsilon \in H_\varepsilon$ and right-hand sides $ \vect f^\varepsilon \in L^2([0,T];V_\varepsilon^*)$, $ \partial_t\vect f^\varepsilon \in L^2([0,T];V_\varepsilon^*)$ such that the sequences  $(\vect f^\varepsilon)_{\eps>0}$, $(\partial_t\vect f^\varepsilon)_{\eps>0}$ are bounded in $L^2([0,T];\RRR V^*_{\eps} \BBB)$ and
	\begin{equation*}
		\begin{aligned}
			&\vect u_0^\varepsilon \xrightarrow{V_\varepsilon} \vect u_0 \in V, \quad \vect u_1^\varepsilon \xrightarrow{H_\varepsilon} \vect u_1\in H , 
			\\[0.3em]
			&\left(\mathcal{A}_\varepsilon + \RRR \mathcal{I} \BBB \right)^{-1} \vect f^\varepsilon  \xrightharpoonup{t, V_\varepsilon } \left(\mathcal{A} + \RRR \mathcal{I} \BBB \right)^{-1} \vect f  \in L^2([0,T];V),  %\\
			%	\left(\mathcal{A}_\varepsilon + I \right)^{-1} \partial_t\vect f^\varepsilon  \xrightharpoonup{t, V_\varepsilon } \left(\mathcal{A} + I \right)^{-1} \partial_t\vect f \in L^2([0,T];V),
		\end{aligned}
	\end{equation*}
	where $\vect f, \partial_t \vect f \in L^2(O,T;V^*)$. 
	Then one has
	\begin{equation*}
		\vect u^\varepsilon \xrightharpoonup{t,V_\varepsilon} \vect u \in L^2([0,T];V), \quad \partial_t \vect u^\varepsilon \xrightharpoonup{t,H_\varepsilon} \partial_t \vect u \in L^2([0,T];H_0), 
	\end{equation*}
	where $\vect u$ is the weak solution of the evolution problem \eqref{abstractevolutionproblem} for the operator $\mathcal{A}$, with the initial data $\vect u_0 \in V$, $P \vect u_1 \in H_0$ and the right-hand side $\vect f \in L^2([0,T];V^{*})$.
	Furthermore, if we assume that 
	\begin{equation}
		\label{ivan561} 
		\begin{aligned}
			&\vect u_0^\varepsilon \xrightarrow{V_\varepsilon} \vect u_0 \in V, \quad \vect u_1^\varepsilon \xrightarrow{H_\varepsilon} \vect u_1\in H_0 , \\[0.3em]
			&\left(\mathcal{A}_\varepsilon + \RRR \mathcal{I} \BBB \right)^{-1} \vect f^\varepsilon(0)  \xrightarrow{V_\varepsilon } \left(\mathcal{A} + \RRR \mathcal{I} \BBB \right)^{-1} \vect f(0) \in  V, \\[0.3em]
			&\left(\mathcal{A}_\varepsilon + \RRR \mathcal{I} \BBB \right)^{-1} \partial_t\vect f^\varepsilon (t) \xrightarrow{ V_\varepsilon } \left(\mathcal{A} + \RRR \mathcal{I} \BBB \right)^{-1} \partial_t\vect f(t) \in V, \quad \mbox{ for a.e. $t\in [0,T]$}, \\[0.3em]
			&\int_0^T \bigl\|\left(\mathcal{A}_{\eps} + \RRR \mathcal{I} \BBB \right)^{-1}\partial_s \vect f^{\varepsilon}(s)\bigr\|_{V_{\eps}}^2\,ds \to \int_0^T \bigl\|\left(\mathcal{A} + \RRR \mathcal{I} \BBB \right)^{-1}\partial_s \vect f(s)\bigr\|_V^2\,ds, 
		\end{aligned}
	\end{equation}
	where $\vect f,\ \partial_t \vect f \in L^2([0,T];V^*),$ 
	then we have 
	\begin{equation*}
		\vect u_\varepsilon(t) \xrightarrow{V_\varepsilon} \vect u(t) \in V, \quad \partial_t \vect u_\varepsilon(t) \xrightarrow{H_\varepsilon} \partial_t \vect u(t) \in H_0, \quad \RRR \forall t \in [0,T],\BBB
	\end{equation*}
	where $\vect u$ is the weak solution of the evolution problem \eqref{abstractevolutionproblem} for the operator $\mathcal{A}$, with the initial data $\vect u_0 \in V$, $ \vect u_1 \in H_0$ and the right hand side $\vect f \in L^2([0,T];V^*)$.
\end{theorem}
\begin{proof}
	The argument follows the proofs of Theorem \ref{thmivan551} and Theorem \ref{thmivan552}, by using the formula \eqref{exponentialsemigroupsolution} instead of \eqref{ivan903} (see also \eqref{ivan955}). 
	
	To prove the first part, notice that the Laplace transform of the solution $\Vec{\vect u}^{\eps}$, respectively $\Vec{\vect u},$ is given by formula \eqref{ivan111}, respectively \eqref{ivan112}. (Note that in \eqref{ivan112} we replace $P\Vec{\vect f}$ with $\Vec{\vect f}$.) This is established by a density argument, using the fact that $L^2([0,T];H_{\eps})$ is dense in  $L^2([0,T];V^*_{\eps})$, respectively that $L^2([0,T];H)$ is dense in  $L^2([0,T];V^*)$. 
	
	To prove the second part, use the second part of Theorem \ref{thmivan104} and notice that \eqref{ivan561} implies that for all $t\in[0,T]$ one has
	 \begin{align*}
	 &\left(\mathcal{A}_{\eps} + \RRR \mathcal{I} \BBB \right)^{-1}\vect f^{\eps}(t)  \xrightarrow{V_\varepsilon } \left(\mathcal{A}_{\eps} + \RRR \mathcal{I} \BBB \right)^{-1}\vect f(t),\\[0.4em]
	&\int_0^t e^{(s-t)\mathbb{A}_{\eps}} (\mathbb{A}_{\eps}+\RRR \mathbb{I} \BBB)^{-1}\bigl(0, \vect f^{\eps}(s)\bigr)^\top ds \xrightarrow{E_{\eps}} \int_0^t e^{(s-t)\mathbb{A}} (\mathbb{A}+\RRR \mathbb{I} \BBB)^{-1}\bigl(0, \vect f(s)\bigr)^\top ds, \\[0.4em]
	&\int_0^t e^{(s-t)\mathbb{A}_{\eps}} (\mathbb{A}_{\eps}+\RRR \mathbb{I} \BBB)^{-1}\bigl(0,\partial_s \vect f^{\eps}(s)\bigr)^\top ds \xrightarrow{E_{\eps}}\int_0^t e^{(s-t)\mathbb{A}} (\mathbb{A}+\RRR \mathbb{I} \BBB)^{-1} \bigl(0,\partial_s \vect f(s)\bigr)^\top ds,
	\end{align*} 
	as a consequence of the dominated convergence theorem and Lemma \ref{lmivan401}. 
\end{proof}
\begin{remark} \label{ivan1013} 
	It is easy to see that in Theorem \ref{thmivan551}, Theorem \ref{thmivan552}, and Theorem \ref{thmivan553} the sequences $(\|\vect u^{\eps}\|_{V_{\eps}})_{\varepsilon>0}$ and $(\|\partial_t \vect u^{\eps}\|_{H_{\eps}})_{\varepsilon>0}$ are bounded in $L^{\infty}(0,T)$. 
\end{remark} 	
\begin{remark} \label{ivan71} 
	The claims of Theorem \ref{thmivan551}, Theorem \ref{thmivan552} and Theorem \ref{thmivan553} can be strengthened slightly. Namely, it suffices to require that $\vect f^{\eps} \xrightharpoonup{t,1,H_{\eps}} \vect f \in L^1\bigl([0,T]; H\bigr)$ in \eqref{ivan91111} to obtain the convergence  	
	\[
	\vect u^\varepsilon \xrightharpoonup{t,\infty,V_\varepsilon} \vect u \in L^\infty\bigl([0,T]; V\bigr), \qquad \partial_t \vect u^\varepsilon \xrightharpoonup{t,\infty,H_\varepsilon} \partial_t \vect u \in L^\infty\bigl([0,T]; H_0\bigr). 
	\]
	Similarly, it suffices to require $\int_0^T\|\vect f^\eps\|_{H_{\eps}} \,ds \to \int_0^T\|\vect f(s)\|\,ds$ in \eqref{ivan912} to obtain \eqref{ivan913}. The statement of Theorem \ref{thmivan553} can also be strengthened accordingly. 
\end{remark} 	
\subsection{Additional claims} 
In the membrane space, under no additional assumptions on the symmetries of the set $Y_0$, the function $\tilde{\beta}^{\rm memb}(\lambda)$ (see (\ref{betamemb1}) as well as the parts A of Sections \ref{hlleps}, \ref{epsllh}) is a symmetric matrix that is not necessarily diagonal. We will first prove two lemmata and then a proposition  concerning the problem \eqref{generalizedeigenvaluemembrane}, which involves the function \RRR $\tilde{\beta}_{\delta}^{\rm memb}(\lambda)$  defined in Section \ref{limreseq}. \BBB
\begin{lemma}\label{lmivanocjena1} There exists $C>0$ such that for every $\lambda > 0$ we have:
	\begin{equation*}
		\left( \bigl(\tilde{\beta}_{\delta}^{\rm memb}\bigr)'(\lambda) \xi, \xi \right) > C|\xi|^2, \quad \forall \xi \in \R^2, \xi \neq 0.
	\end{equation*}
\end{lemma}
\begin{proof}
	We will give the proof for $\delta \in (0,\infty),$ and the other cases can be treated analogously. Notice that for the function $\lambda \mapsto {(\tilde{\eta}_n - \lambda)}^{-1}{\lambda^2}$ one has
	\begin{equation*}
		\left(\frac{\lambda^2}{\tilde{\eta}_n - \lambda}\right)' = -1 + \frac{\tilde{\eta}_n^2}{(\tilde{\eta}_n - \lambda)^2}.
	\end{equation*}
		It follows that
		\begin{equation*}
			\begin{split}
				\bigl(\tilde{\beta}_{\delta}^{\rm memb}\bigr)'(\lambda)  &= \vect{I}_{2 \times 2}\langle \rho \rangle+ \sum_{n \in \N} \frac{\tilde{\eta}_n^2}{(\tilde{\eta}_n - \lambda)^2}\bigl\langle\rho_0\overline{(\tilde{\vect \varphi}_n)_* }\bigr\rangle \cdot \bigl\langle\rho_0\overline{ (\tilde{\vect \varphi}_n)_* }\bigr\rangle^\top - \sum_{n \in \N}\bigl\langle\rho_0\overline{(\tilde{\vect \varphi}_n)_* }\bigr\rangle \cdot \bigl\langle\rho_0\overline{ (\tilde{\vect \varphi}_n)_* }\bigr\rangle^\top \\[0.4em]
				&=
				\vect{I}_{2 \times 2} \langle \rho_1 \rangle + \sum_{n \in \N} \frac{\tilde{\eta}_n^2}{(\tilde{\eta}_n - \lambda)^2}\bigl\langle\rho_0 \overline{(\tilde{\vect \varphi}_n)_* }\bigr\rangle \cdot \bigl\langle\rho_0\overline{ (\tilde{\vect \varphi}_n)_* }\bigr\rangle^\top.
			\end{split}
		\end{equation*}
	Here $\tilde {\vect{\varphi}}_n,$ $n\in{\mathbb N},$ are those eigenfunctions of the operator $\tilde{\mathcal{A}}_{00,\delta}$ associated with the eigenvalues $\tilde{\eta}_n,$ $n\in{\mathbb N},$ that satisfy
	$$ \bigl\langle\rho_0 \overline{(\tilde{\vect \varphi}_n)_* }\bigl\rangle \neq 0,  $$
	and we have used the identity
	$$\vect{I}_{2 \times 2} \langle \rho_0 \rangle=\sum_{n \in \N}\bigl\langle\rho_0\overline{(\tilde{\vect \varphi}_n)_* }\bigr\rangle \cdot \bigl\langle\rho_0\overline{ (\tilde{\vect \varphi}_n)_* }\bigr\rangle^\top,\quad \alpha,\beta=1,2. $$
	
	The proof of the required estimate is concluded by noting that for $\xi \neq 0$ one has 
	\begin{equation*}
		\left(\bigl(\tilde{\beta}_{\delta}^{\rm memb}\bigr)'(\lambda) \xi, \xi \right) = \langle \rho_1 \rangle | \xi |^2 + \sum_{n \in \N} \frac{\tilde{\eta}_n^2}{(\tilde{\eta}_n - \lambda)^2} \left(\bigl\langle\rho_0\overline{ (\tilde{\vect \varphi}_n)_* }\bigr\rangle, \xi \right)^2 > \langle \rho_1\rangle |\xi|^2.
	\end{equation*}
\end{proof}
\begin{lemma}\label{lmivan2}
	Let $\lambda_0 >0$ be such that there exists a nontrivial solution $\vect \funcA \in H^1_{\gamma_D}(\omega;\R^2)$ of the problem \eqref{generalizedeigenvaluemembrane}. Then there exists $\eta>0$ such that for each  $\lambda \in ( \lambda_0, \lambda_0 + \eta)$ the problem \eqref{generalizedeigenvaluemembrane} has only the trivial solution.
\end{lemma}
\begin{proof}
	The problem \eqref{generalizedeigenvaluemembrane} can be reformulated as follows: 
	\begin{equation*}
		\begin{split} 
			\left( (\mathcal{A}_{\delta}^{\rm memb})^{1/2} \vect \funcA, (\mathcal{A}_{\delta}^{\rm memb})^{1/2}\vect \varphi\right)  =  \left( \RRR \tilde{\beta}_{\delta}^{\rm memb}(\lambda)\BBB\vect \funcA, \vect \varphi\right), \quad \forall \vect \varphi \in H^1_{\gamma_D}(\omega;\R^2),   
		\end{split}
	\end{equation*}
	where $(\mathcal{A}_{\delta}^{\rm memb})^{1/2}$ is a self-adjoint positive square root of the operator $\mathcal{A}_{\delta}^{\rm memb},$ which has compact inverse. Thus, the above problem can be rewritten as
	\begin{equation*}
		\left( (\mathcal{A}_{\delta}^{\rm memb})^{1/2} \vect \funcA, (\mathcal{A}^{\rm memb}_{\delta})^{1/2}\vect \varphi\right)  = \left( (\mathcal{A}_{\delta}^{\rm memb})^{-1/2} \tilde{\beta}_{\delta}^{\rm memb} (\lambda) \vect \funcA,  (\mathcal{A}_{\delta}^{\rm memb})^{1/2} \vect\varphi\right), \quad \forall \vect \varphi \in H^1_{\gamma_D}(\omega;\R^2). 
	\end{equation*}
	By substituting $\vect v = (\mathcal{A}_{\delta}^{\rm memb})^{1/2} \vect \funcA$, we have reduced the problem \eqref{generalizedeigenvaluemembrane} to the following equivalent problem: find $\vect v \in L^2(\omega;\R^2)$ that is an eigenfunction for $(\mathcal{A}_{\delta}^{\rm memb})^{-1/2} \tilde{\beta}_{\delta}^{\rm memb}(\lambda) (\mathcal{A}_{\delta}^{\rm memb})^{-1/2}$ with eigenvalue $\mu^\lambda = 1$, i.e.,
	\begin{equation*}
		(\mathcal{A}_{\delta}^{\rm memb})^{-1/2} \tilde{\beta}_{\delta}^{\rm memb}(\lambda) (\mathcal{A}_{\delta}^{\rm memb})^{-1/2} \vect v = \vect v.
	\end{equation*}
	The operator $ (\mathcal{A}_{\delta}^{\rm memb})^{-1/2} \tilde{\beta}_{\delta}^{\rm memb}(\lambda) (\mathcal{A}_{\delta}^{\rm memb})^{-1/2}$ is compact and its positive eigenvalues, in decreasing order, are characterised by the variational principle
	 \begin{equation*}
		\mu_k^\lambda = \max_{V<L^2(\omega;\R^2),\ \dim V=k}\ \ \min_{x \in V,\  ||x||=1} \left( \mathcal{A}_{\rm memb}^{-1/2}\RRR \tilde{\beta}^{\rm memb}(\lambda) \BBB \mathcal{A}_{\rm memb}^{-1/2} x,x \right),\qquad k=1,2,\dots.
	\end{equation*}
	Denote by $k_1$ the index of the eigenvalue $1 = \mu_{k_1}^{\lambda_0},$ which can clearly be done  due to the assumption on $\lambda_0$. Next, denote by $k_2$ the index of the next smaller eigenvalue $\mu_{k_2}^{\lambda_0} < 1$.
	Furthermore, notice that for $\lambda > \lambda_0$ one has 
	\begin{equation*}
		\begin{split}
			(\mathcal{A}_{\delta}^{\rm memb})^{-1/2} \tilde{\beta}_{\delta}^{\rm memb}(\lambda) (\mathcal{A}_{\delta}^{\rm memb})^{-1/2} &= (\mathcal{A}_{\delta}^{\rm memb})^{-1/2} \tilde{\beta}_{\delta}^{\rm memb}(\lambda_0) (\mathcal{A}_{\delta}^{\rm memb})^{-1/2}\\[0.35em] 
			& + (\lambda - \lambda_0)(\mathcal{A}_{\delta}^{\rm memb})^{-1/2} (\tilde{\beta}_{\delta}^{\rm memb})'(\lambda_0) (\mathcal{A}_{\delta}^{\rm memb})^{-1/2} 
			%%\\
			%%&
			+O(|\lambda-\lambda_0|^2),
		\end{split}
	\end{equation*}
	where $\|O(|\lambda-\lambda_0|^2\|\leq C|\lambda-\lambda_0|^2$ for some $C>0$. 
	For this reason, by virtue of Lemma \ref{lmivanocjena1}, one has $\mu_{k_1}^\lambda > 1$, \dots, $\mu_{k_2-1}^{\lambda}>1$, $\mu_{k_2}^{\lambda}>\mu_{k_2}^{\lambda_0}$ whenever $\lambda \in (\lambda_0,\lambda_0+\eta)$, for some $\eta>0$. Due to the continuity of $\mu_{k_2}^\lambda$ with respect to $\lambda$, we can redefine $\eta>0$ so that for every $\lambda \in (\lambda_0, \lambda_0 + \eta)$ we have  $\mu_{k_2}^{\lambda_0}<\mu_{k_2}^\lambda< 1$. Thus, for $\lambda \in  (\lambda_0, \lambda_0 + \eta)$, unity is not the eigenvalue of $\mathcal{A}_{\rm memb}^{-1/2} \RRR \tilde{\beta}^{\rm memb}(\lambda)\BBB \mathcal{A}_{\rm memb}^{-1/2}$.
\end{proof}
\begin{proposition}
	\label{generalizedbetatheorem}
	The set of all $\lambda>0$ for which the problem \eqref{generalizedeigenvaluemembrane} has a nontrivial solution $\funcA \in H^1_{\gamma_D}(\omega;\R^2)$ is at most countable.  
\end{proposition}

\begin{proof}
	Using the preceding lemma, we can define the following family of disjoint intervals: $$\mathcal{F}=\{[ \lambda, \lambda + \eta_\lambda) \subset \R, \quad \mbox{ $\lambda> 0$ is such that the problem \eqref{generalizedeigenvaluemembrane} has a nontrivial solution}   \},
	$$ 
	where $\eta_\lambda$ is provided by Lemma \ref{lmivan2}.
	%%By using Lemma \ref{lmivan2} this family is disjoint and 
	Such a family can be at most countable, from which the claim follows. 
\end{proof}

Finally, we prove a lemma whose variants we used on several occasions within Section \ref{proofs_sec}. We begin by introducing some notation. We define the following norm and seminorm on $H^1_{\Gamma_{\rm D}} (\Omega;\R^3):$ 
\begin{align*}
	\| \vect u\|_{h,\epsh}&= \|\vect u\|_{L^2(\R^3)}+  \| \sym \nabla_h  \tilde{\vect u} \|_{L^2(\Omega;\R^{3\times 3})}+\epsh\|\sym \nabla_h \mathring{\vect u} \|_{L^2(\Omega;\R^{3 \times 3})}, \\[0.4em]
	\| \vect u\|_{s,h,\varepsilon}&=   \| \sym \nabla_h  \tilde{\vect u} \|_{L^2(\Omega;\R^{3\times 3})}+\epsh\|\sym \nabla_h \mathring{\vect u} \|_{L^2(\Omega;\R^{3 \times 3})}, 
\end{align*}
where we have for every $\vect u \in  H^1_{\Gamma_{\rm D}}$ a decomposition $\vect u=\tilde{\vect u}+\mathring{\vect u}$ is employed, with both $\tilde{\vect u}$ and $\mathring{\vect u}$ depending on $\vect u$ in a linear manner. We also assume that $\mathcal{A}$  is a non-negative self-adjoint  operator whose domain is a subset of  $L^2(\Omega;\R^3)$ such that there exist $c_1,c_2>0$ such that
\begin{equation}\label{matteo1}  	c_1\| \vect u\|^2_{s,h,\epsh}\leq (\mathcal{A} \vect u, \vect u) \leq c_2 \| \vect u\|^2_{s,h,\epsh} \quad \forall {\vect u} \in \mathcal{D} (\mathcal{A}).  
\end{equation} 

\begin{lemma} \label{nenad100} 
	Suppose that $\mathcal{A}$ is as above and let $\lambda \notin \sigma(\mathcal{A})$. Assume that $\vect u=\tilde{\vect u}+\mathring{\vect u} \in H^1_{\Gamma_{\rm D}}(\Omega;\mathbb{R}^3)$ satisfies
	\begin{equation*}
		\bigl(\mathcal{A}^{1/2} \vect u, \mathcal{A}^{1/2}\vect \xi\bigr)  -\lambda (\vect u,\vect \xi) =\int_{\Omega} \left(\vect f_1:\sym \nabla_h \tilde{\vect \xi}+ \epsh \vect f_2:\sym \nabla_h \mathring{\vect \xi}+\vect f_3\cdot\vect  \xi\right)\, dx \qquad  \forall \vect \xi=\tilde{\vect \xi}+\mathring{\vect \xi} \in  H^1_{\Gamma_{\rm D}}(\Omega;\R^3).  
	\end{equation*}
	where $\vect f_1 \in L^2(\Omega, \mathbb{R}^{3 \times 3})$, $\vect f_2 \in L^2(\Omega, \mathbb{R}^{3 \times 3})$ and $\vect f_3 \in L^2(\Omega;\R^3)$. Then one has
	\begin{equation*} 
	\| \vect u\|_{h,\epsh} \leq \frac{C(\lambda)}{\dist\bigl(\lambda, \sigma (\mathcal{A})\bigr)} \left(\| \vect f_1\|_{L^2(\Omega;\R^{3 \times 3})}+\|\vect f_2\|_{L^2(\Omega;\R^{3 \times 3})}+\|\vect f_3\|_{L^2(\Omega;\R^3)}\right), 
	%\label{refest1}
\end{equation*}
	for some $C(\lambda)$ that is bounded on bounded intervals.  
\end{lemma}
\begin{proof}
	By virtue of the Riesz representation theorem and \eqref{matteo1}, we know that there exists $\vect f \in L^2 (\mathbb{R}^3)$ and $C>0,$ which depends on $c_1$, $c_2$ only, such that
	\begin{align}
		&\nonumber\| \vect f\|_{L^2(\Omega;\R^3)}\leq C\left(\| \vect f_1\|_{L^2(\Omega;\R^{3 \times 3})}+\|\vect f_2\|_{L^2(\Omega;\R^{3 \times 3})}+\|\vect f_3\|_{L^2(\Omega;\R^3)}\right),
		\\[0.4em]
		\label{nenad3} 
		&(\mathcal{A} \vect u,\vect \xi )  -\lambda (\vect u,\vect \xi)= \int_{\Omega} \vect f\cdot\bigl(\mathcal{A}^{1/2} \vect \xi+\vect \xi\bigr)\quad\ \forall \vect \xi \in H^1_{\Gamma_{\rm D}} (\Omega;\R^3).  
	\end{align}
	We can now use the spectral theorem (see, e.g., \cite{reedsimon}): there exists a measurable space $(M, \mu)$ with a finite measure $\mu$ and a unitary operator
	$$ \mathcal{U}: L^2 (\Omega;\R^3) \to L^2(M)$$ 
	and a non-negative real-valued function $a,$ which is an element of $L^p(M)$, $ p \in [1, \infty)$,
	such that 
	\begin{itemize}
		\item  $\RRR \vect{\psi} \BBB \in \RRR \mathcal{D} \BBB(\mathcal{A}) $ if and only if $a(\cdot) \mathcal{U}\RRR \vect{\psi} \BBB (\cdot) \in L^2 (M) $; 
		\item $\RRR \vect{\psi} \BBB \in \mathcal{U}\,\RRR \mathcal{D} \BBB (\mathcal{A}) $,\ \   $\mathcal{U} \mathcal{A} \mathcal{U}^{-1} \RRR \vect{\psi} \BBB   (\cdot) = a(\cdot) \RRR \vect{\psi} \BBB (\cdot). $  
	\end{itemize}
	Notice that the second claim implies 
	\begin{itemize}
		\item $\mathcal{U} \mathcal{A}^{1/2} \mathcal{U}^{-1} \RRR \vect{\psi} \BBB  (\cdot) = \sqrt{a(\cdot)} \RRR \vect{\psi} \BBB  (\cdot). $ 
		\item $\sigma(\mathcal{A}) ={\rm EssRan} \, a:= \bigl\{ r: \forall \varepsilon>0, \ \mu\{ m \in M: r-\varepsilon \leq a(m) \leq r+\varepsilon\}>0   \bigr\}. $
	\end{itemize}
	Furthermore, \eqref{nenad3} implies 
	$$ \bigl(a(\cdot)-\lambda\bigr) \mathcal{U}\RRR \vect u \BBB (\cdot)=\mathcal{U}\RRR \vect f \BBB(\cdot) \bigl(\sqrt{a(\cdot)}+1\bigr),   $$
	from which, by virtue of 
	$\lambda \notin {\rm EssRan} \, a,$ 
	one has 
	$$  
	\bigl(\sqrt{a(\cdot)}+1\bigr)
	\mathcal{U}\RRR \vect u\BBB (\cdot) 
	=\mathcal{U}\RRR \vect f\BBB (\cdot) \frac{\bigl(\sqrt{a(\cdot)}+1\bigr)^2}
	{\bigl(a(\cdot)-\lambda\bigr)}.
	$$
Therefore, there exists $C(\lambda)>0,$ which is uniformly bounded on compact intervals of $\lambda,$ such that 
	\begin{equation}
		%\| \vect u\|_{h,\epsh} 
		\bigl\Vert(\mathcal{A}^{1/2}+\RRR \mathcal{I} \BBB)\RRR \vect{u} \BBB \bigr\Vert_{L_2}\leq \frac{C(\lambda)}{\dist\bigl(\lambda, \sigma (\mathcal{A})\bigr)}\| \vect f\|_{L^2},
		%%	\bigl\Vert(\mathcal{A}^{1/2}+I)u\bigr\Vert\le C{\rm dist}\bigl(\sigma({\mathcal A})\bigr),
		\label{refest}
	\end{equation} 	
	from which the claim follows immediately.
\end{proof}
	\begin{remark} 
		\label{ivan501} 
		Throughout the paper, we also use some variants of the above lemma, see the discussions around (\ref{ref1}), (\ref{ref2}), (\ref{ref3}), (\ref{ref4}). They generically apply to setups that can be put the form (\ref{nenad3}), and they result in estimates of the type (\ref{refest}).   The key ingredient for their validity is the fact that the right-hand side of the equation is in the dual of $\mathcal{D}(\mathcal{A}^{1/2}+\RRR \mathcal{I} \BBB)$ with respect to the graph norm. 
		
	\end{remark}

\end{document}